\documentclass{amsart}
\usepackage{mathtools}
\usepackage{euscript}
\usepackage{amsmath}
\usepackage{amscd}
\usepackage{amssymb}
\usepackage{enumerate}
\usepackage{stmaryrd}
 \usepackage{amsfonts}

\newcommand{\mw}{14.17}
\newlength{\lw}
\usepackage{setspace}
\usepackage{array}

\usepackage{tikz}
\usetikzlibrary{matrix,arrows,calc,intersections,fit}
\usepackage{tikz-cd}

\newcommand{\bC}{{\mathbb C}}

\newcommand{\bR}{{\mathbb R}}
\newcommand{\bT}{{\mathbb T}}
\newcommand{\bZ}{{\mathbb Z}}

\newcommand{\cA}{\mathcal A}

\newcommand{\cC}{\mathcal C}

\newcommand{\cI}{\mathcal I}
\newcommand{\Ibar}{\overline{\cI}}

\newcommand{\cS}{\mathcal S}
\newcommand{\cN}{\mathcal N}
\newcommand{\cM}{\mathcal M}
\newcommand{\cP}{\mathcal P}

\newcommand{\cL}{\mathcal L}
\newcommand{\cR}{\mathcal R}
\newcommand{\cT}{\mathcal T}

\newcommand{\scrS}{\EuScript S}
\newcommand{\scrC}{\EuScript C}

\newcommand{\scrF}{\EuScript F}

\newcommand{\scrL}{\EuScript L}
\newcommand{\scrM}{\EuScript M}
\newcommand{\scrN}{\EuScript N}
\newcommand{\scrO}{\EuScript O}

\newcommand{\scrR}{\EuScript R}
\newcommand{\scrV}{\EuScript V}

\newcommand{\scrY}{\EuScript Y}

\newcommand{\Q}{Q}

\newcommand{\Rbar}{\overline{\cR}}

\newcommand{\Tbar}{\overline{\cT}}
\newcommand{\RTbar}{\overline{\cT\cR}}

\newcommand{\Sbar}{\overline{\cS}}

\newcommand{\ext}{\mathrm{ext}}
\newcommand{\triv}{\iota}
\newcommand{\fin}{\mathrm{int}}
\newcommand{\inp}{\mathrm{inp}}
\newcommand{\out}{\mathrm{out}}

\newcommand{\ro}{{\mathrm o}}

\newcommand{\scrJ}{{\EuScript J}}

\newcommand{\bfk}{\mathbf{k}}

\newcommand{\basepoint}{\sigma}

\newcommand{\codim}{\operatorname{codim}}

\newcommand{\id}{\operatorname{id}}

\renewcommand{\mod}{\operatorname{mod}}

\newcommand{\Fuk}{\scrF}
\newcommand{\HFuk}{\mathrm{H}\Fuk}

\newcommand{\Hom}{\operatorname{Hom}}

\newcommand{\Lab}{\Upsilon}
\newcommand{\uLab}{\underline{\Upsilon}}

\newcommand{\Pin}{\operatorname{Pin}}

\newcommand{\tr}{\mathrm{tr}}

\renewcommand{\det}{\operatorname{det}}

\newcommand{\val}{\operatorname{val}}

\renewcommand{\min}{\operatornamewithlimits{min}}
\renewcommand{\max}{\operatornamewithlimits{max}}
\newcommand{\diam}{\operatorname{diam}}
\def\co{\colon\thinspace}

\newcommand{\Cone}{\operatorname{Cone}}
\newcommand{\Coeq}{\operatorname{Coeq}}
\newcommand{\Crit}{\operatorname{Crit}}
\newcommand{\Poly}{\operatorname{Po}}
\newcommand{\HPoly}{\operatorname{HPo}}

\newcommand{\Ham}{\operatorname{Ham}}

\newcommand{\Delbar}{\overline{\Delta}}

\newcommand{\Fl}{\mathrm{Fl}}
\newcommand{\Mo}{\mathrm{Mo}}
\newcommand{\kk}{}

\numberwithin{equation}{section}

\newtheorem{thm}{Theorem}[section]
\newtheorem{cor}[thm]{Corollary}
\newtheorem{lem}[thm]{Lemma}
\newtheorem{prop}[thm]{Proposition}
\newtheorem{defin}[thm]{Definition}
\newtheorem{def-lem}[thm]{Definition-Lemma}

\theoremstyle{remark}

\newtheorem{rem}[thm]{Remark}

\newcommand{\superscript}[1]{\ensuremath{^{\textrm{#1}}} }

\renewcommand{\th}[0]{\superscript{th}}

\newcommand{\comment}[1]{}
\begin{document}
\title[HMS without correction]{Homological mirror symmetry without correction}

\author[M.~Abouzaid]{Mohammed Abouzaid}
\address{Columbia University, 2990 Broadway Ave, New York, NY, 10027, USA}
\email{abouzaid@math.columbia.edu}
 \thanks{The author was supported by NSF grants DMS-1308179,  DMS-1609148, and DMS-1564172, the Simons Foundation through its ``Homological Mirror Symmetry'' Collaboration grant, the Erik Ellentuck Fellowship, and the IAS Fund of Math.}
 \date{\today}

\begin{abstract}
  Let $X$ be a closed symplectic manifold equipped a Lagrangian torus fibration over a base $Q$. A construction first considered by Kontsevich and Soibelman produces from this data a rigid analytic space $Y$, which can be considered as a variant of the $T$-dual introduced by Strominger, Yau, and Zaslow. We prove that the Fukaya category of tautologically unobstructed graded Lagrangians in $X$ embeds fully faithfully in the derived category of (twisted) coherent sheaves on $Y$, under the technical assumption that $\pi_2(Q)$ vanishes (all known examples satisfy this assumption). The main new tool is the construction and computation of Floer cohomology groups of Lagrangian fibres equipped with topological infinite rank local systems that correspond, under mirror symmetry, to the affinoid rings introduced by Tate, equipped with their natural topologies as Banach algebras. 
\end{abstract}
\maketitle
\tableofcontents

\section{Introduction}
\label{sec:introduction}

\subsection{Statement of the main result}
\label{sec:cont-stat-main}

Let $X$ be a closed symplectic manifold. One of the key tools used to understand the symplectic topology of $X$ is its Fukaya category, placing the computation of this category as a central problem in the subject. Such computations would ideally rely on geometric features of $X$. For example, the fact that cotangent bundles are fibered in Lagrangian planes ultimately accounts for the computation of their (derived) Fukaya categories as categories of modules over the chains on the based loop space of the base \cite{Abouzaid2011}.

The case of a symplectic manifold equipped with a Lagrangian torus fibration has been the focus of much interest because of its relevance to mirror symmetry via the Strominger-Yau-Zaslow conjecture \cite{StromingerYauZaslow1996}. Kontsevich and Soibelman \cite{KontsevichSoibelman2001} were the first to propose a mathematically precise conjecture in this context: under the assumption that the torus fibration admits a Lagrangian section, they showed that one can associate a rigid analytic space $Y$ (in the sense of Tate \cite{Tate1971}) to every such symplectic manifold, and conjectured that the Fukaya category is equivalent to the category of (rigid analytic) coherent sheaves on $Y$. They also took the first step in this direction by assigning a mirror line bundle on $Y$ to each Lagrangian section of $X$, and comparing the multiplication on the Floer cohomology groups of Lagrangian sections with the composition of (derived) maps between the mirror line bundles.

The next step was taken by  Fukaya in  \cite{Fukaya2010}, showing in complete generality that the Floer cohomology of Lagrangian submanifolds varies \emph{analytically} with respect to the natural local coordinates on the space of Lagrangians modulo Hamiltonian isotopy coming from the flux homomorphism. Fukaya's work was phrased in terms of the self-Floer cohomology of Lagrangians, but a minor adaptation shows that one can assign to an object $L$ of the Fukaya category of $X$ a \emph{sheaf $\cL_L$ of coherent analytic complexes} on the mirror space \cite{Abouzaid2014} (the notation $\cL$ refers to the fact that this complex will arise for us as a \emph{left module} over a category associated to a cover of the mirror).

To avoid technical difficulties, we shall from now on restrict attention to Lagrangian torus fibrations whose bases have vanishing second homotopy group, and consider only the subcategory of the Fukaya category consisting of \emph{tautologically unobstructed Lagrangians,} i.e. those for which there is a choice of almost complex structure so that all the holomorphic discs which they bound are constant. Since the homotopy groups of the base of a fibration are isomorphic to the relative homotopy groups of the inclusion of each fibre in the total space, the first condition implies that $\pi_2(X,X_q)$ vanishes for each fibre $X_q$, hence that any holomorphic disc bounded by such fibres is constant. Since the main constructions of this paper rely on studying the Floer theory of such torus fibres, this condition implies that we may construct Floer cochain complexes without having to choose bounding cochains as in the work of Fukaya, Oh, Ohta, and Ono \cite{FukayaOhOhtaOno2009}.

In this context,  the family Floer functor was shown to be faithful in \cite{Abouzaid2014a}. More precisely, associated to a Lagrangian torus fibration is a gerbe $\beta$ on $Y$, which is classified by an element of $H^2(Y, \scrO^*)$. An $A_\infty$ functor from the Fukaya category of $X$ to the derived category of  $\beta$-twisted coherent sheaves on $Y$ was constructed, and the corresponding family Floer map
\begin{equation} \label{eq:family_Floer}
HF^*(L, L') \to H^*(Y; \Hom_{\scrO_Y}( \cL_{L'}, \cL_{L}))  
\end{equation}
was  proved to be injective for every pair of Lagrangians. Here, and in contrast to the rest of the paper, we use the notation $ H^*(Y;  \_)$ to indicate the cohomology of $Y$ with coefficients in a sheaf, noting that the morphisms between $\beta$-twisted sheaves form an ordinary sheaf on $Y$. In order to state the main result of this paper, we observe that the structure of a Lagrangian torus fibration equips the total space $X$ with a canonical grading in the sense of Seidel \cite{Seidel2000}, which allows us to restrict our attention to graded Lagrangians, i.e. those which define objects of the $\bZ$-graded Fukaya category: 
\begin{thm}
  \label{thm:main_thm}
If $X$ is a closed symplectic manifold equipped with a Lagrangian torus fibration $X \to Q$ whose base has vanishing second homotopy group, then the family Floer map in Equation \eqref{eq:family_Floer} is surjective for graded Lagrangians.
\end{thm}
Using the main result of \cite{Abouzaid2014a}, which asserts the injectivity of the Family Floer map, we conclude:
\begin{cor} \label{cor:fully_faithful_embedding}
There is a fully faithful embedding of the  Fukaya category of tautologically unobstructed and graded Lagrangians in $X$ in the $\beta$-twisted derived category of coherent sheaves on $Y$. \qed
\end{cor}

\begin{rem}
The restriction to graded Lagrangians is required for technical reasons, in particular in Section \ref{sec:perf-left-modul}. We expect that it can be eliminated at the cost of more carefully using the natural filtration of operations in Floer theory by energy, and of the structure sheaf of $Y$ by the valuation of functions.
\end{rem}
The main deficiencies of this result are its restrictive assumptions that (1) all Lagrangians are tautologically unobstructed (and that $\pi_2(Q)=0$), and (2) that the ambient symplectic manifold is equipped with a \emph{non-singular} fibration. Removing the first assumption would require the use of a package of virtual fundamental chains; the one developed by Fukaya-Oh-Ohta-Ono \cite{FukayaOhOhtaOno2009} would be sufficient for the task at hand. The decision not to use it to prove a theorem for general Lagrangians amounts to the desire not to add another layer of complexity to the paper.  On the other hand, admitting singular Lagrangians as fibres, e.g. immersed Lagrangians, will require some new insights about Floer theory in families, though the first steps have been taken by Fukaya in his announced results about mirror symmetry for $K3$ surfaces.

\begin{rem} \label{rem:only_need_linear}
  The construction of the family Floer functor, which is implemented in the Appendix of \cite{Abouzaid2014a}, goes beyond the construction of the chain level analogue of Equation \eqref{eq:family_Floer}, and requires a comparison between composition in the two categories.
Since the Fukaya category is an $A_\infty$ category, there is a tower of maps, generalising the chain level analogue of Equation \eqref{eq:family_Floer}, which are required to specify the functor, but these are not discussed in this paper (except in Appendix \ref{sec:invar-family-floer}) since the assertion that a functor is fully faithful is a statement only about Equation \eqref{eq:family_Floer} being an isomorphism. 
\end{rem}

\begin{rem} \label{rem:twisted_categories}
  A similar result holds if we twist the Fukaya category of $X$ with a background or bulk class arising from $Q$: letting $\Lambda_+$ denote the elements of the Novikov field with strictly positive valuation (see Equation \eqref{eq:valuation_Novikov}), each class $\alpha \in H^2(X; \bZ_2) \oplus H^2(X; \Lambda_+) $ determines a deformation of the Fukaya category of $X$. If this class is obtained by pullback from $Q$, it lifts to a class in $H^2(Y; \scrO^*)$, which we also denote $\alpha$. The $\alpha$-twisted Fukaya category of $X$ can be shown by our methods to embed fully faithfully in the $\alpha + \beta$ twist of the derived category of coherent sheaves on $Y$.

  Our methods are currently insufficient to address the case of more general twists: these give rise to locally non-commutative deformations of the rings of functions on $Y$, and we do not currently know how to formulate some of the required algebraic results in this context (in particular, Tate's acyclicity theorem as discussed in Appendix \ref{sec:null-homotopy-tates}).
\end{rem}

\subsection{Affinoid rings and local systems}
\label{sec:affinoid-rings-local}

In order to introduce the new ideas in this paper, we begin by explaining its background in more detail: for specificity, we fix a field   $\bfk$, and denote by $\Lambda$ the corresponding  Novikov field consisting of series
\begin{equation}
  \sum_{\lambda \in \bR} c_\lambda T^{\lambda}, \, c_\lambda \in \bfk
\end{equation}
with the property that the set of exponents $\lambda$ for which the coefficients $c_\lambda$ do not vanish is discrete and bounded below. The Fukaya category that we consider will be defined over $\Lambda$, in the sense that morphism are $\Lambda$-vector spaces and all operations are linear over $\Lambda$.  The Novikov ring admits a valuation 
\begin{equation} \label{eq:valuation_Novikov}
\val(  \sum_{\lambda \in \bR} c_\lambda T^{\lambda}) =  \lambda_{0}
\end{equation}
where $\lambda_{0}$ is the smallest exponent  whose coefficient does not vanish (we set $\val 0 = + \infty$). %

The points of the space $Y$ are in bijective correspondence with (isomorphism classes) of simple objects of the Fukaya category of $X$, supported on the fibres of the projection to $Q$. As discussed above, our assumption that $\pi_2(Q)$ vanishes implies that the fibres do not bound any non-constant holomorphic disc, and hence that such isomorphism classes naturally correspond to isomorphism classes of unitary rank-$1$ local systems on the fibres. 

In order to describe $Y$ as an analytic space, one must describe a topology and a sheaf of functions. The simplest description is to observe that the base of a Lagrangian torus fibration is equipped with an integral affine structure; in particular $Q$ is obtained by gluing integral affine polytopes (i.e. bounded subsets of $\bR^n$ defined by affine equations whose linear terms are integral), along integral affine transformations of $\bR^n$ (i.e. compositions of translations and elements of $GL(n,\bZ)$). In much of this paper, we shall express such a decomposition by considering a partially ordered set $\Sigma$ whose elements index a polyhedral cover $\{ P_\sigma\}_{\sigma \in \Sigma}$ of $Q$.  Under the valuation map from $(\Lambda \setminus \{0\})^n$ to $\bR^n$, the inverse image of each such polytope $P_\sigma \subset \bR^n$ is an affinoid domain $Y_\sigma$ in the sense of Tate \cite{Tate1971}. Since the valuation map intertwines the operations
\begin{align}
  (z_{1}, \ldots, z_n) \mapsto & (T^{\lambda_{1}} \prod z_{j}^{a_{1j}}, \ldots,   T^{\lambda_n} \prod z_{j}^{a_{nj}}) \\
  (u_{1}, \ldots, u_n) \mapsto & (\lambda_{1} +  \sum a_{1j} u_{j}, \ldots,   \lambda_n + \sum a_{nj} u_{j}) 
\end{align}
on  $(\Lambda \setminus \{0\})^n$ and $\bR^n$, which are associated to each collection $(\lambda_{1}, \ldots, \lambda_n)$ of real numbers and integral matrix $(a_{ij})$, the gluing of polytopes which yields $Q$ lifts to a gluing of the affinoid domains $\{ Y_\sigma\}_{\sigma \in \Sigma}$  which yields a rigid analytic space which we denote $Y$, and this space is independent of the choice of cover. By construction, the space $Y$ is equipped with a projection map to $Q$, and the fibres are easily seen, via the above description, to correspond to unitary local systems on the fibres of the Lagrangian torus fibration over $Q$.

The starting point of the results of \cite{Abouzaid2014a} is that affinoid rings are related to local systems on torus fibres as follows: picking a basepoint $q_\sigma \in P_\sigma$, we can identify the ring $\Gamma^{P_\sigma}$ of rigid analytic functions on $Y_\sigma$ as a completion of the group ring of the fundamental group of $X_{q_\sigma}$, which is isomorphic to $H_{1}(X_{q_\sigma}; \bZ)$ because $X_{q_\sigma}$ is a torus. The basic idea is that $P_{\sigma}$ can be identified with a polytope in $H^1(X_{q_\sigma}; \bR)$, so it makes sense to complete the group ring of $ H_{1}(X_{q_\sigma}; \bZ)$ by allowing infinite series with the property that norm of the pairing of the terms in the series with all elements of $P_{\sigma}$ converges to $0$ (i.e. the valuation goes to $+\infty$). The relationship with the familiar idea of studying Floer theory with coefficients in rank-$1$ local systems is that elements of $H_{1}(X_{q_\sigma}; \bZ) $ define regular functions on the space of rank-$1$ local systems (over the Novikov ring). The set of rank-$1$ local systems is parametrised by the first cohomology with coefficients in $\Lambda \setminus \{0\}$, and the valuation defines a natural map
\begin{equation}
 H^1(X_{q_\sigma}; \Lambda \setminus \{0\}) \to  H^1(X_{q_\sigma}; \bR).
\end{equation}
The completion associated to the polytope $P_\sigma$ considered as a subset of $H^1(X_{q_\sigma}; \bR) $ thus corresponds to the ring of analytic functions on the inverse image in $ H^1(X_{q_\sigma}; \Lambda \setminus \{0\}) $. We thus conclude that the spaces $Y_\sigma$ can be thought of as spaces of rank-$1$ local systems with constraints on their valuation.

\subsection{From Lagrangians to coherent sheaves}
\label{sec:from-lagr-coher}

In order to prove Homological mirror symmetry, we must specify a model for the derived category of $Y$. As in \cite{Abouzaid2014a} we shall use the affinoid covering of $Y$ to express this category in terms of a \v{C}ech model. Indeed, a classical result of Kiehl \cite{Kiehl1967} implies that the derived category of coherent sheaves on each affinoid domain is equivalent to the derived category of finitely presented modules over the corresponding ring of functions which we denoted $\Gamma^{P_\sigma}$, so that derived morphisms between sheaves can computed as \v{C}ech cohomology groups, as is done for example in \cite{Bosch-Guntzer-Remmert1984}. In particular the \v{C}ech cohomology groups are independent of the choice of affinoid cover.

We shall use the differential graded model of the derived category of coherent sheaves on $Y$ given by a subcategory of the category of $A_\infty$-modules over a category $\Fuk$ with objects elements of the cover, with morphisms from $\tau$ to $\sigma$ vanishing unless $\tau$ precedes $\sigma$ in the partial order, in which case they are given by the affinoid ring $\Gamma^{P_\tau}$: the subcategory of modules over $\Fuk$ that corresponds to coherent sheaves consists of those $A_\infty$ modules whose associated cohomological modules assign to each element $\sigma$ of the cover a finitely generated graded module over $\Gamma^{P_\sigma}$, and for which the restriction map associated to an arrow $\tau \to \sigma$ in $\Fuk$ induces an isomorphism on cohomology after tensoring over $\Gamma^{P_\tau}$ with $\Gamma^{P_\sigma}$.  In the most basic case of a space covered by two affinoid domains, this encodes the idea that a (coherent) sheaf on the ambient space can be thought of as a pair of modules over the rings of functions on the two elements of the cover, a module over the ring of functions on the intersection, together with an isomorphism between the module associated to the intersection and the restrictions of the modules associated to the two elements of the cover. This is equivalent, though less economical, than the data of the two modules together with an isomorphism between their restrictions, but it proves to be much more convenient when formulating the corresponding notion at the level of cochain complexes.

Returning to the point of view that $\Gamma^{P_\sigma}$ is a completion of the group ring of the fundamental group of  $X_{q_\sigma}$, the usual correspondence between modules over the group ring and local systems associates to $\sigma$ a local system $U^\sigma$ on the Lagrangian $X_{q_\sigma}$, which yields an object of an enlargement of the Fukaya category in which infinite rank local systems are allowed. We can then associate to each (tautologically unobstructed) Lagrangian $L$, and each element of the cover a cochain complex
\begin{equation}
  \cL_{L}(\sigma) \coloneqq CF^*(L, (X_{q_\sigma}, U^\sigma))
\end{equation}
via Floer theory. This complex admits a natural action (on the right) by the ring $\Gamma^{P_\sigma}$, and the collection of all such cochain complexes underlies the left module $\cL_L$ over $\Fuk$, which is the image of $L$ under the Family Floer functor. In particular, given an arrow $\sigma \to \tau$, we have a restriction map
\begin{equation}
\cL_{L}(\tau) \to \cL_{L}(\sigma)
\end{equation}
induced by the map of local systems $U^\sigma \to U^\tau $.

Having introduced the functor at the level of objects, we now consider morphisms: given a pair $(L,L')$ of Lagrangians, the pair of pants product defines a map
\begin{equation}
  CF^*\left( L', (X_{q_\sigma}, U^\sigma) \right)  \otimes_\Lambda  CF^*(L, L')  \to CF^*\left(L, (X_{q_\sigma}, U^\sigma) \right).     
\end{equation}
Identifying the Floer cochains with coefficients in $ U^\sigma $  with the values of the modules $\cL_{L}$ and $\cL_{L'}$ at $\sigma$, this defines a $\Lambda$-linear map
\begin{equation}
  CF^*(L, L') \to \Hom_{\Lambda}( \cL_{L}(\sigma) , \cL_{L'}(\sigma)).
\end{equation}
The usual arguments for associativity of multiplication in Floer theory imply that this map is naturally compatible with the above restriction map (and the module action by the affinoid rings associated to our cover), yielding a map
\begin{equation} \label{eq:cochain_map_CF_to_left_module}
   CF^*(L, L') \to \Hom_{\Fuk}( \cL_{L} , \cL_{L'}),
\end{equation}
which is the cochain level model that we use for Equation (\ref{eq:family_Floer}). As discussed in Remark \ref{rem:only_need_linear}, one can go further and construct and $A_\infty$ functor from the Fukaya category to the category of modules over $\Fuk$, of which this is the first order map, but this extension will only be enter our discussion in Appendix \ref{sec:invar-family-floer}.

\subsection{A summary of the proof}
\label{sec:new-ideas}

The proof of Theorem \ref{thm:main_thm} is based on the following idea: in addition to the left module $\cL_L$ considered above, Floer theory associates to each Lagrangian $L$ a right module $\scrR_{L}$ over $\Fuk $, given by
\begin{equation}
    \cR_{L}(\sigma) \coloneqq CF^*((X_{q_\sigma}, U^\sigma), L).  
\end{equation}
There is a natural map, which was essentially already used in \cite{Abouzaid2014a}, from the derived tensor product of these modules to the Lagrangian Floer cochains
\begin{equation} \label{eq:cochain_map_tensor_product_to_CF}
 \scrR_{L'} \otimes_{\Fuk} \cL_{L} \to CF^*(L, L').
\end{equation}
The surjectivity of Equation (\ref{eq:family_Floer}) can thus be deduced from the sujectivity of the composition
\begin{equation} \label{eq:composition_through_HF}
 \scrR_{L'} \otimes_{\Fuk} \cL_{L} \to CF^*(L, L') \to \Hom_{\Fuk}(\cL_{L'}, \cL_{L})
\end{equation}
at the level of cohomology; in fact, we shall show that this composition is a quasi-isomorphism.

The holomorphic discs giving rise to this composition are illustrated on the right of Figure \ref{fig:homotopy_compositions_labels_category}, where the label by polytopes indicates that we are considering fibres equipped with the corresponding local systems, the arrows associated to inputs point towards the interior of the disc, and the outward pointing arrow corresponds to the output. Each disc is labelled by the corresponding map or diagram. 
\begin{figure}[h]
  \centering
\begin{tikzpicture}
\newcommand*{\bigradius}{1.5}
\newcommand*{\tinyradius}{.05}
\node at  (0,0) {\eqref{eq:cochain_diagram_modules}};
\coordinate [label=right:$L$] (L) at  (0:\bigradius);
\coordinate [label=left:$P_\sigma$] (P) at  (180:\bigradius);
\coordinate [label=above:$L'$] (L') at  (100:\bigradius);
\coordinate [label=above:$P_\tau$] (P-) at  (60:\bigradius);

\draw (0,0)  ([shift=(-88:\bigradius)]0,0) arc (-88:48:\bigradius);
\draw[->] (-90:\bigradius-\tinyradius) -- (-90:\bigradius+3*\tinyradius);
\draw ([shift=(122:\bigradius)]0,0) arc (122:268:\bigradius);
\draw[->] (50:\bigradius+\tinyradius) -- (50:\bigradius-3*\tinyradius);
\draw ([shift=(52:\bigradius)]0,0) arc (52:78:\bigradius);
\draw[->] (80:\bigradius+\tinyradius) -- (80:\bigradius-3*\tinyradius);
\draw ([shift=(82:\bigradius)]0,0) arc (82:118:\bigradius);
\draw[->] (120:\bigradius+\tinyradius) -- (120:\bigradius-3*\tinyradius);
\newcommand*{\radius}{1}
\begin{scope}[shift={(-4,-.5)}]
\coordinate [label=right:$L$] (L) at  (-10:\radius);
\coordinate [label=left:$P_\sigma$] (P) at  (190:\radius);
\coordinate [label=above:$P_\tau$] (P-) at  (90:\radius-\tinyradius);
\coordinate [label=above:$L'$] (L') at  (150:3*\radius);
\node at  (0,0) {\eqref{eq:map_tensor_perturbed_diagonal_left}};
\draw (0,0)  ([shift=(-88:\radius)]0,0) arc (-88:28:\radius);
\draw[->] (-90:\radius-\tinyradius) -- (-90:\radius+3*\tinyradius);
\draw ([shift=(152:\radius)]0,0) arc (152:268:\radius);
\draw[->] (30:\radius+\tinyradius) -- (30:\radius-3*\tinyradius);
\draw ([shift=(32:\radius)]0,0) arc (32:148:\radius);
\draw[->] (150:\radius+\tinyradius) -- (150:\radius-3*\tinyradius);
\begin{scope}[shift={(150:2*\radius)}, rotate=60]
 \node at  (0,0) {\eqref{eq:map_right_module_Hom_left}};
\draw (0,0)  ([shift=(-88:\radius)]0,0) arc (-88:28:\radius);
\draw ([shift=(152:\radius)]0,0) arc (152:268:\radius);
\draw[->] (30:\radius+\tinyradius) -- (30:\radius-3*\tinyradius);
\draw ([shift=(32:\radius)]0,0) arc (32:148:\radius);
\draw[->] (150:\radius+\tinyradius) -- (150:\radius-3*\tinyradius);
\end{scope}
 \end{scope}

\begin{scope}[shift={(4,-.5)}]
\coordinate [label=right:$L$] (L) at  (-10:\radius);
\coordinate [label=left:$P_\sigma$] (P) at  (190:\radius);
\coordinate [label=above:$P_\tau$] (P-) at  (30:3.1*\radius);
\coordinate [label=above:$L'$] (L') at  (90:\radius);
\node at  (0,0) {\eqref{eq:cochain_map_CF_to_left_module}};
\draw (0,0)  ([shift=(-88:\radius)]0,0) arc (-88:28:\radius);
\draw[->] (-90:\radius-\tinyradius) -- (-90:\radius+3*\tinyradius);
\draw ([shift=(152:\radius)]0,0) arc (152:268:\radius);
\draw[->] (30:\radius+\tinyradius) -- (30:\radius-3*\tinyradius);
\draw ([shift=(32:\radius)]0,0) arc (32:148:\radius);
\draw[->] (150:\radius+\tinyradius) -- (150:\radius-3*\tinyradius);
\begin{scope}[shift={(30:2*\radius)}, rotate=-60]
 \node at  (0,0) {\eqref{eq:cochain_map_tensor_product_to_CF}};
\draw (0,0)  ([shift=(-88:\radius)]0,0) arc (-88:28:\radius);
\draw ([shift=(152:\radius)]0,0) arc (152:268:\radius);
\draw[->] (30:\radius+\tinyradius) -- (30:\radius-3*\tinyradius);
\draw ([shift=(32:\radius)]0,0) arc (32:148:\radius);
\draw[->] (150:\radius+\tinyradius) -- (150:\radius-3*\tinyradius);
\end{scope}
 \end{scope}

\end{tikzpicture}
  \caption{An informal representation of the moduli spaces giving rise to two compositions in Diagram \eqref{eq:cochain_diagram_modules}, together with the moduli space giving rise to the homotopy between them. The arrows indicate inputs (pointing towards the disc), or outputs.}
  \label{fig:homotopy_compositions_labels_category}
\end{figure}
One should interpret Figure \ref{fig:homotopy_compositions_labels_category} as illustrating a cobordism between the moduli spaces of discs used to define the two maps we are considering and a moduli space of discs which we shall use to prove that the composition is an isomorphism. The essential problem is to give a meaning, both in  geometry and algebra, to the picture shown on the left. At the level of geometry, the idea used in \cite{Abouzaid2014a} of considering families of fibres over the two polytopes runs into a transversality problem, because two fibres intersect if and only if they are equal, and it seems difficult to arrange for such a parametrised problem (in which the intersections of Lagrangians change), to give moduli spaces of the correct dimension.

As is standard in Floer theory, we resolve this issue by introducing Hamiltonian perturbations, and associate to the node in Figure \ref{fig:homotopy_compositions_labels_category} with label $(P_\tau, P_\sigma)$ the set of intersections between a fibre over $P_\tau$, and the image of a fibre over $P_\sigma$ under a Hamiltonian isotopy. This leads to a new problem, which is responsible for the first real innovation of this paper: we need the Floer cohomology groups
\begin{equation} \label{eq:isomorphism_affinoid_ring_Floer_HF}
  HF^*((X_{q_\tau}, U^\tau), (X_{q_\sigma}, U^\sigma))   , 
\end{equation}
associated to a perturbed pair to have appropriate meaning under mirror symmetry. At the least, we need to know that the perturbed Floer cohomology of the pair $P_\sigma = P_\tau$ is isomorphic to the affinoid ring associated to the corresponding subset of $Y$. This will require extending Floer theory to a setting in which Lagrangians are equipped with topologised local systems.

\subsection{Topological local systems in Floer theory}
\label{sec:topol-local-syst}

The reason that the standard definition of Floer cohomology with local coefficients, as discussed e.g. in \cite{Abouzaid2012a} in the infinite rank case, does not give an adequate definition of the group in Equation \eqref{eq:isomorphism_affinoid_ring_Floer_HF} is that this homology group is far larger than the desired one. This is analogous to the statement that, if we consider a power series ring $\bfk[[x]]$ as a module over the polynomial ring $\bfk[x]$ in the natural way, then the map
\begin{equation}
  \bfk[[x]] \to \Hom_{\bfk[x]}(\bfk[[x]], \bfk[[x]])  
\end{equation}
is far from being an isomorphism. A solution is offered by the fact that a power series ring acquires a natural topology from its description as an inverse limit, with respect to which it is complete. If we consider instead \emph{continuous} morphisms, then the above map becomes an isomorphism.
\begin{rem}
While the foundations of rigid analytic geometry, following Tate \cite{Tate1971}, relies upon equipping rings of affinoid functions with a topology, subsequent developments in the subject largely avoid using topological methods after the first initial steps. The approach we shall therefore take goes against the standard framework in rigid analytic geometry, though it ties with the recent introduction of completions with respect to the action filtration in Floer theory \cite{Seidel2012b,Groman2015,Venkatesh2018,Varolgunes2018}. 
\end{rem}

We are thus led to consider the Floer theory of Lagrangians equipped with \emph{topological} local systems, taking the topologies into account when defining Floer cohomology groups; the reader may find  Section \ref{sec:constr-oper}, where we consider the simplified setting of Morse theory, to be a good place to first read about how one precisely takes the topology into account. The idea of incorporating topological vector spaces in the study of Floer theory probably goes back to Fukaya \cite{Fukaya1998}, who intended to use it to study Lagrangian foliations. We do not develop the general theory, limiting ourselves to those properties required for the proof of the main theorem. Among the indications that this approach gives the right answer is that we succeed in proving that the self-Floer cohomology of a Lagrangian with such a local system can be computed using either a Morse-theoretic model or appropriate, though not all, Hamiltonian perturbations (see Section \ref{sec:computing-bimodule-nearby}).

The fact that the Floer theory that we consider is not invariant under all Hamiltonian isotopies will be essential to the success of our approach to proving Theorem \ref{thm:main_thm}, since it will allow us to produce an isomorphism in Equation \eqref{eq:isomorphism_affinoid_ring_Floer_HF} even when the fibres are distinct, as long as they are sufficiently close, and the polytopes are sufficiently small. To understand the underlying ideas, recall that a basic consequence of Gromov's compactness theorem is that Floer homology for Lagrangians equipped with unitary local systems is invariant under Hamiltonian isotopies, but that Lagrangians equipped with non-unitary local systems do not have well-defined Floer theory because the sums required to define the differential need not converge. Fukaya observed in \cite{Fukaya2010} that, if the local systems are sufficiently close to being unitary, then in fact the sums do converge.

Unfortunately, Fukaya's indirect approach, which uses the openness of the space of tame almost complex structures, does not seem to be adapted to the problem we face. The solution we adopt uses instead the reverse isoperimetric inequality of Groman and Solomon \cite{GromanSolomon2014} (with a simplified proof by DuVal \cite{Duval2016}). In its simplest form, it states:
\begin{thm} \label{thm:reverse_isoperimetric}
  If $L \subset M$ is an embedded Lagrangian, and $J$ is a tame almost complex structure on $M$, then there exists a constant $C$ such that, for each holomorphic map $u \co D^2 \to M$ with boundary mapping to $L$, the length  $\ell(\partial u)$ of the boundary curve and the energy  $E(u)$  of the map satisfy the inequality 
  \begin{equation}
    \ell(\partial u) \leq C E(u).
  \end{equation}
 \qed
\end{thm}
To understand the relationship with Floer theory, we let $[\partial u] \in H_{1}(L; \bZ)$ denote the homology class of the boundary curve, and note that, if we consider a norm on $H_{1}(L; \bZ)$ so that the norm of a homology class is smaller than the minimal length of a representative, then the above reverse isoperimetric inequality implies
\begin{equation} \label{eq:isoperimetric_bound_homology}
  | [\partial u]| \leq C E(u).  
\end{equation}
The contribution of any such curve to a Floer theoretic operation with respect to a rank-$1$ local system classified by $\alpha \in H^1(L, \Lambda \setminus \{0\})$ has valuation given by
\begin{equation}
  E(u) + \langle \val \alpha, [ \partial u] \rangle.  
\end{equation}
Assuming that the valuation of $\alpha$ is bounded by $1/2C$, we thus find that the valuation of this contribution is bounded by $E(u)/2$. Gromov compactness then implies that the sum of the contributions of all holomorphic curves converges (in the adic topology), and hence that the Floer homology of $L$ equipped with such a local system is well defined, for the given choice of almost complex structure. Passing from points to affinoid domains, the same argument implies that Floer theory with coefficients in a local system associated to a polytope $P \subset H^1(L; \bR)$ is well-defined whenever such a polytope lies within the ball of radius $1/2C$ about the origin.

In this paper, we shall be concerned not just with the Floer theory of a single Lagrangian $L$, but also with the Floer theory of the pair $(L,X_q)$ for a point $q \in Q$, and more generally with operations involving $L$ together with multiple fibres. For any such operation, there is a variant of the reverse isoperimetric inequality with multiple Lagrangian boundary conditions which specifies a maximal diameter of polytopes which give rise to convergent Floer operation. Since there are infinitely many fibres, and infinitely many operations, there may a priori be no constant which is sufficiently large so that all the desired operations converge. In Section \ref{sec:famil-cont-equat}, we ensure that this problem does not arise by constructing a compact parameter space for all the choices used in the paper, which allows us to establish a reverse isoperimetric inequality in this context using the results of Appendix \ref{sec:geometric-setup}.  %

\subsection{The geometric diagonal of rigid analytic spaces}
\label{sec:geom-diag-rigid}

Having introduced the relevant Floer groups, we would like to interpret the picture on the left of Figure \ref{fig:homotopy_compositions_labels_category} as a factorisation
\begin{equation} \label{eq:factorisation_in_the_language_of_sheaves}
 \scrR_{L'} \otimes_{\Fuk} \cL_{L} \to \Hom_{\Fuk}( \cL_{L'}, \Delbar) \otimes_{\Fuk} \cL_{L} \to \Hom_{\Fuk}(\cL_{L'}, \cL_{L}) ,
\end{equation}
where $\Delbar$  is a bimodule over $\Fuk$, which is equipped with a pair of module maps
  \begin{align} \label{eq:map_right_module_Hom_left}
    \scrR_{L'} & \to \Hom_{\Fuk}(\cL_{L'},\Delbar) \\ \label{eq:map_tensor_perturbed_diagonal_left}
   \Delbar \otimes_{\Fuk} \cL_L & \to  \cL_L.
  \end{align}
In particular, the morphism space $ \Hom_{\Fuk}( \cL_{L'}, \Delbar)$ in the middle uses the left module structure on $\Delbar$, and is equipped with a residual right action arising from the right action on $\Delbar$.

These various maps are related to each other by the following result, which is proved in Section \ref{sec:directed-category}:
\begin{prop} \label{prop:cochain_diagram_commutes}
Given pairs  $(L,L') \in \cA$, there is a homotopy commutative diagram
\begin{equation} \label{eq:cochain_diagram_modules}
  \begin{tikzcd}[column sep=-30]
      \scrR_{L'} \otimes_{\Fuk}  \cL_L \ar[rr,"\eqref{eq:cochain_map_tensor_product_to_CF}"] \ar[d,"\eqref{eq:map_right_module_Hom_left}"] &  &  CF^*(L,L')  \ar[d,"\eqref{eq:cochain_map_CF_to_left_module}"]  \\
  \Hom_{\Fuk}( \cL_{L'},  \Delbar_{\Fuk}) \otimes_{\Fuk}  \cL_L   \arrow{dr} &  &   \Hom_{\Fuk}( \cL_{L'}, \cL_L ) \\
&  \Hom_{\Fuk}( \cL_{L'},  \Delbar_{\Fuk} \otimes_{\Fuk}  \cL_L)   \ar[ur,"\eqref{eq:map_tensor_perturbed_diagonal_left}"] & 
  \end{tikzcd}
  \end{equation}
\end{prop}

Having arranged for Figure \ref{fig:homotopy_compositions_labels_category} to correspond to a cobordism of moduli spaces of curves, and for the left hand side to have an algebraic interpretation, we are left with the problem of showing that this alternate factorisation of the composition is a quasi-isomorphism. Showing that the middle arrow in this composition is by far the easiest: it is a purely algebraic consequence of Lemma \ref{lem:left_module_perfect}, which asserts that $\cL_L$ is a perfect module over $\Fuk$. 

We encounter the final difficulty when trying to prove that the other two arrows in this compositions are isomorphisms: the intuition is that the bimodule $\Delbar$ corresponds to the geometric diagonal of $Y$, but the model of sheaves on $Y$ provided by (a subcategory of) modules over $\Fuk$ is not adequate to define $\Delbar$ as the diagonal bimodule of $\Fuk$. The key problem is that we have set morphisms in $\Fuk$ from $\tau$ to $\sigma$ to vanish unless $\tau$ precedes $\sigma$ in the ordering of the set indexing the cover, but the corresponding Floer cohomology groups only vanish if the polytopes are disjoint. On the algebraic side, the intuition is that Floer theory recovers a model for an enlargement of the category of coherent sheaves on $Y$, which admits as objects pushforwards of the structure sheaves of affinoid sub-domains. In algebraic geometry, the analogous problem of enlarging the category of coherent sheaves to admit pushforwards of the structure sheaves of open subsets (in the Zariski topology) is usually handled by considering the category of quasi-coherent sheaves, but the analogue in the analytic setting is poorly understood. In some sense, the paper's main contribution is thus the insight that, despite the fact that such the rigid analytic geometry literature lacks a satisfactory notion of quasi-coherence in general, Floer theory provides a model for morphisms between pushforwards of structure sheaves in the context that we are studying, and it is not difficult to show that this model gives the expected answer for inclusions, as discussed in Appendix \ref{sec:based-loops-laurent}.

While it would be possible to proceed along this route, and construct a category $\Poly$ in which morphisms are not artificially set to vanish for non-inclusions, we choose instead a shortcut:  we define $\Delbar$ as the bimodule over $\Fuk$ which would correspond to the pullback of the diagonal of the hypothetical category $\Poly$ under a natural embedding $\Fuk \to \Poly$. This approach has the advantage of allowing us to study only the local analogue of the category $\Poly$, which is the subject of Sections \ref{sec:cohom-categ-polyt} and \ref{sec:local-category-based}. To understand the basic ideas behind the proof that these local categories allow us to show that the desired maps are quasi-isomorphism, we refer the reader to Section \ref{sec:statement-results}. The discussion in Section \ref{sec:stat-main-results} below, which is limited to the cohomological aspects of this construction, may also be useful in understanding the main ideas.

\subsection{Outline of the paper}
\label{sec:outline-paper}

In  Section \ref{sec:statement-results}, we reduce the proof of Theorem \ref{thm:main_thm} to algebraic properties of the modules $\cL_{L}$ and $\cR_{L'}$ and of the bimodule $\Delbar_{\Fuk}$. The proof that these properties hold accounts for the remainder of Section \ref{sec:homological-algebra}. These algebraic arguments in turn rely on the studying families of pseudo-holomorphic curve  equations, constructed compatibly over various abstract moduli spaces, as explained in the preceding Section \ref{sec:famil-cont-equat}.

Because of the inherent combinatorial complexity of such constructions, we have chosen to start the paper with Section \ref{sec:cohom-constr}, which constructs all the operations at the cohomological level, and in fact proves a weak version of Theorem \ref{thm:main_thm} for Lagrangian sections.  We have organised this section to proceed from the simplest construction to the most complicated one, highlighting at each step the additional constraints on the choices we make that are required in order to carry out the next step.

Given the content of Section \ref{sec:cohom-constr} and the three Appendices on reverse isoperimetric inequalities, Tate's acyclicity results, and the computation of Floer cohomology groups with coefficients in completions of the homology of the group ring of the torus, filling in the rest of the paper is just a matter of lifting cohomological constructions to the chain level.

\subsection*{Acknowledgments}
The writing of this paper took a long and meandering route, including drafts that explored different strategies for the Floer-theoretic and algebraic parts. Discussions with Brian Conrad, Pierre Deligne, Kenji Fukaya, David Hansen, Johan de Jong, Thomas Kragh, Amnon Neeman, Paul Seidel and Katrin Wehrheim, were crucial at various points to abandon potential strategies as hopeless or too technically difficult, or to confirm that Floer-theoretic ideas could have mirror analogues. 

I would like to thank the anonymous referees for extensive and helpful comments; in particular for pointing out that it is inadequate in the statement of Theorem \ref{thm:main_thm} to impose the weaker condition that the total space $X$ have vanishing second homotopy group, as was the case in the first version of this paper.

The author was supported by the Simons Foundation through its ``Homological Mirror Symmetry'' Collaboration grant SIMONS 385571, and by NSF grants DMS-1308179, DMS-1609148, and DMS-1564172. He was also partially supported by the Erik Ellentuck Fellowship and the IAS Fund of Math during the ``Homological Mirror Symmetry'' program at the  Institute for Advanced Study.

\section{Cohomological constructions}
\label{sec:cohom-constr}

\subsection{Statement of the cohomological results}
\label{sec:stat-main-results}

Let $X \to \Q$  be a Lagrangian $n$-torus fibration and consider a (finite) collection $\cA$ of Lagrangians in $X$, which are graded in the sense of \cite{Seidel2000} (see Section \ref{sec:floer-compl-pairs} below). In order to bypass difficulties involving the foundations of holomorphic curve theory, we shall assume that each $L \in \cA$ is tautologically unobstructed in the sense that we are given 
\begin{equation}
   \label{eq:no_bubbling_L}
  \parbox{35em}{a tame almost complex structure $J_L$ on $X$, with respect to which all holomorphic discs with boundary on $L$ are constant.}   
\end{equation}
As discussed in the introduction, we also assume that $ \pi_2(Q) = 0$ which implies that, for any compatible almost complex structure, each fibre $X_q$ of this fibration  over a point $q \in Q$ can only bound constant holomorphic discs. Moreover, we shall assume for simplicity that all Lagrangians in $\cA$ are mutually transverse. %

Since $Q$ is the base of a Lagrangian torus fibration, the Arnol'd-Liouville theorem implies the existence of a lattice $ T^{*,\bZ}_{q}  Q \subset T^*_qQ$ locally spanned by closed $1$-forms. We define an \emph{integral affine polytope} $P \subset Q$ to be the embedded image of a polytope in $\bR^n$, defined by  inequalities $\langle \_ , \alpha_i \rangle \geq \lambda_i $ with $\alpha_i \in \bZ^n $ and $\lambda_i \in \bR$, under a map which pulls back $T^{*,\bZ}_{q}  Q  $ to the standard lattice $\bZ^n$ in $\bR^n$. Note that our conditions allow for $\dim P$ to be strictly less than $\dim Q$. As recalled in Section \ref{sec:affin-rings-assoc} below, we may associate to each such polytope  an \emph{affinoid ring} $\Gamma^P$ in the sense of Tate \cite{Tate1971}; such a ring has appeared in other contexts in symplectic topology, most notably in Fukaya, Oh, Ohta, and Ono's work \cite{FOOO2016} on toric mirror symmetry, where it is denoted $\Lambda \langle y, y^{-1} \rangle^P$.

We consider a finite set $A$ indexing a cover of $Q$ by integral affine polytopes, and the associated cover $\{ P_\sigma \}_{\sigma \in \Sigma}$ by the partially ordered set $\Sigma$ whose elements are the subsets of $A$ labelling non-empty intersections, and which we order by reverse inclusion, so that $\tau \leq \sigma$ implies that $P_\sigma \subseteq P_\tau$. We require that
 \begin{equation} \label{eq:nerve_cover_locally_contractible}
\parbox{33em}{the maximal length of a totally ordered subset of $\Sigma$ is $n+1$, and all polytopes $P_\sigma$ have non-empty interior.}
\end{equation}

In order to achieve this condition, we appeal to Lemma 2.3 of \cite{Abouzaid2014a}: one may construct a triangulation of $Q$ into integral affine simplices. Let $A$ denote the set of simplices. By induction on the dimension of simplices, we pick a cover indexed by $A$, consisting of integral affine polytopes which only intersect for nested simplices, and so that all intersections have non-empty interior (the base case is to choose sufficiently small non-intersecting polygonal neighbourhoods of all vertices, and the next step is to remove these neighbourhood from all edges, and take non-intersecting polygonal neighbourhoods of the resulting intervals).
 In this way, totally ordered subsets of $\Sigma$ correspond to sequences of nested simplices in a triangulation, whose length is bounded as above because $Q$ is a manifold of dimension $n$.  We shall later require that the cover $\{ P_\sigma \}_{\sigma \in \Sigma}$ be sufficiently fine with respect to increasingly strong conditions (see Condition \eqref{eq:nested_cover_conditions} where we impose the first constraint), and that we can pick basepoints $q_\sigma$ in $P_\sigma$ so that the fibres $X_{q_\sigma}$ are transverse to our chosen collection $\cA$ of Lagrangians  (see Section \ref{sec:floer-compl-pairs}).

For convenience, we introduce the notation $\Sigma_q$ for the partially ordered subset of $\Sigma$ corresponding to elements of the cover which contain a given point $q \in Q$. This set has a maximal element, which corresponds to the intersection of all the elements of our cover which contain $q$.

\begin{rem}
 We are working with a restricted class of covers to facilitate the appeal to the results of Appendix \ref{sec:null-homotopy-tates} concerning Tate's acyclicity theorem. A more natural setting to consider is that of covers indexed by finite categories, such that the nerve of the subcategory of polytopes containing any point is contractible. The bound by $n+1$ in Condition \eqref{eq:nerve_cover_locally_contractible} can be replaced by any choice of finite constant, as long as it is possible to construct arbitrarily fine covers while maintaining the same bound. The assumption that the polytopes $P_\sigma$ have non-empty interior is made for convenience, and makes it easier to achieve transversality e.g. in Section \ref{sec:floer-cohom-pairs}.
\end{rem}

In this section, which can be read as an extended introduction, we construct the following structures at the cohomological level which we shall later lift to the $A_\infty$ level:
\begin{itemize}
\item A category $\HFuk$ whose objects are elements $\sigma \in \Sigma$, with $\HFuk(\tau,\sigma)$ vanishing unless $\tau \leq \sigma$, and otherwise given by a Floer cohomology group $HF^*(P_\tau,P_\sigma)$ which is isomorphic to the affinoid ring $\Gamma^{P_\sigma}$.
\item For each Lagrangian $L \in \cA$ and $\sigma \in \Sigma$, Floer cohomology groups $HF^*(P_\sigma,L)$ and $HF^*(L,P_\sigma) $ which give rise to right and left modules $H \scrR_L$ and $H \cL_L$ over $\HFuk$ (see Section \ref{sec:cohomological-module}).
\item For each pair of Lagrangians $(L,L')$ in $\cA$, a map
  \begin{equation} \label{eq:action_HF_pair_lagr}
    HF^*(L,L') \to \Hom_{\Lambda}( HF^*(L',P),HF^*(L,P)  ) 
  \end{equation}
inducing a map into the space of left-module maps:
\begin{equation} \label{eq:map_HF_pair_hom_left_modules}
     HF^*(L,L') \to \Hom_{\HFuk}( H \cL_{L'} ,H \cL_L ).
\end{equation}
We define as well a map
\begin{equation} \label{eq:HF_map_tensor_over_P_to_HF_L}
 HF^*(P_\sigma,L')  \otimes_\Lambda     HF^*(L,P_\sigma) \to      HF^*(L,L')
\end{equation}
which is compatible with the action of morphisms in $\HFuk$, in the sense that we have an induced map
\begin{equation} \label{eq:Left-Right-module-tensor_to_HF_L}
 H   \scrR_{L'} \otimes_{\HFuk} H \cL_L \to  HF^*(L,L').
\end{equation}
\end{itemize}
At the cohomological level, the structures listed above are the same as those  considered in \cite{Abouzaid2014a}. Our cochain-level construction however will be different as we shall need to consider the following additional structures:
\begin{itemize}  %
\item A \emph{geometric diagonal} bimodule $H\Delbar_{\HFuk}$ over $\HFuk$, given for a pair $(\tau,\sigma)$ by a Floer cohomology group $HF^*(P_\tau,P_\sigma)$. The key point here is that these groups may be non-vanishing even if the condition  $\tau \leq \sigma $ does not hold.
\item For each $L \in \cA$, and for each pair $(\tau, \sigma)$ of polytopes,  a map
  \begin{equation}
      HF^*(L,P_\sigma)  \otimes_{\Lambda} HF^*(P_\tau,L) \to HF^*(P_\tau, P_\sigma),
  \end{equation}
which descends to a map of bimodules:
\begin{equation} \label{eq:HF_map_to_perturbed_diagonal}
  H \cL_L \otimes_\Lambda     H   \scrR_L \to H\Delbar_{\HFuk}.
\end{equation}
\item Finally, we construct a map
\begin{equation} \label{eq:HF_map_tensor_with_perturbed_diagonal}
   HF^*(P_\tau,P_\sigma)  \otimes_\Lambda   HF^*(L,P_\tau)  \to HF^*(L,P_\sigma),
\end{equation}
which gives rise to a map of left modules
\begin{equation} \label{eq:phi_diagonal_tensor_iso}
   \Delbar_{\HFuk} \otimes_{\HFuk} H \cL_L \to H \cL_L.
\end{equation}
\end{itemize}

For the statement of the main result of this section, it will be convenient to interpret Equation \eqref{eq:HF_map_to_perturbed_diagonal} as a map of right modules:
\begin{equation} \label{eq:cohomology_map_right_to_hom_L_Delta_Phi}
   H   \scrR_L \to \Hom_{\HFuk}( H \cL_L,  \Delbar_{\HFuk}).  
\end{equation}

\begin{prop} \label{prop:coh_diagram_commutes}
  Given Lagrangians $(L,L') \in \cA$, the following diagram commutes:
\begin{equation} \label{eq:cohomological_diagram_modules}
  \begin{tikzcd}[column sep=-60]
     H \scrR_{L'} \otimes_{\HFuk} H \cL_L \ar[rr,"\eqref{eq:Left-Right-module-tensor_to_HF_L}"] \ar[d, "\eqref{eq:cohomology_map_right_to_hom_L_Delta_Phi}"] &  &  HF^*(L,L')   \ar[d, "\eqref{eq:map_HF_pair_hom_left_modules}"] \\
  \Hom_{\HFuk}(H \cL_{L'}, H \Delbar_{\HFuk}) \otimes_{\HFuk} H \cL_L   \arrow{dr} &  &   \Hom_{\HFuk}(H \cL_{L'},H \cL_L ) \\
&  \Hom_{\HFuk}(H \cL_{L'}, H \Delbar_{\HFuk} \otimes_{\HFuk} H \cL_L) .  \ar[ur,"\eqref{eq:phi_diagonal_tensor_iso}"] & 
  \end{tikzcd}
  \end{equation}
\end{prop}
This result is proved in Section \ref{sec:proof-proposition}, and is illustrated in Figure \ref{fig:homotopy_compositions_labels_category}.  One key point is that the statement of the above proposition is at the level of cohomological categories.  This means in particular that the tensor product in the top left of Diagram  \eqref{eq:cohomological_diagram_modules} is computed as a quotient
\begin{equation*}
  \bigoplus_{\sigma, \tau}  HF^*(P_\sigma,L')  \otimes_\Lambda  \HFuk(\tau,\sigma) \otimes_\Lambda   HF^*(L,P_\tau)   \to  \bigoplus_{\tau} HF^*(P_\tau,L')  \otimes_\Lambda  HF^*(L,P_\tau) \twoheadrightarrow  H \scrR_{L'} \otimes_{\HFuk} H \cL_L ,
\end{equation*}
while the morphism space in the middle right is computed as a kernel
\begin{equation*}
  \Hom_{\HFuk}(H \cL_{L'},H \cL_L )  \hookrightarrow \prod_{\sigma} \Hom_\Lambda(  HF^*(L',P_\sigma), HF^*(L,P_\sigma)) \to \prod_{\sigma, \tau} \Hom_\Lambda(  HF^*(L',P_\sigma), HF^*(L,P_\sigma))
\end{equation*}
With the above in mind, once all structure maps are constructed, the commutativity of Diagram \eqref{eq:cohomological_diagram_modules} follows from proving that the diagram 
\begin{equation} \label{eq:cohomological_diagram}
  \begin{tikzcd}[column sep=-80]
    HF^*(P_\tau,L')  \otimes_\Lambda  HF^*(L,P_\tau)  \arrow{rr}{} \arrow{d} & &  HF^*(L,L')   \arrow{d} \\
\Hom_{\Lambda}(  HF^*(L',P_\sigma) ,  HF^*(P_\tau,P_\sigma))  \otimes_\Lambda   HF^*(L,P_\tau)  \arrow{dr} &  &  \Hom_\Lambda(  HF^*(L',P_\sigma), HF^*(L,P_\sigma)) \\
& \Hom_{\Lambda}(  HF^*(L',P_\sigma) ,  HF^*(P_\tau,P_\sigma)  \otimes_\Lambda   HF^*(L,P_\tau)) \arrow{ur} &
  \end{tikzcd}
\end{equation}
commutes for each pair $(\tau,\sigma)$ in $\Sigma$. When we revisit this result in Section \ref{sec:homological-algebra}, such an approach will not be sufficient, as the cochain-level version of the above diagram will only commute up to homotopy, and the analogue of Diagram \eqref{eq:cohomological_diagram_modules} will involve a tower of (homotopy) commutative diagrams.

We now state the main assumption under which we can prove that Floer homology is isomorphic to morphisms of modules at the level of cohomological categories; even though this is in some sense a special case of Theorem \ref{thm:main_thm}, we shall discuss the proofs, because it serves as a model for the $A_\infty$ analogue which holds in general:
\begin{prop} \label{prop:main_result_for_sections}
 If $L$ and $L'$ are Lagrangian sections, Equation \eqref{eq:map_HF_pair_hom_left_modules} is surjective.
\end{prop}

Given Proposition \ref{prop:coh_diagram_commutes}, it is clear that the above result would follow from the statement that each of the three arrows which compose counterclockwise in Diagram \eqref{eq:cohomological_diagram_modules} is an isomorphism. Starting with the middle arrow, we have the following result, which is proved in Section \ref{sec:bimodule-maps} below:
\begin{lem} \label{lem:H_map_surjective_sections}
If $L$ and $L'$ are Lagrangian sections, the map
\begin{equation}
  \label{eq:map_commute_Hom_tensor}
   \Hom_{\HFuk}(H \cL_{L'}, H \Delbar_{\HFuk}) \otimes_{\HFuk} H \cL_L    \to  \Hom_{\HFuk}(H \cL_{L'}, H \Delbar_{\HFuk} \otimes_{\HFuk} H \cL_L) , 
\end{equation}
is an isomorphism. 
\end{lem}
\begin{rem}
The fact that Equation (\ref{eq:map_commute_Hom_tensor}) is an isomorphism amounts to the statement that the module $H \cL_L$  behaves with respect to the functors $\Hom_{\HFuk}$ and $\otimes_{\HFuk}$ in much the same way as a projective module over a ring. It is therefore not surprising that this property holds whenever $L$ is a Lagrangian section because a section gives rise, under mirror symmetry, to a line bundle. 
\end{rem}

Regarding the remaining arrows in the counterclockwise composition arising in Diagram \eqref{eq:cohomological_diagram_modules},  we use the fact that the Floer cohomology groups $HF^*(P_\tau,P_\sigma)$ vanish whenever the inputs are disjoint. We can formalise the corresponding property of the bimodule $\Delbar_{H\Fuk}$ as follows:
\begin{defin}
  A bimodule $\cP$ over $\HFuk$ is \emph{local} if $\cP(\tau,\sigma)$ vanishes whenever $P_\tau$ and $P_\sigma$ are disjoint.
\end{defin}
For computations involving local bimodules, it is convenient to consider a local category associated to each element of $\Sigma$. For concreteness, we introduce some terminology:
\begin{defin}
The \emph{star}  of any element $\sigma \in \Sigma$ is the partially ordered set $\Sigma_\sigma \subseteq \Sigma$ consisting of elements $\tau \in \Sigma$ whose associated polytope intersects $P_\sigma$. We denote the full subcategory  of $H\Fuk$, whose set of objects is $\Sigma_\sigma$, by $\HFuk_\sigma$.
\end{defin}
We are interested in computing morphisms and tensor products locally, i.e.\ using the restriction to $\HFuk_\sigma$. More specifically, given any left module $\cL$ and bimodule $\cP$, we have natural comparison maps
 \begin{align}
  \Hom_{\HFuk}(\cL, \cP) &\to  \Hom_{\HFuk_\sigma}(\cL, \cP) \\
   \cP \otimes_{\HFuk_\sigma}  \cL & \to \cP \otimes_{\HFuk}  \cL 
 \end{align}
or right and left modules over $\HFuk_\sigma$, where we abuse notation by using the same symbol for a module on $\Fuk$ and its restriction to $\Fuk_\sigma$. The values of these modules on general objects of $ \HFuk_\sigma$ are difficult to control, but the following result asserts that the value on the object corresponding to $\sigma$ is the same:
\begin{lem}
  \label{lem:cohomological_compute_tensor_Hom_local}
 For each left module $\cL$ over $\HFuk$, and each local bimodule $\cP$, the natural maps
 \begin{align}
  \Hom_{\HFuk}(\cL, \cP)(\sigma) &\to  \Hom_{\HFuk_\sigma}(\cL, \cP)(\sigma) \\
  \left( \cP_{\HFuk} \otimes_{\HFuk_\sigma}  \cL\right)(\sigma) & \to \left(\cP \otimes_{\HFuk}  \cL \right)(\sigma)
 \end{align}
are isomorphisms.
\end{lem}
\begin{proof}
The two arguments are entirely analogous; we explain the second. We can describe the value at $\sigma$ of the tensor product over $\HFuk$ as the cokernel of the map
\begin{equation}
  \bigoplus_{\tau_{-1}, \tau_{-2}}   \cP(\tau_{-1}, \sigma) \otimes_\Lambda  \HFuk(\tau_{-2}, \tau_{-1}) \otimes_\Lambda  \cL(\tau_{-2}) \to   \bigoplus_{\tau}   \cP(\tau, \sigma) \otimes_\Lambda  \cL(\tau),
\end{equation}
with the direct sums being taken over objects and pairs of objects in $\HFuk$, and the arrow being given by the difference between the two possible compositions. The key fact is that, whenever $\HFuk(\tau_{-2}, \tau_{-1}) \neq 0$, the polytope $P_{\tau_{-1}}$ is contained in $P_{\tau_{-2}} $, so that $\tau_{-1} $ being an object of $\HFuk_{\sigma}$ implies the same for $\tau_{-2}$. In particular, both direct sums are in fact taken over objects of $ \HFuk_{\sigma}$, which implies the desired isomorphism.
\end{proof}

Because of the above result, whose $A_\infty$ generalisation holds with the same proof, verifying that Equations \eqref{eq:cohomology_map_right_to_hom_L_Delta_Phi} and \eqref{eq:phi_diagonal_tensor_iso} are isomorphisms is a \emph{local} computation, in the sense that the corresponding computations involve the each of the categories $\HFuk_\sigma$ separately.

To perform these local computations, it is useful to introduce a category $\HPoly_\sigma$, whose objects are the polytopes of $Q$ which are contained in a polytopal neighbourhood of $P_\sigma$ that itself contains all polytopes $P_\tau$ belonging to the star of $\sigma$. The morphisms in $\HPoly_\sigma$ are again given by Floer cohomology groups, and we have a faithful embedding
\begin{equation} \label{eq:j-embedding-HFuk-HPoly}
j \co  \HFuk_\sigma \to  \HPoly_\sigma,
\end{equation}
such that  for each pair of elements $\tau$ and $\rho$ of $\Sigma_\sigma$ for which  $ \rho \leq \tau$, the induced map
\begin{equation}
     \HFuk_\sigma(\rho,\tau) \to  \HPoly_\sigma(P_\rho,P_\tau),
\end{equation}
is an isomorphism. We prove in Section \ref{sec:pertubed-diagonal} that the pullback of the diagonal bimodule of $ \HPoly_\sigma$ is naturally isomorphic to the restriction of $\Delbar_{\HFuk}$ to $\HFuk_\sigma$ (see Lemma \ref{lem:map_diagonal_to_perturbed}). Furthermore we show in Section \ref{sec:bimodule-maps} that, for each $L \in \cA$ there are left and right modules $H \cL_{L,\sigma}$ and  $H \scrR_{L,\sigma}$ over $H \Poly_\sigma $  whose pullbacks to $\HFuk_\sigma$ are naturally isomorphic to the restrictions of $H \cL_{L}$ and  $H \scrR_{L}$ (see Proposition \ref{prop:cohomology_commutative_diagram_local_globarl_Diagonal}).

This allows us to reduce local computations to the category $\HPoly_\sigma$ in certain special cases, and to the corresponding $A_\infty$ category in general. For example, the proof that Equation \eqref{eq:phi_diagonal_tensor_iso} is an isomorphism reduces to the statement that the map of the tensor product of pullbacks of left and right modules
\begin{equation} \label{eq:local_diagonal_tensor_directed_iso}
 j^* \Delta_{\HPoly_\sigma}(\sigma, \_) \otimes_{\HFuk_\sigma} j^*\left( H \cL_{L}  \right) \to j^* H \cL_{L}(\sigma)
\end{equation}
is an isomorphism.  We shall prove that the above follows from Tate acyclicity assuming that $L$ meets $X_{q_\sigma}$ at a single point (see Section \ref{sec:comp-dimens-1}).

\begin{rem}
The sign conventions on the cohomological constructions implemented in this section are those corresponding to $A_\infty$ algebras with trivial differential and higher products. One can translate to the usual $dg$ conventions as discussed in \cite[Section I.1.1a]{Seidel2008a}.
\end{rem}

We complete the introduction to this section by reducing its main conclusion to the results proved in the remainder of the section:
\begin{proof}[Proof of Proposition \ref{prop:main_result_for_sections}]
As noted earlier, Proposition \ref{prop:coh_diagram_commutes} is proved in Section \ref{sec:proof-proposition}, and reduces the desired result to proving that the unlabeled arrow in Diagram \eqref{eq:cohomological_diagram_modules}, as well as Equations \eqref{eq:phi_diagonal_tensor_iso} and \eqref{eq:cohomology_map_right_to_hom_L_Delta_Phi} are isomorphisms. The first isomorphism is established by Lemma \ref{lem:H_map_surjective_sections} above, and the remaining two isomorphisms are reduced to their local versions by Lemma \ref{lem:cohomological_compute_tensor_Hom_local} above, and Proposition \ref{prop:cohomology_commutative_diagram_local_globarl_Diagonal} from Section \ref{sec:bimodule-maps} below. The corresponding local statements are proved in Proposition \ref{prop:tensor_Fuk_iso_projective} and Lemma \ref{lem:compute_right_module_section}. 
\end{proof}

\subsection{Morse theory of local systems associated to polytopes}
\label{sec:constr-oper}

\subsubsection{Affinoid rings associated to polytopes}
\label{sec:affin-rings-assoc}
The first step to construct the ring associated to an integral affine polytope is to observe that each point $p \in \bR^n$  determines a non-archimedean valuation on the ring of Laurent polynomials with coefficients in the Novikov field $\Lambda$, which we think of as the group ring of  $ \bZ^n $, given by
\begin{align}
  \Lambda[ \bZ^n ] & \to \bR \cup \{  + \infty \} \\
 \val_{p}\left( \sum c_\beta z^{\beta} \right) & \coloneqq \min_{ \beta }\left( \val(c_\beta) + \langle \beta, p-q \rangle \right).
\end{align}
Here, $\val(c_\beta)$ denotes the valuation of this element of $\Lambda$ as defined in Equation \eqref{eq:valuation_Novikov}. Taking the minimum over all elements of an integral affine polytope $P$ in $\bR^n$, we obtain the valuation
\begin{equation} \label{eq:norm_P}
  \val_{P}\left(\sum c_\beta z^{\beta}\right)  \coloneqq \min_{p \in P} \val_p\left(\sum c_\beta z^{\beta}\right),
\end{equation}
and define the ring $\Gamma^P $
to be the corresponding completion of the Laurent polynomial ring (i.e.\ consider series such that the number of terms with  valuation bounded above by any fixed constant is finite). Elements of this ring can be written uniquely as series
\begin{equation}
 \sum_{\beta \in \bZ^n} c_\beta z^{\beta} 
\end{equation}
satisfying the condition
\begin{equation}
    \lim_{|\beta| \to +\infty} \val_{P} \left( c_\beta z^{\beta}\right)  = +\infty.
  \end{equation}

If $P$ is an integral affine polytope in $Q$, we can relate this affinoid ring to the fundamental group ring of a fibre $X_q$ as follows: the integral affine structure on $Q$ defines an \emph{affine exponential map} 
\begin{equation}
H^1(X_q, \bR) \cong T_q Q \dasharrow Q
\end{equation}
given on a neighbourhood of the origin, which induces an isomorphism of lattices at each point where it is defined. A path from $q$ to $P$ determines a lift of $P$ to $H^1(X_q, \bR)$. Each point $p \in P$ thus determines a valuation on the group ring $\Lambda[ H_1(X_q; \bZ)] $ of the first homology of the fibre $X_q$, and we obtain a valuation associated to the polytope $P$ by taking the minimum as in Equation \eqref{eq:norm_P}. Given that the integral affine structure on $Q$ and the embedding of $P$ in $\bR^n$ give an isomorphism $\bZ^n \cong H_1(X_q; \bZ)$, we conclude:
\begin{lem} \label{lem:affinoid_ring_description_q}
A choice of homotopy class of paths from $q$ to $P$ determines an isomorphism between $\Gamma^P$ and the completion of $\Lambda[ H_1(X_q; \bZ)]$ at the valuation associated to the image of $P$ in $H^1(X_q, \bR)\cong T_q Q  $.  \qed
\end{lem}

\subsubsection{Loops, paths, and local systems}
\label{sec:loops-paths-local}

To proceed further, it is convenient to eliminate the choice of homotopy class by fixing a contractible subset $\cN \subseteq Q$, which will contain both the basepoint $q$ and the polytope $P$. We equip $\cN$ with a Lagrangian section
\begin{equation} 
\triv \co  \cN \to X_{\cN} 
\end{equation}
of the restriction of the projection $X \to Q$ to the inverse image $X_{\cN}$ of $\cN$. Given this data, we shall associate to each integral affine polytope $P \subseteq  \cN$, and each point $q \in \cN$, a local system $U^P_\triv$ on the Lagrangian $X_q$ with fibre isomorphic to $ \Gamma^P$.

To start, let  $U_{\triv}$ denote the local system on $X_q$ whose value at a point $x$ is the free $\Lambda$-module %
\begin{equation}
  U_{\triv,x} \coloneqq    \Lambda[\pi_{0}(\Omega_{\triv(q),x} X_q)],
\end{equation}
freely generated by the components of the space $\Omega_{\triv(q),x} X_q $ of paths in $X_q$ from $\triv(q)$ to $x$. This space is naturally a module over the loop space of $X_q$ based at $\triv_{\cN}(q)$, so that $  U_{\triv,x}$ is a module over the group ring of the fundamental group of $X_q$. Since $X_q$ is a torus, the fundamental group is isomorphic to $H_{1}( X_q; \bZ)$, so we conclude that $U_{\triv,x}$ is a module over the group ring of $H_{1}( X_q; \bZ)$.

On the other hand, Lemma \ref{lem:affinoid_ring_description_q} implies that the affinoid ring $\Gamma^P$ is naturally a module over the group ring of $ H_{1}( X_q; \bZ)$. The tensor product of $U_{\triv}$ over $ \Lambda[H_{1}( X_q; \bZ) ]$ with $\Gamma^P$ defines a local system $U^P$, with fibre
\begin{equation} \label{eq:completion_universal_local_system_P}
  U^P_{\triv,x} \coloneqq \Lambda[\pi_{0}(\Omega_{\triv(q),x} X_q)] \otimes_{  \Lambda[ H_{1}( X_q; \bZ) ]  }  \Gamma^{P}.
\end{equation}
Note that $ U_{\triv,x}$ is a free rank-$1$ module over $  \Lambda[ H_{1}( X_q; \bZ)] $, with a generator corresponding to a choice of homotopy class of paths from $\triv(q)$ to $x$. Choosing such a path, we obtain a norm on $U_{\triv,x}$, and $U^P_{\triv,x}$ is the completion with respect to this norm.

Next, given a polytop in $Q$, we would like to compare the associated local systems on different Lagrangian fibres. For applications, we shall more precisely need to compare the values at points which are defined as the image of a Lagrangians section; since the Lagrangian sections we encounter are seldom globally defined, we introduce a  contractible subset $B$ of $\cN$, which will be the domain of definition of a Lagrangian section
\begin{equation}
  x \co B \to X_B.
\end{equation}
Given $q \in B$, we write $x^q$ for this intersection of this section with the fibre $X_q$.
\begin{lem} \label{lem:identification_local_systems_Lagrangian_section}
If $q$ and $q'$ lie in $B$, there is an identification
\begin{equation} \label{eq:identify_local_system_nearby_q}
     U^P_{\triv,x^q}  \cong U^P_{\triv,x^{q'}}
\end{equation}
for all points $q, q' \in B$, with the property that, for a triple $q$, $q'$ and $q''$, the  following diagram commutes: 
 \begin{equation}
   \begin{tikzcd}
      U^P_{\triv,x^q} \ar[r] \ar[d] &  U^P_{\triv,x^{q''}}  .  \\
   U^P_{\triv,x^{q'}}    \ar[ur]
   \end{tikzcd}
 \end{equation}
\end{lem}
\begin{proof}
Pick a class $\gamma$ in $\pi_{0}(\Omega_{\triv(q),x} X_q) \cong \pi_{0}(\Omega_{\triv(q'),x} X_{q'}) $. Since $B$ is contractible, there is a unique homotopy class of maps
\begin{equation}
u \co  [0,1]^2 \to X_B  
\end{equation}
such that the restrictions to $\{0\} \times [0,1]$ and $\{1\} \times [0,1]$ map to $X_q$ and $X_{q'}$ and represent the class $\gamma$, and the restrictions to $[0,1] \times \{ 0\}$ and $[0,1] \times \{ 1\}$ map to the sections $\triv$ and $x$. We define
\begin{equation}
F_\gamma(q,q') = \int_{[0,1]^2} u^* \omega,  
\end{equation}
and note that this integral is independent of the choice of representative of the given homotopy class because all boundary conditions are Lagrangian. We then define the map from $ U^P_{\triv, x^q}$ to $ U^P_{\triv, x^{q'}} $ by 
\begin{equation}
z^\gamma \mapsto T^{ F_\gamma(q,q') } z^\gamma. 
\end{equation}
\end{proof}

As the proof shows, the identification in Lemma \ref{lem:identification_local_systems_Lagrangian_section} is compatible with parallel transport maps in the following sense: we can associate to a homotopy class $\gamma$ of paths with endpoints $(x_{0},x_{1})$ on $X_q$, and to $q' \in B$ a real number $F_\gamma(q,q')$ which is the \emph{flux} of this path. Under the identification of Equation \eqref{eq:identify_local_system_nearby_q}, the parallel transport maps associated to this homotopy class of paths in $X_q$ and the corresponding homotopy class in $X_{q'}$ differ by multiplication by $ T^{ F_\gamma(q,q') }$. In the setting of the above Lemma, we shall therefore omit the superscript from $x^q$ unless it is required for clarity of exposition.

A key aspect of the constructions of this paper will be the need to verify the  $T$-adic convergence of operators constructed using holomorphic curves. Consider points $x, y \in X_q$, a path $\gamma$ with endpoints $x$ and $y$, and  the induced parallel transport map
\begin{equation} \label{eq:parallel_transport_map}
 z^{[\gamma]} \co  U^P_{\triv,x}  \to U^P_{\triv,y}  
\end{equation}
for a polytope $P \subseteq Q$ which is in the image of the affine exponential map based at $q$. Fixing paths from $x$ and $y$ to the intersection of $X_q$ with the section associated to $\cN$, we obtain valuations on the fibres of these local systems. These data also assign a homology class  $[\gamma] \in H_{1}(X, \bZ)$ to $\gamma$, and the valuation of $ z^{[\gamma]}$ is then bounded by the product of the norm of $[\gamma]$ with the distance from the origin to the inverse image of $P$ in $T_qQ$, where
\begin{equation} \label{eq:min_length_curve_equals_norm_hlgy}
  |[\gamma]| \coloneqq \min_{[\gamma'] = [\gamma]} \ell(\gamma'),
\end{equation}
and $\ell$ is the length of the loop $\gamma'$, with respect to the flat metric on $X_q$ induced by a choice of Riemannian metric on $Q$.
\begin{lem} \label{lem:estimate_norm_transport_P}
  If the lift of $P$ to $T_qQ$ %
  is contained in the ball of radius $ \epsilon$, the valuation of the parallel transport map $z^{[\gamma]}$ on $ U^P_{\triv}$ is bounded below by $-  \epsilon \cdot |[\gamma]| $, up to adding a constant which is independent of $\gamma$. \qed 
\end{lem}
It is convenient to replace the condition about the lift of $P$ under the exponential map by a condition in $Q$. To this end, we shall assume from now on that:
\begin{equation}
  \label{eq:distortion_bounded}
  \parbox{35em}{the convexity radius of the chosen metric on $Q$ is greater than $2$. Moreover, for all $q \in Q$, the affine exponential map restricts to an embedding on the ball of radius $2$ in $T_qQ$ with distortion bounded by $2$.}
\end{equation}
We recall that the convexity radius of a Riemannian metric is the smallest constant $r$ for which geodesic balls of radius smaller than $r$, centered at any point in $Q$, have the property that any two points are connected by a unique minimal geodesic which is contained in the given ball. This constant is non-zero for any closed Riemannian manifold, as discussed e.g. in \cite[Section 6.5.3]{Berger2003}. The key property we shall use is that the intersection of any balls of radius smaller than the convexity radius is again convex, and hence contractible.

The following result will play a key role in the proof of convergence of various Floer theoretic constructions:
\begin{cor}  \label{cor:convergence_if_coefficients_bounded_lengths}
Let $0 < \delta$ be a positive real number, and let $\epsilon$ be a constant which is smaller than the minimum of $1/2$ and $\delta/2$. Assume that $\{\gamma_i\}_{i=0}^{\infty}$ is a sequence of paths with endpoints $x$ and $y$ in $X_q$ and $\lambda_i \in \bR$ is a sequence going to $+\infty$ such that
\begin{equation}
\delta \cdot |[\gamma_i]| \leq \lambda_i + \textrm{ a constant independent of }i.
\end{equation}
Whenever $P$ is contained in the ball of radius $\kk \epsilon$ about $q$ in $Q$, the map
\begin{equation}
  \sum_{i=1}^{\infty} T^{\lambda_i} z^{[\gamma_i]}  \co U_{\triv,x}^P \to U_{\triv,y}^P 
\end{equation}
converges in the topology induced by $P$.
\end{cor}
\begin{proof}
The valuation of $T^{\lambda_i} z^{[\gamma_i]} $ is given by
\begin{equation}
  \lambda_i + \val_P z^{[\gamma_i]}  \geq  \lambda_i -  2 \kk \epsilon \cdot |\gamma_i| \geq (1 - 2\kk \epsilon/\delta) \lambda_i +  \textrm{ a constant independent of }i.
\end{equation}
The assumptions that $2\kk \epsilon < \delta$ and  $\lambda_i \to \infty$ imply that this valuation also goes to $+\infty$. 
\end{proof}

\begin{rem}
The reader should have in mind that the constant $\delta$ which appears in the statement is obtained from applying the reverse isoperimetric inequality to a collection of holomorphic curves with boundaries $\gamma_i$, and energies $\lambda_i$.  
\end{rem}

\subsubsection{Morse homology groups associated to polytopes}
\label{sec:morphisms}

Given a pair $(\cN_{0}, \cN_{1})$ of contractible subsets of $Q$, with contractible intersection, and equipped with Lagrangian sections $\triv_0$ and $\triv_1$,  consider a point  $q \in  \cN_{0} \cap \cN_{1}$ together with a Morse function
\begin{equation} \label{eq:Morse_function_fibre_intersection}
  f_{0,1} \co X_{q} \to \bR,
\end{equation}
whose critical locus we denote $\Crit(0,1)$.

In order to later be able to study Floer theory over a field of characteristic different from $2$, we pick a $\Pin$ structure on the restriction of $TQ$ to $\cN_i$ for each $\sigma \in \Sigma$; as $\cN_i$ is contractible, such a choice is unique up to isomorphism.  Via the canonical identification, for $q \in Q$, of the tangent space of $X_q$ with the trivial vector bundle with fibre $T^*_q Q$, we obtain a $\Pin$ structure on $X_q$, varying continuously over $q \in \cN_i$. We then denote by $\lambda_{{0},{1}}$ the rank-$1$ free abelian group of isomorphism classes of families of $\Pin$ structures on $T_{q} Q$, parametrised by $[0,1]$ and  twisted by the orientation line of  $T_{q} Q$ as in  \cite[Equation (11.32)]{Seidel2008a}, which agree with the chosen structures associated to $\cN_{0}$ and $\cN_{1}$ at the endpoints.  Given a critical point $x \in \Crit(0,1)$, define 
\begin{equation} \label{eq:orientation_line_critical}
  \delta_x \coloneqq  \lambda_{0,1}  \otimes \det_x,  
\end{equation}
where $\det_x$ is the orientation line of the stable manifold of $x$. Let $\deg(x)$ denote the degree of this graded line.

The choices $\triv_{i}$  of Lagrangian sections over  $\cN_i$ yield local systems $U_{i}$ on $X_{q}$. Given polytopes $P_i \subset T^*_q Q$, we obtain local systems $U^{P_i}_{i}$ which are constructed by completion, and define  
\begin{equation}
  \label{eq:Floer_complex_2_polytopes}
  CM^{*}(X_{q}, \Hom^c_\Lambda(U^{P_{0}}_{{0}}, U^{P_{1}}_{{1}}) \otimes \lambda) \coloneqq \bigoplus_{x \in \Crit(0,1) } \Hom^c_\Lambda(U^{P_{0}}_{{0},x},U^{P_{1}}_{{1},x}) \otimes \delta_x
\end{equation}
where $ \Hom^c_\Lambda(U^{P_{0}}_{{0},x},U^{P_{1}}_{{1},x})  $ is the space of continuous homomorphisms, i.e. those with finite valuation
\begin{equation}
  \val \phi = \inf_{f \in U^{P_{0}}_{{0},x} \setminus \{ 0 \} } \val( \phi(f)) - \val(f).
\end{equation}

We now recall the construction of the differential: let $\cI$ denote the interval $(-\infty,+\infty)$ for which we use the parameter $t$. We define 
\begin{equation}
  \cT({0},{1}) = \Big \{ \gamma \co \cI \to X_q | \frac{d \gamma}{dt} = \nabla f_{0,1} \Big\}/\bR
\end{equation}
where the gradient is taken with respect to a Riemannian metric on $X_q$, and the $\bR$ action is by translation. We have a standard compactification $   \cT({0},{1})  \subset \Tbar({0},{1})  $ consisting of broken gradient trajectories, and an evaluation map
\begin{equation}
 \Tbar({0},{1})  \to \Crit(0,1) \times \Crit(0,1), 
\end{equation}
which on $  \cT({0},{1}) $ is given by taking the limits at  $-\infty$ and $+\infty$. We write $ \Tbar(x_{0}, x_{1}) $ for the fibre over critical points $x_{0}$ and $x_{1}$.

For a generic choice of metric, the moduli space $\Tbar(x_{0}, x_{1}) $ for $x_0 \neq x_1$ is a compact manifold with boundary of dimension
\begin{equation}
  \dim \Tbar(x_{0}, x_{1}) = \deg(x_{0}) - \deg(x_{1}) - 1,
\end{equation}
whose direct sum with the tangent space of the interval is oriented relative the tensor product $ \det_{x_{0}}^{\vee} \otimes \det_{x_{1}} $ in the sense that we have a canonical isomorphism
\begin{equation}
    \ro_{\Tbar(x_{0}, x_{1})} \otimes \ro_{\cI} \otimes  \det_{x_{0}}^{\vee} \otimes \det_{x_{1}} \cong \bZ,
  \end{equation}
  where we use $\ro_{M}$ to denote the orientation line of a manifold $M$ (i.e. the top exterior power of its tangent space),  
Fixing the standard trivialisation of the tangent space of the interval, we obtain,  whenever
\begin{equation} \label{eq:index_1_difference}
\deg(x_{0}) = \deg(x_{1}) + 1   
\end{equation}
 a natural  isomorphism of degree $1$
\begin{equation}
\det_{\gamma} \co \det_{x_{1}} \to \det_{x_{0}}.  
\end{equation}
associated to each element of $\Tbar(x_{0}, x_{1})$. Parallel transport also induces an isomorphism of topological vector spaces
\begin{align}
\Pi_\gamma \co \Hom^c_\Lambda(U^{P_{0}}_{{0},x_{1}},U^{P_{1}}_{{1},x_{1}}) & \to \Hom^c_\Lambda(U^{P_{0}}_{{0},x_{0}},U^{P_{1}}_{{1},x_{0}}) \\
\psi & \mapsto z^{[\gamma]} \circ \psi \circ  z^{-[\gamma]}.
\end{align}
Taking the tensor product of these two maps with the identity on $\lambda_{0,1}$, we define
\begin{align}
  \mu^1 \co    CM^{*}\left(X_{q}, \Hom^c_\Lambda(U^{P_{0}}_{{0}}, U^{P_{1}}_{{1}}) \otimes \lambda \right)  \to   &  CM^{*}\left(X_{q}, \Hom^c_\Lambda(U^{P_{0}}_{{0}}, U^{P_{1}}_{{1}}) \otimes \lambda \right) \\
\mu^1 \coloneqq & \sum_{[\gamma] \in  \Tbar^0({0},{1})} (-1)^{\deg(x_{1})} \Pi_{\gamma} \otimes \id_{\lambda} \otimes \det_{\gamma},
\end{align}
where $ \Tbar^0({0},{1})$ is the space of rigid gradient flow lines, i.e.\ those of virtual dimension $0$, which in this case corresponds to the union of components $\Tbar(x_{0}, x_{1}) $ for critical points $x_0$ and $x_1$ satisfying Equation \eqref{eq:index_1_difference}. Since the sum is necessarily finite, this map is continuous, and defines a differential by the usual analysis of the boundary of the moduli space of flow lines of dimension $1$, and the fact that the parallel transport maps are invariant under homotopy, and compatible with concatenation of paths.

Choosing a basepoint  $q_{{0}, {1}} \in \cN_{0} \cap \cN_{1}$,  we define
\begin{equation} \label{eq:definition_Floer_complex_polytopes}
   CF^*((\cN_{0},P_{0}), (\cN_{1},P_{1})) \coloneqq CM^{*}\left(X_{q_{0,1}}, \Hom^c_\Lambda(U^{P_{0}}, U^{P_{1}}) \otimes \lambda \right).
\end{equation}

We now prove that, if we choose a different basepoint, we can arrange Morse data so that the resulting complex is isomorphic.  Recall that the choice of section $\triv_{{0}}$ induces an identification of symplectic manifolds
\begin{equation}
X_{\cN_{0}} \cong T^* \cN_{0} /T^{*,\bZ}  \cN_{0}
\end{equation}
over the base $\cN_{0}$, hence diffeomorphisms $  X_{q} \cong X_{q'} $
for all pairs $(q,q') \in \cN_{0}$, which are compatible for triples. We shall leave these diffeomorphisms implicit in our notation. With this in mind, the function $f_{0,1}$ appearing in the right hand side of Equation \eqref{eq:Floer_complex_2_polytopes} can therefore be thought of as a Morse function on \emph{any fibre} over $\cN_{0} \cap \cN_{1}$, and the metric on $X_q$ induces a metric on $X_{q'}$. Since the base is contractible, Lemma \ref{lem:identification_local_systems_Lagrangian_section} provides an isomorphism of local systems, so we obtain an isomorphism of cochain complexes
\begin{equation} \label{eq:Morse_complex_independent_of_basepoints}
CM^{*}(X_{q}, \Hom^c_\Lambda(U^{P_{0}}_{{0}}, U^{P_{1}}_{{1}}) \otimes \lambda) \to CM^{*}(X_{q'}, \Hom^c_\Lambda(U^{P_{0}}_{{0}}, U^{P_{1}}_{{1}}) \otimes \lambda)
\end{equation}
which is compatible with composition for triples $(q, q', q'')$. This establishes that this Morse complex is indeed independent of the choice of basepoint in $\cN_{0}$.

\begin{rem} \label{rem:any_trivialisation_works}
Note that we broke symmetry and chose $\triv_{{0}}$ to identify fibres over $\cN_{0} \cap \cN_{1}$. We could have chosen $\triv_{{1}}$, or in fact any other trivialisation. In cochain-level constructions, we shall need to verify compatibility between various trivalisations, by ensuring that they arise from a contractible set of choices.
\end{rem}

\subsubsection{Morse theoretic product}
\label{sec:morse-theor-prod}

Consider a triple $\{\cN_{i}\}_{i=0}^{2}$ of contractible subsets of $Q$, equipped with sections $\{\triv_{i}\}_{i=0}^{2}$ of the Lagrangian torus fibration. We assume that the triple intersection, which we denote $\cN$, is contractible.  We use the above discussion to identify the functions $f_{i,j}$ as Morse functions on a single fibre over a point $q \in \cN$. Morse theory thus induces a product on the Floer cohomology groups for pairs, whose construction, while standard,  we now recall in order to set up notation for future use. 

Consider the semi-infinite intervals
\begin{equation}
\cI_+ \coloneqq   [0,\infty) \textrm{ and } \cI_-  \coloneqq   (-\infty,0].
\end{equation}
We define the space of Morse data 
\begin{align} \label{eq:Morse_data_product}
\scrV_{\pm}(i, j) \subset C^{\infty}(\cI_\pm, C^\infty(X_q, TX_q)) 
\end{align}
to consist of families of vector fields on $X_q$, parametrised by $\cI_\pm$, which agree with $\nabla f_{i,j}$ outside a compact set. The gradient flow is taken with respect to the metric chosen in the construction of the Floer complex for the pair $(i, j)$.

Given $\xi_{ij}^{\pm} \in \scrV_{\pm}(i, j)$, we then define
\begin{equation}
\cT_{\pm}(i,j) \subset   C^{\infty}(\cI_{\pm}, X_q) 
\end{equation}
to be the set of perturbed gradient flow lines, i.e. maps $\gamma$ from $\cI_\pm$ to $X_q$ satisfying
\begin{equation}
\frac{d \gamma}{dt} = \xi_{ij}^{\pm}. 
\end{equation}
The limit of $\gamma$ at $\pm \infty$ yields a natural evaluation map
\begin{equation}
  \cT_{\pm}(i,j)  \to    \Crit(i,j).
\end{equation}
We define the spaces of  \emph{broken semi-infinite perturbed gradient flow lines}, to be
\begin{equation}
  \Tbar_{\pm}(i, j) \coloneqq  \cT_{\pm}(i,j) \cup \cT_{\pm}(i,j)  \times_{\Crit(i,j)} \Tbar(i,j).
\end{equation}
where we use the evaluation map at $\mp \infty$ for $\Tbar(i, j)$. An element of this fibre product can be thought of as a broken flow line together with a semi-infinite flow line, with matching asymptotic limits.

There is a natural evaluation map
\begin{equation}
  \label{eq:map_universal_semi_inf_gradient}
  \Tbar_{\pm}(i,j) \to X_q \times \Crit(i,j),
\end{equation}
given by evaluation at $0$ and of the limit of $\gamma$ at $\pm \infty$.  We denote the coequaliser of the three evaluation maps from the product to $X_q$ by
\begin{equation}
  \Tbar(0,1,2) \coloneqq   \Coeq\left(\Tbar_{-}({0},2) \times \Tbar_{+}({1},2) \times \Tbar_{+}({0},{1}) \to X_q\right),
\end{equation}
This is the space of flow lines which map $0$ to the same point. %

Given critical points $x_i \in  \Crit(i,{i+1})$ (with index counted modulo $3$), we define
\begin{equation}
\Tbar(x_{0}, x_2, x_{1})   
\end{equation}
to be the inverse image of $( x_{1}, x_2, x_{0})$ under the evaluation map
\begin{equation}
      \Tbar(0,1,2) \to \prod_{i=0}^{2}  \Crit_{i,{i+1}}.
\end{equation}

For a generic choice of triples of Morse perturbations $(\xi^+_{01}, \xi^+_{12}, \xi^-_{02})$, the space $ \Tbar(x_{0}, x_2, x_{1})$ is a compact topological manifold with boundary, naturally oriented relative
\begin{equation}
\det^{\vee}_{x_{0}} \otimes \det_{x_2}  \otimes  \det_{x_{1}}.
\end{equation}
Dualising, we obtain, for each rigid element $\gamma \in \Tbar(x_{0}, x_2, x_{1})$, a natural map
\begin{equation} \label{eq:map_determinant_lines_product_morse}
\det_{\gamma} \co  \det_{x_2}  \otimes  \det_{x_{1}} \to \det_{x_{0}}.
\end{equation}

Let us now assume that we are given polytopes $P_i \subseteq \cN_i$.   Parallel transport (and composition) also induces a map
\begin{equation} \label{eq:map_HOM_local_systems_product_morse}
\Pi_\gamma \co  \Hom^c_\Lambda(U^{P_{1}}_{{1},x_2},U^{P_2}_{2,x_2}) \otimes_\Lambda  \Hom^c_\Lambda(U^{P_{0}}_{{0},x_{1}},U^{P_{1}}_{{1},x_{1}}) \to \Hom^c_\Lambda(U^{P_{0}}_{{0},x_{0}},U^{P_2}_{2,x_{0}})
\end{equation}
which can be defined as follows: let $\gamma_{ij}$ denote the path from $x_i$ to $x_j$ determined by $\gamma$. We have
\begin{equation}
\Pi_\gamma \left( \psi_{12} \otimes \psi_{01} \right) = z^{[\gamma_{20}]}  \circ  \psi_{12} \circ z^{[\gamma_{12}]} \circ \psi_{01} \circ z^{[\gamma_{01}]}.
\end{equation}
We also have a natural map
\begin{equation} \label{eq:map_twisting_product_morse}
\lambda_{0,1,2} \co \lambda_{{1},2} \otimes  \lambda_{0,1}  \to \lambda_{{0},2}
\end{equation}
induced by concatenating paths.  

Taking the sum, over all triples of critical points, of the tensor products of Equations \eqref{eq:map_determinant_lines_product_morse}, \eqref{eq:map_HOM_local_systems_product_morse}, and \eqref{eq:map_twisting_product_morse}, we obtain a map
\begin{multline}
CM^{*}\left(X_{q}, \Hom^c_\Lambda(U^{P_{1}}_{{1}}, U^{P_2}_{2} ) \otimes \lambda  \right) \otimes CM^{*}\left(X_{q}, \Hom^c_\Lambda(U^{P_{0}}_{{0}}, U^{P_{1}}_{{1}} ) \otimes \lambda  \right)\\ \to CM^{*}\left(X_{q}, \Hom^c_\Lambda(U^{P_{0}}_{{0}}, U^{P_2}_{2} ) \otimes \lambda  \right)
\end{multline}
which is given by the formula
\begin{align} 
\label{eq:mu_2_polytopes}
\mu^2 \coloneqq & \bigoplus_{x^\vee_{0} \in \Crit(0,2)  } \sum_{\substack{x_{1} \in \Crit(0,1) \\ x_2 \in \Crit(1,2)} } (-1)^{\deg(x_{1})}  \Pi_{\gamma} \otimes \lambda_{0,1,2} \otimes  \delta_{\gamma}.
\end{align}
Since the sum is finite, this is necessarily a continuous map. Composing the left and right hand sides with the isomorphisms of Equation \eqref{eq:Morse_complex_independent_of_basepoints}, we obtain the product
\begin{multline}
  \mu^2 \co CF^*((\cN_{1},P_{1}), (\cN_2,P_2)) \otimes_\Lambda  CF^*((\cN_{0},P_{0}),(\cN_{1},P_{1}))\\
  \to CF^*((\cN_{0},P_{0}),(\cN_{2},P_{2})).
\end{multline}
The fact that this is a cochain map is again a combination of the standard Morse-theoretic description of the boundary of the moduli space of gradient trees, with the invariance of parallel transport maps under homotopies, and their compatibility with concatenation.

\subsection{The cohomological category of polytopes}
\label{sec:cohom-categ-polyt}

We now return to the setting discussed at the beginning of Section \ref{sec:stat-main-results}, and consider a cover of $Q$ by integral affine polytopes $\{P_{\sigma} \}_{\sigma \in \Sigma} $, satisfying Condition \eqref{eq:nerve_cover_locally_contractible}. In addition, we assume that the cover is sufficiently small so that we may choose a contractible neighbourhood $\cN_\sigma$ of $P_\sigma$ for each element of this cover with the property that:
\begin{equation}
  \label{eq:nested_cover_conditions}
  \parbox{33em}{ all non-empty intersections among the sets $\{\cN_\sigma \}_{\sigma \in \Sigma}$ are contractible, and such that $\cN_\sigma$ contain $P_\tau$ whenever $\tau \in \Sigma_\sigma$, i.e. whenever $P_\sigma \cap P_\tau $ is non-empty.}
\end{equation}

To see that such a cover exists, start by picking an arbitrary cover $\{ \cN_{i} \}_{i=1}^{N}$ by the interiors of integral affine polytopes so that all non-empty intersections are contractible. Then pick a simplicial triangulation $\Sigma$ which is finer than this cover in the sense that the open star of any simplex lies in some $\cN_{i}$. Then define $\cN_\sigma$, for $\sigma \in \Sigma$, to be some choice of $\cN_{i}$ which contains the open star of $\sigma$.

\subsubsection{Definition of the cohomological categories}
\label{sec:descr-cohom-categ}
Let us fix, for each $\sigma \in \Sigma$, a local  Lagrangian section 
\begin{equation} 
\triv_\sigma \co  \cN_\sigma \to X 
\end{equation}
of the projection $X \to Q$. We pick a $\Pin$ structure on $\cN_\sigma$, as in the previous section, as well as a Morse function $f_{\tau, \sigma}$ on the fibre over a basepoint $q_{\tau, \sigma} \in \cN_\sigma \cap \cN_\tau$. Whenever $\tau = \sigma$, we assume for consistency with the previous discussion that this basepoint agrees with $q_{\sigma} \in P_\sigma$. For each pair $(P_0,P_1)$ of polytopes in $\cN_\sigma \cap \cN_\tau $, we then define
\begin{equation}
  CF^*( (\tau, P_0), (\sigma,P_1))  \coloneqq  CF^*( (\cN_\tau, P_0), (\cN_\sigma,P_1)).
\end{equation}

The most important case of the above construction is for the pair $(P_\tau, P_\sigma)$, which allows us to define a category with morphisms
\begin{equation} \label{eq:morphism_category_polytopes}
  \Fuk(\tau, \sigma) \coloneqq
  \begin{cases}
    CF^*( (\tau, P_\tau), (\sigma, P_\sigma)) &  \tau \leq \sigma \\
0 & \textrm{otherwise.}
  \end{cases}
\end{equation}
Letting $\HFuk(\tau, \sigma)$ denote the corresponding cohomology group, we obtain a category $\HFuk$ with compositions given as in Section \ref{sec:morse-theor-prod}. We omit the verification that the associativity conditions hold, as this will follow from the construction of an $A_\infty$ category in Section \ref{sec:homological-algebra}. 
 We denote by $\HFuk_\sigma$ the full subcategory with objects given by $\tau \in \Sigma_\sigma $.

\begin{rem}
  In \cite{Abouzaid2014a}, the product was twisted by an explicit term obtained from a \v{C}ech  cocycle representing the obstruction to the existence of a Lagrangian section of $X \to Q$ which is equipped with a $\Pin$ structure. To see that the construction of this paper is equivalent, note that the obstruction to a consistent trivialisation of the local systems $\lambda_{0,1}$ is exactly $w_2(Q) \in H^2(Q, \bZ_2)$, while resolving the ambiguity in the construction of the local systems $U^P$ over all basepoints corresponds to the choice of a global Lagrangian section of $X \to Q$.

  In order to formulate and prove the twisted version of Theorem \ref{thm:main_thm} discussed in Remark \ref{rem:twisted_categories}, one would choose a \v{C}ech cochain representative of the deformation class, which would give a deformed category $\HFuk^\alpha$ by the methods of \cite{Abouzaid2014a}.
\end{rem}

For each $\sigma \in \Sigma$ and pair of polytopes $P_{0}, P_{1} \subseteq \cN_\sigma$, we also define
\begin{equation} \label{eq:morphisms_Poly_local_category}
  \Poly_\sigma(P_{0}, P_{1}) \coloneqq CF^*((\sigma,P_{0}), (\sigma,P_{1})).  
\end{equation}
Letting $\HPoly_\sigma(P_{0}, P_{1}) $ denote the corresponding Floer cohomology group, we obtain a category $\HPoly_\sigma$, which will be the local (cohomological) category of polytopes associated to an element of $\Sigma$.
\begin{rem}
In Section \ref{sec:from-local-global}, we shall find it useful to redefine $\Poly_\sigma$ to add the choice of a basepoint $q_i \in P_i$ to each object.   The additional choice gives a category with many more objects, but it will be clear from the computations of this section that objects corresponding to different choices of basepoints on the same polytope are quasi-isomorphic.
\end{rem}

The fact that Floer cohomology groups are independent of the all auxiliary choices yields a faithful embedding $\HFuk_\sigma \to  \HPoly_\sigma$.  One way to make this embedding explicit is as follows: Fix a homotopy of sections between the restrictions of  $\triv_{\tau}$ and  $\triv_{\sigma}$ to $\cN_\sigma \cap  \cN_\tau$.  This induces an isomorphism of local systems:
\begin{equation}
  U^{P_\tau}_{\tau} \to  U^{P_\tau}_{\sigma}.
\end{equation}
 Taking the sum of these isomorphisms over all maxima of the Morse function $f_{\sigma_{0}, \sigma_{1}}$, we obtain a \emph{continuation} element
\begin{equation}  \label{eq:continuation_element_change_section}
 \kappa \in   CF^0((\tau,P_{\tau}), (\sigma,P_{\tau}))
\end{equation}
which is closed and whose cohomology class is canonical. Given a pair $(\tau_{0}, \tau_{1})$ of objects of $\HFuk_\sigma$, the left and right  products with the corresponding continuation elements induce a map
\begin{equation} \label{eq:quasi_isomorphism_change_section}
  \Fuk(\tau_{0},\tau_{1}) \cong CF^0((\tau_{0},P_{\tau_{0}}), (\tau_{1},P_{\tau_{1}})) \to  CF^0((\sigma,P_{\tau_{0}}), (\sigma,P_{\tau_{1}})) \cong   \Poly_\sigma(P_{\tau_{0}},P_{\tau_{1}}). 
\end{equation}
Passing to cohomology, we obtain the functor
\begin{equation}
  \HFuk_\sigma \to \HPoly_\sigma.  
\end{equation}
\begin{rem}
The standard way of constructing a map of Morse complexes for different choices of Morse functions is to consider a $1$-parameter family of vector fields interpolating between the two gradients, and counting solutions of the corresponding flow lines. Keeping in mind that a generic point in a manifold lies on a unique negative gradient flow line starting at a maximum, one sees that the count of pairs of gradient trees defining the product is the same count that defines a composition of continuation maps
\begin{equation}
CF^0((\tau_{0},P_{\tau_{0}}), (\tau_{1},P_{\tau_{1}}))  \to  CF^0((\sigma,P_{\tau_{0}}), (\sigma,P_{\tau_{1}})).
\end{equation}
\end{rem}

To show that this is a fully faithful embedding one may construct a continuation element in $ CF^0((\sigma,P_{\tau}), (\tau,P_{\tau}))$ as in Equation \eqref{eq:continuation_element_change_section}, and show that the product
\begin{equation}
  CF^0((\tau,P_{\tau}), (\sigma,P_{\tau})) \otimes_\Lambda   CF^0((\sigma,P_{\tau}), (\tau,P_{\tau})) \to CF^0((\sigma,P_{\tau}), (\sigma,P_{\tau}))
\end{equation}
maps the tensor products of the two continuation elements to the multiplicative unit.

\subsubsection{Computation of morphisms in the category}
\label{sec:comp-morph-categ}

To understand the categories $\HFuk$ and $\HPoly_\sigma$, we summarise some computations established in the Appendices. The first result is a computation for the Floer cohomology groups associated to inclusions: given a point $q$ in $\cN_{\sigma}$, and a critical point $x$ of $f_{\sigma, \sigma}$, we have a natural map
\begin{equation}
 \Lambda[H_{1}( X_q; \bZ) ]   \to \Hom^c_\Lambda(U_{\sigma,x},U_{\sigma,x}).
\end{equation}
If $P_0$ and $P_1$ are polytopes lying in $\cN_{\sigma}$, so that $P_{1} \subseteq  P_{0}$, the above map has a unique continuous extension to a map
\begin{equation} \label{eq:morphism_Gamma_Hom_fibres}
    \Gamma^{P_{1}}  \to \Hom^c_\Lambda(U^{P_0}_{\sigma,x},U^{P_1}_{\sigma,x}).
\end{equation}
The following result immediately follows from the first half Proposition \ref{prop:computation_inclusion-restated} proved in Appendix \ref{sec:constr-bound-null}:
\begin{prop} \label{prop:computation_inclusion}
If $P_{1} \subseteq  P_{0}$, the sum of the homomorphisms \eqref{eq:morphism_Gamma_Hom_fibres} over all minima of $f_{\sigma, \sigma} $ defines a natural quasi-isomorphism
  \begin{equation}
    \Gamma^{P_{1}}  \to CF^*((\sigma,P_{0}), (\sigma,P_{1})).
  \end{equation}
\qed
\end{prop}
Combining this with continuation maps, we have:
\begin{cor} 
If $P_{1} \subseteq  P_{0} \subseteq \cN_\sigma \cap \cN_\tau$, there is an isomorphism 
  \begin{equation}
    \Gamma^{P_{1}}  \to HF^*((\sigma,P_{0}), (\tau,P_{1})).
  \end{equation}
    In particular, the right hand side vanishes except in degree $0$.
\qed
\end{cor}
\begin{rem}
  We remind the reader that the above isomorphism is not natural; it depends on the choice of path between the sections associated to $\sigma$ and $\tau$. In particular, it may not be compatible with products, and this is ultimately the reason for the appearance of the gerbe $\beta$, twisting the derived category of coherent sheaves, in our statement of mirror symmetry.
\end{rem}

The above result implies that the morphism spaces in $\HFuk$ are given by
\begin{equation} \label{eq:morphisms_cohomological_directed}
    \HFuk(\tau, \sigma) =
  \begin{cases}
    \Gamma^{P_\tau} &  \tau \leq \sigma \\
0 & \textrm{otherwise.}
  \end{cases}  
\end{equation}
Indeed, Equation \eqref{eq:morphism_category_polytopes} stipulates that the only morphisms in $ \HFuk$ arise as Floer cohomology groups for pairs of nested polytopes, which agree with affinoid rings by the above result.  In particular, $\HFuk$ is isomorphic to the category denoted $\Fuk$ in \cite{Abouzaid2014a}, allowing us to tie the constructions of the two papers. 

Next, we consider a polytope $P \subseteq \cN_\sigma$ for $\sigma \in \Sigma$, and a cover $\{P_{\alpha}\}_{\alpha \in A}$ of $P$, indexed by a finite ordered set. The natural (restriction) map $\Gamma^{P_\alpha} \to \Gamma^{P_\alpha \cap P_\beta}$ gives rise to a map of local systems $U^{P_\alpha}_\sigma \to U^{P_\alpha \cap P_\beta}_\sigma$, allowing us to form the  \v{C}ech complex
\begin{equation}
\check{U}(P;A) \coloneqq  \bigoplus_{ \alpha_{0} \in A} U^{P_{\alpha_{0}}}_\sigma \to \bigoplus_{ \alpha_{0} < \alpha_{1} \in A} U^{P_{\alpha_{0}} \cap P_{\alpha_{1}}}_\sigma \to \bigoplus_{ \alpha_{0} < \alpha_{1} < \alpha_2 \in A} U^{P_{\alpha_{0}} \cap P_{\alpha_{1}} \cap P_{\alpha_{2}}}_\sigma  \to \cdots
\end{equation}
as a complex of (topological) local systems over $X_{q_\sigma}$. Note that this is a finite direct sum of topological local systems, and thus there is no ambiguity in the construction of the topology on $\check{U}(P;A)$. Moreover, we have a canonical map of local systems
\begin{equation}
 U^{P}_\sigma  \to  \check{U}(P;A)
\end{equation}
given by the restriction to $U^{P_{\alpha}}_\sigma $ for all $\alpha \in A$.

We now consider an enlargement of the category of polytopes, including objects $  \check{T}(P;A)$ which correspond to the complex of local systems $ \check{U}(P;A)$. Explicitly, for each $P' \subseteq \cN_\sigma$, we can  define cochain groups
\begin{align}
    \Poly_\sigma( \check{T}(P;A), P') & \coloneqq CM^{*}(X_{q}, \Hom^c_\Lambda( \check{U}(P;A), U^{P'}_{\sigma}) \otimes \lambda)  \\
 \Poly_\sigma(P',\check{T}(P;A))   & \coloneqq CM^{*}(X_{q}, \Hom^c_\Lambda( U^{P'}_{\sigma}, \check{U}(P;A)) \otimes \lambda)
\end{align}
equipped with the sum of the Morse differential and the internal differential of $\check{U}(P;A)$.

Because $\check{U}(P;A) $ is built in finitely many steps, the Floer complex $   \Poly_\sigma( P', \check{T}(P;A))$ is isomorphic to the complex
\begin{equation}
  \label{eq:decompose_Hom_with_twisted_complex}
   \bigoplus_{ \alpha_{0} \in A}\Poly_\sigma(P',P_{\alpha_{0}})  \to \bigoplus_{ \alpha_{0} < \alpha_{1} \in A} \Poly_\sigma(P',P_{\alpha_{0}})  \to \bigoplus_{ \alpha_{0} < \alpha_{1} < \alpha_2 \in A} \Poly_\sigma(P',P_{\alpha_{0}})   \to \cdots,
\end{equation}
and similarly for $\Poly_\sigma(\check{T}(P;A), P')$. This implies that these complexes are isomorphic to the ones discussed in Appendix \ref{sec:null-homotopy-tates}, where we describe $\check{T}(P;A)$ as a twisted complex in the category $\Poly_\sigma $.  We can now state an immediate consequence of Proposition \ref{prop:cech-complex-acyclic}, which is the version of Tate acyclicity which we shall use for computations:
\begin{lem} \label{lem:tate_acyclic_hom_from_to}
The map from $ U^{P}_\sigma $ to $  \check{T}(P;A)$ is a quasi-isomorphism of complexes of local systems, and induces quasi-isomorphisms
\begin{align}
    \Poly_\sigma( \check{T}(P;A), P') & \to \Poly_\sigma(P, P') \\
\Poly_\sigma(P', P) & \to \Poly_\sigma(P',\check{T}(P;A))  
\end{align}
for all $P' \subseteq \cN_\sigma$. \qed
\end{lem}
The above result allows us to reduce global computations to local computations. To fully make use of locality, we need the fact, proved in Appendix \ref{sec:constr-bound-null}, that the Floer complex $ CF^*((\sigma_{0},P_{0}), (\sigma_{1},P_{1}))$ is acyclic whenever $P_{0}$ and $P_{1}$ are disjoint.
\begin{cor}
  \label{cor:computation_morphisms_local}
Let $(P,P',P'')$ be polytopes contained in $\cN_\sigma$. If the intersections of $P'$ and $P''$  with an open neighbourhood of $P$ agree, there are natural isomorphisms
\begin{align}
  \HPoly_\sigma^*(P,P') & \cong  \HPoly_\sigma^*(P,P'')  \\
   \HPoly_\sigma^*(P',P) & \cong  \HPoly_\sigma^*(P'',P).
\end{align}
\end{cor}
\begin{proof}
By taking the intersection of $P'$ and $P''$, it suffices to prove the result under the assumption that $P'' \subseteq P'$. In this case, we can extend $P''$ to a cover of $P'$ with the property that all other elements of the cover intersect $P$ trivially, hence the Floer cohomology of all other elements of the cover with $P$ vanish. From Equation \eqref{eq:decompose_Hom_with_twisted_complex}, we obtain an isomorphism
\begin{equation}
  \HPoly_\sigma( \check{T}(P;A), P')  \cong \HPoly_\sigma(P,P'').
\end{equation}
  The result then follows from Lemma \ref{lem:tate_acyclic_hom_from_to}.
\end{proof}
By applying the above result twice, we conclude:
\begin{cor}
  The morphisms between objects in $\HPoly_\sigma$ are local in the sense that they only depend on a neighbourhood of the intersection of the corresponding polytopes. \qed
\end{cor}

\subsection{Local cohomological modules}
\label{sec:cohomological-module}

In this section, we shall assign, to each Lagrangian  $L$ in $\cA$, left and right modules over the categories $\HPoly_\sigma$, and use Tate acyclicity to compute these modules whenever $L$ is Hamiltonian isotopic to a Lagrangian meeting $X_{q_\sigma}$ at a point.

We begin by imposing a condition which can be achieved by a small Hamiltonian perturbation of $L$:
\begin{equation}
  \label{eq:isolated_intersections}
\parbox{35em}{the intersection of $L$ with every fibre $X_q$ is contained in the interior of a disjoint union of closed balls.}
\end{equation}
\begin{rem}
In fact, after a generic perturbation, the intersection of $L$ with any fibre is finite: to prove this, observe that the problem is equivalent to proving that the fibres of the projection to the base are generically finite. The study of projection of Lagrangian submanifolds to the base of Lagrangian fibration is the subject of Lagrangian singularities \cite{Arnold1990}. We shall only need a very small part of the theory: by a generic Hamiltonian perturbation, one may achieve transversality for the projection map $L \to Q$ with respect to any jet condition. This implies that, at any point and in any direction, there is some sufficiently high derivatives of the projection map which does not vanish, and hence that the intersection with all fibres are isolated.  
\end{rem}

For the remainder of the paper, we fix the disjoint union of balls appearing in Equation \eqref{eq:isolated_intersections}, assuming further that the distance between these balls is bounded above uniformly over $q \in Q$. For each element $L$ of $\cA$, we fix a compact codimension $0$ submanifold $\nu_X L \subset X$  of $L$ so that the intersection $\nu_X L \cap X_{q} $ is a finite union of path connected components which lie in this union of balls (the fact that this is a finite union is a consequence of the much stronger fact that the fibres of a generic smooth map are triangulable, which is the main result of \cite{Verona1984}).  We denote by $\scrJ_L$ the space of tame almost complex structures which agree with $J_L$ away from a fixed compact subset of the interior of $\nu_X L$, where $J_L$ is the almost complex structure with respect to which we have assumed that $L$ does not bound any non-constant holomorphic disc.  

We now impose constraints on the diameter of $\cN_\sigma$, which will be essential in ensuring that, for each $\sigma \in \Sigma$,  we can define a Floer cohomology group for pairs $L \in \cA$ and $P \in \cN_\sigma$. Consider the quotient $X_{q_\sigma}/{\sim}$ by the equivalence relation which collapses each %
component of $X_{q_\sigma} \cap \nu_X L$ to a point. Since $\nu_X L$ is compact, the quotient space is compact and Haussdorff.

We shall presently see that a Riemannian metric on $ X_{q_\sigma}$ induces a length metric on $X_{q_\sigma}/{\sim} $.  We recall that the length of a path with target any metric space may be defined as the supremum of the distance between successive points in a partition of the domain interval (see e.g. \cite[Definition I.1.18]{BridsonHaefliger1999}). A \emph{length metric} is one with the property that the distance between two points is given by the infimum of the lengths of paths between them \cite[Definition I.3.1]{BridsonHaefliger1999}. Since the spaces we shall consider are compact, the Hopf-Rinow theorem holds and the infimum is achieved \cite[Proposition I.3.7]{BridsonHaefliger1999}.

\begin{lem} \label{lem:good-metric-quotient}
Let $M$ be a compact Riemannian manifold, $K$ a compact subset all of whose components are path components, and $\sim$ the equivalence relation given by collapsing each component of $K$ to a point. There is a length metric on $X/{\sim}$, determined by the property that, for each path $\gamma \co I \to X$, the length $\ell(\gamma/\sim)$ of the image in $X/{\sim}$ is given by the sum of the lengths of the subpaths which lie away from $K$.
\end{lem}
\begin{proof}
  Writing $ \chi_{K}$ for the characteristic function of $K$ (which vanishes on $K$ and is unity away from it), we can write the formula for the length as
  \begin{equation}
     \ell(\gamma/\sim) =   \int_{0}^{1}  |\frac{d \gamma}{dt}| \cdot (1 - \chi_{K}(\gamma(t))) dt.
  \end{equation}
The main claim is therefore that, for points $x$ and $y$ in $X/{\sim} $ the formula
      \begin{equation}
    d(x,y) \coloneqq \inf_{\gamma \co [0,1] \to M} \ell(\gamma/\sim) 
  \end{equation}
defines a metric, where the infimum is taken over all paths whose endpoints project to $x$ and $y$. Symmetry is obvious, and the fact that $d$ separates points follows from the assumption that $K$ is closed, hence the Hausdorff distance between any two components is strictly positive. The triangle inequality follows from the assumption that each component is path connected, since this allows us to connect any two paths with endpoints in a component $K_y$ via a path lying entirely in $K_y$, and this procedure does not change the length of the projection as we have defined it.
\end{proof}

We now return to our specific setting, in which $M = X_{q_\sigma}$ and $K$ is the intersection with $\nu_X L$.  Using the fact that the quotient of a manifold by a ball does not change the homeomorphism type and the inclusion of the components of $\nu_X L \cap X_{q_\sigma}$ in a disjoint union of balls, we obtain a retraction
\begin{equation}
   X_{q_\sigma}  \to  X_{q_\sigma}/{\sim} \to X_{q_\sigma}
 \end{equation}
 so the map on first homology groups
\begin{equation}
H_{1}(    X_{q_\sigma}; \bZ)  \to H_{1}( X_{q_\sigma}/{\sim}; \bZ)
\end{equation}
is injective. The metric on $ X_{q_\sigma}/{\sim} $ equips the right hand side with a norm given by the minimal length of a representative.  The first homology of $  X_{q_\sigma} $ inherits a norm satisfying the following property: for any loop $\gamma$ in $X_{q_\sigma}$ the length $ \ell(\gamma/{\sim})$ of the projection is greater than or equal to the norm $|[\gamma]|$ of the associated homology class.

We can therefore equip $X_{q_\sigma}/{\sim}$ with the structure of a metric space so that the norm $|[\gamma]|$ of the homology class of any loop in $X_{q_\sigma}$ is bounded above by the length $ \ell(\gamma/{\sim})$ of the projection.

For each $J_L$-holomorphic curve $u$ from a strip to $ X$, with  boundary conditions on a fixed compact subset of $\nu_X L$ along $t=0$, and $X_{q_\sigma}$ at $t=1$ (the coordinates on the strip  $\bR \times [0,1]$ are  $(s,t)$), we consider the energy $E(u)= \int u^* \omega$ and the length  $\ell(\partial u/{\sim}) $ of the projection  to $X_{q_\sigma}/{\sim}$ of the boundary component of the strip labelled $X_{q_\sigma}$. 
According to Corollary \ref{cor:isoperimetric_constant_lagrangian_quotient}, we may choose a constant $C$ independent of $u$ such that this length is bounded by $C E(u)$. We require that
\begin{equation} \label{eq:condition_convergence-basic}
\parbox{35em}{the diameter of $\cN_\sigma$ is bounded by $1/ 4 \kk C$.}
\end{equation} 
By construction, we note that the constant $C$ depends on the collection $\cA$ of chosen Lagrangians and the complex structure $J_L$ (as well at the fibration $X \to Q$).
\begin{rem}
This is the first of many places where we impose a condition on the diameters of the covers $\{ \cN_\sigma \}$ and $\{ P_\sigma \}$, and we will repeatedly be imposing conditions of this nature. The only difficulty with this idea is that the reverse isoperimetric inequality depends on the choice of Lagrangian boundary conditions, and that imposing further conditions on the cover entails changing which fibres are associated to elements of the cover (since we require $q_\sigma \in \cN_\sigma$). As discussed in the introduction, the solution implemented in Section \ref{sec:famil-cont-equat} is to establish a reverse isoperimetric inequality for compact families of choices, and then pick a cover which is sufficiently fine with respect to this uniform constant. In the case of the condition in Equation \eqref{eq:condition_convergence-basic}, we note that the constant $C$ can be bounded from the distance between the balls which cover $X_{q} \cap \nu_X L$, minimised over all $q \in Q$.

The reason that the inductive procedure can ultimately be implemented is that the ordering of choices proceeds by increasing complexity: the algebraic constructions we perform at any given stage depend on reverse isoperimetric constants for some moduli space of holomorphic curves and all moduli spaces which appear in its boundary; these constants dictate how fine the cover has to be for the given algebraic construction to converge. Later constructions consider different moduli spaces of holomorphic curves which may include previous ones in their boundary, and simply impose additional constraints on the cover.
\end{rem}

In order to construct a module over a field of characteristic different from $2$, we assume that each  $L \in \cA$ is equipped with a $\Pin$ structure. In order for the categories we construct to be $\bZ$ graded, we use the fact that the (canonical up to homotopy) trivialisation of the square of the top exterior power of $TQ$ induces a trivialisation of the square of the top (complex) exterior power of $TX$ with respect to any compatible almost complex structure, i.e. a complex quadratic volume form. We then assume that the Lagrangians are graded  with respect to the chosen quadratic volume form on $X$ (in the sense of \cite{Seidel2000}). Note that all fibres  $X_q$ are canonically graded in this sense.

\begin{rem}
In the twisted setting of Remark \ref{rem:twisted_categories}, one would equip $L$ with the data of a $\Pin$ structure relative the pull-back of a vector bundle on the $3$-skeleton of $Q$, which represents the given class in $H^2(Q; \bZ_2)$. We would also add to our assumptions the vanishing of the pullback of the chosen class in $ H^2(Q; \Lambda_+)$ to $L$.
\end{rem}

\subsubsection{Floer complexes between polytopes and Lagrangians}
\label{sec:floer-compl-pairs}

In order to construct the modules associated to elements of $\cA$, we begin by requiring that
\begin{equation} \label{eq:transversality_Lagrangian_pair}
  \parbox{33em}{for each  $\sigma \in \Sigma$  and $L \in \cA$, the Lagrangian pair $(X_{q_\sigma}, L)$ is transverse.}  
\end{equation}
We note that this condition can be achieved by an arbitrarily small Hamiltonian perturbation, in particular by a perturbation taking place in $\nu_X L$, and that the condition that $L$ be tautologically unobstructed is preserved by such a perturbation, since we may replace $J_L$ by its pushforward under the chosen Hamiltonian isotopy. Let  $\Crit(\sigma,L)$ and $\Crit(L,\sigma)$ denote the intersection of $L$ with $X_{q_\sigma}$. 
The chosen $\Pin$ structures on $X_{q_\sigma}$ and $L$ give rise to an assignment $\delta_x$ of a $\bZ_2$-graded free abelian group of rank $1$ associated to each element $x$ of $ \Crit(\sigma,L)$ or $  \Crit(L,\sigma) $. 
 The canonical grading of the fibre with respect to the standard quadratic complex volume form on $X$, together with a choice of grading on $L$ determine a $\bZ$-grading on $\delta_x$.

With this in mind, we define, for each $P \subseteq \cN_\sigma $  the Floer complex
\begin{align} \label{eq:complex_L_to_P}
CF^*(L,(\sigma,P)) & \coloneqq \bigoplus_{x \in \Crit(\sigma,L)} U^P_{\sigma,x} \otimes \delta_{x}. 
\end{align}
 A choice of paths connecting the endpoints of orbits to a basepoint on $X_{q_\sigma}$ induces a (complete) norm on these complexes, and the corresponding topology is independent of choice.

In order to define the differential, pick a family of almost complex structures
\begin{equation}
J(L,\sigma) \co [0,1] \to \scrJ_L
\end{equation}
which restricts at $0$ to $J_L$.  %
We obtain a moduli space $\Rbar(L,\sigma)$ of finite energy $J(L,\sigma)$ stable holomorphic strips with boundary conditions $L$ along $t=0$, and $X_{q_\sigma}$ at $t=1$ (the coordinates on the strip  $\bR \times [0,1]$ are  $(s,t)$).

There is a natural evaluation
\begin{equation}
\Rbar(L,\sigma)  \to   \Crit(\sigma,L)\times  \Crit(\sigma,L)
\end{equation}
given by the asymptotic conditions at $\pm \infty$. We denote the fibre over $(x_{0},x_{1})$ by $\Rbar(x_{0},x_{1}) $. Choosing $J(L,\sigma) $ generically ensures that this is a topological manifold of dimension 
\begin{equation}
  \dim \Rbar(x_{0},x_{1}) = \deg(x_{0}) - \deg(x_2) - 1,
\end{equation}
whose boundary is covered by codimension-$1$ strata corresponding to breaking of strips:
\begin{equation}
  \partial \Rbar(x_{0},x_{1}) = \bigcup_{x \in \Crit(\sigma,L)} \Rbar(x_{0},x) \times \Rbar(x,x_{1}).
\end{equation}
The output of Floer theory is that, whenever $\Rbar(x_{0},x_{1}) $ has dimension $0$, each element induces a map:
\begin{equation}
  \delta_{u} \co \delta_{x_{1}} \to \delta_{x_{0}}. 
\end{equation}

In order to define the differential in Equation \eqref{eq:complex_L_to_P}, we recall from Section \ref{sec:loops-paths-local} that a path from $x_{0}$ to $x_{1}$ induces a parallel transport map from $ U_{\sigma,x_{1}}^{P}$ to $ U_{\sigma,x_{0}}^{P}$. The boundary of an element  $u \in \Rbar(x_{0},x_{1}) $  gives rise to such a path which we denote $\partial u$, so that we have a parallel transport map $z^{[\partial u]}$ as in Equation \eqref{eq:parallel_transport_map}.
 For the statement of the next result, we recall that the energy $E(u)$ of a holomorphic curve is its area.
\begin{lem} \label{lem:bound_for_strips}
Assuming that the cover satisfies Condition \ref{eq:condition_convergence-basic}, there is a constant $A$, independent of $u$, such that whenever $P \subseteq \cN_\sigma$, we have
\begin{equation} \label{eq:energy_plus_valuation_positive}
E(u) +  \val_{P} z^{[\partial u]} \geq E(u)/2 + A.  
\end{equation}
\end{lem}
\begin{proof}
It suffices to bound $\val_{q} z^{[\partial u]}$ for any $ q \in \cN_\sigma$. The condition that the distortion is bounded by $2$ implies that the image of $P$ in $T_{q}X$ is contained in the ball of radius $2 \diam \cN_\sigma$. Thus
\begin{equation}
  \val_{q} z^{[\partial u]} \geq - 2 \diam \cN_\sigma  |[\partial u ]| \geq - \frac{|[\partial u]|}{2C},
\end{equation}
where the second inequality follows from Equation \eqref{eq:condition_convergence-basic}. At the cost of introducing an additive constant, the choice of metric fixed in the discussion preceding \eqref{eq:condition_convergence-basic} allows us to replace $|[\partial u ]| $ by the length $\ell(\partial u / {\sim})$. We can then apply the reverse isoperimetric inequality: the key point is that, according to  Corollary \ref{cor:isoperimetric_constant_lagrangian_quotient} the reverse isoperimetric inequality for $J_L$-holomorphic with boundary conditions on $\nu_X L$ and $X_{q_\sigma}$ applies (with the same constant) to $J(L,\sigma)$ holomorphic curves, because the two almost complex structures agree by assumption away from a fixed compact subset of the interior of $\nu_X L$. The result thus follows.
\end{proof}
The bound in Equation \eqref{eq:energy_plus_valuation_positive}, which can of course be simplified, should be read as follows: the right hand side is the valuation of the operation $
 T^{E(u)} z^{[\partial u]} $, and the left hand side goes to $+\infty$ as the energy of $u$ goes to infinity. From Gromov compactness, and as in Corollary \ref{cor:convergence_if_coefficients_bounded_lengths}, we therefore conclude:
\begin{cor} \label{cor:convergence_differential}
For each pair $(x_{0},x_{1})$ of intersection points, the expression
\begin{equation}
 \sum_{u \in  \Rbar_{q}(x_{0},x_{1})} T^{E(u)} z^{[\partial u]}
\end{equation}
converges with respect to the topology defined by $P$, hence induces a map from $  U_{\sigma,x_{1}}^{P} $ to  $  U_{\sigma,x_{0}}^{P}$.  \qed
\end{cor}

We conclude that, for each pair $(x_{0},x_{1})$ of intersection points, the expression
\begin{equation}
\partial_{x_{0},x_{1}} \coloneqq  \sum_{u \in  \Rbar_{q}(x_{0},x_{1})} T^{E(u)} z^{[\partial u]} \otimes \delta_{u}
\end{equation}
gives a well-defined map from $  U_{\sigma,x_{1}}^{P} \otimes \delta_{x_{1}} $ to  $  U_{\sigma,x_{0}}^{P} \otimes \delta_{x_{0}}$. 

\begin{defin}
The differential on $CF^*(L,(\sigma, P)) $ is given by 
\begin{equation} \label{eq:differential_left_module}
 \bigoplus_{x_{0}} \sum_{x_{1}} (-1)^{\deg(x_{1})+1}  \partial_{x_{0},x_{1}}.
  \end{equation}
\end{defin}
By the previous discussion, this differential is a continuous operator with respect to the natural topology on Floer complexes (i.e. bounded with respect to the norm induced by a choice of homotopy classes of paths to the basepoint). 
\begin{rem}
Note that the sign in Equation \eqref{eq:differential_left_module} differs by one from the sign in Equation \eqref{eq:Floer_complex_2_polytopes}. The reason for this choice is that the sign conventions for modules are more intuitive if they are based on unreduced gradings.
\end{rem}

Reversing the r\^oles of the Lagrangian and the polytope,  we construct a complex
\begin{equation} \label{eq:complex_P_to_L}
CF^*((\sigma, P),L)  \coloneqq  \bigoplus_{x \in \Crit(\sigma,L)} \Hom^c_\Lambda(U^P_{\sigma,x} , \Lambda) \otimes \delta_{x^{q}}, 
\end{equation}
using moduli spaces of holomorphic strips $\Rbar(\sigma,L)$ with boundary $X_{q_\sigma}$ along $t=0$ and $L$ along $t=1$. To simplify the discussion later, we define $\Rbar(\sigma,L) $ using the $t$-dependent family of almost complex structure
\begin{equation}
  J(\sigma,L) (t) = J(L,\sigma)(1-t).  
\end{equation}
Following the above procedure, and using parallel transport along the boundary arc labelled $X_{q_\sigma}$, we obtain the differential on $ CF^*((\sigma, P),L) $. Given the choices we have made, this complex is naturally isomorphic to the $\Lambda$-linear dual of $CF^*(L,(\sigma, P)) $.

\subsubsection{Modules over the local categories}
\label{sec:left-right-modules-1}

For each $\sigma \in \Sigma$, $L \in \cA$, and $P \subseteq \cN_\sigma$, we define 
\begin{align}
\cL^*_{L,\sigma}(P) & \coloneqq  CF^*(L,(\sigma,P)) \\
\scrR_{L,\sigma}^*(P) & \coloneqq   CF^*((\sigma, P),L).
\end{align}
The differentials  $ \mu^{1|0}_{\cL_{L,\sigma}} $ and $\mu^{0|1}_{\scrR_{L,\sigma}}$  are given by Equation \eqref{eq:differential_left_module} and its analogue when the boundary conditions are reversed.

We now construct the module structure on $H\cL_{L,\sigma} $, i.e. the left action of morphism spaces in $\HPoly_\sigma$. The construction of the right module action is entirely similar, as we shall explain at the end. 

\begin{figure}[h]
  \centering
  \begin{tikzpicture}

\draw[line width=2*\lw,dotted] (-6*\mw pt,1*\mw pt) -- ( -7*\mw pt, 1*\mw pt);
\draw[line width=2*\lw,dotted] (-6*\mw pt,-1*\mw pt) -- ( -7*\mw pt, -1*\mw pt);
\draw[line width=2*\lw,dotted] (6*\mw pt,1*\mw pt) -- ( 7*\mw pt, 1*\mw pt);
\draw[line width=2*\lw,dotted] (6*\mw pt,-1*\mw pt) -- ( 7*\mw pt, -1*\mw pt);
\draw[line width=2*\lw] (-6.5*\mw pt,1*\mw pt) -- ( 6.5*\mw pt, 1*\mw pt);
\draw[line width=2*\lw] (-6.5*\mw pt,-1*\mw pt) -- ( 6.5*\mw pt, -1*\mw pt);

\coordinate [label=above:$X_{q_\sigma}$] () at (2*\mw pt, 1*\mw pt);
\coordinate [label=above:$X_{q_\sigma}$] () at (-2*\mw pt, 1*\mw pt);
\coordinate [label=below:$L$] () at (0, -1*\mw pt);
\draw[line width=4*\lw] (0,1*\mw pt +1/8*\mw pt) -- (0,1*\mw pt -1/8*\mw pt);
\draw[line width=4*\lw]  [->] (0, 1*\mw pt) -- (0,3*\mw pt);
\draw[line width=4*\lw] (0,3*\mw pt) -- (0, 5*\mw pt);
\coordinate [label=left:$\nabla f_{\sigma,\sigma}$] () at (0, 3*\mw pt) ;
\end{tikzpicture}
  \caption{ An element of the moduli space $\RTbar( L, \sigma, \sigma) $ defining the left module over $\HPoly_\sigma $. }
  \label{fig:action_left_module}
\end{figure}

Define $ \Rbar(L,\sigma,\sigma)  $ to be the moduli space of strips in $\Rbar(L,\sigma)$, with an additional marked point along the segment mapping to $X_{q_\sigma}$.  We have a natural evaluation map
\begin{equation}
\Rbar(L,\sigma,\sigma) \to \Crit(L,\sigma) \times X_{q_\sigma} \times \Crit(L,\sigma).
\end{equation}
Let $\RTbar(L,\sigma,\sigma)$ be the fibre product over $X_{q_\sigma}$ of $\Rbar(L,\sigma,\sigma)$  with the space $\Tbar_+(\sigma,\sigma)$ of (perturbed) positive half-gradient flow lines of the function $f_{\sigma,\sigma}$.  We shall call such moduli spaces \emph{mixed moduli spaces,} as they consist of gradient flow lines and pseudo-holomorphic discs with matching evaluation maps to Lagrangians in $X$. Their use is standard in constructions combining Morse and Floer theory, going all the way back to \cite{Fukaya1997b}.

We have a natural evaluation map
\begin{equation} \label{eq:mixed_moduli_space_}
\RTbar(L,\sigma,\sigma) \to \Crit(L,\sigma) \times \Crit(\sigma,\sigma) \times \Crit(L,\sigma),
\end{equation}
with the ordering given counterclockwise around the boundary starting at the outgoing end. Denote by $E$ the set of ends and marked points, which we decompose into $E^{\inp} = \{e_{1}, e_2\}$ and $E^{\out} = \{ e_{0}\}$, with $E^{\inp}$ consisting of the incoming (positive) end and the marked point, and $E^{\out}$ consisting of the singleton output.  We denote the fibre at a triple $x = \{ x_e \}_{e \in E}$ in the left hand side of Equation \eqref{eq:mixed_moduli_space_} by $ \RTbar (x) $. By parallel transport along the gradient flow line and the part of the boundary mapping to $X_{q_\sigma}$, we obtain a map
\begin{equation} \label{eq:map_parallel_transport_left_modules}
 \Hom^c_\Lambda(U_{\sigma,x_{e_2}}, U_{\sigma,x_{e_2}}) \otimes_\Lambda  U_{\sigma,x_{e_{1}}} \to  U_{\sigma,x_{e_{0}}}.
\end{equation}

For generic choices of perturbations, the fibre product defining $\RTbar (x) $ is transverse, so that it is a topological manifold with boundary of dimension
\begin{equation}
\deg(x_{e_{0}}) - \sum_{e \in E^\inp}  \deg(x_e),
\end{equation}
and which is naturally oriented relative
\begin{equation}
 \delta_{x_{e_{0}}}^{\vee} \otimes  \bigotimes_{e \in E^\inp}  \delta_{x_e}.
\end{equation}
A rigid element  $ u \in \RTbar(\Lab) $ thus induces a map
\begin{equation} \label{eq:map_product_orientation_lines}
\bigotimes_{e \in E^\inp}  \delta_{x_e} \to  \delta_{x_{e_{0}}}.
\end{equation}

Given a triple $\Lab = (L, P_{1}, P_{2})$, with $P_{1}$ and $P_{2}$ contained in $\cN_\sigma$, we combine Equations \eqref{eq:map_parallel_transport_left_modules} and \eqref{eq:map_product_orientation_lines} to obtain a map
\begin{equation}
\Poly_\sigma(P_{1},P_{2})  \otimes_\Lambda    \cL_{L,\sigma}(P_{1})  \to \cL_{L,\sigma}(P_{2})
\end{equation}
which, upon twisting by $(-1)^{\deg(x_{1})+1}$, gives rise to the structure map $ \mu^{1|1}_{\cL_{L,\sigma}} $.  Passing to cohomology, we obtain
\begin{equation}
\HPoly_{\sigma}(P_{1}, P_{2}) \otimes_\Lambda  H \cL_{L,\sigma}(P_{1}) \to   H \cL_{L,\sigma}(P_{2}). 
\end{equation}

The construction of the right module map proceeds as follows:  we construct a moduli space $\Rbar(\sigma,\sigma,L)$ by considering strips in $\Rbar(\sigma,L)$ with an additional marked point on the boundary with label $X_{q_\sigma}$, then define $\RTbar(\sigma,\sigma,L)$ to be the fibre product over $X_{q_\sigma}$ with the moduli space of semi-infinite gradient flow lines of $f_{\sigma,\sigma}$ (see Figure \ref{fig:action_right_module}).

\begin{figure}[h]
  \centering
  \begin{tikzpicture}

\draw[line width=2*\lw,dotted] (-6*\mw pt,-1*\mw pt) -- ( -7*\mw pt, -1*\mw pt);
\draw[line width=2*\lw,dotted] (-6*\mw pt,1*\mw pt) -- ( -7*\mw pt, 1*\mw pt);
\draw[line width=2*\lw,dotted] (6*\mw pt,-1*\mw pt) -- ( 7*\mw pt, -1*\mw pt);
\draw[line width=2*\lw,dotted] (6*\mw pt,1*\mw pt) -- ( 7*\mw pt, 1*\mw pt);
\draw[line width=2*\lw] (-6.5*\mw pt,-1*\mw pt) -- ( 6.5*\mw pt, -1*\mw pt);
\draw[line width=2*\lw] (-6.5*\mw pt,1*\mw pt) -- ( 6.5*\mw pt, 1*\mw pt);

\coordinate [label=below:$X_{q_\sigma}$] () at (2*\mw pt, -1*\mw pt);
\coordinate [label=below:$X_{q_\sigma}$] () at (-2*\mw pt, -1*\mw pt);
\coordinate [label=above:$L$] () at (0, 1*\mw pt);
\draw[line width=4*\lw] (0,-1*\mw pt -1/8*\mw pt) -- (0,-1*\mw pt +1/8*\mw pt);
\draw[line width=4*\lw]  [->] (0, -1*\mw pt) -- (0,-3*\mw pt);
\draw[line width=4*\lw] (0,-3*\mw pt) -- (0, -5*\mw pt);
\coordinate [label=left:$\nabla f_{\sigma,\sigma}$] () at (0, -3*\mw pt) ;
\end{tikzpicture}
  \caption{ An element of the moduli space $\RTbar(\sigma, \sigma, L) $ defining the right module over $\HPoly_\sigma $. }
  \label{fig:action_right_module}
\end{figure}

The count of rigid elements of these moduli spaces defines a map $ \mu^{1|1}_{\scrR_{L,\sigma}}$ on Floer cochains
\begin{equation}
 \scrR_{L,\sigma}(P_{-1})  \otimes_\Lambda  \Poly_\sigma (P_{-2},P_{-1})    \to \scrR_{L,\sigma}(P_{-2})
\end{equation}
for each pair of polytopes $(P_{-1}, P_{-2})$ in $\cN_\sigma$. At the level of cohomology, we obtain the desired map:
\begin{equation}
H \scrR_{L,\sigma}(P_{-1}) \otimes_\Lambda  \HPoly_{\sigma}(P_{-2}, P_{-1}) \to   H \scrR_{L,\sigma}(P_{-2}). 
\end{equation}

\subsubsection{Computation of the module structure over $\HFuk_\sigma$}
\label{sec:comp-module-struct}

Equation \eqref{eq:morphisms_cohomological_directed} gives a particularly simple description of the category $\HFuk_\sigma$, with morphisms given by affinoid algebras. Identifying $\Gamma^{P}$ with the completion of the homology of the based loop space of $X_{q_\sigma}$, the construction of the Floer complexes yields natural maps
\begin{align} 
\Gamma^{P'}  \otimes_\Lambda  \cL_{L}(P) & \to  \cL_{L}(P') \\
\scrR_{L}(P') \otimes_\Lambda  \Gamma^{P'} & \to \scrR_{L}(P),
\end{align}
whenever $P' \subseteq P$, arising from the map
\begin{equation} \label{eq:local_system_module_based_loop}
U^{P}_{\sigma,x} \otimes_\Lambda  \Gamma^{P'} \to U^{P}_{\sigma,x} \otimes_{\Gamma^P}  \Gamma^{P'} \cong    U^{P'}_{\sigma,x}.  
\end{equation}
Passing to cohomology, we conclude that the groups $H \cL_{L,\sigma}(P_{\tau}) $  and $H \scrR_{L,\sigma}(P_{\tau}) $ give rise to modules over $\HFuk_\sigma$. This construction, which  does not use any gradient trees or holomorphic discs in the module structure maps, was used in \cite{Abouzaid2014a}.  In this section, we show:
\begin{lem} \label{lem:pullback_modules_agree}
The pullbacks of $H \cL_{L,\sigma} $ and $H \scrR_{L,\sigma}  $ under the inclusion of $\HFuk_\sigma$ in $\HPoly_\sigma$ are naturally isomorphic to the modules constructed from Equation \eqref{eq:local_system_module_based_loop}.
\end{lem}

The key point is the following result:
\begin{lem} \label{lem:only_constant_contribute_to_local_modules}
All contributions to the maps
\begin{align}
\Poly^0(P_{1},P_{2})  \otimes_\Lambda    \cL^k_{L} (P_{1}) &  \to    \cL^k_{L} (P_{2})    \\
\scrR^k_{L} (P_{-1}) \otimes_\Lambda  \Poly^0(P_{-2},P_{-1})  &  \to \scrR^k_{L} (P_{-2}) 
\end{align}
are given by configurations whose holomorphic component is a constant strip.
\end{lem}
\begin{proof}
By construction, the holomorphic component of such an element is a strip with endpoints $x$ and $y$ of equal Maslov index, hence has Fredholm index $0$. However, the fact that the almost complex structure on the strip is translation invariant and our assumption that the moduli space is regular implies that the minimal Fredholm index of a non-constant disc is $1$.
\end{proof}
The identification between the two ways of constructing left and right modules can be derived as follows:
\begin{proof}[Proof of Lemma \ref{lem:pullback_modules_agree}]
 By Lemma \ref{lem:only_constant_contribute_to_local_modules}, if $\max x$ denotes the unique maximum of $f_{\sigma,\sigma}$ which is the endpoint of a (perturbed) gradient flow line starting at $x$, the left and right module actions are given by the sums of the maps
\begin{align}
& \Poly^0(P_{1},P_{2})  \otimes_\Lambda    \cL^k_{L} (P_{1})  \to \\ \notag
& \qquad \bigoplus_{x \in \Crit(L,P)} \Hom^c_\Lambda(U^{P_{1}}_{\sigma,\max x}, U^{P_{2}}_{\sigma,\max x}) \otimes_\Lambda   U^{P_{1}}_{\sigma,x}  \to  U^{P_{2}}_{\sigma,x}    \\ 
&   \scrR^k_{L} (P_{-1}) \otimes_\Lambda  \Poly^0(P_{-2},P_{-1})   \to \\ \notag
& \qquad  \bigoplus_{x \in \Crit(L,P)} \Hom^c_\Lambda(U^{P_{-1}}_{\sigma,x}, \Lambda) \otimes_\Lambda  \Hom(U^{P_{-2}}_{\sigma,\max x}, U^{P_{-1}}_{\sigma,\max x})   \to \Hom^c_\Lambda(U^{P_{-2}}_{\sigma,x}, \Lambda)
\end{align}
where the second map (in both cases) is induced by parallel transport along the flow line from $x$ to $\max x$.  Using the inclusion $\Gamma^{P_{2}} \subseteq \Poly^0(P_{1},P_{2})  $, we obtain the desired result.
\end{proof}

We end this section by noting that the above discussion, together with the result stated in Proposition \ref{prop:computation_inclusion}, implies:
\begin{cor} \label{cor:left_module_action_obvious}
If $P_{2} \subseteq P_{1}$, the pullback of $ \cL^k_{L} (P_{1})  $ to $\Gamma^{P_{2}} \subseteq   \Poly^0(P_{1},P_{2}) $ is a free module of rank equal to the number of elements of $\Crit(L,\sigma)$ of degree $k$. \qed
\end{cor}

\subsubsection{Computation for sections}
\label{sec:comp-dimens-1}

In this section, we prove that Equation \eqref{eq:local_diagonal_tensor_directed_iso} is an isomorphism  for sections, i.e.:
\begin{prop}
  \label{prop:tensor_Fuk_iso_projective}
If $L$ meets the fibre over $q_\sigma$ transversely at a single point, the map
\begin{equation}
 j^* \Delta_{\HPoly_\sigma}(\sigma, \_)  \otimes_{\HFuk_\sigma} j^*\left( H \cL_{L,\sigma}  \right) \to j^* H \cL_{L,\sigma}(\sigma)
\end{equation} 
is an isomorphism.
\end{prop}

The tensor product on the left of Equation \eqref{eq:local_diagonal_tensor_directed_iso} is the cokernel of the map
\begin{equation}
  \bigoplus_{ \rho_{0} \leq \rho_{1} \in \Sigma_\sigma} \HPoly_\sigma(P_{\rho_{1}}, P_\sigma) \otimes_\Lambda  \HFuk(P_{\rho_{0}}, P_{\rho_{1}})    \otimes_\Lambda  H\cL_{\sigma}(P_{\rho_{0}})    \to \bigoplus_{ \rho_{0}\in \Sigma_\sigma}  \HPoly_\sigma (P_{\rho_{0}}, P_\sigma)  \otimes_\Lambda  H\cL_{\sigma}(P_{\rho_{0}}).
\end{equation}
If we restrict to $\rho_{0} = \rho_{1} $ in the above complex, the cokernel computes the tensor product  over $ \Gamma^{P_{\rho_{0}}} =  \HFuk(P_{\rho_{0}}, P_{\rho_{0}})  $  of the modules $\HPoly_\sigma(P_{\rho_{0}}, P_\sigma)$ and $ H\cL_{L,\sigma}(P_{\rho_{0}})   $. Thus the cokernel is the same as that of the map
\begin{equation}
    \bigoplus_{ \rho_{0} < \rho_{1} \in \Sigma_\sigma}\HPoly_\sigma(P_{\rho_{1}}, P_\sigma) \otimes_\Lambda  \HFuk(P_{\rho_{0}}, P_{\rho_{1}})  \otimes_\Lambda  H\cL_{\sigma}(P_{\rho_{0}})       \to \bigoplus_{ \rho_{0}\in \Sigma_\sigma} \HPoly_\sigma(P_{\rho_{0}}, P_\sigma) \otimes_{ \Gamma^{P_{\rho_{0}}} } H\cL_{L,\sigma}(P_{\rho_{0}}) .
\end{equation}
The above map factors through the direct sum of the surjections
\begin{equation}
 \HPoly_\sigma(P_{\rho_{1}}, P_\sigma) \otimes_\Lambda  \HFuk(P_{\rho_{0}}, P_{\rho_{1}})  \otimes_\Lambda  H\cL_{\sigma}(P_{\rho_{0}})   \to   \HPoly_\sigma(P_{\rho_{1}}, P_\sigma) \otimes_{\Gamma^{P_{\rho_{1}}}} \HFuk(P_{\rho_{0}}, P_{\rho_{1}})  \otimes_{\Gamma^{P_{\rho_{0}}}} H\cL_{\sigma}(P_{\rho_{0}}) 
\end{equation}
Since  $ \HFuk(P_{\rho_{0}}, P_{\rho_{1}})   $ is a free rank-$1$ module over $ \Gamma^{P_{\rho_{1}}} $, the flatness of the map of rings $\Gamma^{P_{\rho_{0}}} \to  \Gamma^{P_{\rho_{1}}}$  (see e.g. \cite[Lemma 8.6]{Tate1971}) yields an isomorphism
\begin{equation}
\HFuk(P_{\rho_{0}}, P_{\rho_{1}})  \otimes_{\Gamma^{P_{\rho_{0}}}} H\cL_{\sigma}(P_{\rho_{0}})  \to   H\cL_{\sigma}(P_{\rho_{1}}).
\end{equation}
It thus suffices to compute the cokernel of the map
\begin{equation}
  \bigoplus_{ \rho_{0} < \rho_{1} \in \Sigma_\sigma}  \HPoly_\sigma(P_{\rho_{1}}, P_\sigma)  \otimes_{\Gamma^{P_{\rho_{1}}} } H\cL_{\sigma}(P_{\rho_{1}})   \to \bigoplus_{ \rho_{0}\in \Sigma_\sigma}  \HPoly_\sigma(P_{\rho_{0}}, P_\sigma) \otimes_{ \Gamma^{P_{\rho_{0}}} } H\cL_{L,\sigma}(P_{\rho_{0}}) . 
\end{equation}

So far, the discussion has been completely general. We now use the assumption that $L$ meets $X_{q_\sigma}$ at a point: let $x$ denote the intersection of $L$ with $X_{q_\sigma}$. For each $P \subseteq \cN_\sigma$,  we have a canonical isomorphism
\begin{equation}
 \cL_{L,\sigma}(P)  =  U^{P}_{\sigma,x},
\end{equation}
with trivial differential. Using Corollary \ref{cor:left_module_action_obvious}, we conclude that each module $ H\cL_{L,\sigma}(P_{\rho}) $ is free of rank $1$ over $ \Gamma^{P_\rho}$, and we are thus reduced to proving:
\begin{lem}
  \label{lem:first_application_Tate_acyclicity}
The following sequence
  \begin{equation}
  \bigoplus_{ \rho_{0} < \rho_{1} \in \Sigma_\sigma}  \HPoly_\sigma(P_{\rho_{1}}, P_\sigma)  \to \bigoplus_{ \rho_{0} \in \Sigma_\sigma}  \HPoly_\sigma(P_{\rho_{0}}, P_\sigma) \to \HPoly_\sigma(P_{\sigma}, P_\sigma) \to 0
\end{equation}
is right exact, i.e.\ the middle arrow is an isomorphism from the cokernel of the first map to $ \HPoly_\sigma(P_{\sigma}, P_\sigma)$.
\end{lem}
\begin{proof}
Let $P$ be a polytope containing $P_\sigma$ in its interior, and contained in the union of polytopes $P_\tau$ with $\tau \in \Sigma_\sigma$. By Corollary \ref{cor:computation_morphisms_local}, the restriction
\begin{equation}
  \HPoly_\sigma(P_{\tau} \cap P, P_\sigma) \to \HPoly_\sigma(P_{\tau} , P_\sigma) ,
\end{equation}
is an isomorphism, so that the Lemma follows from the right exactness of the complex
  \begin{equation}
  \bigoplus_{ \rho_{0} < \rho_{1} \in \Sigma_\sigma} \HPoly_\sigma(P_{\rho_{1}}\cap P, P_\sigma)  \to \bigoplus_{ \rho_{0} \in \Sigma_\sigma} \HPoly_\sigma(P_{\rho_{0}}\cap P, P_\sigma)  \to \HPoly_\sigma(P_\sigma, P_\sigma) \to 0.
\end{equation}
This is a consequence of Tate acyclicity: Lemma \ref{lem:tate_acyclic_hom_from_to} implies that the map from the  \v{C}ech complex
\begin{equation}
 \left( \bigoplus_{0 \leq k} \bigoplus_{\rho_{0} < \cdots < \rho_k} \HPoly_\sigma(P_{\rho_{1}}\cap P, P_\sigma)[-k], \check{\delta} \right)
 \end{equation}
to $\HPoly_\sigma(P_\sigma, P_\sigma) $ is a quasi-isomorphism. Passing to cohomology, the quotient complex corresponding to $1 \leq k $ yields a map of cohomology groups which fits in a commutative diagram
\begin{equation}
    \begin{tikzcd}[column sep=-30]
  \displaystyle{   \bigoplus_{ \rho_{0} < \rho_{1} \in \Sigma_\sigma}} \HPoly_\sigma(P_{\rho_{1}}\cap P, P_\sigma)   \arrow{rr}{}  \arrow{dr}{} & &   H^* \left(   \displaystyle{ \bigoplus_{1 \leq k} \bigoplus_{\rho_{0} < \cdots < \rho_k} }  \HPoly_\sigma(P_{\rho_{1}}\cap P, P_\sigma)[-k], \check{\delta} \right)  \arrow{dl}{} \\
&  \displaystyle{  \bigoplus_{ \rho_{0} \in \Sigma_\sigma}} \HPoly_\sigma(P_{\rho_{0}}, P_\sigma) . & 
  \end{tikzcd}
\end{equation}
Moreover, the images of these two maps agree.  Applying the long exact sequence on cohomology and using the fact that $ \HPoly_\sigma(P_\sigma, P_\sigma) $ vanishes except in degree $0$ implies the desired result.
\end{proof}
Proposition \ref{prop:tensor_Fuk_iso_projective} is an immediate consequence of the above result.

\subsection{Local module duality}
\label{sec:local-module-duality}

In this section, we fix $\sigma \in \Sigma$, and construct a map of bimodules
\begin{equation}
H \cL_{L,\sigma}  \otimes_\Lambda  H\scrR_{L,\sigma}  \to \Delta_{\HPoly_\sigma },
\end{equation}
with target the diagonal bimodule of $\HPoly_\sigma$. We then prove that, if $L$ is a Lagrangian section, this map defines an isomorphism between the values at $P_\sigma$ of $H\scrR_{L,\sigma}$ and the bimodule dual of $H \cL_{L,\sigma}$, i.e. the right module defined as the space of left module maps from $H \cL_{L,\sigma}$ to the diagonal bimodule.
\begin{rem}
  The results of this section will be used again when considering general Lagrangians in Section \ref{sec:homological-algebra}. The key idea is that such a Lagrangian gives rise to a filtered module whose associated graded module is isomorphic to a (finite) direct sum of modules corresponding to local Lagrangian sections. The argument given here can then be used for each of the summands of this associated graded module, and the statement that the left and right $A_\infty$ modules associated to a Lagrangian are bimodule duals then follows by homological algebra (see Section \ref{sec:local-computations}).
\end{rem}

\subsubsection{A map to the local diagonal bimodule}
\label{sec:map-diag-bimod}

Consider the triple $\Lab = (\sigma,L,\sigma)$. We define $\Rbar(\Lab)$ to consist of elements of $\Rbar(\sigma,L)$ equipped with a marked point along the boundary component labelled $\sigma$, and define $\RTbar(\Lab)$ to be the fibre product of $\Rbar(\Lab)$  with the space $\Tbar_-(\sigma,\sigma)$ of (perturbed) negative half-gradient flow lines.  We have a natural evaluation map
\begin{equation}
\RTbar(\Lab) \to  \Crit(\sigma,\sigma) \times \Crit(\sigma,L) \times \Crit(L,\sigma).
\end{equation}
For each pair $(P,P_-)$ of polytopes contained in $\cN_\sigma$,  the count of rigid elements of this moduli space thus defines a map
\begin{equation}
 \cL_{L,\sigma} (P)   \otimes_\Lambda  \scrR_{L,\sigma}(P_-) \to \Poly_\sigma(P_-,P)
\end{equation}
which, on co-homology, induces a map 
\begin{equation} \label{eq:product_into_HF_polygons}
H\cL_{L,\sigma} (P) \otimes_\Lambda  H\scrR_{L,\sigma}(P_-) \to \HPoly_\sigma(P_-,P).  
\end{equation}
We omit, as usual, the straightforward verification that this arises from a map of bimodules. Dualising, we obtain 
\begin{equation} \label{eq:map_right_module_to_Hom}
  H \scrR_{L,\sigma}  \to \Hom_{\HPoly_\sigma}( H \cL_{L,\sigma}, \Delta_{\HPoly_\sigma}).
\end{equation}

We have the following variant of Lemma \ref{lem:only_constant_contribute_to_local_modules}:
\begin{lem}
All contributions to the maps
\begin{align}
  \cL^{n-k}_{L} (P)   \otimes_\Lambda  \scrR^k_{L}(P_-) \to \Poly^n_\sigma(P_-,P)
\end{align}
are given by configurations whose holomorphic component is a constant strip. \qed
\end{lem}
Letting $\min x$ denote the unique minimum of $f_{\sigma,\sigma}$ which is the endpoint of a (perturbed) gradient flow line starting at $x$, we conclude
\begin{cor} \label{cor:compute_map_to_diagonal}
In degree $n$, the map in Equation \ref{eq:product_into_HF_polygons} is given by the sum of the maps
\begin{multline}
  \cL^{n-k}_{L} (P)   \otimes_\Lambda  \scrR^k_{L}(P_-)   \to \bigoplus_{x \in \Crit(L,\sigma)}  U^{P}_{\sigma,x} \otimes_\Lambda  \Hom^c_\Lambda(U^{P_-}_{\sigma,x}, \Lambda) \to \Hom(U^{P_-}_{\sigma,\min x}, U^{P}_{\sigma,\min x})  
\end{multline}
associated to each critical point $x$, where the second map is induced by parallel transport along the flow line from $x$ to $\min x$. \qed
\end{cor}

\subsubsection{Computing the right module for sections}
\label{sec:comp-right-module}

In this section, we prove:
\begin{lem} \label{lem:compute_right_module_section}
 The restriction of Equation \eqref{eq:map_right_module_to_Hom} to $P_\sigma$
\begin{equation} %
  H \scrR_{L,\sigma}(P_\sigma)  \to \Hom_{\HPoly_\sigma}( H \cL_{L,\sigma}, \Delta_{\HPoly_\sigma}(P_\sigma, \_))
\end{equation}
is an isomorphism whenever $L$ meets $X_{q_\sigma}$ transversely at one point.
\end{lem}

Before giving the proof, we need the following result, which should be thought of as the dual of Proposition \ref{prop:computation_inclusion}, and is a restatement of the second half of Proposition \ref{prop:computation_inclusion-restated}:
\begin{prop} \label{prop:computation_reverse_inclusion}
If $P_0$ and $P_1$ are polytopes so that $P_{0} $ is contained in the interior of $P_{1}$, a choice of orientation of $P_0$ determines a trace
  \begin{equation}
\tr \co  \Hom_{\Lambda}^{c}(\Gamma^{P_{0}}, \Gamma^{P_{1}}) \to \Lambda.      
  \end{equation}
Composing the trace with the module action, we obtain a map
  \begin{align}
     CF^n((\sigma,P_{0}), (\sigma,P_{1})) & \to \Hom_{\Lambda}^{c}(\Gamma^{P_{0}}, \Lambda) \\    
\theta & \mapsto \tr \circ \theta \circ m
  \end{align}
(where $n$ is the dimension of $Q$), which induces an isomorphism 
\begin{equation}
 HF^*((\sigma,P_{0}), (\sigma,P_{1})) \cong   \Hom_{\Lambda}^{c}(\Gamma^{P_{0}}, \Lambda).
\end{equation} \qed
\end{prop}

We now prove the Lemma stated above:

\begin{proof}[Proof of Lemma \ref{lem:compute_right_module_section}]
The argument is formally dual to the one given in Section \ref{sec:comp-dimens-1}. We start by noting that the right hand side is the kernel of the map
\begin{equation}
\prod_{ \rho_{0}\in \Sigma_\sigma } \Hom_{\Lambda}( H \cL_{L,\sigma}(P_{\rho_{0}}), \HPoly_\sigma(P_\sigma, P_{\rho_{0}}))  \to  \prod_{ \rho_{0} \leq \rho_{1} \in \Sigma_\sigma} \Hom_{\Lambda}(\HFuk(P_{\rho_{0}}, P_{\rho_{1}})    \otimes_\Lambda   H \cL_{L,\sigma}(P_{\rho_{0}}), \HPoly_\sigma(P_\sigma, P_{\rho_{1}})).
\end{equation}
We use the isomorphism $  \Gamma^{P_{\rho_{0}}} =  \HFuk(P_{\rho_{0}}, P_{\rho_{0}}) $, the fact that $\HFuk(P_{\rho_{0}}, P_{\rho_{1}}) $ is a free rank-$1$ module over $\Gamma^{P_{\rho_{1}}} $, and the flatness of the map $\Gamma^{P_{\rho_{0}}} \to \Gamma^{P_{\rho_{1}}}$ to rewrite this as the kernel of the map
\begin{equation}
\prod_{ \rho_{0}\in \Sigma_\sigma } \Hom_{\Gamma^{P_{\rho_{0}}}}( H \cL_{L,\sigma}(P_{\rho_{0}}), \HPoly_\sigma(P_\sigma, P_{\rho_{0}}))  \to  \prod_{ \rho_{0} < \rho_{1} \in \Sigma_\sigma} \Hom_{\Gamma^{P_{\rho_{1}}}}( H \cL_{L,\sigma}(P_{\rho_{1}}), \HPoly_\sigma(P_\sigma, P_{\rho_{1}})).
\end{equation}
Assuming that $L $ meets $X_{q_\sigma}$ at one point implies that $H \cL_{L,\sigma}(P_{\rho}) $ is free of rank one over $ \Gamma^{P_{\rho}}$, hence simplifying the above to:
\begin{equation}
\prod_{ \rho_{0}\in \Sigma_\sigma }  \HPoly_\sigma(P_\sigma, P_{\rho_{0}})  \to  \prod_{ \rho_{0} < \rho_{1} \in \Sigma_\sigma} \HPoly_\sigma(P_\sigma, P_{\rho_{1}}).
\end{equation}
We now introduce a polytope $P$, containing $P_\sigma$ in its interior, and covered by its intersection with $P_\tau$. By Corollary \ref{cor:computation_morphisms_local}, the restriction
\begin{equation}
 \HPoly_\sigma(P_\sigma, P_{\tau}) \to \HPoly_\sigma( P_\sigma, P \cap P_\tau)
\end{equation}
is an isomorphism. The desired computation then follows from the fact that the complex
  \begin{equation}
0 \to \HPoly_\sigma( P_\sigma, P) \to \prod_{ \rho_{0}\in \Sigma_\sigma }  \HPoly_\sigma(P_\sigma, P \cap P_{\rho_{0}})  \to  \prod_{ \rho_{0} < \rho_{1} \in \Sigma_\sigma} \HPoly_\sigma(P_\sigma, P \cap P_{\rho_{1}}) 
\end{equation}
is left exact, which is a consequence of Tate acyclicity.

Since $P_\sigma$ is contained in the interior of $P$, the computations of Appendix \ref{sec:based-loops-laurent}, summarised above in Proposition \ref{prop:computation_reverse_inclusion}, together with the definition of $\scrR_{L,\sigma}(P_\sigma) $ yield an isomorphism
\begin{equation}
   \HPoly_\sigma( P_\sigma, P) \cong \Hom^c_\Lambda( U^{P_\sigma}_{\sigma,x}, \Lambda)  \cong  H \scrR_{L,\sigma}(P_\sigma).
\end{equation}
As in Section \ref{sec:comp-dimens-1}, this isomorphism is compatible with module action, implying Lemma \ref{lem:compute_right_module_section}.  
\end{proof}

\subsection{Global cohomological modules}
\label{sec:glob-cohom-modul}

Given $L \in \cA$ and $\sigma \in \Sigma$, we define
\begin{align}
\cL^*_L(\sigma) & \coloneqq  CF^*(L,(\sigma,P_\sigma)) \\  %
\scrR_L^*(\sigma) & \coloneqq   CF^*((\sigma, P_\sigma),L),
\end{align}
with differentials  $ \mu^{1|0}_{\cL_L} $ and $\mu^{0|1}_{\scrR_L}$   obtained from Equation \eqref{eq:differential_left_module} and its dual. In this section, we show that the corresponding cohomology groups give rise to left and right modules $H\cL_L$ and $H\scrR_L$ over $\HFuk$. In Section \ref{sec:comp-local-glob} below, we show that the restrictions of these modules to $\HFuk_\sigma$ are isomorphic to the pullbacks of the corresponding local modules, while in Section \ref{sec:maps-relating-left}, we construct the maps relating the Floer cohomology of Lagrangians in $\cA$ with the corresponding left and right modules.   All these constructions are minor variants of those appearing  in \cite{Abouzaid2014a}.

\subsubsection{Continuation maps with non-Hamiltonian boundary conditions}
\label{sec:non-hamilt-cont}

Consider a triple
\begin{equation}
  \Lab = (\sigma_{-2}, \sigma_{-1}, L) \textrm{ or }   \Lab = (L, \sigma_{0}, \sigma_{1}),
\end{equation} such that $ P_{\sigma_{-2}} \cap P_{\sigma_{-1}} \neq \emptyset $, or $ P_{\sigma_{0}} \cap P_{\sigma_{1}} \neq \emptyset $. %
The continuation maps between modules associated to $L$ and the two elements of the cover will be defined by a count of pseudo-holomorphic strips, so we begin by labelling the positive end $e_\inp$ of the strip $S= \bR \times [0,1]$ with the pair $\Lab_{e_\inp}$ of labels $(\sigma_{0},L)$ or $(L,\sigma_{-1})$, and the negative end $e_\out$ with the pair $\Lab_{e_\out}$ given by  $(\sigma_{1},L)$ or $(L,\sigma_{-2})$. We thus obtain assignments $(\sigma_e,q_{\Lab_e})$ of  elements of $\Sigma$ and $Q$ for each end $e$.  Fix disjoint neighbourhoods $\nu_S e$ of these ends, and write $\partial_\Sigma S$ for the boundary component of $S$ which is labelled by elements of $ \Sigma$.

Consider the straight path $q_{\Lab}$  from $q_{\sigma_{e_\inp}}$ to $q_{\sigma_{e_\out}}$ determined by the affine structure; we parametrise this path by $\partial_\Sigma S$, with the condition that it agree with $ q_{\sigma_e}$ in $\nu_S e$. The quotient of  $\partial_\Sigma S$ by the intersection with $\nu_S e$ is a closed interval.

If we choose a family
\begin{equation}
J(\Lab) \co S \to \scrJ
\end{equation}
of almost complex structures which agree with $J(\Lab_e)$ along the end $e$, and whose restriction to the boundary $t=1$ (the coordinates are $(s,t)$ on the strip) agrees with 
$J_L$ 
we obtain a moduli space $\Rbar(\Lab)$ of stable $J(\Lab)$-holomorphic strips with boundary conditions given by the path $X_{q_{\Lab}}$ along the boundary $t=0$ (ordered compatibly with the counter-clockwise orientation), and by $L$
along $t=1$.  By evaluation along the boundary $t=0$, we associate to each element $u$ of $\Rbar(\Lab)$ a path
\begin{equation}
\partial_Q u \co \bR \to X_S,
\end{equation}
where the right hand side denotes the fibre bundle over a closed interval, obtained from the Lagrangian boundary conditions by collapsing the inverse image of the neighbourhoods $\nu_S e$ to the fibres $X_{q_{\sigma_e}}$. 

Consider as in Section \ref{sec:cohomological-module} the quotient  $X_{S}/{\sim}$ by the equivalence relation which collapses the components of the intersection of $\nu_X L$ with the boundaries $X_{q_{\sigma_{e_\inp}}}$ and  $X_{q_{\sigma_{e_\out}}}$ to points.  According to Lemma \ref{lem:basic_isoperimetric_moving_L}, there is a constant $C$ such that, up to an additive constant, the length of $\partial_Q u / {\sim}$ is  bounded by $C E^{geo}(u)$, where the \emph{geometric energy} is given by
\begin{equation} \label{eq:geometric_energy_strip_moving}
  E^{geo}(u) = \int |du|^2 = \int u^* \omega
\end{equation}
with the norm taken with respect to the metric induced by the almost complex structure. Using the fact that we have an isomorphism $H_{1}(X_{q}, \bZ) \cong H_{1}(X_{q_\Lab}, \bZ) $ for any $q$ in the path $q_\Lab$, we arrange as before for the metric on $X_{q_{\Lab}}/{\sim}$ to have the property that  for any  $q \in q_{\Lab}$ the norm of the homology class in $H_{1}(X_{q}, \bZ)$  associated to a loop in $X_{q_{\Lab}} $  is bounded by the length of the image of this loop in $X_{q_{\Lab}}/{\sim} $.  We then require that
\begin{equation} \label{eq:condition_convergence}
\parbox{35em}{the diameter of $\cN_\sigma $ is bounded by $1/8 \kk C$.}
\end{equation} 

\begin{rem}
We are slightly abusing notation by using $C$ for the constant $C$ appearing above, instead of using a symbol that distinguishes it from the constant in Condition \eqref{eq:condition_convergence-basic}. One justification for the abuse of notation is that  Condition \eqref{eq:condition_convergence-basic} holds for a given constant $C$ whenever Condition \eqref{eq:condition_convergence} holds for the same constant: the former arises from a reverse isoperimetric inequality for holomorphic strips, and the latter for continuation maps, and the moduli spaces of strips arise as boundary strata in the moduli spaces of continuation maps.
\end{rem}

\subsubsection{Energy of continuation maps}
\label{sec:energy-cont-maps}

The  moduli space $ \Rbar(\Lab)$ is slightly unusual because the moving Lagrangian boundary conditions along $t=0$ do not form a Hamiltonian family. The standard description of the Gromov-Floer bordification (in terms of breaking of strips at the ends, bubbling of discs at the boundary and of spheres at the interior) applies with unmodified proof, as does the treatment of regularity.

The main difference with the standard case of Hamiltonian families is that Gromov compactness fails in general. We now explain that it remains valid in our setting: the section $\triv_{\sigma_{e_\inp}}$ determines an identification of $X_{\cN_{\sigma_{e_\inp}}}$ with a neighbourhood of the $0$-section in the cotangent bundle of the fibre over the basepoint $q_{\sigma_e}$, in such a way that the section corresponds to a cotangent fibre; let $ \lambda_{\sigma_{e_\inp}}$ denote the (locally defined) Liouville $1$-form obtained from this identification.  Since the boundary conditions are not Lagrangian, the expression for geometric energy in Equation \ref{eq:geometric_energy_strip_moving} is not invariant under homotopies of the map $u$. This leads us to define the (topological) energy of an element $u \in \Rbar(\Lab)$ as
\begin{equation} \label{eq:topological_energy_definition}
E(u) \coloneqq   \int u^* \omega - \int (\partial_Q u)^{*} \lambda_{\sigma_{e_\inp}}. %
\end{equation}
\begin{lem} \label{lem:top_energy_homotopy_invariant}
The (topological) energy of $u$ depends only on its relative homotopy class; i.e. if $\{u_r\}_{r \in [0,1]} $ is a $1$-parameter family of strips with boundary conditions given by the path $X_{q_{\Lab}}$ along the boundary $t=0$, and by $L$ along $t=1$, with constant asymptotic conditions  at the ends, then $E(u_{0}) = E(u_{1})$.   
\end{lem}
\begin{proof}
This is  a straightforward application of Stokes's theorem: consider the associated map
\begin{equation}
  [0,1] \times  \bR \times [0,1] \to X.
\end{equation}
Since $\omega$ is closed, and we have assumed that the asymptotic conditions are constant, the integral of $\omega$ over the boundary vanishes. This sum decomposes in two $4$ terms; the integrals over the faces corresponding to the boundary of the first factor give rise to the terms involving $\omega$ in $E(u_{0})$ and $E(u_{1})$. For the integral over the face corresponding to $t=0$, we note that the image of this face is contained in $X_{\cN_{\sigma_{e_\inp}}}$, and that the choice of primitive $\lambda_{\sigma_{e_\inp}}  $ allows us to apply Stokes's theorem to yield the two corresponding terms in $E(u_i)$, using the fact that contributions from the end $s = \pm \infty$ vanish because the asymptotic conditions are constant. Finally, the restriction of the map to the face corresponding to $t=1$ factors through $L$, so the Lagrangian condition implies that the corresponding term vanishes. %
Having accounted for all the terms in $E(u_{0}) - E(u_{1})$, the result follows.
\end{proof}

We now prove that Gromov compactness holds, under Condition \eqref{eq:condition_convergence}:
\begin{lem} \label{lem:estimate_top_energy}
For $u \in \Rbar(\Lab)$, we have
\begin{equation}
    E(u) \geq \frac{3}{4}E^{geo}(u) + \textrm{ a constant independent of u.}
\end{equation}
In particular, for any positive real number $E$, the subset of $\Rbar(\Lab) $  consisting of curves of topological energy bounded by $E$ is compact.   
\end{lem}
\begin{proof}
Gromov's argument \cite{Gromov1985} shows that the geometric energy $\int u^* \omega$  defines a proper map $\Rbar(\Lab) \to [0,\infty)$, so that the first statement implies the second, and it suffices to bound the difference between the two energies.

Identifying the path $q_\Lab$ with a path in $H^1(X_{q_{\sigma_{e_\inp}}}; \bR) $, and lifting $\partial_Q u$ to a path $\widetilde{\partial_Q u}$ in $H_{1}(X_{q_{\sigma_{e_\inp}}}; \bR) $, we can express the integral of the $1$-form $\lambda_{\sigma_{e_\inp}}$ as
\begin{align} \label{eq:express_integral_primitive_universal_cover}
\int (\partial_Q u)^{*} \lambda_{\sigma_{e_\inp}} & =  \int_{-\infty}^{\infty}  \langle q_\Lab(s), \partial_s \widetilde{\partial_Q u} \rangle ds \\
& =  \langle q_\Lab(-\infty), \widetilde{\partial_Q u}(-\infty) \rangle - \int_{-\infty}^{\infty}  \langle \partial_s  q_\Lab(s), \widetilde{\partial_Q u} \rangle ds \\
& \leq \frac{1}{8C} \cdot \ell(\partial_Q u) + \ell(\partial_Q u) \int_{-\infty}^{\infty}  |\partial_s  q_\Lab(s)| ds \\
& \leq \frac{\ell(\partial_Q u)}{4C}.
\end{align}   
Above, we have used Equation \eqref{eq:condition_convergence} to bound  the length of $q_\Lab $. The reverse isoperimetric inequality thus implies that
\begin{equation}
  \left|\int (\partial_Q u)^{*} \lambda_{\sigma_{e_\inp}}\right|  \leq  E^{geo}(u)/4.
\end{equation}
We conclude the desired result.
\end{proof}

We now restrict to the union of components corresponding to a pair of intersection points $(x_-,x_+)$. Choosing a path, for each end $e$, from these points to the basepoint in $X_{q_{\sigma_{e}}}$ obtained by intersecting with the section associated to $\sigma_{e_\inp}$, we obtain a homology class $[\partial_Q u] $ in $H_{1}(X_{q_{\sigma_{e}}}; \bZ)$ associated to $u \in  \Rbar(x_-,x_+)$. Using Condition \eqref{eq:condition_convergence}, we find that
\begin{equation} \label{eq:bound_norm_boundary_curve_moving_lagrangian}
 |[\partial_Q u] | \leq C  E^{geo}(u) \leq 3 C/4  E(u) +  \textrm{ a constant independent of } u.
\end{equation}

\subsubsection{Mixed moduli spaces}
\label{sec:mixed-moduli-spaces}

We equip the strip with the marked point $e = (0,0)$ in the case $\Lab = (\sigma_{-2}, \sigma_{-1}, L)$, and  $e = (0,1)$ in the case $\Lab = (L, \sigma_{0}, \sigma_{1})$. We assign to this marked point the label $\Lab_{e} =  (\sigma_{-2}, \sigma_{-1}) $ in the first case and  $\Lab_{e} =  (\sigma_{0}, \sigma_{1})$ in the second.

This choice of marked point equips the moduli space $\Rbar(\Lab)$ from the previous section with a natural evaluation map
\begin{equation}
\Rbar(\Lab) \to X_{q_{\Lab_{e}}}
\end{equation}
where we use the trivialisation associated to the first element of $\Lab_e$ to identify the two fibres. The manifold $X_{q_{\Lab_e}}$ carries a Morse function $f_{\Lab_{e}}$ from Section \ref{sec:morphisms}, and a choice of family of vector fields parametrised by $[0,\infty)$, which agree with the gradient flow outside a compact set, yields a moduli space $\Tbar_{+}(\Lab_{e})$, which also admits an evaluation map to $X_{q_{\Lab_{e}}} $. We define the mixed moduli space as a fibre product (see Figure \ref{fig:action_left_module_directed})
\begin{equation}
\RTbar(\Lab) \coloneqq  \Rbar(\Lab) \times_{ X_{q_{\Lab_e}}}   \Tbar_{+}(\Lab_{e}) .
\end{equation}
As before, this space admits an evaluation map to the product of the intersection points associated to the triples  $\{ e, e_\out, e_\inp \}$, and we denote the fibre over a triple $x$ of generators by $\RTbar(x) $

\begin{figure}
  \centering
  \begin{tikzpicture}

\draw[line width=2*\lw,dotted] (-6*\mw pt,-1*\mw pt) -- ( -7*\mw pt, -1*\mw pt);
\draw[line width=2*\lw,dotted] (-6*\mw pt,1*\mw pt) -- ( -7*\mw pt, 1*\mw pt);
\draw[line width=2*\lw,dotted] (6*\mw pt,-1*\mw pt) -- ( 7*\mw pt, -1*\mw pt);
\draw[line width=2*\lw,dotted] (6*\mw pt,1*\mw pt) -- ( 7*\mw pt, 1*\mw pt);
\draw[line width=2*\lw] (-6.5*\mw pt,-1*\mw pt) -- ( 6.5*\mw pt, -1*\mw pt);
\draw[line width=2*\lw] (-6.5*\mw pt,1*\mw pt) -- ( 6.5*\mw pt, 1*\mw pt);

\coordinate [label=above:$X_{q_{\sigma_{0}}}$] () at (6*\mw pt, 1*\mw pt);
\coordinate [label=above:$X_{q_{\sigma_{1}}}$] () at (-6*\mw pt, 1*\mw pt);
\coordinate [label=below:$L$] () at (0, -1*\mw pt);
\draw[line width=4*\lw] (0,1*\mw pt -1/8*\mw pt) -- (0,1*\mw pt +1/8*\mw pt);
\draw[line width=4*\lw]  [->] (0, 1*\mw pt) -- (0,3*\mw pt);
\draw[line width=4*\lw] (0,3*\mw pt) -- (0, 5*\mw pt);
\coordinate [label=left:$\nabla f_{0,1}$] () at (0, 3*\mw pt) ;
\end{tikzpicture}
  \caption{ An element of the moduli space $\RTbar(L, \sigma_{0}, \sigma_{1}) $ defining the left module over $\HFuk $. }
  \label{fig:action_left_module_directed}
\end{figure}

\subsubsection{The global left and right multiplications}
\label{sec:left-right-mult}

If $\Lab = (L, \sigma_{0}, \sigma_{1})$, an element of $ \RTbar (x)$ defines a map
\begin{align} \label{eq:map_product_local_systems_{1}}
\Hom(  U^{P_{\sigma_{0}}}_{\sigma_{0}, x_{e}}, U^{P_{\sigma_{1}}}_{\sigma_{1},x_{e}}) \otimes_\Lambda  U^{P_{\sigma_{0}}}_{\sigma_{0},x_{e_\inp}}  & \to U^{P_{\sigma_{1}}}_{\sigma_{1},x_{e_\out}},
\end{align}
where we have used the canonical identifications from Lemma \ref{lem:identification_local_systems_Lagrangian_section} to omit the superscript from the intersection points of Lagrangians. The map is defined as follows:  given an element $u $ in this moduli space represented by a pair $(v,\gamma)$,  the trivialisation $\triv_{\sigma_{e_\inp}}$, allows us to identify $x_{e}$ and $x_{e_\out}$  with points in $X_{q_{e_\inp}}$. Applying this trivalisation to  the path along the boundary of the elements of $v$ from the positive end to $v(e)$ and the gradient flow line $\gamma$, we obtain a parallel transport map
\begin{equation}
   U^{P_{\sigma_{0}}}_{\sigma_{0}, x_{e_\inp}}  \to   U^{P_{\sigma_{0}}}_{\sigma_{0}, x_{e}}.
\end{equation}
Given an element of the left hand side of Equation \eqref{eq:map_product_local_systems_{1}}, the composition of this isomorphism with the given linear map in $ \Hom^c_\Lambda(  U^{P_{\sigma_{0}}}_{\sigma_{0},x_{e}}, U^{P_{\sigma_{1}}}_{\sigma_{1},x_{e}})$ yields an element of $U^{P_{\sigma_{1}}}_{\sigma_{1},x_{e}}$ which we transport, along the image in $X_{q_{e_\inp}}$ of  the gradient flow line and the boundary of the holomorphic strip to the fibre of $U^{P_{\sigma_{1}}}_{\sigma_{1}}$ at $x_{e_\out}$, which is the right hand side of Equation \eqref{eq:map_product_local_systems_{1}}. 

\begin{rem}
As explained in the discussion following Lemma \ref{lem:identification_local_systems_Lagrangian_section},  one must specify the fibre along which one performs parallel transport because the identifications of the local systems over different fibres are only compatible with parallel transport maps up to multiplication by the exponential of the flux. 
\end{rem}

Multiplying the tensor product of the isomorphism in Equations \eqref{eq:map_product_orientation_lines} and that in Equation \eqref{eq:map_product_local_systems_{1}}  by $(-1)^{\deg(x_{1})+1}$, and restricting to continuous morphisms from $U^{P_{\sigma_{0}}}_{\sigma_{0}, x_{e}}$ to $U^{P_{\sigma_{1}}}_{\sigma_{1},x_{e}} $  we obtain, for $u $ a rigid element in $\RTbar(\Lab)$, a map
\begin{equation}
\mu_u \co  \Fuk(\sigma_{0},\sigma_{1})    \otimes_\Lambda     \cL_L(\sigma_{0})  \to   \cL_L(\sigma_{1}).
\end{equation}
\begin{lem}
  \label{lem:convergence_left_module_structure_map}
Under assumption \eqref{eq:condition_convergence}, the sum
\begin{align}
\sum_{ u \in  \RTbar(\sigma_{1}, \sigma_{0}, L)} T^{E(u)} \mu_u
\end{align}
is convergent.
\end{lem}
\begin{proof}
The key point is to bound the valuation of $ T^{E(u)} \mu_u$. The map $\mu_u$ is a composition of two parallel transport maps with an evaluation map. The evaluation map has trivial valuation and the valuation of the two parallel transport maps is bounded by the sum of the norms of the two corresponding loops obtained by concatenating with fixed paths to a basepoint. Up to an additive constant independent of $u$, the sum of norms of these loops is bounded by the length of the projection of  $\partial_Q u $ to $X_{q_\Lab} / {\sim}$.  Applying the reverse isoperimetric inequality, we conclude that the valuation of the parallel transport maps is bounded by $C E^{geo}(u)$.  By Condition \eqref{eq:condition_convergence}, the polytopes $P_{\sigma_i}$ are contained in the ball of radius $1/4C$ in $T_{q_{e_\inp}} Q$, so that applying Lemma \ref{lem:estimate_norm_transport_P} yields
\begin{equation}
  \val \mu_u \geq  - E^{geo}(u)/4 + \textrm{ a constant independent of u.}
\end{equation}
From  Lemma \ref{lem:estimate_top_energy}, we conclude that the valuation of $ T^{E(u)} \mu_u $ is bounded above by  $E(u)/2$, up to a constant term independent of $u$. Gromov compactness (for the geometric energy) thus implies that the sum is convergent.
\end{proof}
\begin{rem}
Note that the definitions of $E(u)$ and $\mu_u$ are asymmetric, since they rely on distinguishing the fibre $X_{q_\inp}$ corresponding to the incoming end.  However, the product $T^{E(u)} \mu_u$ is independent of this choice.
\end{rem}

We denote the sum in Lemma \ref{lem:convergence_left_module_structure_map} by
\begin{equation}
    \mu^{1|1}_{\cL_L} \co  \Fuk^*\left( \sigma_{0},\sigma_{1} \right)   \otimes_\Lambda     \cL_L^*(\sigma_{0})  \to  \cL^*_L(\sigma_{1}).
\end{equation}
Applying Lemma \ref{lem:top_energy_homotopy_invariant} (and using the fact that the topological energy is additive for broken curves), the standard description of the boundary of $\RTbar(\Lab)$ implies that $ \mu^{1|1}_{\cL_L} $ is a cochain map, hence that it induces a map on cohomology
\begin{equation}
  \HFuk^*(\sigma_{0},\sigma_{1})   \otimes_\Lambda     H\cL_L^*(\sigma_{0})  \to  H\cL^*_L(\sigma_{1}),
\end{equation}
which makes the groups $H\cL_L$ into a left module over the category $\HFuk$. The last statement requires the construction of a homotopy corresponding, in the usual language of $A_\infty$ modules to the operation $\mu^{2|1}_{\cL_L}$, and hence will be subsumed by later constructions.

In the case $\Lab = (\sigma_{-2}, \sigma_{-1}, L)$,  elements of the moduli space $  \RTbar (x_{e_\inp}; x_{e_\out}, x_{e})$ give rise to a map 
 \begin{align} 
\label{eq:map_product_local_systems_2}
 \Hom^c_\Lambda(U^{P_{\sigma_{-1}}}_{\sigma_{-1}, x_{e_\inp}}, \Lab)  \otimes_\Lambda   \Hom^c_\Lambda(  U^{P_{\sigma_{-2}}}_{\sigma_{-2}, x_{e}}, U^{P_{\sigma_{-1}}}_{\sigma_{-1}, x_{e}})  & \to  \Hom^c_\Lambda(U^{P_{\sigma_{-2}}}_{\sigma_{-2}, x_\out}, \Lab).
\end{align}
The sum of maps associated to this moduli space defines the cochain-level product
\begin{equation}
\mu^{1|1}_{\scrR_L} \co \scrR^*_L(\sigma_{-1})   \otimes_\Lambda     \Fuk^*(\sigma_{-2}, \sigma_{-1})  \to  \scrR_L^*(\sigma_{-2}).
\end{equation}
which induces the structure map of the right module $H\scrR_L$ at the cohomological level.

\subsubsection{Comparison of local and global modules}
\label{sec:comp-local-glob}

Recall that we denote by $\HFuk_\sigma$ the subcategory of $\HFuk$ with objects the elements $\tau \in \Sigma$ such that $P_\tau \cap P_\sigma \neq \emptyset$. We have a faithful embedding $\HFuk_\sigma \to \HPoly_\sigma$. In this section, we prove that the pullback of $H\cL_{L,\sigma}$ under this functor is naturally isomorphic to the restriction of $H\cL_L$.

We could define such a map using previously constructed moduli spaces, but it is convenient for later purposes to have a variant, in which we replace the boundary marked point by the data of a map to the strip: we denote by %
$\Rbar_{2,\underline{1}}$ the moduli space of degree $1$ stable maps to the disc $D^2$ whose domain is a pre-stable disc with two boundary marked points, one of which is marked as incoming and required to map to $+1$, and the other as outgoing, which is required to map to $-1$.

It is straightforward to see that $\Rbar_{2,\underline{1}}$ is a single point: the assumption that the map is stable forces the domain to consist of a single component, which is identified with the target $D^2$ by the given map. We find it convenient to label such a domain with an extra interior marked point, corresponding to the inverse image of the origin under this map, since the position of this marked point determines the map.
\begin{rem}
We could describe the construction of this section in terms of maps whose domain is $D^2$, specifying that the data we use is not invariant under reparametrisation. The convenience of the additional formalism (which we called \emph{popsicles} in \cite{AbouzaidSeidel2010}) lies in its ability to efficiently handle additional boundary marked points as discussed in Section \ref{sec:abstr-moduli-spac-2}.
\end{rem}

Consider a basepoint $\basepoint \in \Sigma$, and a pair $\Lab = (L,  \tau)$, with $\tau \in \Sigma_{\basepoint} $. Pick a parametrisation of the path $q_{\basepoint,\tau}$ by the boundary component $\bR \times \{1\} $ of the strip,  which agrees with $q_{\basepoint}$ near the positive end and with $q_\tau$ near the other end. We equip the strip with boundary conditions given by the corresponding path of fibres, and by the Lagrangian $L$ along the boundary $\bR \times \{0\}$.%

\begin{figure}[h]
  \centering
\begin{tikzpicture}
\newcommand*{\radius}{1.5}
\newcommand*{\tinyradius}{.05}
\coordinate [label=above:$ X_{q_{\Lab}} $] () at  (90:1.1*\radius);
\coordinate [label=above:$X_{q_{\sigma}}$] () at  (1.3*\radius,0);
\coordinate [label=below:$ L$] () at  (-90:1.1*\radius);
\coordinate [label=above:$X_{q_{\tau}}$] () at  (-1.3*\radius,0);
\draw [fill=black] (0,0) circle (\tinyradius);

\draw[->] (0:\radius+\tinyradius) -- (0:\radius-3*\tinyradius);
\draw ([shift=(2:\radius)]0,0) arc (2:178:\radius);
\draw[->] (180:\radius-\tinyradius) -- (180:\radius+3*\tinyradius);
\draw ([shift=(182:\radius)]0,0) arc (182:358:\radius);

\begin{scope}[shift={(-4*\radius,0)}]
\draw [fill=black] (0,0) circle (\tinyradius);
\coordinate [label=below:$ X_{q_{\Lab}} $] () at  (-90:1.1*\radius);
\coordinate [label=below:$X_{q_{\sigma}}$] () at  (1.3*\radius,0);
\coordinate [label=above:$ L$] () at  (90:1.1*\radius);
\coordinate [label=below:$X_{q_{\tau}}$] () at  (-1.3*\radius,0);

\draw[->] (0:\radius+\tinyradius) -- (0:\radius-3*\tinyradius);
\draw ([shift=(2:\radius)]0,0) arc (2:178:\radius);

\draw[->] (180:\radius-\tinyradius) -- (180:\radius+3*\tinyradius);
\draw ([shift=(182:\radius)]0,0) arc (182:358:\radius);
\end{scope}

\end{tikzpicture}
  \caption{The boundary conditions of elements of $\Rbar(\uLab)$ for $\Lab = (\sigma, \tau, L)$ and $\Lab = (L, \sigma, \tau)$. The interior marked point corresponds to the origin, and indicates the fact that the data we choose on the strip is not invariant under translation.}
 \label{fig:continuation_map_HFuk_bimodule}
\end{figure}

 Choosing a family of almost complex structures as before,  we obtain a moduli space $\Rbar(\uLab)$. Choosing a homotopy between the sections $\triv_{\basepoint}$ and $\triv_\tau$, the count of rigid elements of this space defines a map
\begin{equation}
\cL_{L,\basepoint}(P_\tau) \coloneqq CF^*(L, (\basepoint,P_{\tau}))  \to   CF^*(L, (\tau,P_{\tau})) \coloneqq \cL_L(\tau),
\end{equation}
which is a quasi-isomorphism. Swapping the roles of the boundary conditions at $0$ and $1$, we obtain
\begin{equation}
\scrR_{L,\basepoint}(P_\tau) \coloneqq CF^*( (\basepoint,P_{\tau}), L)  \to   CF^*((\tau,P_{\tau}), L) \coloneqq \scrR_L(\tau).  
\end{equation}
These constructions give rise to maps of cohomological modules by the use of mixed moduli spaces as in previous sections.

\subsubsection{Floer cohomology of Lagrangians, and global modules}
\label{sec:maps-relating-left}

We now proceed to construct the top horizontal and right vertical maps in Diagram \eqref{eq:cohomological_diagram_modules}; this is a minor variant of the maps considered in \cite{Abouzaid2014a}. Given $L \neq L' \in \cA$, and $\sigma \in \Sigma$, we consider the triple $\Lab = (L, L', \sigma)$ or $(L,\sigma,L')$. Let $S$ be a disc with $3$ punctures. The triple $\Lab$ induces a unique labelling of the boundary components of the complement of the punctures which respects the ordering counterclockwise around the boundary; we write $\Lab_e$ for the pair associated to each end $e$.  We pick a strip-like end on each puncture, which is positive if the corresponding edge is incoming, and negative otherwise.

\begin{figure}[h]
  \centering
\begin{tikzpicture}
\newcommand*{\radius}{1}
\newcommand*{\tinyradius}{.025}
\newcommand*{\bigradius}{1.5}
\begin{scope}[shift={(-4*\bigradius,0)},rotate=30]
\node at (0,0) {$\Rbar(L, \sigma, L'  ) $};

\draw (0,0)  ([shift=(-89:\bigradius)]0,0) arc (-89:29:\bigradius);

\draw[->] (-90:\bigradius+2*\tinyradius) -- (-90:\bigradius-6*\tinyradius);
\draw ([shift=(151:\bigradius)]0,0) arc (151:269:\bigradius);

\coordinate [label=above:$L'$] () at  (90:\bigradius);
\coordinate [label=right:$X_{q_{\sigma}}$] () at  (-30:\bigradius);

\coordinate [label=below:$L$] () at  (180:1.1*\bigradius);

\draw[->] (30:\bigradius+2*\tinyradius) -- (30:\bigradius-6*\tinyradius);
\draw ([shift=(31:\bigradius)]0,0) arc (31:149:\bigradius);

\draw[->] (150:\bigradius-2*\tinyradius) -- (150:\bigradius+6*\tinyradius);
\end{scope}

\begin{scope}[rotate=30]
\node at (0,0) {$\Rbar(L, L', \sigma ) $};

\draw (0,0)  ([shift=(-89:\bigradius)]0,0) arc (-89:29:\bigradius);

\draw[->] (-90:\bigradius+2*\tinyradius) -- (-90:\bigradius-6*\tinyradius);
\draw ([shift=(151:\bigradius)]0,0) arc (151:269:\bigradius);
 \coordinate [label=below:$L$] () at  (180:1.1*\bigradius);
\coordinate [label=above:$ X_{q_{\sigma}} $] () at  (90:\bigradius);

\coordinate [label=right:$L'$] () at  (-30:1*\bigradius);

\draw[->] (30:\bigradius+2*\tinyradius) -- (30:\bigradius-6*\tinyradius);
\draw ([shift=(31:\bigradius)]0,0) arc (31:149:\bigradius);

\draw[->] (150:\bigradius-2*\tinyradius) -- (150:\bigradius+6*\tinyradius);
\end{scope}

\end{tikzpicture}
  \caption{The boundary conditions for elements of the moduli spaces $\Rbar(L, L', \sigma)$ and $\Rbar(L, \sigma, L' )$.}
\label{fig:maps_relating_left_right_module_to_Lag}
\end{figure}

We define
\begin{equation}
\scrJ(\Lab) \subset C^{\infty}(S, \scrJ)
\end{equation}
to be the space of almost complex structures parametrised by $S$ which respectively agree with $J_L$ and $J_{L'}$ along the boundary components labelled by these Lagrangians, and whose restriction to each end $e$ agrees, under a choice of strip-like ends, with the families $J(\Lab_e)$ assigned earlier to pairs of labels containing an element of $\Sigma$, and given by a fixed regular choice of path from $J_L$ to $J_{L'}$ for $\Lab_e = (L,L')$.

If we impose constant Lagrangian conditions $X_{q_\sigma}$ along the remaining boundary component, the choice of an element $J(\Lab)$ of $\scrJ(\Lab)$ determines a moduli space $\Rbar(\Lab)$ of holomorphic triangles equipped with an evaluation map
\begin{equation}
\Rbar(\Lab) \to \prod_{e} \Crit(\Lab_e),   
\end{equation}
where the product is taken over all the ends. Assuming that the family of almost complex structure is generic, and after imposing a further constraint on the diameter of $\cN_\sigma$ in terms of the corresponding reverse isoperimetric inequality, the count of rigid elements of $\Rbar(\Lab)$ (together with parallel transport) induces maps
\begin{align}  \label{eq:cochain_map_left_module_action}
  \cL_{L'}(\sigma) \otimes_\Lambda   CF^*(L,L') & \to \cL_{L}(\sigma) \\
   \scrR_{L'}(\sigma) \otimes_\Lambda  \cL_{L}(\sigma)   &  \to CF^*(L,L'),
\end{align}
depending on whether  $\Lab = (L, L', \sigma)$ or $(L,\sigma,L')$.

Dualising Equation  \eqref{eq:cochain_map_left_module_action} gives rise to a map
\begin{equation}
  CF^*(L,L') \to \Hom^c_\Lambda(   \cL^*_{L'}(\sigma),\cL^*_{L}(\sigma) ).
\end{equation}
Using a mixed moduli space as in Section \ref{sec:mixed-moduli-spaces}, and a family of moving Lagrangian boundary conditions, we find that this map commutes with the action of morphisms groups in $\HFuk$, and obtain the map from Equation \eqref{eq:map_HF_pair_hom_left_modules}.

We also obtain a commutative diagram
\begin{equation}
  \begin{tikzcd}
    H\scrR_{L'}^*(\sigma)  \otimes_\Lambda  \HFuk^*(\sigma_{-}, \sigma) \otimes_\Lambda    H\cL^*_L(\sigma_{-}) \ar[r] \ar[d] & H\scrR_{L'}^*(\sigma)   \otimes_\Lambda    H\cL^*_L(\sigma) \ar[d] \\
 H\scrR_{L'}^*(\sigma_{-}) \otimes_\Lambda    H\cL^*_L(\sigma_{-})  \ar[r] &   HF^*(L,L'),
  \end{tikzcd}
\end{equation}
which proves that we have defined a map from the tensor product of left and right modules over $\HFuk$ to the Floer cohomology of the pair $(L,L')$.

\subsection{The pertubed diagonal}
\label{sec:pertubed-diagonal}

The constructions of Section \ref{sec:glob-cohom-modul} generalise those of Section \ref{sec:cohomological-module} to the global category. In this section, we extend the results of Section \ref{sec:local-module-duality}: the key point is to construct a bimodule over $\HFuk$ whose restriction to $\HFuk_\sigma$, for each $\sigma \in \Sigma$, agrees with the pullback of the diagonal bimodule of $\HPoly_\sigma$.

The construction will be implemented via a perturbed Lagrangian Floer cohomology group of fibres. To this end, we fix a  Hamiltonian diffeomorphism $\phi \co X \to X$ which is generic in the sense that
\begin{equation} \label{eq:discrete_intersection_perturb_diagonal}
  \parbox{35em}{for all pairs $(q_-, q) \in Q^2$, $X_{q_-} \cap \phi X_{q}$  is discrete,}  
\end{equation}
which can be achieved by jet transversality. We also fix a neighbourhood $   \nu(X \cap \phi X) $ in $Q^2 \times X$ of the set
\begin{equation}
 \coprod_{(q_-, q) \in Q^2} X_{q_-} \cap \phi X_{q}  \subset  Q^2 \times X.
\end{equation}
We denote by $\nu_X(X_{q_-} \cap \phi X_{q} )$ the inverse image over $(q_- , q) \in  Q^2$, and require that this subset be sufficiently small so that 
\begin{equation} \label{eq:intersection_perturbation_neighbourhood_contractible}
   \parbox{35em}{for all pairs $(q_-, q)$, the intersections of  $\nu_X(X_{q_-} \cap \phi X_{q})$ with $X_{q_-}$ and $\phi X_{q} $ are respectively contained in disjoint unions of closed balls in $X_{q_-}$ and $\phi X_{q} $.} 
\end{equation}

Given a path $J(q_-,q)$ of almost complex structures on $X$ which are constant away from a fixed compact subset of the interior of $\nu_X(X_{q_-} \cap \phi X_{q})$, Lemma \ref{lem:isoperimetric_constant_lagrangian} provides a reverse isoperimetric constant $C$ for $J(q_-,q)$ holomorphic strips with boundary conditions $X_{q_{-}}$ and $\phi X_{q}$, with the length measured in the quotient of these Lagrangians by the equivalence relation which identifies two points in the same component of the intersection with $\nu_X(X_{q_-} \cap \phi X_{q})$, and which is uniform in the choice of points in $Q^2$, because we have assumed that $Q$ is compact. As before, the reverse isoperimetric inequality is independent of the restriction of $J(q_-,q)$ to the chosen compact subset of the interior of $\nu_X(X_{q_-} \cap \phi X_{q})$, and of the intersection of the Lagrangians with this region. This will allow us to pick perturbations in order to achieve transversality.

With the isoperimetic constant from the above paragraph, we impose an additional constraint on the cover by requiring that
\begin{equation} \label{eq:isoperimetric_condition_Phi}
  \diam  \cN_\sigma < 1/ 4 C,
\end{equation}
where we note that the constant $C$ now depends additionally on the Hamiltonian diffeomorphism $\phi$, since it incorporates a reverse isoperimetric constraint for the pair of Lagrangians $X_{q_{-}}$ and $\phi X_{q}$.

\begin{rem}
We note that we do not need to impose any assumption on $\phi$ other than genericity, since the existence of a uniform reverse isoperimetric constant is a very general fact about families of Lagrangians. Indeed, the key geometric quantity controlling the reverse isoperimetric constant is the distance between the two Lagrangians. This implies that any family of pairs of Lagrangians parametrised by a compact manifold admits a $C^2$-small perturbation for which there is a uniform reverse isoperimetric constant. In particular, we need not require that $\phi$ be a small perturbation of the identity.
\end{rem}

\subsubsection{Perturbed Floer cohomology for pairs of polytopes}
\label{sec:floer-cohom-pairs}
Let $\Sigma^\phi $ denote the set $\Sigma \times \{ \phi \}$; in other words, we think of $\Sigma^\phi$ as another copy of $\Sigma$, whose elements we shall denote  $\sigma^\phi$ as a shorthand for $(\sigma, \phi)$. This set parametrises the same cover of $Q$ as $\Sigma$, but we associate to it the Lagrangian fibration obtained by composing the projection $X \to Q$ with $\phi$.  Given a pair $\Lab = (\sigma_-, \sigma^\phi) \in \Sigma \times \Sigma^\phi$, we introduce the notation 
\begin{align}   \label{eq:definition_nbd_mixed_pair}  
\nu_X \Lab & \coloneqq \nu_X(X_{q_{\sigma_-}} \cap \phi X_{q_\sigma}).
\end{align}

We shall impose the following additional condition on the basepoints $q_\sigma$:
\begin{equation}
  \label{eq:transversality_image_under_diagonal}
  \parbox{34em}{for each pair $(\sigma_-,\sigma)$ of elements of $\Sigma^2$, the Lagrangians $X_{q_\sigma}$ and $\phi X_{q_\sigma}$ are transverse.}
\end{equation}
Having assumed that each polytope $P_\sigma$ has non-empty interior, Sard's theorem implies that this condition is achieved for generic choices of basepoints. We write 
\begin{equation}
  \Crit(\Lab) \coloneqq   X_{q_{\sigma_-}} \cap \phi X_{q_\sigma}
\end{equation}
for the set of intersection points between these Lagrangians.

By assumption, this is a discrete space, and, given the choice of $\Pin$ structures from Section \ref{sec:morphisms}, we obtain a $1$-dimensional free abelian group $\delta_x$ associated to each element $x$ of $  \Crit(\Lab)$.  Given a choice $ J(\Lab)$ of an almost complex structure on $X$, which agrees with $J$ away from $\nu_X \Lab $, we obtain a moduli space
\begin{equation}
\Rbar(\Lab) \to  \Crit(\Lab)^2
\end{equation}
of stable finite energy  $J(\Lab)$-holomorphic strips with boundary condition $X_{q_{\sigma_-}}$ along $t=0$ and $\phi X_{q_{\sigma}}$ along $t=1$, equipped with its natural evaluation map at the ends to intersection points of the boundary Lagrangians. 
Choosing the almost complex structure generically, the fibre  $ \Rbar(x_{0},x_{1})  $ over a pair $(x_{0},x_{1})$ has the expected dimension, and whenever the moduli space is rigid, we obtain a map $\delta_u$ of orientation lines for each $u \in  \Rbar(x_{0},x_{1})$.

Given a curve $u \in  \Rbar(x_{0},x_{1})$,  let $\partial_\Sigma u$ and $\partial_\Sigma^\phi u$ denote the boundary components labelled by $\sigma_-$ and $\sigma^\phi$.  Parallel transport   along these boundaries  defines maps
\begin{align}
z^{[\partial_\Sigma^\phi u]} \co  U^{P_\sigma}_{\sigma, x_{1}} &  \to U^{P_\sigma}_{\sigma,x_{0}} \\
z^{[\partial_\Sigma u]} \co  U^{P_{\sigma_-}}_{\sigma_-,x_{0}} &  \to U^{P_{\sigma_-}}_{\sigma_-,x_{1}}.
\end{align}
 We define the Floer complex
\begin{align}
 CF^*\left((\sigma_-,P_{\sigma_-}),  (\sigma^\phi,P_\sigma)\right) & \coloneqq \bigoplus_{x \in \Crit(\Lab) } \Hom^c_\Lambda(U^{P_{\sigma_-}}_{\sigma_-,x}, U^{P_\sigma}_{\sigma,x}) \otimes \delta_x,
\end{align}
with differential $\mu^{0|1|0}_{\phi}$ given by
\begin{align}
\mu^{0|1|0}_{\phi} |  \psi \otimes \delta_{x_{0}, x_{1}} & \coloneqq \sum_{u \in \cR(x_{0},x_{1})} (-1)^{\deg(x_{1})+1} T^{E(u)} z^{[\partial^\phi_\Sigma u]} \cdot \psi \cdot z^{[\partial_\Sigma u]} \otimes \delta_u.
\end{align}
Condition \eqref{eq:isoperimetric_condition_Phi} implies that  $P_\sigma$ and $P_{\sigma_-}$ are contained in the ball of radius $1/4C$ about  $q_\sigma$ and $q_{\sigma_-}$. As in Corollary \ref{cor:convergence_if_coefficients_bounded_lengths}, we conclude that this differential is well-defined and continuous with respect to the topology on these Floer complexes.

To define a bimodule over $\HFuk$, we set
\begin{equation} \label{eq:definition_bimodule}
\Delbar(\sigma_-, \sigma  ) \coloneqq  CF^*\left((\sigma_-,P_{\sigma_-}),  (\sigma^\phi,P_\sigma)\right)
\end{equation}
and denote the cohomology by $H\Delbar(\sigma_-,  \sigma )$.

As it is completely analogous to the discussion from Section \ref{sec:left-right-mult}, we omit for the moment the details of the construction of the bimodule structure maps (they will appear in greater generality when we discuss the $A_\infty$ refinement).

\subsubsection{Map from the diagonal bimodule}
\label{sec:computing-bimodule-nearby}

The main result of this section is the existence of a map from the diagonal bimodule on $\HFuk$ to the perturbed diagonal. It will require a further constraint on the cover, to be specified in Condition \eqref{eq:isoperimetric_condition_Phi_z} below:
\begin{lem} \label{lem:map_diagonal_to_perturbed}
  If $\Sigma$ labels a sufficiently fine cover, then for each triple   $(\sigma, \tau,\tau_-)$ in $\Sigma$ such that $\tau$ and $ \tau_-$ both lie in $\Sigma_\sigma$, there is a natural quasi-isomorphism
\begin{equation}
  \Poly_\sigma( P_{\tau_-},P_\tau) \to \Delbar(\tau_-,\tau).
\end{equation}
 In particular, the restriction of $H\Delbar $ to $\HFuk_\sigma $ is isomorphic to the pullback of the diagonal bimodule of $\HPoly_\sigma$.
\end{lem}

Fix a basepoint $\basepoint \in \Sigma$, and consider elements $\tau$ and $\tau_-$ of $\Sigma_{\basepoint}$. Let $\Lab =  \{ \tau_-,  \tau^\phi \}$, and denote by  $\Rbar_{\uLab}$  a copy of $ \Rbar_{2, \underline{1}}$, which we think of as parametrising a disc $S$ with a puncture  (we denote the corresponding end by $e$), and a boundary marked point, together with a stabilising map to the strip $D^2$ mapping the puncture to $-1$ and the marked point to $+1$. The basic idea is that the puncture corresponds to the perturbed bimodule, the marked point at $1$ to the diagonal of $\HPoly_{\basepoint}$, and the interior marked point to the fact that we shall interpolate between the data defining these two bimodules. We denote the intervals from $1$ to $-1$ along the upper and lower part of the boundary by $\partial^\phi_\Sigma  S$ and $ \partial_\Sigma S $.

We begin by strengthening the conditions imposed at the beginning of Section \ref{sec:pertubed-diagonal} by imposing the relevant isoperimetric constraint on an additional moduli space. We pick maps from  $\partial^\phi_\Sigma  $ and $ \partial_\Sigma S $ to the interval $[0,1]$, which respectively map neighbourhoods of the points $1$ and $-1$ to the endpoints $0$ and $1$ of the interval. We obtain paths $q_{\tau_-,\basepoint}$ and $q_{\tau,\basepoint}$
\begin{align}
  q_{\uLab} \co \partial_\Sigma S & \to   Q \\
q^\phi_{\uLab} \co \partial^\phi_\Sigma S  & \to  Q
\end{align}
which agree with $q_{\basepoint}$ near the marked point $1$, and have value $(q_{\tau_-}, q_{\tau})$ near the end.   We also obtain a map
\begin{equation}
  \phi_{\uLab} \co   \partial^\phi_\Sigma S \to \Ham(X),
\end{equation}
which is the identity near $1$, and agrees with $\phi$ near the end. %

Given this data, we form Lagrangian boundary conditions given by the path $X_{q_{\uLab}}$ on $\partial_\Sigma S $ and $\phi_{\uLab} X_{q^\phi_{\uLab}} $ on $\partial^\phi_\Sigma S $ (see Figure  \ref{fig:continuation_map_Poly_bimodule}); by construction, these boundary conditions agree with $X_{q_\sigma}$ at the marked point. Let $X_S$ to be the fibre bundle of Lagrangian boundary conditions over the quotient of the boundary of $S$ by the two components of its intersection with a fixed neighbourhood  $\nu_S e$ of the negative end over which these paths are locally constant. Define $X_{S} /{\sim} $  to be the quotient by the equivalence relation that collapses the intersections of $\nu_X \Lab_e$ with the fibres $X_{q_{\tau_-}}$ and $\phi X_{q_\tau}$ over the endpoints. We can associate to each map $u$ from $S$ to $X$ with such boundary conditions a map
\begin{align}
\partial u \co & \partial S \to X_{S} /{\sim}.
\end{align}
We equip this quotient with a metric so that, for a loop in $X_S$, the norm of the homology class in $H_{1}(X_{q_{\sigma}}, \bZ)$  is bounded by the length in the quotient.

Let $J(\uLab)$ be a family of almost complex structures on the strip which agree with $J$ outside of $\nu_X \Lab_e$ and agree with  $J(\Lab_e)$ near the end $e$.
\begin{lem} \label{lem:isoperimetric_constant_map_between_bimodules}
There is a constant $C$, depending on the collection $\cA$ of Lagrangians and the choices of almost complex structures, but independent of the choice of cover as long as it is sufficiently fine, so that for each pseudoholomorphic map $u$ with the above boundary conditions, the length of the path $ \partial u / {\sim} $ is bounded by the product of the energy with $C$.
\end{lem}
\begin{proof}
  This is a direct consequence of Lemma \ref{lem:basic_isoperimetric_moving_L}, for an appropriate parametrised problem that restricts to the above moduli space for each sufficiently fine cover indexed by $\Sigma$, and for each choice of basepoints $q_\sigma$, $q_{\tau}$, and $q_{\tau_-}$ associated to an element $\sigma$ of $\Sigma$ and a pair of elements $\tau$ and $\tau_-$ of $\Sigma_\sigma$.

  We can use the same strategy described before imposing the constraint in Equation \eqref{eq:isoperimetric_condition_Phi}: once a Hamiltonian isotopy $\Phi$ from the identity to a diffeomorphism $\phi$ satisfying Condition \eqref{eq:discrete_intersection_perturb_diagonal} is fixed, we may consider the family of boundary conditions on $S$ parametrised by a triple $\vec{q} = (q, q',q'_-) \in Q^3$ such that the distance from $q$ to $q'$ and $q'_-$  is less than or equal to $1$: the boundary  conditions are given by the unique short affine linear path from $q$ to $q'$ along one boundary segment, and by applying $\Phi$ to the path from $q$ to $q'_-$ along the other. As above, we obtain an equivalence relation $\sigma$ on the boundary conditions $X_{S, \vec{q} } $, by collapsing the intersection of $\nu_X(X_{q'_{-}} \cap \phi X_{q'})$ with the boundary conditions over the endpoints.  

  According to Lemma \ref{lem:basic_isoperimetric_moving_L},  there is a constant $C$ independent of $u$, such that, whenever $u \co S \to X$ is a $J$-holomorphic curve with any of the above boundary conditions, the length of the path $ \partial u / {\sim} $ is bounded by the product of the energy with $C$, up to an additive constant which is also independent of $u$. Moreover, the same reverse isoperimetric constant holds after perturbing the almost complex structure, along the end, as long as the perturbation takes place in $\nu_X(X_{q'_{-}} \cap \phi X_{q'})$.
\end{proof}

Given the above result, we may require our choice of cover to satisfy the conditions that
\begin{equation} \label{eq:isoperimetric_condition_Phi_z}
\diam \cN_\sigma < 1/ 8C
\end{equation}
for the constant $C$ in Lemma \ref{lem:isoperimetric_constant_map_between_bimodules}, in addition to all previous imposed constraints. This means in particular that $C$ depends on the collection $\cA$ of Lagrangian, on the fibration $X \to Q$, on the Hamiltonian diffeomorphism $\phi$, as well as on the path $\Phi$ (in addition to the auxiliary choices of almost complex structures).

Let $  \Rbar(\uLab)$ denote the moduli space of finite energy stable $J(\uLab)$ holomorphic strips, with moving Lagrangian boundary conditions given by the paths $X_{q_{\uLab}} $ and $\Phi_{\uLab} X^\phi_{q_{\uLab}}$. We have a natural evaluation map at $\pm 1$:
\begin{equation}
      \Rbar(\uLab) \to  \Crit(\Lab) \times X_{q_{\sigma}}.
\end{equation}
By construction, the reverse isoperimetric inequality from Equation \eqref{eq:isoperimetric_condition_Phi_z} applies to this moduli space.

\begin{figure}[h]
  \centering
\begin{tikzpicture}
\newcommand*{\radius}{1.5}
\newcommand*{\tinyradius}{.05}
\draw [fill=black] (0,0) circle (2*\tinyradius);
\coordinate [label=below:$X_{q_{\tau_-}}$] () at  (-1.3*\radius,0);
\coordinate [label=above:$\phi X_{q_{\tau}}$] () at  (-1.4*\radius,0);
\coordinate [label=below:$X_{q_{\uLab}} $] () at  (0,-\radius);
\coordinate [label=above:$\Phi_{\uLab} X_{q^\phi_{\uLab}}$] () at  (0,\radius);
\coordinate [label=above:$X_{q_{\basepoint}}$] () at  (1.3*\radius,0);
\draw [fill=black] (0:\radius) circle (\tinyradius);
\draw ([shift=(2:\radius)]0,0) arc (2:178:\radius);

\draw[->] (180:\radius-\tinyradius) -- (180:\radius+3*\tinyradius);
\draw ([shift=(182:\radius)]0,0) arc (182:358:\radius);
\draw (2*\radius,0) -- (3*\radius,0);
\draw [->] (1*\radius,0)  -- (2*\radius,0);
\coordinate [label=above:$\nabla f_{\Lab}$] () at  (2*\radius,0);
\end{tikzpicture}
  \caption{A representation of an element of the moduli space $\RTbar(\uLab)$, for $\Lab = \{ \tau_-, \tau^\phi \} $, and basepoint $\basepoint$.}
 \label{fig:continuation_map_Poly_bimodule}
\end{figure}

We now consider the fibre product
\begin{equation}
 \RTbar(\uLab)  \coloneqq   \Rbar(\uLab)  \times_{  X_{q_{\sigma}} } \Tbar_+(\basepoint,\basepoint)
\end{equation}
where $ \Tbar_+(\basepoint,\basepoint) $ is the moduli space of perturbed half-gradient flow lines introduced in Section \ref{sec:morse-theor-prod}. We have a natural evaluation map
\begin{equation}
  \RTbar(\uLab)  \to   \Crit(\tau_-,\tau^\phi) \times  \Crit(\basepoint,\basepoint),
\end{equation}
whose fibre at a pair $(x_{0}, x_{1})$ we denote
\begin{equation}
  \RTbar(x_{0},x_{1}).    
\end{equation}
Choosing paths connecting the intersections of $X_{q_{\basepoint}}$ with the sections associated to $\tau$, $\tau_-$, and $\basepoint$ yields isomorphisms of local systems $U^{P_{\tau_-}}_{\tau_-} \cong U^{P_{\tau_-}}_{\basepoint} $ and $ U^{P_\tau}_{\tau} \cong U^{P_\tau}_{\basepoint}$. Using this together with parallel transport along the boundary of an element $u$ of this moduli space yields a map
\begin{align}
   \Hom^c_\Lambda(U^{P_{\tau_-}}_{\basepoint,x_{1}}, U^{P_\tau}_{\basepoint,x_{1}}) & \to  \Hom^c_\Lambda(U^{P_{\tau_-}}_{\tau_-,x_{0}}, U^{P_\tau}_{\tau,x_{0}}) \\ \label{eq:parallel_transport_map_associated_to_continuation_id_phi}
\phi \mapsto & z^{[\partial_\Sigma u]} \cdot \psi \cdot z^{[\partial^{\phi}_\Sigma u]},
\end{align}
where $\partial_\Sigma u$ and $\partial^\phi_{\Sigma} u $  are the restrictions of $u$ to the boundary components $\partial_\Sigma S$ and $\partial^\phi_\Sigma S$.

For generic choices of the almost complex structure and the Morse perturbation, $ \RTbar(x_{0},x_{1}) $ is a manifold with boundary of dimension $\deg(x_{0}) - \deg(x_{1})$, and rigid elements of this moduli space induce an isomorphism $\delta_u$ of orientation lines. Tensoring these with the parallel transport maps, we define
\begin{equation} \label{eq:cochain_map_diagonal_to_perturbed_diagonal}
\kappa \co   \Poly_\sigma( P_{\tau_-},P_\tau)  \to    \Delbar(\tau_-,\tau)
\end{equation}
to be given by the sum
\begin{equation}
  \kappa | \phi \otimes \delta_{x_{1}} \coloneqq - \sum_{u \in \cT \cR^1_{q}(x_{0},x_{1})} T^{E(u)}  z^{[\partial u]} \cdot \psi \cdot z^{[\partial_- u]}\otimes \delta_u.
\end{equation}
The reverse isoperimetric inequality for moduli problems with moving Lagrangian boundary conditions, together with Condition \eqref{eq:isoperimetric_condition_Phi_z}, implies that this map is well-defined and continuous. 
\begin{rem}
  The minus sign above accounts for the fact that we use reduced gradings to define the differential on $\Poly_\sigma( P_{\tau_-},P_\tau)$, but unreduced gradings for $\Delbar(\tau_-,\tau)$.
\end{rem}

To prove that this map is a cochain equivalence, we construct a two-sided homotopy inverse:  Given a sequence $\Lab = ( \tau_-, \tau^\phi)$, we denote by $\Rbar_{\uLab}$ a copy of $\Rbar_{2, \underline{1}}$ which we now think of as equipped with a marked point at $-1$ and a puncture at $1$. By reflecting the data chosen above, we obtain moving Lagrangian boundary conditions which agree with $X_{q_{\basepoint}} $ near $-1$ and with the pair $(X_{q_{\tau_-}},\phi X_{q_\tau})$, near  the puncture.  We obtain a moduli space $\Rbar(\uLab)$ with an evaluation map to $X_{q_{\basepoint}} \times  \Crit(\Lab_e)$. Taking the fibre product over $X_{q_{\basepoint}}$ with the moduli space  $ \cT_-(\basepoint,\basepoint) $ of perturbed flow lines, we obtain the moduli space $  \RTbar(\uLab)$, equipped with an evaluation map to the product of  $ \Crit(\basepoint,\basepoint) $ at the negative end, and $\Crit(\tau_-,\tau^\phi)$ at the positive end. Counts of rigid elements of this moduli space define a map
\begin{equation}
  \Delbar(\tau_-,\tau) \to  \Poly_\sigma( P_{\tau_-},P_\tau). 
\end{equation}

We omit the proof that this is a left and right homotopy inverse to Equation (\ref{eq:cochain_map_diagonal_to_perturbed_diagonal}), noting only that this requires a further reverse isoperimetric constraint on the cover.  Passing to cohomology proves Lemma \ref{lem:map_diagonal_to_perturbed}.

\subsubsection{Computing the bimodule: distant polytopes}
\label{sec:comp-bimod-dist}

In this section, we compute the group $ H\Delbar(\tau_-, \tau)$ associated to the Hamiltonian diffeomorphism $\phi$ chosen at the beginning of Section \ref{sec:pertubed-diagonal}. We fix as well the Hamiltonian isotopy $\Phi$ from the identity to $\phi$ chosen in the previous section:
\begin{lem} \label{lem:disjoint_small_polytopes_trivial_perturbed_HF}
If the cover $\Sigma$ is sufficiently fine, and $\diam P_\sigma \ll \diam \cN_\sigma$ for all elements  $\sigma \in \Sigma$, then the group $H\Delbar(\tau_-, \tau)$ vanishes whenever $P_\tau$ and $P_{\tau_-}$ are not contained in a common chart $\cN_\sigma$.
\end{lem}
\begin{proof}[Sketch of proof:]
  
The difference in scale between $P_\sigma$ and $\cN_\sigma$ will be set by a reverse isoperimetric condition. Consider the space of discs with punctures at $\pm 1$, and two interior marked points, which lie on the real axis and are equidistant from the origin. This moduli space has two boundaries, corresponding to the two marked points lying at the origin, and to the breaking of the domain into two components.  We equip the curves over the moduli space with Lagrangian boundary conditions as follows (see Figure \ref{fig:null_homotopy}): for each of these curves, we equip the  boundary segment from $-1$ to $+1$, going along the lower semi-circle, with the constant boundary condition $X_{q_{\tau_-}}$. Along the other end, we take a family of paths from $ \phi X_{q_\tau}$ to itself, which is the constant path when the two marked points agree, and with the concatenation of the path $\Phi X_{q_\tau}$ and its inverse when the disc breaks into two components.

\begin{figure}[h]
  \centering
\begin{tikzpicture}
\newcommand*{\radius}{1}
\newcommand*{\tinyradius}{.025}
\newcommand*{\bigradius}{1.5}
\draw [fill=black] (.5*\bigradius,0) circle (2*\tinyradius); 
\draw [fill=black] (-.5*\bigradius,0) circle (2*\tinyradius); 
\coordinate [label=above:$ \phi X_{q_{\tau}}$] () at  (1.3*\bigradius,0);
\coordinate [label=above:$ \phi X_{q_{\tau}}$] () at  (-1.3*\bigradius,0);
\coordinate [label=below:$X_{q_{\tau_-}}$] () at  (0,-\bigradius);
\coordinate [label=above:$X_{q_{\tau}}$] () at  (0,\bigradius);

\draw[->] (0:\bigradius+2*\tinyradius) -- (0:\bigradius-6*\tinyradius);
\draw ([shift=(1:\bigradius)]0,0) arc (2:179:\bigradius);

\draw[->] (180:\bigradius-2*\tinyradius) -- (180:\bigradius+6*\tinyradius);
\draw ([shift=(181:\bigradius)]0,0) arc (181:359:\bigradius);

\begin{scope}[shift={(-3*\bigradius,0)}]

\draw [fill=black] (0,0) circle (2*\tinyradius); 
\coordinate [label=above:$ \phi X_{q_{\tau}}$] () at  (0,\bigradius);
\coordinate [label=below:$X_{q_{\tau_-}}$] () at  (0,-\bigradius);

\draw[->] (0:\bigradius+2*\tinyradius) -- (0:\bigradius-6*\tinyradius);
\draw ([shift=(1:\bigradius)]0,0) arc (2:179:\bigradius);

\draw[->] (180:\bigradius-2*\tinyradius) -- (180:\bigradius+6*\tinyradius);
\draw ([shift=(181:\bigradius)]0,0) arc (181:359:\bigradius);
\end{scope}

\begin{scope}[shift={(2.5*\bigradius,0)}]

\coordinate [label=below:$X_{q_{\tau_-}}$] () at  (-90:\radius);
\coordinate [label=above:$\Phi X_{q_{\tau}}$] () at  (90:\radius);
\draw [fill=black] (0,0) circle (2*\tinyradius); 
\draw[->] (0:\radius+2*\tinyradius) -- (0:\radius-6*\tinyradius);

\draw[->] (180:\radius-2*\tinyradius) -- (180:\radius+6*\tinyradius);

\draw ([shift=(2:\radius)]0,0) arc (2:178:\radius);

\draw ([shift=(182:\radius)]0,0) arc (182:358:\radius);

\begin{scope}[shift={(2*\radius,0)}]
\coordinate [label=above:$X_{q_\tau}$] () at  (-\radius, .5);

\draw[->] (180:\radius-2*\tinyradius) -- (180:\radius+6*\tinyradius);
\coordinate [label=below:$X_{q_{\tau_-}}$] () at  (-90:\radius);
\coordinate [label=above:$\Phi^{-1} X_{q_{\tau}}$] () at  (90:\radius);
\draw [fill=black] (0,0) circle (2*\tinyradius); 
\coordinate [label=above:$ \phi  X_{q_{\tau}}$] () at  (1.1*\bigradius,0);

\draw[->] (0:\radius+2*\tinyradius) -- (0:\radius-6*\tinyradius);
\draw ([shift=(2:\radius)]0,0) arc (2:178:\radius);
\draw ([shift=(182:\radius)]0,0) arc (182:358:\radius);
\end{scope}
\end{scope}

\newcommand*{\smalldash}{.1}
\newcommand*{\length}{3*\bigradius}
\draw[line width=1]   (-\length,-\bigradius-1.2*\radius) -- (\length,-\bigradius-1.2*\radius);
 \begin{scope}[shift={(-\length,-\bigradius-1.2*\radius )}]
 \draw [line width=1] (0,-\smalldash)--(0,\smalldash);
\coordinate [label=below:{$r=0$}] () at  (0,0);
\end{scope}
 \begin{scope}[shift={(\length,-\bigradius-1.2*\radius)}]
\draw [line width=1] (0,-\smalldash)--(0,\smalldash);
\coordinate [label=below:{$r=\infty$}] () at  (0,0);
\end{scope}
\end{tikzpicture}
  \caption{The boundary conditions, near from the ends, of the moduli space which induces a null-homotopy of Floer complexes with $\Lab = (\tau_-, \tau)$. The notation $\Phi^{-1}$ denotes the path $\Phi$ composed with the orientation reversing involution $t \to 1-t$ of the parametrising interval.}
 \label{fig:null_homotopy}
\end{figure}

We now pick almost complex structures for this family of Lagrangian boundary conditions: for the constant boundary conditions, we use the translation-invariant almost complex structure $J(\tau_-, \tau^\phi)$ chosen in Section \ref{sec:floer-cohom-pairs}, while near the broken curve, we assume that the almost complex structure is obtained by gluing, and agrees with $J(\tau_-, \tau^\phi)$ at the ends. Applying  Lemma \ref{lem:isoperimetric_constant_indep_gluing}, we obtain a reverse isoperimetric constant for holomorphic curves in this family.

Using the fact that $X_{q_\tau}$ and $X_{q_{\tau_-}}$ are disjoint, we find that the composition of maps associated to the broken curve vanishes, since it corresponds to a map factoring through a Floer complex that is the $0$ group by definition. Choosing the diameter of the elements of the cover $\{ P_\sigma \}_{\sigma \in \Sigma}$ to be much smaller than that of the cover $ \{ \cN_\sigma \}_{\sigma \in \Sigma}$, we obtain a null homotopy for the identity map on $\Delbar(\tau_-, \tau)$,  implying Lemma \ref{lem:disjoint_small_polytopes_trivial_perturbed_HF}. 
\end{proof}

\subsection{Global bimodule maps}
\label{sec:bimodule-maps}
Given a Lagrangian $L \in \cA$ we have defined left and right modules over $\HFuk$. In Section \ref{sec:moduli-space-holom}, we construct a map of bimodules
\begin{equation}
  H\cL_L  \otimes_\Lambda  H\scrR_L   \to \Delbar
\end{equation}
by using a moduli space of holomorphic triangles, with one moving Lagrangian boundary condition. We also construct a map of left modules
\begin{equation}
\Delbar \otimes_{\HFuk}   H\cL_L \to  H\cL_L.
\end{equation}
These are the missing ingredients in the statement of Proposition \ref{prop:coh_diagram_commutes}, which we proceed to prove in Section \ref{sec:proof-proposition}.

In order to compare these constructions with the local ones, we recall that Lemma \ref{lem:cohomological_compute_tensor_Hom_local} shows that the computation of tensor products with $\Delbar$ and morphisms of left module to $\Delbar$ are local. The main results of this section are summarised by the following, for which we recall that we denote by $j$ the embedding of the category $\HFuk_\sigma $ in $\HPoly_\sigma$, as in Equation \eqref{eq:j-embedding-HFuk-HPoly}:
\begin{prop}
  \label{prop:cohomology_commutative_diagram_local_globarl_Diagonal}
For each $\sigma \in \Sigma$, there are commutative diagrams of bimodules over $\HFuk_\sigma$  \begin{equation} 
  \begin{tikzcd}
 j^* H\cL_{L,\sigma} \otimes_\Lambda  j^* H\scrR_{L,\sigma}  \arrow{r}{}\arrow{d}{}  & j^* \Delta_{\HPoly_\sigma}   \arrow{d}{}  \\
 H\cL_{L} \otimes_\Lambda  H\scrR_L  \arrow{r}{}   & H \Delbar
  \end{tikzcd}
\end{equation}
and of left $\HFuk_\sigma$-modules:
\begin{equation} \label{eq:product_compatible_pertubed_diagonal}
  \begin{tikzcd}
j^*H \Delta_{\HPoly_\sigma} \otimes_{\HFuk_\sigma} j^* H\cL_{L,\sigma}   \arrow{r}{}  \arrow{d}{} & j^* H\cL_{L,\sigma}  \arrow{d}{} \\
 H\Delbar  \otimes_{\HFuk_\sigma} H\cL_L  \arrow{r}{} &  H\cL_{L}. 
  \end{tikzcd}
\end{equation}
\end{prop}
Note that the first commutative diagram in the above Proposition is equivalent to 
\begin{equation} \label{eq:diagram_right_modules_local}
   \begin{tikzcd}
j^* H\scrR_{L,\sigma}  \arrow{r}{}  \arrow{dd}{} & \Hom_{\HFuk_\sigma }\left( j^* H\cL_{L,\sigma}, j^* \Delta_{\HPoly_\sigma}\right)  \arrow{d}{}  \\
&  \Hom_{\HFuk_\sigma } \left(  j^* H\cL_{L,\sigma} ,  H \Delbar  \right)  \\
H\scrR_{L} \arrow{r}{}  & \Hom_{\HFuk_\sigma } \left( H\cL_{L},  H \Delbar  \right) \arrow{u}{} .
  \end{tikzcd} 
\end{equation}

Our previous results establish that, whenever $L$ meets $X_{q_\sigma}$ at one point, all the vertical arrows in Diagrams \eqref{eq:product_compatible_pertubed_diagonal} and \eqref{eq:diagram_right_modules_local} are isomorphisms if we restrict to the object $\sigma \in H \Fuk_\sigma $. Together with Lemma \ref{lem:cohomological_compute_tensor_Hom_local}, we conclude
\begin{lem}
  If $L$ meets $X_{q_\sigma}$ at one point, there are natural isomorphisms
  \begin{align}
  \Delbar(\sigma, \_) \otimes_{\HFuk} H\cL_L  \to & H\cL_{L}(\sigma) \\
H\scrR_L(\sigma)  \to  & \Hom_{\HFuk}\left( H\cL_{L},  \Delbar(\_, \sigma) \right).
  \end{align} \qed
\end{lem}

At this stage, we also note that we can establish the following result using the techniques developed in Section \ref{sec:comp-dimens-1}:
\begin{proof}[Proof of Lemma \ref{lem:H_map_surjective_sections}]
Choosing a generator for $\HFuk(P_{\tau}, P_{\rho})$ and $ H \cL_L(\rho)$ as rank-$1$ modules over $\Gamma^{P_{\rho}}$, we can write morphisms from $H \cL_{L'}$ to any module $M$ as the kernel of the map
\begin{equation}
\bigoplus_{  \sigma_{0}  \in \Sigma}   M(\sigma_{0}) \to \bigoplus_{  \sigma_{0} < \sigma_{1}  \in \Sigma}  M(\sigma_{1}). 
\end{equation}
This implies that $\Hom_{H\Fuk}(  H \cL_{L'}, \_)$ is right exact. Identifying   $\Delbar_{\HFuk} \otimes_{\HFuk} H \cL_L $, as a left module, with the cokernel of
\begin{equation}
\bigoplus_{  \rho_{0} < \rho_{1}  \in \Sigma} \Delbar_{\HFuk}(\rho_{0}, \_) \otimes_{\Gamma^{\rho_{0}}} \Gamma^{\rho_{1}} \to  \bigoplus_{ \rho_{0}\in \Sigma} \Delbar_{\HFuk}(\rho_{0}, \_) 
\end{equation}
we obtain the desired result.

\end{proof}
\begin{rem}
Combining the results of Section \ref{sec:perf-left-modul} with the discussion of Corollary \ref{cor:reduction_to_isomorphism_Hom_tensor}, gives a less computational proof of the $A_\infty$ refinement of this statement. The reader is invited to adapt that proof to the cohomological level, or alternatively adapt this computational proof to the $A_\infty$ refinement. 
\end{rem}

\subsubsection{Construction of the maps}
\label{sec:moduli-space-holom}

Let $S$ be a thrice-punctured disc, equipped with a fixed neighbourhood $\nu_s e$ of each end $e$. Let $\Lab$ be a  set of labels for the boundary components which is either (i) $(\sigma_-, L, \sigma^\phi)$,  or (ii) $( L , \sigma_-, \sigma^\phi)$, with $L \in \cA$, and $\sigma, \sigma_- \in \Sigma$. Pick a map
\begin{equation}
\Phi_\Lab \co \partial S \to \Ham(X)
\end{equation} 
which agrees with the identity on the components labelled $\sigma_-$ and $L$, and with a parametrisation of the path $\Phi$ along on the component labelled by $\sigma^\phi$ (specifically, we require that it interpolate between the identity along the end corresponding to the label $L$ and the map $\phi$ on the end labelled by $\sigma_-$).   Applying these Hamiltonians to the Lagrangians $L$, $X_{q_\sigma}$, and $X_{q_{\sigma_-}}$ yields moving Lagrangian boundary conditions on $S$ (see Figure \ref{fig:compatibility_bimodule_left_multiplication}).

In order to straightforwardly appeal to the reverse isoperimetric inequality, we require that the restriction of the path $\Phi_\Lab$ to each end $\nu_S e$ agree away from a fixed compact subset of the interior of  $\nu_S \Lab_e$ with (i) $\phi$ along the boundary component labelled by $\sigma^\phi$ and (ii) the identity on every other component. This restriction is exactly the same as the one made in the choice of perturbations used to define the Floer cohomology groups of fibres with $L$, or in Section \ref{sec:pertubed-diagonal} for pairs for fibres.  We choose a family of almost complex structures on $X$ parametrised by $S$, whose restriction to $\nu_S e$ is obtained from $J(\Lab_e)$ by a choice of strip-like ends. We denote by $\Rbar(\Lab)$ the corresponding moduli space of stable holomorphic discs.

\begin{figure}[h]
  \centering
\begin{tikzpicture}
\newcommand*{\radius}{1}
\newcommand*{\tinyradius}{.025}
\newcommand*{\bigradius}{1.5}
\begin{scope}[shift={(-4*\bigradius,0)},rotate=30]
\node at (0,0) {$\Rbar( L, \sigma_-, \sigma^\phi  ) $};

\draw (0,0)  ([shift=(-89:\bigradius)]0,0) arc (-89:29:\bigradius);

\draw[->] (-90:\bigradius+2*\tinyradius) -- (-90:\bigradius-6*\tinyradius);
\draw ([shift=(151:\bigradius)]0,0) arc (151:269:\bigradius);

\coordinate [label=below:$ L$] () at  (210:\bigradius);

\coordinate [label=right:$X_{q_{\sigma_-}}$] () at  (-30:1*\bigradius);

\draw[->] (30:\bigradius+2*\tinyradius) -- (30:\bigradius-6*\tinyradius);
\coordinate [label=above:$\phi X_{q_{\sigma}}$] () at  (30:\bigradius);
\draw ([shift=(31:\bigradius)]0,0) arc (31:149:\bigradius);

\draw[->] (150:\bigradius-2*\tinyradius) -- (150:\bigradius+6*\tinyradius);
\coordinate [label=above:$X_{q_{\sigma}}$] () at  (150:1.2*\bigradius);
\end{scope}

\begin{scope}[rotate=30]
\node at (0,0) {$\Rbar( \sigma_-, L, \sigma^\phi) $};

\draw (0,0)  ([shift=(-89:\bigradius)]0,0) arc (-89:29:\bigradius);

\draw[->] (-90:\bigradius+2*\tinyradius) -- (-90:\bigradius-6*\tinyradius);
\draw ([shift=(151:\bigradius)]0,0) arc (151:269:\bigradius);

\coordinate [label=right:$ L$] () at  (-30:\bigradius);

\coordinate [label=below:$X_{q_{\sigma_-}}$] () at  (210:1*\bigradius);

\coordinate [label=above:$X_{q_{\sigma}}$] () at  (30:\bigradius);
\draw[->] (30:\bigradius+2*\tinyradius) -- (30:\bigradius-6*\tinyradius);
\draw ([shift=(31:\bigradius)]0,0) arc (31:149:\bigradius);

\draw[->] (150:\bigradius-2*\tinyradius) -- (150:\bigradius+6*\tinyradius);
\coordinate [label=above:$\phi X_{q_{\sigma}}$] () at  (150:1.3*\bigradius);
\end{scope}

\end{tikzpicture}
  \caption{The boundary conditions for elements of the moduli spaces $\Rbar( \sigma_-, \sigma^\phi, L )$ and $\Rbar(\sigma_-, L, \sigma^\phi)$.}
\label{fig:compatibility_bimodule_left_multiplication}
\end{figure}

We write $X_S$  for the  total space of the boundary conditions labelled $\sigma$ and $\sigma^\phi$, modulo the relation which identifies each component of the inverse image of $\nu_S e$ to the corresponding fibre; this is a fibre bundle over a pair of closed intervals. We consider the equivalence relation $\sim$, which is the identity in the interior, and which is given, at a boundary Lagrangian labelled by an end $e$, by collapsing the components of the intersection with $\nu_X \Lab_e$ to a point. We equip this quotient with a metric such that norm of the homology class of a loop is bounded by the length of the projection.

At this stage, we need to impose a constraint on the diameter of the polytopes labelled by the elements of $\Sigma$. As before, the key point is that the moduli space $\Rbar(\Lab)$ can be described as the fibre over the pair $ (q_{\sigma_-}, q_{\sigma_-})$ of a moduli space parametrised by a neighbourhood of the diagonal in $Q \times Q$, consisting of pairs of points whose distance is smaller than the convexity radius. Lemma \ref{lem:basic_isoperimetric_moving_L} provides a reverse isoperimetric constant $C$ for elements of this parametrised moduli space, which allows us to impose the condition that the our chosen cover of $Q$ satisfy the constraint that
\begin{equation} \label{eq:epsilon_small_convergence_continuation_diagonal}
\diam \cN_\sigma  <   \frac{1}{8 C}
\end{equation}
for each $\sigma \in \Sigma$. 

This condition ensures the convergence of the count of rigid elements of the moduli space $\Rbar(\Lab)$, yielding maps
\begin{align}
\Delbar(\sigma_-,\sigma) \otimes_\Lambda    \cL_L(\sigma_-) & \to  \cL_L(\sigma) \\
\cL_L(\sigma)  \otimes_\Lambda  \scrR_L(\sigma_-)   &\to \Delbar(\sigma_-,\sigma).
\end{align}
Passing to cohomology, we obtain the maps
\begin{align}
H\Delbar(\sigma_-,\sigma) \otimes_\Lambda    H\cL_L(\sigma_-) & \to  H\cL_L(\sigma) \\
  H\cL_L(\sigma)  \otimes_\Lambda  H\scrR_L(\sigma_-)  & \to H\Delbar(\sigma_-,\sigma).
\end{align}

\subsubsection{Proof of Proposition \ref{prop:coh_diagram_commutes}}
\label{sec:proof-proposition}
Consider the moduli space $\Rbar_4$ of holomorphic discs with $4$ boundary punctures, one of which is distinguished as outgoing. Given a quadruple $\Lab = (L, \sigma_-, L', \sigma^\phi)$,  with  $\sigma, \sigma_- \in Q$ and $L, L' \in \cA$, we denote by $\Rbar_\Lab$ a copy of $\Rbar_4$ with the corresponding boundary labels, and by 
\begin{equation}
 \Sbar_{\Lab} \to \Rbar_{\Lab}   
\end{equation}
the universal curve over this moduli space.  Recall that $\Rbar_\Lab$ is homeomorphic to a closed interval, with boundary given by the two possible configurations of stable discs consisting of two components each of which is disc with $3$ punctures. These configurations are distinguished by the fact that the node is labelled by the pair $(L,L')$ in one case and by $(\sigma_-, \sigma^\phi)$ in the second.

\begin{figure}[h]
  \centering
\begin{tikzpicture}
\newcommand*{\bigradius}{1.6}
\newcommand*{\tinyradius}{.05}
\begin{scope}[rotate=-90]
\node at  (0,0) {$\Rbar(L,\sigma_-,L',\sigma^\phi)$};

\node at (-20:1.2*\bigradius) {$L$};
\node at (210:1.2*\bigradius) {$X_{q_\sigma}$};

\node at (140:1.2*\bigradius) {$L'$};
\node at (60:1.3*\bigradius) {$X_{q_{\sigma_-}}$};

\draw (0,0)  ([shift=(-88:\bigradius)]0,0) arc (-88:48:\bigradius);

\draw[->] (-90:\bigradius-\tinyradius) -- (-90:\bigradius+3*\tinyradius);
\draw ([shift=(182:\bigradius)]0,0) arc (182:268:\bigradius);

\draw[->] (50:\bigradius+\tinyradius) -- (50:\bigradius-3*\tinyradius);
\draw ([shift=(52:\bigradius)]0,0) arc (52:78:\bigradius);

\draw[->] (80:\bigradius+\tinyradius) -- (80:\bigradius-3*\tinyradius);
\draw ([shift=(82:\bigradius)]0,0) arc (82:178:\bigradius);

\draw[->] (180:\bigradius+\tinyradius) -- (180:\bigradius-3*\tinyradius);
\end{scope}
\newcommand*{\radius}{1.3}
\begin{scope}[shift={(-4.5,\radius-\bigradius)},rotate=-90]
\node at (-35:1.3*\radius) {$L$};
\node at (110:1.4*\radius) {$X_{q_{\sigma_-}}$};
\node at (-175:1.2*\radius) {$\phi  X_{q_\sigma}$};
\node at (-110:1.3*\radius) {$X_{q_\sigma}$};

\node at  (0,0) {$\Rbar(L,\sigma_-,\sigma^\phi)$};
\draw (0,0)  ([shift=(-88:\radius)]0,0) arc (-88:28:\radius);

\draw[->] (-90:\radius-\tinyradius) -- (-90:\radius+3*\tinyradius);
\draw ([shift=(152:\radius)]0,0) arc (152:268:\radius);

\draw[->] (30:\radius+\tinyradius) -- (30:\radius-3*\tinyradius);
\draw ([shift=(32:\radius)]0,0) arc (32:148:\radius);

\draw[->] (150:\radius+\tinyradius) -- (150:\radius-3*\tinyradius);
\begin{scope}[shift={(150:2*\radius)}, rotate=60]
\node at (75:1.3*\radius) {$L'$};

\node at (160:1.3*\radius) {$X_{q_\sigma}$};
 \node at  (0,0) {$\Rbar(\sigma_-,L',\sigma^\phi)$};
\draw (0,0)  ([shift=(-88:\radius)]0,0) arc (-88:28:\radius);

\draw[->] (-90:\radius-\tinyradius) -- (-90:\radius+3*\tinyradius);
\draw ([shift=(152:\radius)]0,0) arc (152:268:\radius);

\draw[->] (30:\radius+\tinyradius) -- (30:\radius-3*\tinyradius);
\draw ([shift=(32:\radius)]0,0) arc (32:148:\radius);

\draw[->] (150:\radius+\tinyradius) -- (150:\radius-3*\tinyradius);
\end{scope}
 \end{scope}

\begin{scope}[shift={(4,1.5*\radius)},rotate=-90]
\node at (-5:1.3*\radius) {$L$};
\node at (70:1.3*\radius) {$L'$};
\node at (-150:1.3*\radius) {$X_{q_\sigma}$};

\node at  (0,0) {$\Rbar(L,L',\sigma)$ };
\draw (0,0)  ([shift=(-88:\radius)]0,0) arc (-88:28:\radius);

\draw[->] (-90:\radius-\tinyradius) -- (-90:\radius+3*\tinyradius);
\draw ([shift=(152:\radius)]0,0) arc (152:268:\radius);

\draw[->] (30:\radius+\tinyradius) -- (30:\radius-3*\tinyradius);
\draw ([shift=(32:\radius)]0,0) arc (32:148:\radius);

\draw[->] (150:\radius+\tinyradius) -- (150:\radius-3*\tinyradius);
\begin{scope}[shift={(30:2*\radius)}, rotate=-60]
 \node at  (0,0) {$\Rbar(L,\sigma_-,L')$ };

\node at (90:1.3*\radius) {$X_{q_{\sigma_-}}$};

\draw (0,0)  ([shift=(-88:\radius)]0,0) arc (-88:28:\radius);

\draw ([shift=(152:\radius)]0,0) arc (152:268:\radius);

\draw[->] (30:\radius+\tinyradius) -- (30:\radius-3*\tinyradius);
\draw ([shift=(32:\radius)]0,0) arc (32:148:\radius);

\draw[->] (150:\radius+\tinyradius) -- (150:\radius-3*\tinyradius);
\end{scope}
 \end{scope}

\end{tikzpicture}
  \caption{The boundary conditions for different elements of $\Rbar(L,\sigma_-,L',\sigma^\phi)$.}
  \label{fig:homotopy_compositions_labels_boundary}
\end{figure}

We choose families of strip-like ends for all surfaces over $\Rbar_\Lab$, which we assume are compatible with those made in Sections \ref{sec:maps-relating-left}, \ref{sec:floer-cohom-pairs},  and \ref{sec:moduli-space-holom} over the boundary strata.  The choices made in these sections also determine moving Lagrangian boundary conditions on the fibres over the boundary. We extend these choices to  arbitrary surfaces representing elements of $ \Rbar_\Lab$, compatibly with gluing near the boundary of the moduli space, in such a way that the moving Lagrangian boundary conditions are (i) constant (with value $X_{q_{\sigma_-}}$, $L$ and $L'$) along the boundary segments with label $\sigma_-$, $L$ and $L'$, and (ii) given by a family of paths with endpoints on $X_{q_\sigma}$, interpolating between the constant path and the concatenation of the path $\Phi X_{q_{\sigma}}$ from $X_{q_\sigma} $ to $\phi (X_{q_\sigma}) $ and its inverse along the remaining boundary. Our previous choices also determine families of almost complex structures on $X$ parametrised by the fibres over the endpoints of the moduli space; we extend them to a family $J_S$ over each curve $S$ in $\Rbar_\Lab$  in such a way that the almost complex structure along the boundary components labelled $L$ and $L'$ respectively agree with $J_L$ and $J_{L'}$, and the restrictions to the ends agree with the choices made in the construction of the Floer cohomology groups.

Each such surface is equipped with a decomposition which we call the thick-thin decomposition; for a surface far from the boundary, the thin parts are neighbourhoods of the ends where the boundary conditions are locally constant and the almost complex structure is obtained by pull-back from an interval by a choice of strip-like end. For a surface near the boundary, there is an additional component coming from the gluing region.  Each component $\Theta$ of the thin part has a corresponding subset $\nu_X \Theta$ of $X$ equipped with a compact subset of the interior away from which the restrictions to $\Theta$ of the boundary conditions and the almost complex structures are constant.

For each surface $S \in \Sbar_\Lab$, let $X_S$ denote the fibre bundle over the pair of closed intervals obtained from the subset of $\partial S$ labelled by $\sigma^\phi$ and $\sigma_-$ by collapsing the components of the intersection with the thin part.  We denote by $X_S/{\sim}$ the quotient of $X_S$ by the equivalence relation which, in each fibre corresponding to a component $\Theta$ of the thin part, identifies the components of the intersection with $\nu_X \Theta$ to points. As before, we pick a metric on this quotient so that the lengths of loops in the quotient provide a bound for the norm of the homology class in a fibre.

By the results of Appendix \ref{sec:geometric-setup}, we can find a constant $C$ which uniformly controls the length of the projection to $X_S / {\sim}$ of the boundary of $J_S$-holomorphic curves in $X$ with the above boundary conditions.  We then require that
 \begin{equation}
\diam \nu_X \sigma < \frac{1}{8 C}.
\end{equation}

By construction,  we have an evaluation map
\begin{equation}
  \Rbar(\Lab) \to  \Crit(L,\sigma_-) \times   \Crit(\sigma_-,L') \times \Crit(L',\sigma) \times \Crit(L,\sigma).
\end{equation}
Choosing the data generically, the count of rigid elements of these moduli spaces thus induces a map
\begin{equation}
  \scrR_{L'}(\sigma_-)   \otimes_\Lambda  \cL_L(\sigma_-)  \to \Hom_{\Lambda}( \cL_{L'}(\sigma), \cL_{L}(\sigma)) 
\end{equation}
which gives a homotopy for the following diagram
\begin{equation}
  \begin{tikzcd}
    \scrR_{L'}(\sigma_-)   \otimes_\Lambda  \cL_L(\sigma_-)   \arrow{d} \arrow{r} &     CF^*(L,L')        \arrow{d} \\
\Hom_{\Lambda}(  \cL_{L'}(\sigma)  , \Delbar(\sigma_-, \sigma)  )  \otimes_\Lambda   \cL_L(\sigma_-)  \arrow{r} &  \Hom_{\Lambda}( \cL_{L'}(\sigma), \cL_{L}(\sigma)),
  \end{tikzcd}
\end{equation}
because the boundary strata of $ \Rbar(\Lab) $ over the boundary of $\Rbar_\Lab$ correspond to the two compositions. 
Passing to cohomology, we obtain the commutativity of Diagram \eqref{eq:cohomological_diagram}.

\subsubsection{Compatibility of the maps of left modules}
\label{sec:comp-with-bimod}

We now prove the commutativity of Diagram \eqref{eq:product_compatible_pertubed_diagonal}. Consider  pre-stable discs equipped with two boundary punctures, one of which is distinguished as outgoing and the other as incoming, and with one boundary marked point. Let $\Rbar_{3,2}$  denote the moduli space of stable maps from such discs to $(D^2)^2$ such that the projection to each factor has degree $1$, the outgoing end maps to $-1$ in both factors, the incoming end maps to $+1$ in the first factor, and the boundary marked point maps to $+1$ in the second factor. This moduli space is a $2$ dimensional manifold with corners which is homeomorphic as a topological manifold with boundary to a closed disc.

We can label each element of this moduli space by a disc equipped with a pair of interior marked points (in addition to the given boundary punctures and marked points), corresponding to the inverse image of the origin under the two maps: there is a unique point on the boundary of this moduli space where these marked points agree, lying on the stratum where both maps to the strip are supported on a component of the domain which is collapsed to the outgoing end by the forgetful map to $\Rbar_3$ (see the left side of Figure \ref{fig:comparison_maps_left_modules}). This corresponds to an embedding
\begin{equation}
\Rbar_{2,\underline{1}}  \times  \Rbar_{3} \subset \Rbar_{3,2},    
\end{equation}
and gives rise to the composition of the top and right arrows in Diagram \eqref{eq:product_compatible_pertubed_diagonal}.

We have as well an embedding
\begin{equation}
\Rbar_{3} \times \left(\Rbar_{2,\underline{1}} \right)^2 \subset \Rbar_{3,2},      
\end{equation}
shown on the right of Figure \ref{fig:comparison_maps_left_modules}, and corresponding to the corner stratum of the right hand side given by the configuration where the components carrying the degree $1$ maps project (under the forgetful map to $\Rbar_3$) to the boundary marked point and to the incoming puncture.  For the appropriate choices of Lagrangian boundary conditions and families of almost complex structures, this embedding corresponds to the composition of the bottom and left arrows in Diagram \eqref{eq:product_compatible_pertubed_diagonal}. %

To construct a homotopy in  Diagram \eqref{eq:product_compatible_pertubed_diagonal}, we fix a path
\begin{equation}
    \Rbar_{3,\underline{2}} \subset \Rbar_{3,2}
\end{equation}
interpolating between these two strata. 

\begin{figure}[h]
  \centering
\begin{tikzpicture}
\newcommand*{\bigradius}{1.5}
\newcommand*{\tinyradius}{.05}

\newcommand*{\radius}{1}
\draw ([shift=(-178:\radius)]0,0) arc (-178:-62:\radius);

\draw[->] (-60:\radius+\tinyradius)--(-60:\radius-3*\tinyradius);
\node at  (-120:1.2*\radius) {$L$};
\draw ([shift=(-58:\radius)]0,0) arc (-58:58:\radius);
\draw [fill=black] (60:\radius) circle (\tinyradius);
\node at  (0:1.3*\radius) {$X_{q_\sigma}$};
\draw ([shift=(62:\radius)]0,0) arc (62:178:\radius);

\draw[->] (180:\radius-\tinyradius)--(180:\radius+3*\tinyradius);
\node at  (120:1.3*\radius) {$X_{q_\sigma}$};
\begin{scope}[shift={(-2*\radius,0)}]
\node  at  (170:1.3*\radius) {$X_{q_\tau}$};
\node  at  (90:1.2*\radius) {$X_{q_{\sigma,\tau}}$};
\node  at  (-90:1.2*\radius) {$L$};
\draw ([shift=(-178:\radius)]0,0) arc (-178:-2:\radius);
\draw ([shift=(2:\radius)]0,0) arc (2:178:\radius);

\draw[->] (180:\radius-\tinyradius)--(180:\radius+3*\tinyradius);
\draw [fill=black] (0,0) circle (\tinyradius);

\end{scope}

\begin{scope}[shift={(4,0)}]
\draw ([shift=(-178:\radius)]0,0) arc (-178:-62:\radius);

\draw[->] (-60:\radius+\tinyradius)--(-60:\radius-3*\tinyradius);
\node  at  (-120:1.2*\radius) {$L$};
\draw ([shift=(-58:\radius)]0,0) arc (-58:58:\radius);

\draw[->] (60:\radius+\tinyradius)--(60:\radius-3*\tinyradius);
\node  at  (0:1.4*\radius) {$X_{q_{\tau_-}}$};
\draw ([shift=(62:\radius)]0,0) arc (62:178:\radius);

\draw[->] (180:\radius-\tinyradius)--(180:\radius+3*\tinyradius);
\node  at  (170:1.3*\radius) {$X_{q_\tau}$};

\begin{scope}[shift={(-60:\radius+\radius)}, rotate=-60]
\node  at  (15:1.3*\radius) {$X_{q_\sigma}$};
\draw ([shift=(-178:\radius)]0,0) arc (-178:-2:\radius);

\draw[->] (0:\radius+\tinyradius)--(0:\radius-3*\tinyradius);
\draw [fill=black] (0,0) circle (\tinyradius);
\draw ([shift=(2:\radius)]0,0) arc (2:178:\radius);

\node  at  (-90:1.2*\radius) {$L$};
\node  at  (90:1.35*\radius) {$X_{q_{\sigma,\tau}}$};
\end{scope}
\begin{scope}[shift={(60:\radius+\radius)}, rotate=60]
\node  at  (0:1.2*\radius) {$X_{q_\sigma}$};
\draw ([shift=(-178:\radius)]0,0) arc (-178:178:\radius);

\draw [fill=black] (0,0) circle (\tinyradius);
\draw [fill=black] (0:\radius) circle (\tinyradius);
\node  at  (145:1.3*\radius) {$\phi X_{q_\tau}$};

\end{scope}
\end{scope}

\end{tikzpicture}
  \caption{The boundary of the moduli space giving rise to the homotopy between the two compositions in Diagram \eqref{eq:product_compatible_pertubed_diagonal}. On the right, the component carrying the output corresponds to the left side of Figure \ref{fig:compatibility_bimodule_left_multiplication} while the other two components correspond to the right side of Figure \ref{fig:continuation_map_HFuk_bimodule} and to  Figure \ref{fig:continuation_map_Poly_bimodule}. On the left, the component carrying the output corresponds to the right side of Figure \ref{fig:continuation_map_Poly_bimodule}, while the other component corresponds to the right side of Figure \ref{fig:continuation_map_HFuk_bimodule}. The two interior marked points on the right have collided on the left.}
  \label{fig:comparison_maps_left_modules}
\end{figure}

Let $\basepoint$ be a basepoint in $\Sigma$, $L$ a Lagrangian in $\cA$, and $\tau, \tau_-$ elements of $\Sigma_{\basepoint}$. Let $\Lab$ denote the sequence $(L,  \tau_-,  \tau^\phi)$, and denote by $\Rbar_{\uLab}$ a copy of the space $\Rbar_{3, \underline{2}}$, with corresponding boundary labels on the complement of the boundary marked point. Let $\Sbar_{\uLab}$ denote the universal curve over $\Rbar_{\uLab}$.

Over the boundary boundary stratum of $\Rbar_{\uLab}$ labelled by a configuration with two disc components, the boundary labels are given by $(L,   \tau^\phi)$ on the disc carrying an interior marked point, and $(L,  \tau_-, \tau^\phi)$  on the other. The first case was already considered in Section \ref{sec:comp-local-glob}, where we chose a path of Lagrangians between the pairs $(L_{\basepoint}, X_{q_{\basepoint}})$ and $(L_{\tau}, X_{q_\tau})$. In the second case, we use the triple of constant Lagrangians  $(L_{\basepoint}, X_{q_{\basepoint}}, X_{q_{\basepoint}})$.

Over the other boundary stratum of $\Rbar_{\uLab}$ the choice of moving Lagrangian boundary conditions is immediately obtained from the choices made in Sections \ref{sec:comp-local-glob} for the disc with labels $(L,  \tau_-^\phi) $, in  Section \ref{sec:computing-bimodule-nearby} for the disc with interior marked point with labels $( \tau_-,  \tau^\phi)$, and  in  Section \ref{sec:moduli-space-holom} for the disc with labels $(L, \tau_-, \tau^\phi)$.

We now pick families of moving Lagrangian boundary conditions, which are Hamiltonian along the boundary labelled $L$, and almost complex structures on the universal curve over $\Rbar_{\uLab}$, which near the endpoints of this interval, are obtained by gluing, up to a perturbation term supported in the thick part away from the boundaries in which we apply the reverse isoperimetric inequality. In this way, we achieve transversality using standard methods for the corresponding moduli space $\Rbar(\uLab)$, while still being able to appeal to Lemma  \ref{lem:isoperimetric_constant_indep_gluing} to ensure the existence of a uniform reverse isoperimetric constant for elements of this moduli space. To avoid bubbling problems, we assume that the family of almost complex structures at each point $z$ along the boundary condition labelled $L$ agrees with $J_L$. We also assume that the boundary condition at the marked point with label $(  \tau_-, \tau^\phi)$  agrees with $X_{q_{\basepoint}}$, so that we have an evaluation map
\begin{equation}
\Rbar(\uLab) \to  X_{q_{\basepoint}}. 
\end{equation}

Assuming that the diameter of $\nu_Q \basepoint$ is sufficiently small relative to the reverse isoperimetric constant for this moduli space, the count of elements of the fibre product
\begin{equation}
 \RTbar(\uLab) \coloneqq \Rbar(\uLab) \times_{X_{q_{\basepoint}}} \Tbar_+(\basepoint,\basepoint)
\end{equation}
space defines a homotopy for the diagram
\begin{equation}
  \begin{tikzcd}
   \Poly_\sigma(P_{\tau_-} ,  P_\tau) \otimes_\Lambda  \cL_{L}(P_{\tau_-}) \arrow{r}{}  \arrow{d}{} & \cL_{L}(P_{\tau})   \arrow{d}{} \\
  \Delbar(\tau_-, \tau) \otimes_\Lambda  \cL_L(\tau) \arrow{r}{}    & \cL_L(\tau)  .
  \end{tikzcd}
\end{equation}
Passing to cohomology yields the commutativity of Diagram \eqref{eq:product_compatible_pertubed_diagonal}.

\subsubsection{Comparison with the map to the perturbed diagonal}
\label{sec:comparison-with-map}

In this section, we prove the commutativity of Diagram (\ref{eq:diagram_right_modules_local}); most of the arguments are essentially the same as in the previous section, except that we shall encounter a moduli space with a boundary facet containing a node for which the boundary label is given by a pair of Lagrangians which agree. This requires us to add  to the moduli space an additional component consisting of pairs of discs connected by a gradient flow line of varying length. This type of additional cobordism appears throughout the literature when combining Morse-theoretic and Floer-theoretic moduli spaces.

\begin{figure}[h]
  \centering
\begin{tikzpicture}
\newcommand*{\bigradius}{1.5}
\newcommand*{\tinyradius}{.05}

\newcommand*{\radius}{1}
\begin{scope}[shift={(-5,0)}]
\draw ([shift=(-178:\radius)]0,0) arc (-178:-62:\radius);

\draw[->] (-60:\radius+\tinyradius)--(-60:\radius-3*\tinyradius);
\node at  (-120:1.2*\radius) {$\basepoint$};
\draw ([shift=(-58:\radius)]0,0) arc (-58:58:\radius);

\draw[->] (60:\radius+\tinyradius)--(60:\radius-3*\tinyradius);
\node at  (0:1.2*\radius) {$L$};
\draw ([shift=(62:\radius)]0,0) arc (62:178:\radius);
\draw [fill=black](180:\radius) circle (\tinyradius);
\node at  (120:1.3*\radius) {$\basepoint$};

\draw [fill=black](180:1.5*\radius) circle (\tinyradius);
\draw (180:2*\radius)--(180:\radius);
\begin{scope}[shift={(-3*\radius,0)}]
\draw ([shift=(-178:\radius)]0,0) arc (-178:-2:\radius);
\draw ([shift=(2:\radius)]0,0) arc (2:178:\radius);

\draw[->] (180:\radius-\tinyradius)--(180:\radius+3*\tinyradius);
\draw [fill=black] (0,0) circle (\tinyradius);
\draw [fill=black] (0:\radius) circle (\tinyradius);
\end{scope}
\end{scope}

\draw ([shift=(-178:\radius)]0,0) arc (-178:-62:\radius);

\draw[->] (-60:\radius+\tinyradius)--(-60:\radius-3*\tinyradius);
\node at  (-120:1.2*\radius) {$\basepoint$};
\draw ([shift=(-58:\radius)]0,0) arc (-58:58:\radius);

\draw[->] (60:\radius+\tinyradius)--(60:\radius-3*\tinyradius);
\node at  (0:1.2*\radius) {$L$};
\draw ([shift=(62:\radius)]0,0) arc (62:178:\radius);
\draw (180:\radius) circle (\tinyradius);
\node at  (120:1.3*\radius) {$\basepoint$};
\begin{scope}[shift={(-2*\radius,0)}]
\node  at  (170:1.3*\radius) {$\tau^\phi$};
\node  at  (-170:1.3*\radius) {$\tau_-$};
\draw ([shift=(-178:\radius)]0,0) arc (-178:-2:\radius);
\draw ([shift=(2:\radius)]0,0) arc (2:178:\radius);

\draw[->] (180:\radius-\tinyradius)--(180:\radius+3*\tinyradius);
\draw [fill=black] (0,0) circle (\tinyradius);
\draw [fill=black] (0:\radius) circle (\tinyradius);
\end{scope}

\begin{scope}[shift={(3,0)}]
\draw ([shift=(-178:\radius)]0,0) arc (-178:-62:\radius);

\draw[->] (-60:\radius+\tinyradius)--(-60:\radius-3*\tinyradius);
\node  at  (-120:1.2*\radius) {$\tau_-$};
\draw ([shift=(-58:\radius)]0,0) arc (-58:58:\radius);

\draw[->] (60:\radius+\tinyradius)--(60:\radius-3*\tinyradius);
\node  at  (0:1.2*\radius) {$L$};
\draw ([shift=(62:\radius)]0,0) arc (62:178:\radius);

\draw[->] (180:\radius-\tinyradius)--(180:\radius+3*\tinyradius);
\node  at  (120:1.3*\radius) {$\tau^\phi$};
\begin{scope}[shift={(-60:\radius+\radius)}, rotate=-60]
\node  at  (-45:1.2*\radius) {$\basepoint$};
\draw ([shift=(-178:\radius)]0,0) arc (-178:-2:\radius);

\draw[->] (0:\radius+\tinyradius)--(0:\radius-3*\tinyradius);
\draw [fill=black] (0,0) circle (\tinyradius);
\draw ([shift=(2:\radius)]0,0) arc (2:178:\radius);

\end{scope}
\begin{scope}[shift={(60:\radius+\radius)}, rotate=60]
\node  at  (45:1.2*\radius) {$\basepoint$};
\draw ([shift=(-178:\radius)]0,0) arc (-178:-2:\radius);
\draw[->] (0:\radius+\tinyradius)--(0:\radius-3*\tinyradius);

\draw [fill=black] (0,0) circle (\tinyradius);
\draw ([shift=(2:\radius)]0,0) arc (2:178:\radius);
\end{scope}
\end{scope}

\newcommand*{\smalldash}{.1}
\draw[line width=1]   (-6.5,-2.5*\bigradius) -- (4,-2.5*\bigradius);
 \begin{scope}[shift={(-\radius,-2.5*\bigradius )}]
 \draw [line width=1] (0,-\smalldash)--(0,\smalldash);
\coordinate [label=below:{$r=-\infty$}] () at  (0,0);
\coordinate [label=above:{$\ell=0$}] () at  (0,0);
\end{scope}
 \begin{scope}[shift={(4,-2.5*\bigradius)}]
\draw [line width=1] (0,-\smalldash)--(0,\smalldash);
\coordinate [label=below:{$r=0$}] () at  (0,0);
\end{scope}
 \begin{scope}[shift={(-6.5,-2.5*\bigradius)}]
\draw [line width=1] (0,-\smalldash)--(0,\smalldash);
\coordinate [label=above:{$\ell=\infty$}] () at  (0,0);
\end{scope}
\end{tikzpicture}
  \caption{The boundary of the moduli space giving rise to the homotopy between the two compositions in Diagram (\ref{eq:diagram_right_modules_local}).}
 \label{fig:comparison_map_right_modules}
\end{figure}

Given a triple $\Lab = (\tau_-,L,\tau)$ with  $L \in \cA$ and $\tau,\tau_{-} \subseteq \cN_\sigma$ as before, we shall construct a moduli space interpolating between the two moduli spaces defining the compositions being compared. To this end, we let  $\Rbar_{\uLab}$ denote a copy of the space $\Rbar_{3, \underline{2}}$, which we now think of as parametrising discs with $3$ boundary punctures and $2$ interior marked points whose position is constrained, as illustrated in the middle and right of Figure \ref{fig:comparison_map_right_modules}.

Over the boundary stratum of $\Rbar_{\uLab}$ labelled by a configuration consisting of three discs, the moving Lagrangian boundary conditions are given by the choices made in Section \ref{sec:comp-local-glob} for the two discs carrying marked points, and by those in Section \ref{sec:moduli-space-holom} for the thrice-punctured disc.

Over the other boundary stratum of $\Rbar_{\uLab}$, the choice of basepoint $\sigma \in \Sigma$ yields moving Lagrangian conditions on the boundary of the disc carrying the interior marked point, which were fixed in Section \ref{sec:computing-bimodule-nearby}. We forget the labels on the second component, and consider the constant boundary conditions $(X_{q_\sigma}, L, X_{q_{\sigma}})$. Note that these two conditions are compatible because, near the marked point, the boundary conditions considered in Section \ref{sec:computing-bimodule-nearby} are constant with value $X_{q_\sigma}$; this is why we label the middle of Figure \ref{fig:comparison_map_right_modules} with a node connecting the two components (instead of a pair of ends with matching labels). 

As in the previous section we interpolate between these two boundary conditions, and the corresponding choices of almost complex structures, to obtain a moduli space $\Rbar(\uLab)$, equipped with an evaluation map:
\begin{equation}
\Rbar(\uLab) \to \Crit(L, \sigma) \times \Crit(\sigma, L) \times \Crit(\tau_-, \tau^\phi).   
\end{equation}
There is a reverse-isoperimetric constant for these moduli space, with respect to which we impose the condition that the cover be sufficiently small.

The stratum of the boundary of this moduli space corresponding to discs with $3$ components evidently gives rise to the composition
\begin{equation}
 j^* H\scrR_{L,\sigma} \to H\scrR_{L}  \to \Hom_{\HFuk_\sigma } \left( j^* H\cL_{L,\sigma}  ,   \Delbar  \right),
\end{equation}
around the left and bottom of Diagram (\ref{eq:diagram_right_modules_local}). The other boundary does not, however, correspond to the other composition, because it should involve a Floer cohomology group of local systems over $X_{q_\sigma}$, which has been defined using Morse theory.

We therefore introduce a space $\Tbar_{[0,\infty]}(\sigma,\sigma)$ of (perturbed) gradient flow lines of $f_{\sigma,\sigma}$ with arbitrary length, which is equipped with a natural evaluation map
\begin{equation}
 \Tbar_{[0,\infty]}(\sigma,\sigma)  \to X^2_{q_{\sigma}}
\end{equation}
corresponding to the two ends of the flow line. The boundary is covered by a copy of $X_{q_{\sigma}}$ on which the evaluation map is the diagonal (corresponding to the gradient flow line of length $0$), and the closure of the codimension $1$ stratum
\begin{equation}
\cT_+(\basepoint,\basepoint) \times_{\Crit(\sigma,\sigma)} \cT_-(\basepoint,\basepoint)
\end{equation}
which corresponds to gradient lines of infinite length. The compactification is obtained by allowing broken flow lines with multiple components, i.e. by adding the strata:
\begin{equation}
\cT_+(\basepoint,\basepoint) \times_{\Crit(\sigma,\sigma)} \Tbar(\sigma,\sigma) \times_{\Crit(\sigma,\sigma)} \cT_-(\basepoint,\basepoint).
\end{equation}
We now define $\RTbar(\uLab)$ to be the union of $\Rbar(\uLab)$ with the fibre product
\begin{equation}
\Rbar(\uLab_{e_\out}) \times_{X_{q_{\sigma}}}  \Tbar_{[0,\infty]}(\sigma,\sigma)  \times_{X_{q_{\sigma}}}  \Rbar(\sigma,L,\sigma),
\end{equation}
where $\Lab_{e_\out} = (\tau_-, \tau^\phi)$. By construction, we have an evaluation map
\begin{equation}
   \RTbar(\uLab) \to   \Crit(L, \sigma) \times \Crit(\sigma, L) \times \Crit(\tau_-, \tau^\phi).   
\end{equation}
For generic choices of Floer and Morse data, this moduli space gives rise to a homotopy for the diagram
\begin{equation}
  \begin{tikzcd}
    \cL_{L}(P_{\tau}) \otimes_\Lambda  \scrR_{L}(\tau_-) \arrow{d} \arrow{r} &  \Poly_\sigma(\tau_-,\tau) \arrow{d} \\
  \cL_{L}(\tau) \otimes_\Lambda  \scrR(\tau_-)  \arrow{r}  &  \Delbar(\tau_-, \tau) .
  \end{tikzcd}
\end{equation}
Passing to cohomology, and using adjunction, we obtain the commutativity of Diagram (\ref{eq:diagram_right_modules_local}).

\section{Higher moduli spaces}
\label{sec:famil-cont-equat}

In this section, we consider the families of holomorphic curves which will be used in the $A_\infty$ refinement of the constuctions of Section \ref{sec:cohom-constr}. We separate considerations of convergence and transversality by proceeding in two steps: (i) we first write families of equations with moving Lagrangian boundary conditions  for which we state reverse isoperimetric inequalities that are robust under a certain class of perturbations, then (ii) we perturb these equations to achieve transversality among the Lagrangian boundary conditions and regularity for the moduli spaces of holomorphic curves, within the allowable class.

\subsection{Abstract moduli spaces of discs}
\label{sec:abstr-moduli-spac-2}

A \emph{stable tree}  is a finite tree $T$ with edges $E(T)$ and vertices $V(T)$ such that the valency of every vertex is larger than or equal to $3$; we write $F(T)$ for the set of flags $(v,e)$ (i.e. a pair consisting of an edge and an adjacent vertex). We allow edges which are adjacent to a single vertex, and which are called \emph{external} and form a set denoted $E^{\ext}(T) $. Given a vertex $v$ of $T$, we write  $T_v$ for the tree consisting of those edges adjacent to $v$; this is a tree with a single vertex, and with edge set that we denote  $E_v(T)$.

A stable \emph{ribbon} tree is a tree as above, equipped with a cyclic ordering of the edges $E_v(T)$ for each vertex.  We shall more precisely consider \emph{rooted ribbon trees,} i.e. ribbon trees equipped with a distinguished external edge (the root) which we call outgoing, forming the singleton $E^\out(T) \subset E(T)$.  The remaining external edges are called incoming and are denoted $E^\inp(T)$. The path from any vertex $v$ to the root determines a unique outgoing edge $e^\out_v$ of $T_v$;  we call the remaining edges adjacent to $v$ incoming, and we denote the sets of incoming and outgoing edges at a vertex $v$ by $E_v^\inp(T)$ and $E_v^\out(T)$. The cyclic ordering and the choice of outgoing edge give rise to a unique ordering of $E_v(T)$ compatible with the cyclic structure with the property that $e^\out_v$ is the last element.

The cyclic ordering also determines an isotopy class of proper embeddings of the corresponding topological tree in the plane, and the choice of outgoing edge determines an ordering of  the components of $\bR^2 \setminus T$ given counterclockwise starting with the component which is to the left of the outgoing edge when directed outwards. Given a set $A$, each finite sequence $\Lab$ of elements of $A$ consisting of $E^{\ext}(T)$ elements induces a map
\begin{equation} \label{eq:label_by_tree_embedding}
\pi_{0}(\bR^2 \setminus T) \to A.
\end{equation}
We assign to each edge $e$ of $T$ the ordered pair $\Lab_e$ consisting of the labels of the two regions adjacent to $e$, with the convention that, if $e$ is oriented towards the outgoing edge, the label which is to the left appears first.  We denote by $\Lab_v$ the induced label on the tree $T_v$ associated to each vertex $v$.

Let us fix a Riemannian metric on $Q$ satisfying Condition \eqref{eq:distortion_bounded}.  In this section, we shall consider  as labels (finite) sequences $\Lab$ of elements $Q \amalg Q^\phi \amalg \cA $ (where $Q^\phi = Q \times \{\phi\}$) such that \begin{equation}
  \label{eq:conditions_labels_sigma_F}
  \parbox{35em}{(i) there are at most $\dim Q+1$ distinct elements of $Q$ (respectively $Q^\phi$), and at most $2$ elements of $\cA$ appearing in $\Lab$ (ii) the elements of $Q$ (respectively $Q^\phi$) are consecutive in $\Lab$ and are contained in a ball of radius $1$, and (iii) any elements of $Q^\phi$ appear last. }
\end{equation}
In other words,  the sequence is of the form 
\begin{equation}
( L_{1}, \ldots, L_j, q_{-r}, \ldots,  q_{-1}, L'_{1}, \ldots, L'_k,  q_{1}^\phi, \ldots, q_\ell^\phi)   
\end{equation}
where $0 \leq j+k \leq 2$ and $0 \leq \ell, r \leq \dim Q +1 $.

Given such a label $\Lab$, let $\Rbar_{\Lab}$ denote the moduli space of stable holomorphic discs with (i) a boundary marked point for each cyclically successive pair of elements of $Q$ or $Q^\phi$ in $\Lab$ and (ii) a boundary puncture for each other cyclically successive pair of elements of $\Lab$. We require that the cyclic ordering induced by the ordering of $\Lab$ correspond to the counterclockwise ordering around the boundary of the disc. There are thus $|\Lab|$ punctures and marked points in total, and the boundary components of the complement of the marked points are labelled by the elements of $\Lab$. In addition, there is a distinguished puncture corresponding to the first and last element of $\Lab$, which we call outgoing. For each tree $T$ labelled by $\Lab$, we then define
\begin{equation}
  \Rbar^T_\Lab \coloneqq \prod_{v \in V(T)} \Rbar_{\Lab_v}.  
\end{equation}
We distinguish the subset
\begin{equation} \label{eq:degenerate_vertex}
 V^{\deg}(T) \subseteq V(T) 
\end{equation}
consisting of vertices with \emph{degenerate labels,} i.e. such that the label  is contained in $Q$ or $Q^\phi$.

Each map $T \to T'$ which collapses internal edges induces an inclusion
\begin{equation}
    \Rbar^T_\Lab  \to   \Rbar^{T'}_\Lab   
\end{equation}
and the tree with a unique vertex corresponds to the moduli space $\Rbar_\Lab$. Moreover, if $\Lab \to \Lab'$ is a map of ordered sets which collapses successive elements that are equal, we obtain a (forgetful) map of moduli spaces
\begin{equation}
\Rbar_{\Lab}  \to \Rbar_{\Lab'}.
\end{equation}

Let $\Sbar_{\Lab}$ denote the universal curve over $\Rbar_{\Lab}$.  For $L \in \cA$, $q \in Q$, and $q^\phi \in Q^\phi$, we denote by  $\partial_L \Sbar_\Lab$, $\partial_q \Sbar_\Lab$,  and $\partial_q^\phi \Sbar_\Lab$  the corresponding boundary segment of the complement of the marked points. We write $\partial_Q \Sbar_\Lab$ and $\partial^\phi_Q \Sbar_\Lab$ for the union of boundary components labelled by elements of $Q$ or $Q^\phi$.

For each tree $T$, we write $ \Sbar^T_{\Lab} $ for the fibre over the stratum $\Rbar^T_\Lab$. This space decomposes as a disjoint union of components labelled by the vertices of $T$:
\begin{align} \label{eq:decomposition_fibre_universal_curve}
\Sbar^T_{\Lab} &\coloneqq \coprod_{v \in V(T)}  \Sbar^v_{\Lab} \\ 
\Sbar^v_{\Lab} & \coloneqq \Sbar_{\Lab_v} \times_{\Rbar_{\Lab_v}} \Rbar^T_\Lab.
\end{align}

We shall need to consider certain moduli spaces of discs with conformal constraints, which depend on a choice of basepoint $q_\ast$ on $Q$. Consider ordered subsets $\Lab$ of $ Q \amalg Q^\phi \amalg \cA $  of the form
\begin{align}  \label{eq:sequence_map_left_modules}
& (L,  q_{1}, \ldots , q_\ell ) \\
& ( q_{-r}, \ldots,  q_{-1}  , L) \\ \label{eq:sequence_map_from_diagonal}
 &  (q_{-r}, \ldots,  q_{-1},q_{1}^\phi, \ldots, q_\ell^\phi)   \\  \label{eq:sequence_tensor_product_comparison}
 &  (L,  q_{-r}, \ldots,  q_{-1},q_{1}^\phi, \ldots, q_\ell^\phi)   \\  \label{eq:sequence_Hom_comparison-Q}
&  (q_{-r}, \ldots,  q_{-1}, L, q_{1}^\phi, \ldots, q_\ell^\phi),
\end{align}
where all elements of $Q$ and $Q^\phi$ are within distance $1$ of $q_\ast$.

In the first three cases above, we define $\Rbar_{\uLab}$ to be the moduli space of stable degree $1$ maps to $D^2$ whose domain is a pre-stable disc with boundary marked points labelled by successive pairs of elements of the sequence $\Lab$ both of which lie in $Q$ of $Q^\phi$, and a puncture for the remaining two cyclically successive elements of $\Lab$ (again, we require that the counterclockwise ordering on the boundary of the disc corresponds to the order of the elements of $\Lab$). We only consider maps so that the outgoing end maps to $-1$, and the incoming puncture labelled by $(L,q_1)$, $(q_{-1},L)$, or $ (q_{-1}, q_{1}^\phi) $ maps $+1$.

In the last two cases, we define $\Rbar_{\uLab}$ to be the inverse image of the $1$-dimensional submanifold $\Rbar_{3, \underline{2}} \subset \Rbar_{3,2}$ fixed in Section \ref{sec:comp-with-bimod}, in the moduli space of stable maps from a pre-stable disc with boundary marked points or punctures given as before by cyclically successive elements of $\Lab$ to $(D^2)^2$, so that the image of the outgoing end under both factors is $-1$, and the ends labelled by $(L,q_{-r})$ and $(q_{-1},q_1^{\phi})$ in Equation \eqref{eq:sequence_tensor_product_comparison} (respectively by $(q_{-1},L)$ and $(q_{1}^\phi, q_{2}^\phi)$ in Equation \eqref{eq:sequence_Hom_comparison-Q}) map to $+1$ under the two factors.  The map to $\Rbar_{3,2}$ is thus obtained by forgetting all marked points labelled by pairs consisting of two elements of $Q$ or two elements of $Q^\phi$.

\begin{rem}
There is a minor difference between the moduli spaces $\Rbar_{\uLab}$ for $\Lab$ given by Equations (\ref{eq:sequence_tensor_product_comparison}) and (\ref{eq:sequence_Hom_comparison-Q}). In the first case, we consider $\Rbar_{3,2}$ as parametrising discs with one boundary marked point, and two punctures, while in the second, we have three boundary punctures. The difference will be apparent when we discuss the boundary conditions that will be imposed on these moduli spaces.
\end{rem}

\begin{figure}[h]
  \centering
\begin{tikzpicture}
\newcommand*{\bigradius}{2.5}
\newcommand*{\tinyradius}{.05}
\newcommand*{\radius}{1}
\newcommand*{\smallradius}{.5}

\draw  [fill=black] (0,0) circle (2*\tinyradius);

\draw[->] (180:\bigradius-\tinyradius) -- (180:\bigradius+3*\tinyradius);
\draw (0,0)  ([shift=(-178:\bigradius)]0,0) arc (-178:-62:\bigradius);
\coordinate [label=left:$q_{-3}$] (q-2) at  (-165:1*\bigradius);
\draw [fill=black] (-150:\bigradius) circle (\tinyradius);
\coordinate [label=below:$q_{-2}$] (q-1) at  (-120:1*\bigradius);
\draw [fill=black] (-90:\bigradius) circle (\tinyradius);
\coordinate [label=below:$q_{-1}$] (q-0) at  (-75:1*\bigradius);

\draw[->] (-60:\bigradius+\tinyradius) -- (-60:\bigradius-3*\tinyradius);
\draw ([shift=(-58:\bigradius)]0,0) arc (-58:58:\bigradius);
\coordinate [label=right:$L$] (L) at  (0:\bigradius);

\draw[->] (60:\bigradius+\tinyradius) -- (60:\bigradius-3*\tinyradius);
\draw ([shift=(62:\bigradius)]0,0) arc (62:178:\bigradius);
\coordinate [label=above:$q_{1}^\phi$] (q0) at  (70:1*\bigradius);
\draw [fill=black] (80:\bigradius) circle (\tinyradius);
\coordinate [label=above:$q_{2}^\phi$] (q1) at  (110:1*\bigradius);
\draw [fill=black] (140:\bigradius) circle (\tinyradius);
\coordinate [label=above:$q_{3}^\phi$] (q2) at  (160:1*\bigradius);

\begin{scope}[shift={(0,-2.5*\bigradius)}]
\draw [->] (0,0)--(-180:\bigradius);
\coordinate [label=left:$e_{\out}$] (eout) at  (-180:1*\bigradius);
\coordinate [label=left:$q_{-3}$] (q-2) at  (-165:.75*\bigradius);
\draw [>-] (-150:\bigradius)--(0,0);
\coordinate [label=below:$q_{-2}$] (q-1) at  (-120:.75*\bigradius);
\draw [>-] (-90:\bigradius)--(0,0);
\coordinate [label=below:$q_{-1}$] (q-0) at  (-75:.75*\bigradius);
\draw [>-] (-60:\bigradius)--(0,0);
\coordinate [label=right:$L$] (L) at  (0:.5*\bigradius);
\draw [>-] (60:\bigradius)--(0,0);
\coordinate [label=above:$q_{1}^\phi$] (q0) at  (70:.75*\bigradius);
\draw [>-] (80:\bigradius)--(0,0);
\coordinate [label=above:$q_{2}^\phi$] (q1) at  (110:.75*\bigradius);
\draw [>-] (140:\bigradius)--(0,0);
\coordinate [label=above:$q_{3}^\phi$] (q2) at  (160:.75*\bigradius);

\draw  [fill=black]  (-2*\tinyradius,-2*\tinyradius) -- (-2*\tinyradius,2*\tinyradius) -- (2*\tinyradius,2*\tinyradius) -- (2*\tinyradius,-2*\tinyradius) -- cycle; 
\end{scope}

\begin{scope}[shift={(6,0)}]

\draw[->] (-180:\radius-\tinyradius) -- (-180:\radius+3*\tinyradius);
  \draw (0,0)  ([shift=(-178:\radius)]0,0) arc (-178:-62:\radius);
\coordinate [label=left:$q_{-3}$] (q-2) at  (-165:1*\radius);
\draw [fill=black] (-110:\radius) circle (\tinyradius);

\draw[->] (-60:\radius+\tinyradius) -- (-60:\radius-3*\tinyradius);
\draw ([shift=(-58:\radius)]0,0) arc (-58:58:\radius);
\coordinate [label=right:$L$] (L) at  (0:\radius);

\draw[->] (60:\radius+\tinyradius) -- (60:\radius-3*\tinyradius);
\draw ([shift=(62:\radius)]0,0) arc (62:178:\radius);
\coordinate [label=above:$q_{3}^\phi$] (q2) at  (120:1*\radius);

\begin{scope}[shift={(-110:1.5*\radius)}, rotate=-110]
\draw [fill=black] (-180:\smallradius) circle (\tinyradius);
\draw (0,0)  ([shift=(-178:\smallradius)]0,0) arc (-178:-62:\smallradius);
\draw [fill=black] (-60:\smallradius) circle (\tinyradius);
\draw ([shift=(-58:\smallradius)]0,0) arc (-58:58:\smallradius);
\coordinate [label=below:$q_{-2}$] (q-1) at  (0:1*\smallradius);
\draw [fill=black] (60:\smallradius) circle (\tinyradius);
\draw ([shift=(62:\smallradius)]0,0) arc (62:178:\smallradius);  
\end{scope}

\begin{scope}[shift={(-60:2*\radius)}, rotate=-60]
\draw [fill=black] (0,0) circle (\tinyradius);
\draw (0,0)  ([shift=(-178:\radius)]0,0) arc (-178:-2:\radius);
\coordinate [label=below:$q_{-1}$] (q-0) at  (-90:1.1*\radius);

\draw[->] (0:\radius+\tinyradius) -- (0:\radius-3*\tinyradius);
\draw ([shift=(2:\radius)]0,0) arc (2:178:\radius);
\end{scope}

\begin{scope}[shift={(60:2*\radius)}, rotate=60]
\draw  [fill=black] (0,0) circle (\tinyradius);
\draw (0,0)  ([shift=(-178:\radius)]0,0) arc (-178:-2:\radius);

\draw[->] (0:\radius+\tinyradius) -- (0:\radius-3*\tinyradius);
\draw ([shift=(2:\radius)]0,0) arc (2:178:\radius);  
\draw [fill=black] (90:\radius) circle (\tinyradius);
\coordinate [label=above:$q^\phi_{2}$] (q-1) at  (60:1*\radius);
\begin{scope}[shift={(0:2*\radius)}, rotate=0]
\draw (0,0)  ([shift=(-178:\radius)]0,0) arc (-178:-2:\radius);

\draw[->] (0:\radius+\tinyradius) -- (0:\radius-3*\tinyradius);
\draw ([shift=(2:\radius)]0,0) arc (2:178:\radius);  
\coordinate [label=left:$q^\phi_{2}$] (q-1) at  (120:1*\radius);
\draw [fill=black] (90:\radius) circle (\tinyradius);
\coordinate [label=above:$q^\phi_{1}$] (q-1) at  (60:1*\radius);
\end{scope}
\end{scope}

\begin{scope}[shift={(0,-2.5*\bigradius)}]
\draw  [fill=black] (0,0) circle (\tinyradius);
\draw [->] (0,0)--(-180:\bigradius);
\coordinate [label=left:$e_{\out}$] (eout) at  (-180:1*\bigradius);
\coordinate [label=left:$q_{-3}$] (q-2) at  (-165:.75*\bigradius);
\draw (-110:\smallradius)--(0,0);
\draw [>-] (-150:\bigradius)--(-110:\smallradius);
\draw [>-] (-90:\bigradius)--(-110:\smallradius);
\draw  [fill=white] (-110:\smallradius) circle (\tinyradius);
\coordinate [label=below:$q_{-2}$] (q-1) at  (-120:.75*\bigradius);
\coordinate [label=below:$q_{-1}$] (q-0) at  (-75:.75*\bigradius);
\draw [>-] (-60:\bigradius)--(0,0);

\begin{scope}[shift={(-60:.5*\bigradius)} ] \draw  [fill=black]  (-\tinyradius,-\tinyradius) -- (-\tinyradius,\tinyradius) -- (\tinyradius,\tinyradius) -- (\tinyradius,-\tinyradius) -- cycle;
\end{scope}

\coordinate [label=right:$L$] (L) at  (0:.5*\bigradius);
\draw [>-] (60:\bigradius)--(0,0);
\coordinate [label=above:$q_{1}^\phi$] (q0) at  (65:.8*\bigradius);
\draw [>-] (80:\bigradius)--(60:2/3*\bigradius);

\begin{scope}[shift={(60:2/3*\bigradius)} ] \draw  [fill=white]  (-\tinyradius,-\tinyradius) -- (-\tinyradius,\tinyradius) -- (\tinyradius,\tinyradius) -- (\tinyradius,-\tinyradius) -- cycle;
\end{scope}
\coordinate [label=above:$q_{2}^\phi$] (q1) at  (110:.75*\bigradius);
\draw [>-] (140:\bigradius)--(60:1/3*\bigradius);

\begin{scope}[shift={(60:1/3*\bigradius)} ] \draw  [fill=black]  (-\tinyradius,-\tinyradius) -- (-\tinyradius,\tinyradius) -- (\tinyradius,\tinyradius) -- (\tinyradius,-\tinyradius) -- cycle;
\end{scope}
\coordinate [label=above:$q_{3}^\phi$] (q2) at  (160:.75*\bigradius);
\end{scope}
\end{scope}

\end{tikzpicture}
  \caption{A representation of two elements of the same moduli space $\Rbar(\uLab)$, and the trees labeling the strata they lie on. The solid square vertices are in $M(T)$, the open square vertex in $V^\ast(T)$, and the open circle in $V^{\deg}(T)$.}
\label{fig:moduli_space_constraints_figure}
\end{figure} 
The boundary strata of $\Rbar_{\uLab}$ are also labelled by trees. To describe them, we begin by associating to each stratum of $\Rbar_{2,\underline{1}}$ or $\Rbar_{3,\underline{2}}$ a  rooted ribbon tree $T$ equipped with a distinguished subset of vertices denoted $M(T)$ (\emph{marked vertices}) consisting of those vertices which correspond to disc components carrying a non-trivial map to a disc. For $\Rbar_{2,\underline{1}}$, we thus obtain a tree with a unique bivalent node, and no other vertices. In the same way, each boundary stratum of $\Rbar_{\uLab}$ corresponds to a ribbon tree with a distinguished set of vertices $M(T) \subset V(T)$ which are allowed to be bivalent, and such that the tree obtained by forgetting all incoming edges with labels pairs of elements in $Q$ or $Q^\phi$ corresponds to a tree labelling a boundary stratum of $\Rbar_{2,\underline{1}}$ or $\Rbar_{3,\underline{2}}$. 

We distinguish the set of vertices $V^{\ast}(T) \subset V(T)$ with the property that they separate an incoming Floer edge from all nodes. Writing $\Lab^\ast_v$ for the sequence obtained from $\Lab_v$ by replacing all points in $Q$ or $Q^\phi$ by $q_\ast$, we have natural product decomposition
\begin{equation} \label{eq:decomposition_stratum_constrained_moduli_space}
\Rbar^T_{\uLab}  \coloneqq  \prod_{v \in M(T)}  \Rbar_{\uLab_v} \times \prod_{v \in V^\ast(T)}  \Rbar_{\Lab^\ast_v} \times  \prod_{v \notin  M(T) \cup V^\ast(T)}  \Rbar_{\Lab_v}.  
\end{equation}
\begin{rem}
The distinction between $\Rbar_{\Lab^\ast_v}$ and $\Rbar_{\Lab_v}$ is entirely formal, as the two spaces are naturally homeomorphic. As indicated by the labelling of Figure \ref{fig:moduli_space_constraints_figure}, we shall use these moduli spaces to extend the construction of Section \ref{sec:computing-bimodule-nearby}, which should indicate to the reader why we introduce this notation.  
\end{rem}
As in Equation \eqref{eq:degenerate_vertex}, we also define the set $V^{\deg}(T) \subset V(T) \setminus M(T)$ to consist of those vertices with label $\Lab_v$ contained in $Q$ or $Q^\phi$.

By construction, elements of $\Rbar_{\uLab}$ correspond to curves with at most two incoming ends. We set the labels on these incoming ends to be
\begin{equation}
  \label{eq:label_incoming_end_constrained}
  \parbox{33em}{ $\Lab_e = (L, q_\ast)$ or $\Lab_e = (q_\ast, L)$,}
\end{equation}
while for all other ends we use the ribbon embedding from Equation (\ref{eq:label_by_tree_embedding}) to set the label.

Let $\Sbar_{\uLab}$ denote the universal curves over $ \Rbar_{\uLab}$, obtained by restriction of the universal curve over $\Rbar_\Lab$. As before, given a tree $T$, the fibre $ \Sbar^T_{\uLab}$ over $\Rbar^T_{\uLab} $ can be written as a union of components labelled by vertices of $T$, and which we denote  $ \Sbar^v_{\uLab}$. We have
\begin{equation}
  \Sbar^v_{\uLab}  \coloneqq  \begin{cases}  \Sbar_{\uLab_v} \times_{\Rbar_{\uLab_v}}     \Rbar^T_{\uLab} & v \in M(T) \\ 
\Sbar_{\Lab^\ast_v} \times_{\Rbar_{\Lab^\ast_v}}     \Rbar^T_{\uLab} & v \in V^\ast(T)\\
\Sbar_{\Lab_v} \times_{\Rbar_{\Lab_v}}     \Rbar^T_{\uLab} & \textrm{otherwise.}
  \end{cases}
\end{equation}

\subsubsection{Strip-like ends, gluing charts, and thin-thick decompositions}
\label{sec:gluing-thin-parts}

Recall that a positive (respectively negative) strip-like end on a Riemann surface $S$  is a holomorphic map
\begin{equation}
[0,\infty) \times [0,1] \to S \textrm{ or }  (-\infty,0] \times [0,1] \to S
\end{equation}
which is a biholomorphism in a neighbourhood of a puncture. We shall equip the outgoing puncture of a curve in $\Rbar_\Lab$ with a negative strip-like end, and all other punctures with positive strip-like ends.   As in \cite[Section (9f)]{Seidel2008a},  we make this choice consistently for all curves in $\Rbar_{\Lab} $ and for all labels $\Lab$ . The key consistency condition is that, near the boundary strata of $ \Rbar_{\Lab}$, the strip-like ends are compatible with the gluing maps which are constructed as follows: for each tree $T$ labelling a stratum of $\Rbar_\Lab$, the choices of ends induces an open embedding (the gluing map)
\begin{equation}
  \Rbar^T_\Lab \times (0, \infty]^{E(T)} \to  \Rbar_\Lab.
\end{equation}
This open embedding lifts to the level of universal curves, so that the choice of strip-like ends on curves parametrised by $ \Rbar^T_\Lab$ induces strip-like ends on curves in the image of the gluing map, and we require that the choice of ends be induced by the gluing map whenever the gluing parameters are sufficiently large. 

Fixing the identification of $(0,\infty]$  with $(-1,0]$ which takes $R$ to $- e^{-R}$, we obtain charts
\begin{equation}
  \Rbar^T_\Lab \times (-1,0]^{E(T)} \to  \Rbar_\Lab
\end{equation}
which equip the moduli space with the structure of a smooth manifold with corners (the easiest way to prove that this is a smooth structure is to realise these moduli spaces are submanifolds of Deligne-Mumford spaces of closed genus $0$ Riemann surfaces  with marked points).

The gluing map lifts at the level of universal curves to a map
\begin{equation}
  \Sbar^T_\Lab   \times (-1,0]^{E(T)} \to  \Sbar_\Lab,
\end{equation}
which is surjective over the image of the gluing map at the level of moduli spaces.  The same discussion applies to the moduli spaces $  \Rbar_{\uLab}$: choices of strip-like ends induce gluing charts near the boundary strata, which lift to the universal curve.

The gluing charts equip each curve $S$ in $\Rbar_{\Lab}$ or $\Rbar_{\uLab}$ with a thick-thin decomposition: the thin part will be the union of the image under the gluing maps of (i) the strip-like ends, (ii) the curves corresponding to degenerate vertices, and (iii) curves which become unstable upon omitting repeated elements of $Q$ or $Q^\phi$. The thick part is the complement. We associate to each component $\Theta$ of the thin part the subset $\nu_X \Theta$ of $X$ given by (i) the empty set if all labels appearing on the intersection of the boundary with $\Theta$ are contained in $Q$ or $Q_\phi$, (ii) the set $\nu_X \Lab_e$ if $\Theta$ contains the image of an end $e$ under gluing.  Here, we set $\nu_X \Lab$ to be (i) $X$ if $\Lab$ is degenerate and (ii) the set $\nu_X L$ if $\Lab$ consists of an element of $Q \cup Q^\phi$ and $L \in \cA$, and (iii) the set $\nu_X(X_{q_-} \cap \phi X_{q} )$ if $\Lab = (q_-, q^\phi)$. 

By construction, any two ends with images in $\Theta$ have the same label, so this definition is consistent (it is important here to distinguish ends from marked points).

Finally, for each $S$ in $\Sbar_\Lab$, we fix neighbourhoods $\nu \partial_Q S $ and  $\nu \partial^\phi_Q S$ in $S$ of the corresponding boundary components of $S$, which are compatible with gluing, only intersect each other in the thin part, and whose union forms an open set in the universal curve. If $S$ contains a component with degenerate labels, we assume that such a component is contained in $\nu \partial_Q S$ or $\nu \partial_Q^\phi S$ as long as there is a part of the boundary with the corresponding label. We write $\nu \partial_Q \Sbar_\Lab$ and $\nu \partial_Q^\phi \Sbar_\Lab$ for the union of such neighbourhoods over all points $S \in \Rbar_\Lab$.

We also choose such neighbourhoods of the boundary components of $S \in \Sbar_{\uLab}$; the only difference is that the neighbourhood  $\nu \partial_Q S $ and  $\nu \partial^\phi_Q S$  should intersect in the union of the thin part with a neighbourhood of any marked point with label in $Q \amalg  Q^\phi$.

\subsection{Unperturbed families of equations}
\label{sec:unpert-famil-equat}
We continue to work with a fixed metric on the base $Q$ satisfying Condition \eqref{eq:distortion_bounded}. The key property which shall use is that, for such a metric, the intersection of any two geodesic balls of radius less than $2$ is either empty or contractible since it is geodesically convex.

\subsubsection{Paths of fibres}
\label{sec:paths-fibres}

For each non-degenerate sequence $\Lab$ of the form given in Equations \eqref{eq:sequence_map_left_modules}--\eqref{eq:sequence_Hom_comparison-Q}, we choose maps
\begin{align} \label{eq:paths_of_fibres_functor}
q_{\Lab} \co \partial_Q \Sbar_{\Lab} & \to   Q \\
q^\phi_{\Lab} \co \partial^\phi_{Q} \Sbar_{\Lab}  & \to  Q,
\end{align}
whose value along each segment labelled by $q \in Q$ lies in the ball of radius $1$ about this point, and varying smoothly in the choice of labels. We require these maps to satisfy the following additional properties:
\begin{enumerate} 
\item \label{item:value_along_ends} (Values along the ends) The restrictions of $q_\Lab$ and $q^\phi_\Lab$ to the intersection of each end  with $\partial_q \Sbar_\Lab$ and $\partial^\phi_q \Sbar_\Lab$ are constant with value $q$.
\item (Compatibility with gluing) For each tree $T$ labelling a boundary stratum of $ \Rbar_{\Lab}$, the restrictions of $q_\Lab$ and $ q^\phi_{\Lab}$ to a neighbourhood of $\Rbar^T_\Lab$ are obtained by gluing in the sense that the following diagram commutes in a neighbourhood of the origin:
  \begin{equation}
   \begin{tikzcd}
\left(\partial_Q \Sbar^T_{\Lab}  \amalg \partial^\phi_{Q}  \Sbar^T_{\Lab}   \right) \times (-1,0]^{E(T)}  \arrow{dr}{q_\Lab \amalg q^\phi_{\Lab} } \arrow{r}{} & \partial_Q  \Sbar_{\Lab}\amalg \partial^\phi_{Q}  \Sbar_{\Lab} \arrow{d}{q_\Lab \amalg q^\phi_{\Lab}}   \\
 & Q.
\end{tikzcd}  
\end{equation}
\item (Compatibility with boundary) For each tree $T$ labelling a boundary stratum of $ \Rbar_{\Lab}$ and vertex $v \in V^{\deg}(T)$, the restriction of $q_\Lab$ and $q^\phi_{\Lab}$ to $\Sbar^v_{\Lab}$ is constant (note that such components are disjoint from the ends, so this condition is independent of Condition \ref{item:value_along_ends} above). If $v$ is not a degenerate vertex, these data are compatible with $(q_{\Lab_v},q^\phi_{\Lab_v})$ in the sense that we have a commutative diagram
  \begin{equation}
   \begin{tikzcd}
\partial_Q \Sbar^v_{\Lab} \amalg \partial^\phi_{Q}  \Sbar^v_{\Lab}  \arrow{dr}{q_\Lab \amalg q^\phi_{\Lab}} \arrow{r}{} & \partial_{Q} \Sbar_{\Lab_v} \amalg  \partial^\phi_{Q} \Sbar_{\Lab_v} \arrow{d}{q_{\Lab_v} \amalg q^\phi_{\Lab_v}}   \\
  & Q . 
\end{tikzcd}  
\end{equation}
\item (Forgetful maps) If $\Lab$ is obtained from a sequence $\Lab'$ by repeating elements of $Q$ or $Q^\phi$,  the choices of paths are compatible with the induced forgetful maps in the sense that the following diagrams commute:
  \begin{equation}
    \begin{tikzcd}
\partial_Q \Sbar_{\Lab}  \arrow{dr}{} \arrow{r} & \partial_Q \Sbar_{\Lab'}  \arrow{d}{} & \partial^\phi_{Q} \Sbar_{\Lab} \arrow{dr}{} \arrow{r} & \partial^\phi_{Q} \Sbar_{\Lab'}  \arrow{d}{} \\
 & Q &  & Q.
\end{tikzcd}
  \end{equation}
\item (Forgetting $\phi$) If $\Lab$ is an ordered non-degenerate subset of $Q^\phi \amalg \cA$, and $\pi \Lab$ the corresponding ordered subset of  $Q \amalg \cA$, we have a commutative diagram
  \begin{equation}
    \begin{tikzcd}
\partial^\phi_{Q} \Sbar_{\Lab} \arrow{r}{} \arrow{d}{} & \partial_Q \Sbar_{\pi \Lab} \arrow{dl}{}  \\
Q. &
\end{tikzcd}
  \end{equation}
\end{enumerate}
The construction of such paths proceeds by induction on the number of elements of $\Lab$, and the bound on the lengths of the paths ensures the contractibility of the corresponding space of choices, ensuring that there is no obstruction to extending the choices from the boundary strata to the interior of the moduli space.

Next, we consider the moduli spaces associated to a choice of basepoint $q_\ast \in Q$. For each sequence $\Lab$ of the form given in Equations  \eqref{eq:sequence_map_left_modules}--\eqref{eq:sequence_Hom_comparison-Q}, we choose maps
\begin{align}
q_{\uLab} \co \partial_Q \Sbar_{\uLab} & \to   Q \\
q^\phi_{\uLab} \co \partial^\phi_{Q} \Sbar_{\uLab}  & \to  Q
\end{align}
varying smoothly in the choice of labels (and in the choice of basepoint $q_\ast$, which we leave implicit in the notation), and whose restriction to each boundary segment labelled by $q \in Q$ lies in the ball of radius $1$ about this point. The key condition is that:
\begin{enumerate}
\item The maps have value $q_\ast$ near any marked point or node with label given by one element in $Q$ and one in $Q^\phi$.
\end{enumerate}
We require these maps to satisfy the following additional properties, which are entirely analogous to those listed above:
\begin{enumerate} \setcounter{enumi}{1}
\item (Values along the ends) The restrictions of $q_{\uLab}$  and $q^\phi_{\uLab}$ to the intersection of each end $e$ with $\partial_q \Sbar_{\uLab}$  and $\partial^\phi_q \Sbar_{\uLab}$ are constant with value $q$ if $e$ is an outgoing end, and $q_{\ast}$ if $e$ an incoming end. \label{item:condition_along_ends}
\item (Compatibility with gluing) For each tree $T$ labelling a boundary stratum of $ \Rbar_{\uLab}$, the restrictions of $q_{\uLab}$  and $ q^\phi_{\uLab}$ 
  to a neighbourhood of $\Sbar^T_{\uLab}$ are obtained by gluing.
\item (Compatibility with boundary) For each tree $T$ labelling a boundary stratum of $\Rbar_{\Lab}$ and vertex $v \in V(T)$, the restrictions of the maps $q_{\uLab}$  and $q^\phi_{\uLab}$ to $\Sbar^{v}_{\uLab}$  agree with (i) constant functions if $v$ is degenerate, (ii)  the maps $q_{\underline{\Lab_v}}$  and  $q^\phi_{\underline{\Lab_v}}$
  if $v$ is a node, (iii) the constant function $q_\ast$ if $v \in V^\ast(T)$, and (iv) the functions  $q_{\Lab_v}$  and  $q^\phi_{\Lab_v}$ in the remaining cases.
\item (Forgetful maps)  If $\Lab \to \Lab'$ is a map of labels, then $q_{\uLab}$ and $q^\phi_{\uLab'}$ are obtained from $q_{\uLab'}$ and  $q^\phi_{\uLab'}$ by composition with the forgetful map.
\end{enumerate}
Note that Condition \eqref{item:condition_along_ends} above is compatible with Equation (\ref{eq:label_incoming_end_constrained}), which sets the labels for incoming ends to consist of an element of $\cA$ and the basepoint $q_\ast$.

\subsubsection{Deformation of the diagonal}
\label{sec:deformation-diagonal}

Recall that we chose a Hamiltonian diffeomorphism $\phi$ in Section \ref{sec:pertubed-diagonal}. We now pick a map
\begin{equation}
  \Phi_\Lab \co   \partial^\phi_{Q} \Sbar_{\Lab} \to \Ham(X)
\end{equation}
such that the following properties hold:
\begin{enumerate}
\item (Values along the ends) the restriction to an end $e$  agrees with $\phi$ if $\Lab_e = (\sigma_-, \sigma^\phi)$ and with the identity otherwise.
\item (Compatibility with gluing) For each tree $T$ labelling a boundary stratum of $ \Rbar_{\Lab}$ the restriction of $\Phi_\Lab$ to a neighbourhood of $\Rbar^T_\Lab$ is obtained by gluing in the sense that the following diagram commutes near the origin:
  \begin{equation}
   \begin{tikzcd}
 \partial^\phi_{Q}  \Sbar^T_{\Lab}  \times (-1,0]^{E(T)}  \arrow{dr}{\Phi_\Lab} \arrow{r}{} & \partial^\phi_{Q}  \Sbar_{\Lab}  \arrow{d}{\Phi_\Lab}   \\
 &  \Ham(X).
\end{tikzcd}  
\end{equation}
\item (Compatibility with boundary)  For each tree $T$ labelling a boundary stratum of $ \Rbar_{\Lab}$ and vertex $v \in V(T)$  the map $\Phi_\Lab$  is (i) constant if $v$ is degenerate and (ii) otherwise agrees with $\Phi_{\Lab_v}$ in the sense that we have a commutative diagram
  \begin{equation}
   \begin{tikzcd}
\partial^\phi_Q \Sbar^v_{\Lab}  \arrow{dr}{\Phi_\Lab} \arrow{r}{}  & \partial^\phi_{Q} \Sbar_{\Lab_v} \arrow{d}{\Phi_{\Lab_v}}  \\
 & \Ham(X).
\end{tikzcd}  
\end{equation}
\item (Forgetful maps) For each forgetful map $\Lab \to \Lab'$, the following diagrams commute:
  \begin{equation}
    \begin{tikzcd}
 \partial^\phi_{Q} \Sbar_{\Lab} \arrow{dr}{} \arrow{r}{} & \partial^\phi_{Q} \Sbar_{\Lab'} \arrow{d}{} \\
&  \Ham(X).
\end{tikzcd}
  \end{equation}
\item (Forgetting $\phi$) If $\Lab$ is an ordered subset of $Q^\phi \amalg \cA$, $\Phi_\Lab$  is constant with value the identity map.
\end{enumerate}

Again, the construction proceeds by induction on the number of elements of $\Lab$; while the space of Hamiltonian diffeomorphisms may not be contractible, it is easy to lift all choices to the space of paths based at the identity in $\Ham(X)$. In this way, all obstructions to extending choices from boundary strata to the interior vanish.

Similarly we consider a family
\begin{equation}
  \Phi_{\uLab} \co   \partial^\phi_{Q} \Sbar_{\uLab} \to \Ham(X),
\end{equation}
such that the following properties hold:
\begin{enumerate}
\item  The restriction to a neighbourhood of a node or marked point with label in $Q \times Q^\phi$ is constant with value the identity. 
\item (Values along the ends) the restriction to each end $e$ agrees with $\phi$ if $\Lab_e $ contains an element of $Q$,  and with the identity otherwise.
\item (Compatibility with gluing) For each tree $T$ labelling a boundary stratum of $ \Rbar_{\uLab}$ the restriction of $\Phi_\Lab$ to a neighbourhood of $\Sbar^T_{\uLab}$ is obtained by gluing.
\item (Compatibility with boundary) For each tree $T$ labelling a boundary stratum of $ \Rbar_{\uLab}$ and vertex $v \in V(T)$,  the restriction of  $\Phi_\Lab$  to $ \Sbar^v_{\uLab}$ agrees with (i) a constant function if $v$ is degenerate, (ii) the map  $\Phi_{\uLab_v}$ if $v$ is a node, (iii) the  identity if $v \in V^\ast(T)$, and (iv) the map  $\Phi_{\Lab_v}$ otherwise.
\item (Forgetful maps) Whenever $\Lab$ is obtained from $\Lab'$ by repeating elements of $Q$ or $Q^\phi$, $\Phi_{\uLab}$ is obtained from $\Phi_{\uLab'}$ by composition with the forgetful map.
\end{enumerate}

\subsubsection{Choices of almost complex structures}
\label{sec:choic-almost-compl}
Let $\scrJ$ denote the space of $\omega$-tame almost complex structures on $X$. For simplicity, fix an almost complex structure $J_\phi$ on $X$ which we will use to define the Floer theory of the perturbed diagonal (i.e. of fibres with their image under $\phi$), and recall that we have chosen an almost complex structure $J_L$ associated to each $L \in \cA$, and subsets $\nu_X \Lab$ of $X$ for each pair $\Lab$ of labels.  For each sequence $\Lab$ as in the preceding sections, we consider maps
\begin{align} \label{eq:restrict_J_nbd_Q}
J_\Lab^\nu \co  \nu \partial_Q \Sbar_\Lab   \cup \nu \partial_Q^\phi \Sbar_\Lab \to & \scrJ  \\
J_{\uLab}^\nu \co   \nu \partial_Q \Sbar_{\uLab}   \cup \nu \partial_Q^\phi \Sbar_{\uLab}  \to & \scrJ  \label{eq:restrict_J_nbd_Q_phi_under}
\end{align}
satisfying the following conditions:
\begin{enumerate}
\item (Values along the ends) Away from a fixed compact subset of the interior of  $\nu_X \Lab_e$, the restriction to each end $e$ is constant, with value $J_L$ if $L \in \Lab_e$, and $J_\phi$ otherwise. 
\item (Compatibility with gluing) For each tree $T$ labelling a boundary stratum of $ \Rbar_{\Lab}$ or $\Rbar_{\uLab}$ the restrictions to a neighbourhood of $\Rbar^T_\Lab$ or $\Rbar^T_{\uLab}$ are  obtained by gluing.
\item (Compatibility with boundary) For each tree $T$ labelling a boundary stratum of $ \Rbar_{\Lab}$ or $ \Rbar_{\uLab}$ and non-degenerate vertex $v \in V(T)$, the restriction of $J_{\Lab}^\nu$ or $J_{\uLab}^\nu$ to  $ \Sbar^v_{\Lab}$ or $ \Sbar^v_{\uLab}$ 
 agrees with the almost complex structure associated to $\Lab_v$. 
\item (Forgetful maps) The choice of almost complex structure is compatible with forgetful maps associated to repeating elements of $Q^\phi$ or $Q$.
\item (Forgetting $\phi$) The choice of almost complex structure is compatible with the projection $Q^\phi \to Q$ if $\Lab$ contains no element of $Q$.
\end{enumerate}

\subsubsection{Reverse isoperimetric constant}
\label{sec:isop-const}

For each curve $S$ in $\Rbar_\Lab$ or $\Rbar_{\uLab}$, the paths $q_\Lab$ and $q^\phi_\Lab$ or $q_{\uLab}$ and $q^\phi_\Lab$ determine a closed manifold with boundary $X_{S}$, homeomorphic to two copies of  the product of a fibre with $[0,1]$ in the first case, and one such copy in the second case, and obtained from the boundary conditions over $\partial_Q S$ and $\partial^\phi_Q S$ by identifying the fibres over each component $\Theta \subset S$ of the thin part.  We define the equivalence relation $\sim$ on $ X_{S}$ to collapse the intersection of the fibre over $\Theta$ with each component of $\nu_X \Theta$ to a point. The assumption that this intersection is contained in a disjoint union of balls allows us (as in Lemma \ref{lem:good-metric-quotient}) to fix a metric on the quotient so that, for each loop in $X_S$, the length of the projection to $X_S/{\sim}$ bounds the norm of the corresponding homology class in $H_{1}(X_q, \bZ)$ for any $q$ on the path $q_\Lab$ or $q^\phi_\Lab$.

Given a curve $S$ in  $\Rbar_\Lab$ or $\Rbar_{\uLab}$,  let $u \co S \to X$ be a map such that $u(z)$ lies (i) in $\nu_X L$ if $z \in \partial_L S$, (ii) in  $X_{q_\Lab(z)}$ or $X_{q_{\uLab}(z)}$ if $z \in \partial_Q S$, and (iii) in $\Phi_\Lab(z) X_{q^\phi_\Lab(z)}$ or $\Phi_{\uLab}(z) X_{q^\phi_{\uLab}(z)}$  if $z \in \partial^\phi_Q S$. Each such curve has an evaluation map
\begin{equation}
\partial u / {\sim} \co \partial_Q S \amalg  \partial^\phi_Q S \to X_S/ {\sim}.
\end{equation}
Let $\ell(\partial u/ {\sim})$ denote the length of this curve.

\begin{lem}
  \label{lem:universal_constant_isoperimetric}
There is a constant $C$ (depending on the collection of Lagrangians and the choices of almost complex structures) such that, for each choice of labels $\Lab$ satisfying Condition \eqref{eq:conditions_labels_sigma_F}, we have
  \begin{equation}
  \ell(\partial u/ {\sim}) \leq C E^{geo}(u) + \textrm{ a constant independent of } u    
  \end{equation}
for each curve $u$ with boundary conditions as above, whose restriction to $\nu \partial_Q S $ and  $\nu \partial^\phi_Q S$ are holomorphic with respect to the complex structures fixed in Section \ref{sec:choic-almost-compl}. Moreover, for each components $\Theta$ of the thin part of $S$, this constant is independent of the restriction of the complex structure and the Lagrangian boundary conditions to  $\Theta \times \nu_X \Theta$. 
\end{lem}
\begin{proof}
This is a consequence of Proposition \ref{prop:main_result_family_reverse_isoperimetric}, which concerns reverse isoperimetric constants for families. Indeed, the existence of such a constant for each curve $S$ follows from Lemma \ref{lem:basic_isoperimetric_moving_L}. The assumption that the Floer data are compatible with gluing and boundary strata implies that we can use Lemma \ref{lem:isoperimetric_constant_indep_gluing} to conclude that the constant can be chosen uniformly for $S$ in $ \Rbar_\Lab$ or  $\Rbar_{\uLab}$  for any finite collection of labels. The fact that $Q$ is compact implies that the constant may be chosen uniformly if the number of labels on the boundary is bounded.  The assumption that the data are compatible with forgetful maps, and the constraint that $\Lab$ contains at most $n+1$ distinct elements of $Q$ or $Q^\phi$ implies that the constant may be chosen uniformly. 
\end{proof}

\subsection{Perturbed equations}
\label{sec:perturbed-equations}

We now return to the context of Section \ref{sec:cohom-constr}, and consider a finite cover $Q$ by integral affine polygons with the property that the intersection of any collection consisting of more than $n+1$ element is empty. Letting $\Sigma$ denote the set of collections of elements of the cover with non-empty intersection, which is partially ordered by inclusion, we obtain a cover  $\{ P_\sigma  \}_{\sigma \in \Sigma}$  of $Q$ which has dimension $n$ (i.e. the maximal length of any totally ordered subset of $\Sigma$ is $n+1$). We assume that this cover is sufficiently fine that we can pick a contractible set $\cN_{\sigma}$ for each element of $\Sigma$, which includes $P_\tau$ whenever $P_\sigma$ intersects $P_\tau$, and such that 
\begin{equation} \label{eq:diameter_open_star_cover}
  \diam \cN_\sigma \leq \min(1,\frac{1}{8C}),  
\end{equation}
where the constant $C$ is the one from Lemma \ref{lem:universal_constant_isoperimetric}. For each $\sigma \in \Sigma$, we choose a basepoint $q_\sigma \in P_\sigma$, which we think of as a map
\begin{equation}
  \Sigma \to Q.
\end{equation}
As before, we write $\Sigma^\phi$ for $\Sigma \times \{\phi\}$, whose elements are equipped with the same choices of basepoints. We assume that these basepoints are chosen so that the Lagrangian fibre $X_{q_\sigma}$ is transverse to each element $L$ of $\cA$, and that the pair $(X_{q_\sigma}, \phi X_{q_\tau})$ is transverse for each pair $(\sigma, \tau) \in \Sigma$.

Let $\Lab$ denote an ordered subset of $\Sigma \amalg \Sigma^\phi \amalg  \cA$ such that 
\begin{equation}
  \parbox{35em}{(i) all elements of $\Sigma$ (respectively $\Sigma^\phi$) are consecutive and are increasing with respect to the partial order on $\Sigma$, (ii) all elements of $\Sigma^\phi$ appear last, and (iii) all elements of $\cA$ are distinct.}
\end{equation}
In other words,  the sequence is of the form 
\begin{equation}
  \label{eq:elements_of_uLab_ordered_set}
( L_{1}, \ldots, L_j, \sigma_{-r}, \ldots,  \sigma_{-1}, L'_{1}, \ldots, L'_k,  \sigma_{0}^\phi, \sigma_{1}^\phi, \ldots, \sigma_\ell^\phi)   
\end{equation}
where $\sigma_i \leq \sigma_{i+1}$. In this paper, we shall only consider the cases $j+k \leq 2$. 

We have a map of labelling sets 
\begin{equation}
  \Sigma \amalg \Sigma^\phi \amalg  \cA \to Q \amalg Q^\phi \amalg \cA,
\end{equation}
so that any curve $S \in \Rbar_\Lab$ is equipped with moving Lagrangians boundary conditions along the boundary components $\partial_\Sigma S$ and $\partial_\Sigma^\phi S$ by the construction of Section \ref{sec:unpert-famil-equat}. By abuse of notation, we write 
\begin{align}
q_{\Lab} \co \partial_\Sigma \Sbar_{\Lab} & \to   Q \\
q^\phi_{\Lab} \co \partial^\phi_\Sigma \Sbar_{\Lab}  & \to  Q
\end{align}
for these paths.

Next, we fix an element $\sigma \in \Sigma$, and  consider an ordered subset $\Lab$ of $ \Sigma_\sigma \amalg \Sigma^\phi_\sigma \amalg  \cA$, such that
\begin{equation}
  \label{eq:condition)_uLab_with_Sigma}
  \parbox{35em}{the image $\pi \Lab$ under the map $\Sigma \to Q$  is one of the sequences in Equations \eqref{eq:sequence_map_left_modules}--\eqref{eq:sequence_Hom_comparison-Q}, and such that the corresponding ordered subsets of $\Sigma$ and $\Sigma^\phi$ respect the partial ordering. }
\end{equation}
We write  $\Rbar_{\uLab}$ for $\Rbar_{\underline{\pi \Lab}}$, and shall set $q_\ast = q_{\basepoint}$ in all future constructions. We have the same product decomposition as in Equation (\ref{eq:decomposition_stratum_constrained_moduli_space}), replacing $\Lab^\ast_v$ by $\Lab^\sigma_v$, i.e. the sequence obtained by replacing all elements of $\Sigma$ and $\Sigma^\phi$ by $\sigma$.  We also impose the analogue of Equation (\ref{eq:label_incoming_end_constrained}), and set the labels of the incoming ends to be either $\Lab_e = (L, \basepoint)$
 or $\Lab_e = (\basepoint, L)$.

\subsubsection{Families of almost complex structures}
\label{sec:famil-almost-compl}
For each non-degenerate pair $\Lab$ in $\Sigma \amalg \Sigma^\phi \amalg \cA$, we fix a map
\begin{equation}
  J(\Lab) \co [0,1] \to \scrJ  s
\end{equation}
which is constant outside of $\nu_X \Lab$ with value $J_\Lab$, and agrees with $J_L$ on an endpoint corresponding to $L \in \cA$. Recall that we set  $\nu_X \Lab = X$ if $\Lab $ is a subset of $\cA$, so that there is no obstruction to choosing such a family for a pair $(L,L')$.

We pick families of almost complex structures
\begin{align}
 J_\Lab \co    \Sbar_\Lab \to &  \scrJ \\
 J_{\uLab} \co    \Sbar_{\uLab} \to &  \scrJ
\end{align}
subject to the following constraints:
\begin{enumerate}
\item (Restriction to the boundary) In the thick part, the restriction to the neighbourhoods of the boundary components labelled by elements of  $\Sigma \amalg \Sigma^\phi$ agree with those fixed in Equations (\ref{eq:restrict_J_nbd_Q})-(\ref{eq:restrict_J_nbd_Q_phi_under}). %
\item (Restriction to the thin parts)  Over each component $\Theta$ of the thin part, $J(\Lab)$ or $ J(\uLab)  $ are constant away from $\nu_X \Theta$. 
\item (Restriction to the ends) Sufficiently far along each end $e$, $J(\Lab)$ and $J(\uLab)$ agree with $J_{\Lab_e}$ under a choice of strip-like end (it is important here to keep Equation (\ref{eq:label_incoming_end_constrained}) into account).
\item (Compatibility with gluing) For each tree $T$ labelling a boundary stratum of $ \Rbar_{\Lab}$ or $ \Rbar_{\uLab}$,  the restriction of $J_\Lab$  to a neighbourhood of the corresponding boundary stratum asymptotically agrees with the map obtained by gluing. In the case of $ \Rbar_{\Lab} $, this means that the diagram
  \begin{equation}
   \begin{tikzcd}
\partial \Sbar^T_{\Lab}  \times (-1,0]^{E(T)}  \arrow{dr}{J_\Lab} \arrow{r}{} & \partial \Sbar_{\Lab} \arrow{d}{J_\Lab}   \\
 & \scrJ &
\end{tikzcd}  
\end{equation}
commutes up to a perturbation which %
is supported in a compact subset, and which vanishes to infinite order at the origin.
\item (Compatibility with boundary) For each tree $T$ labelling a boundary stratum of $ \Rbar_{\Lab}$ or  $ \Rbar_{\uLab}$ and vertex $v \in V(T)$ such that $\Lab_v$ is non-degenerate, the restriction of $J(\Lab)$ or $J(\uLab)$ to  $\Sbar^v_{\Lab} $ or  $\Sbar^v_{\uLab} $ agree with the pullbacks of $J(\Lab_v)$, $J(\Lab^\sigma_v)$, or $J(\uLab_v)$.
\item (Forgetful maps) $J(\Lab)$ and $J(\uLab) $ are compatible with the forgetful maps associated to repeating elements of $\Sigma$ or $\Sigma^\phi$.
\item (Forgetting $\phi$) For each sequence $\Lab$ in  $ \Sigma^\phi \amalg \cA$, $J(\Lab)$ agrees with $J(\pi \Lab)$.
\end{enumerate}

\subsubsection{Moduli spaces of pseudoholomorphic discs}
\label{sec:moduli-spac-pseud}

Given a choice of Hamiltonian paths $H_\Lab$ and almost complex structure $J_\Lab$, we define the moduli space $\Rbar(\Lab)$ as the space of isomorphism classes consisting of an element  $S \in \Rbar_\Lab$ and a finite energy $J_\Lab$-holomorphic maps $u$ from a pre-stable curve whose stabilisation is $S$ to $X$, with (moving) Lagrangian boundary conditions given by (i) $L$ along $\partial_L S$,  (ii) $\Phi_\Lab(X_{q_\Lab})$ along $\partial_\Sigma S$, and (iii)  $\Phi_\Lab  (X_{q^\phi_\Lab})$ along $\partial^\phi_\Sigma S$.

Returning to the context of Section \ref{sec:isop-const}, we now consider the (topological) energy $E(u)$ defined as in Section \ref{sec:energy-cont-maps}. The following result follows immediately from Lemma \ref{lem:universal_constant_isoperimetric} and the proof of Lemma \ref{lem:estimate_top_energy}:
\begin{cor} \label{cor:uniform_reverse_isoperimetric}
Every element $u$ in $ \Rbar(\Lab)$ or $\Rbar(\uLab)$   satisfies the reverse isoperimetric inequality
  \begin{equation}
  \ell(\partial u/ {\sim}) \leq \frac{3C}{4} E(u) +  \textrm{ a constant independent of u.}
  \end{equation} \qed
\end{cor}

The finite energy condition implies that we have a natural evaluation map
\begin{equation}
  \Rbar(\Lab) \to \prod \Crit(\Lab_e)  
\end{equation}
where the product is taken over all ends of elements of $\Rbar_\Lab$, i.e. over all external edges $e$ of the underlying tree $T$ whose label $\Lab_e$ is non-degenerate (i.e. not contained in $\Sigma$ or  $\Sigma^\phi$). We have a similar map for $\Rbar(\uLab)$, and   denote the fibre over a collection $x_e \in \Crit(\Lab_e)$ of intersection points by $ \Rbar(\{x_e\})$ in either case. The following result is then a standard consequence of regularity theory for holomorphic curves; for its statement, we fix the constant $\codim$ which vanishes for the moduli spaces $ \Rbar(\Lab)$, and equals the codimension of  $\Rbar_{\uLab} \subseteq \Rbar_\Lab$ for the moduli spaces $\Rbar(\uLab)$.
\begin{lem} \label{lem:dimension_moduli_space_of_discs}
For generic choices of Floer data $(H_\Lab, J(\Lab))$ and $(H_{\uLab},J(\uLab))  $  the moduli space $\Rbar(\Lab)$ and $\Rbar(\uLab)$ are regular. In particular,  $\Rbar(\{x_e\} )$ is a topological manifold with boundary of dimension
  \begin{equation}
  |\Lab| - 3 - \codim + \deg( x_{e^\out}) - \sum_{e \neq e^\out}    \deg( x_{e}),
  \end{equation}
which is stratified by the topological type of the underlying curve, and the interior of each stratum is a smooth manifold.

This space is naturally oriented relative the tensor product of the orientation lines of $ \Rbar_\Lab $ or $ \Rbar_{\uLab}$ with  $  \delta_{x_{e^{\out}}} \otimes \bigotimes_{e \in E^{\inp}(T_v)}  \delta^{\vee}_{x_e} $. The  boundary is covered by the inverse image of the boundary strata of $\Rbar_\Lab$ or $\Rbar_{\uLab}$, together with the union, over all ends $e$ of elements of $\Rbar_\Lab$ or $\Rbar_{\uLab}$, of the images of the fibre products
\begin{equation}
   \Rbar(\Lab) \times_{\Crit(\Lab_e)} \Rbar(\Lab_e) \textrm{ or } \Rbar(\uLab) \times_{\Crit(\Lab_e)} \Rbar(\Lab_e).
\end{equation} \qed
\end{lem}

\subsection{Parametrised Morse moduli spaces}
\label{sec:abstr-moduli-spac}

In Section \ref{sec:cohom-constr}, we constructed various moduli spaces as fibre products of moduli spaces of discs and gradient flow lines, along common evaluation maps to Lagrangian fibres (see \cite{Fukaya1993} and \cite{FukayaOh1998}). This construction is not sufficiently flexible in general, as iterated fibre products may not be transverse. The standard solution is to take perturbations of the Floer equations and the gradient flow equations which depend on the abstract configuration being considered. For our applications, it suffices to perturb the gradient flow lines, so we adopt this simplified context.

\subsubsection{Morse and Floer edges}
\label{sec:morse-floer-edges}

Let $T$ be a rooted ribbon tree with a label $\Lab$ by elements of $\Sigma \amalg \Sigma^\phi \amalg \cA $ as in the previous section.  We define a decomposition
\begin{equation}
  E(T) = E^{\Fl}(T) \amalg E^{\Mo}(T)  
\end{equation}
into \emph{Floer} and \emph{Morse} edges as follows:
\begin{equation}
 \parbox{35em}{any edge labelled by a pair in $\Sigma$ or in $\Sigma^\phi$ lies in $E^{\Mo}(T) $. All other edges lie in $  E^{\Fl}(T) $.}
\end{equation}
We have a corresponding decomposition of the set of flags of $T$ into Floer and Morse flags:
\begin{equation}
  F(T) = F^{\Fl}(T) \amalg F^{\Mo}(T).
\end{equation}

To edges in $E^{\Mo}(T)$, we shall associate (perturbed) Morse flow lines in a Lagrangian fibre of $X \to Q$. To this end, we set the space of allowable lengths for each edge $e$ of $T$ to be
\begin{equation}
  \Tbar_e =   \begin{cases}  [0,\infty] & e \in E^{\fin}(T) \\
 \{ \infty \} & e \in E^{\ext}(T).
\end{cases},
\end{equation}

For each vertex $v$, recall that $\Lab_v$ denotes the ordered subset of $\Lab$ obtained from the labels of the components adjacent to $v$, starting again with the component to the left of the outgoing edge $e^\out_v$. Note that we have a canonical identification between the punctures of an element of $\Rbar_{\Lab_v}$ and the Floer edges of $T_v$, and between the boundary marked points and the Morse edges. We define
 \begin{align}
 \RTbar_{T,\Lab}  & \coloneqq \prod_{v \in V(T)} \Rbar_{\Lab_v} \times \prod_{e \in E^{\Mo}(T)} \Tbar_e 
\end{align}
Similarly, given a fixed element $\sigma \in \Sigma$, and  consider an ordered subset $\Lab$ of $ \Sigma_\sigma \amalg \Sigma^\phi_\sigma \amalg  \cA$ satisfying Condition \eqref{eq:condition)_uLab_with_Sigma}. Referring back to Section \ref{sec:abstr-moduli-spac-2}, and setting $V^{\sigma}(T)$  to be $V^\ast(T)$ in order to account for the fact that we have chosen $\sigma$ as a basepoint, we similarly define
\begin{align}
 \RTbar_{T,\uLab}  & \coloneqq \prod_{v \in M(T)} \Rbar_{\uLab_v} \times \prod_{v \in V^\sigma(T)}  \Rbar_{\Lab^\sigma_v} \times \prod_{v \notin V^{\sigma}(T) \cup M(T)} \Rbar_{\Lab_v}  \times \prod_{e \in E^{\Mo}(T)} \Tbar_e .
\end{align}

Identifying $\Tbar_e$ with the interval $[-1,0]$ via the map $R \mapsto - e^{-R}$, we obtain the structure of a manifold with corners on this space.  The corner strata are products of the corner strata of $\Rbar_{\Lab_v} $, $\Rbar_{\Lab^\sigma_v} $,  and $\Rbar_{\uLab_v} $  according to the topological type of the corresponding stable disc, and the stratification of $\Tbar_e $ according to whether the length is $0$, non-zero and finite, or infinite. To have a better description of the boundary strata of $ \RTbar_{T,\uLab} $, we write
\begin{equation}
   \RTbar^\sigma_{T,\Lab}  \coloneqq \prod_{v \in V(T)} \Rbar_{\Lab^\sigma_v} \times \prod_{e \in E^{\Mo}(T)} \Tbar_e,
\end{equation}
for an isomorphic copy of $\RTbar_{T,\Lab}$ whenever $\Lab$ is a sequence of the form
\begin{align} 
& (L,  \sigma_{0}, \ldots , \sigma_\ell ) \\
& ( \sigma_{-r},  \ldots, \sigma_{-1}  , L) \\ 
&  (\sigma_{-r},  \ldots,  \sigma_{-1}, \sigma^\phi_{0}, \cdots,  \sigma^\phi_\ell),
\end{align}
and all elements of $\Sigma$ above lie in $\Sigma_\sigma$.
\begin{rem}
We stress the difference between $ \RTbar^\sigma_{T,\Lab} $ and $ \RTbar_{T,\Lab^\sigma} $. In the first case, the labels for Morse edges are given by pairs of elements of $\Lab$, and hence can include elements of $\Sigma$ which differ from the basepoint $\sigma$. In the second case, all Morse edges have label $(\sigma,\sigma)$.
\end{rem}

For each map $T' \to T$ which collapses internal edges, and $v$ a vertex of $T$, let $T'_v$ denote the subtree of $T'$ whose vertices are those mapping to $v$, and whose edges all are edges adjacent to such vertices.    We have a natural identification of strata
\begin{equation} \label{eq:common_strata}
\RTbar_{T',\Lab} \supset     \RTbar^{T}_{T',\Lab} \subset \RTbar_{T,\Lab} 
\end{equation}
given on one side by the locus where the lengths of all collapsed edges vanish and on the other by the inclusion $\Rbar^{T'_v}_{\Lab_v} \subset \Rbar_{\Lab_v}$.  We obtain the space
\begin{equation}
  \RTbar_{\Lab} \coloneqq \bigcup_{|E^\ext(T)| = |\Lab|} \RTbar_{T,\Lab}/ {\sim}
\end{equation}
by gluing the spaces $ \RTbar_{T,\Lab}$ along their common strata. Restricting to trees labelling the boundary strata of $\Rbar_{\uLab}$, we similarly obtain:
\begin{equation}
  \RTbar_{\uLab} \coloneqq \bigcup_{|E^\ext(T)| = |\Lab|} \RTbar_{T,\uLab}/ {\sim}.
\end{equation}

We can stratify the boundary of $ \RTbar_{\Lab}  $ in such a way that the strata are indexed by trees $T$ labelled by $\Lab$:
\begin{equation}
  \RTbar^T_{\Lab} \coloneqq \prod_{v \in V(T)} \RTbar_{T_v,\Lab_v}. 
\end{equation}
Such a stratum corresponds to collections where each Floer edge of $T$ can be identified with the node of a broken disc, and where the Morse edges of $T$ have infinite length. In particular, the codimension $1$ boundary strata are given by a choice of tree $T$ with $|\Lab|$ external edges,  together with a distinguished  internal edge $e \in E$ which is either (i) the unique internal Floer edge  or (ii) the unique internal Morse edge whose length is specified to equal $\infty$. Such a stratum is naturally homeomorphic to the product
\begin{equation} \label{eq:decompose_tree_above_and_below_e}
  \RTbar_{T^-_e, \Lab^-_e} \times \RTbar_{T^+_e, \Lab^+_e},
\end{equation}
where $T^-_e$ and $T^+_e$ are the trees whose vertices are those vertices of $T$ whose path to the outgoing edge respectively passes through $e$ or avoids it; the labels $ \Lab^\pm_e$ are inherited from $\Lab$, and we note that there are distinguished external edges
\begin{equation} \label{eq:definition_epm}
  e_\pm \in E^\ext(T^\pm_e)
\end{equation}
which is the outgoing edge of $T^-_e$, and the incoming edge of $T^+_e$ corresponding to $e$.

We have a similar description for the moduli spaces $ \RTbar_{T,\uLab}$  with constraints, whose boundary strata are given by
\begin{equation}
  \RTbar^T_{\uLab} \coloneqq \prod_{v \in M(T)} \RTbar_{T_v,\uLab_v} \times  \prod_{v \in V^\sigma(T)} \RTbar^\sigma_{T_v,\Lab_v} \times   \prod_{v \notin V^\sigma(T) \cup M(T)} \RTbar_{T_v,\Lab_v},
\end{equation}
where $T$ is required to label a boundary stratum of $\Rbar_{\uLab}$. With the exception of strata corresponding to the boundary of $\Rbar_{3,\underline{2}}$ in the cases corresponding to Equations \eqref{eq:sequence_tensor_product_comparison} and \eqref{eq:sequence_Hom_comparison-Q}, the boundary strata again have either a unique Floer edge, or a unique Morse edge of infinite length.

\begin{rem} \label{rem:uniform_description_special_case}
  Since we have chosen to give a uniform description, our construction in the case $\Lab$ is a subset of $Q$ is slightly more complicated than strictly necessary for the intended application of construction an $A_\infty$ category with objects given by a sufficiently fine covering of $Q$: the moduli space $\Rbar_\Lab$ can be thought of as the union of the moduli space of discs with points marked by successive elements of $\Lab$ with collars of all its boundary strata. By collapsing the moduli spaces of discs to a point, one obtains a map from $\RTbar_\Lab$ to the corresponding Stasheff associahedron labelling metric ribbon trees (the metric corresponds to the collar coordinate). Having assumed that the Lagrangian fibres do not bound any holomorphic discs, the construction of the $A_\infty$ category can proceed entirely by Morse-theoretic methods, as in \cite{FukayaOh1998,Abouzaid2011a}. As a consequence of Remark \ref{rem:regular_moduli_spaces_despite-forgetful} below, our construction ultimately yields an isomorphic category (i.e. with $A_\infty$ operations that are equal).
  \end{rem}

We associate to each Morse edge of $T$ a universal interval over $\RTbar_{T,\Lab}$:  first, we define $\Ibar_r$ for $r \in [0,\infty)$ to be the interval $[0,r]$. We extend this to  $r = \infty$, by setting
\begin{equation}
\Ibar_\infty \coloneqq \cI_+ \amalg \cI_-.
\end{equation} 
We equip the space
\begin{equation}
\Ibar_{[0,\infty]} \coloneqq \coprod_{r \in [0, \infty]} \Ibar_r
\end{equation}
with the  natural topology away from $r = \infty$, extended to the latter by the requirement that, for each positive real number $R$, the sets
\begin{equation} \label{eq:charts_infinity_tree_break}
\coprod_{R < r \leq \infty} [0,R) \textrm{ and } \coprod_{R < r \leq \infty} (-R,0] 
\end{equation}
be open, where on the right side we identify the interval $(-R,0]$ with $(r-R,r]$ over $r \in (R,\infty]$. We then define
\begin{equation}
\Ibar_e \coloneqq \begin{cases}\Ibar_{[0,\infty]}  & e \in E^{\fin}(T) \\
 [0, \infty) & e \in E^\inp(T) \\
 (-\infty, 0] & e \in E^\out(T)
\end{cases}    
\end{equation}
and note that we have a natural projection map
\begin{equation}
\Ibar_e \to \Tbar_e  
\end{equation}
for all Morse edges of $T$. 

Pulling back the universal interval over $\Tbar_e$ yields the universal interval associated to the edge $e$:
\begin{equation}
\Ibar^{e}_{T,\Lab} \to \RTbar_{T,\Lab}.
\end{equation}
and the union over all Morse edges is the \emph{universal interval over} $\RTbar_{T,\Lab}$
\begin{equation}
\Ibar_{T,\Lab} \to    \RTbar_{T,\Lab},
\end{equation}
which is again equipped with a natural smooth structure. The two universal intervals over the codimension $1$ strata of $\RTbar_{T,\Lab}$ are naturally isomorphic. We have an entirely analogous construction of a universal interval
\begin{equation}
\Ibar_{T,\uLab} \to    \RTbar_{T,\uLab}.
\end{equation}

\subsubsection{Morse data and moduli spaces for pairs}
\label{sec:morse-data}

Given  a pair of elements $\Lab = (\Lab_{0}, \Lab_{1})$  of $ \Sigma $ or $  \Sigma_{\phi} $, let $\nu_Q \Lab$ denote the intersection $\nu_Q \Lab_{0} \cap \nu_Q \Lab_{1}$. Recall that $X_{\nu_Q \Lab }$ denotes the restriction of the torus bundle $X \to Q$ to $\nu_Q \Lab$; we pick a distinguished fibre $X_{q_{\Lab}}$  which we equip with a metric, and with a Morse--Smale function
\begin{equation}
  f_\Lab \co X_{q_{\Lab}} \to \bR.
\end{equation}
Pick in addition a Lagrangian section over $\nu_Q \Lab$, as in Remark \ref{rem:any_trivialisation_works}, which induces a trivialisation of $X_{\nu_Q \Lab}$ by parallel transport. We write
\begin{equation}
\psi_{\Lab}^{q, q'} \co X_q \to X_{q'}  
\end{equation}
for the induced map of fibres over points $q, q' \in \nu_q \Lab$. For simplicity of notation, we may sometimes omit the subscript from the above maps.

We extend the definition of Morse data given in Equation \eqref{eq:Morse_data_product} as follows: for a non-negative real number $r$, we define
\begin{equation}
\scrV_{r}(\Lab)  =   C^\infty(\Ibar_r, C^\infty(X_{q_{\Lab}}, T X_{q_{\Lab}} )). 
\end{equation}
For the case $r=\infty$, we define
\begin{equation} \label{eq:Morse_data_poly_bi_infinite}
\scrV_{\infty}(\Lab)   =   \scrV_{+}(\Lab)  \times \scrV_{-}(\Lab),
\end{equation}
and recall that we require that elements of $ \scrV_\pm(\Lab) $ correspond to a family of vector fields parametrised by $\Ibar_\pm$, which agree with the gradient vector field outside a compact set.  We then define
\begin{equation}
\scrV_{[0,\infty]}(\Lab) \coloneqq \coprod_{r \in [0,\infty]}  \scrV_{r} (\Lab). 
\end{equation}

Exactly as in Section \ref{sec:morse-theor-prod}, we define $ \cT_{\bullet}(\Lab, \scrV) $, for $\bullet \in \{ \pm \} \amalg [0,\infty) $, to be the set of pairs $(\gamma, \xi)$, with $\gamma$ a path in $X_{q_\Lab}$ with domain $\Ibar_{\bullet}$ and $\xi \in  \scrV_{\bullet} (\Lab)$ satisfying the perturbed gradient flow equation $
\frac{d \gamma}{dt} = \xi$.
We also define 
\begin{equation}
 \cT_{\infty}(\Lab, \scrV) =  \cT_{-}(\Lab, \scrV) \times_{ \Crit_{\Lab}}  \cT_{+}(\Lab,\scrV),
\end{equation}
and obtain the union
\begin{equation}
 \cT_{[0,\infty]}(\Lab, \scrV) = \coprod_{r \in [0,\infty]}  \cT_{r} (\Lab). 
\end{equation}

By adding to $\cT_{\bullet}(\Lab, \scrV) $ the fibre products at the ends with the space $\Tbar(\Lab)$ of gradient flow lines, we obtain a fibrewise compactification of $  \cT_{\bullet}(\Lab, \scrV)  $ over $\scrV_{\bullet}$, which we denote $ \Tbar_{\bullet}(\Lab, \scrV) $. %
In the case $\bullet = \infty$, the boundary stratum is given by
\begin{equation}
 \cT_{-}(\Lab, \scrV) \times_{ \Crit_{\Lab}} \Tbar(\Lab) \times_{ \Crit_{\Lab} }  \cT_{+}(\Lab,\scrV).  
\end{equation}

Taking the union over $r \in [0,\infty]$, we obtain the space
\begin{equation}
     \cT_{[0,\infty]}(\Lab, \scrV) \subset  \Tbar_{[0,\infty]}(\Lab, \scrV) 
\end{equation}
which maps to $\scrV_{[0,\infty]}(\Lab)$ with compact fibres.

\subsubsection{Morse moduli spaces for labelled trees}
\label{sec:morse-moduli-spaces-1}

Let $T$ be a tree labelled as in Section \ref{sec:abstr-moduli-spac-2}. For each edge in $ E^\Mo(T)$,   we define
\begin{equation} \label{eq:Morse_data_poly_semi_infinite}
\scrV_{e}(\Lab) \coloneqq \begin{cases} \scrV_{\pm}(\Lab_e) \times \RTbar_{T,\Lab} & \textrm{if } e \textrm{ is external} \\
 \scrV_{[0,\infty]} (\Lab_e) \times_{\Tbar_e} \RTbar_{T,\Lab} &  \textrm{otherwise. }
\end{cases}
\end{equation}

In particular, for any edge $e$, there is thus a natural projection map
\begin{equation} \label{eq:parametrised_Morse_one_edge}
\scrV_e(\Lab) \to \RTbar_{T,\Lab}.
\end{equation}
The fibre product of these spaces over $\RTbar_{T,\Lab}$ for all $e \in E^{\Mo}(T)$ defines the space of Morse data
\begin{equation}
 \scrV_T(\Lab) \to \RTbar_{T,\Lab}.
\end{equation}
We can define spaces $ \scrV^\sigma_T(\Lab) \to \RTbar^\sigma_{T,\Lab} $ and $\scrV_T(\uLab) \to  \RTbar_{T,\uLab}$ in exactly the same way.

A section of  $ \scrV_T(\Lab) $ thus consists of a choice of perturbed gradient flow equation for each edge $e \in E^{\Mo}(T)$, metrised according to the image of this point in $\Tbar_T$, and hence corresponds to a map
\begin{equation}
  \Ibar_e \to   C^\infty(X_{q_{\Lab_e}}, T X_{q_{\Lab_e}} ).
\end{equation}
for each Morse edge $e$. We shall assume that such a map is smooth, and moreover agrees to infinite order, along the boundary of the moduli space, with the map obtained in gluing.

We shall moreover pick the data consistently in the following sense: identifying the elements of $E(T) \setminus \{e\}$ and $E(T/e)$, we assume that the following diagram commutes for each Morse edge $e'$ of $T$:
   \begin{equation}
     \begin{tikzcd}
       \Ibar^{e'}_{T,\Lab}  | \RTbar^{T/e}_{T,\Lab}   \arrow{r} \arrow{dr} &  \Ibar^{e'}_{T/e,\Lab}  | \RTbar^{T/e}_{T,\Lab}   \arrow{d}\\
  & C^\infty(X_{q_{\Lab_{e'}}}, T X_{q_{\Lab_{e'}}} ).
     \end{tikzcd}      
   \end{equation}
In addition, for each Morse edge $e$ of $T$, and for each edge $e' \neq e_\pm$ of $T^{\pm}_e$ (with corresponding edges $e' \in E(T)$), we require the commutativity of the diagram:
   \begin{equation}
     \begin{tikzcd}
\Ibar^{e'}_{T,\Lab}  | \RTbar_{T^-_e, \Lab^-_e} \times \RTbar_{T^+_e, \Lab^+_e} \arrow{r} \arrow{dr} &  \Ibar^{e'}_{T^\pm_e,\Lab^\pm_e} | \RTbar_{T^\pm_e, \Lab^\pm_e} \arrow{d} \\
  & C^\infty(X_{q_{\Lab_{e'}}}, T X_{q_{\Lab_{e'}}}).
     \end{tikzcd}        
   \end{equation}
We impose the analogous condition that the restrictions of the Morse data associated to $e$ agree, upon restriction to this boundary stratum, with the pullbacks of the data associated to $e_\pm$ on the respective components.

Such sections determine  moduli spaces
\begin{equation}
  \Tbar_{e}(\Lab) \to \RTbar_{T,\Lab},
\end{equation}
for each Morse edge $e$, which compactify the parametrised moduli space of (perturbed) gradient solutions corresponding to $e$. The same construction yields moduli spaces
\begin{align}
 \Tbar^\sigma_{e}(\Lab) & \to \RTbar^\sigma_{T,\Lab} \\
  \Tbar_{e}(\uLab) & \to \RTbar_{T,\uLab},
\end{align}
whenever $e$ is a Morse edge of a tree labelling a boundary stratum of $\Rbar^\sigma_{\Lab}$ or $\Rbar_{\uLab}$.
\begin{rem}
Note that, while we have an identification $\RTbar^\sigma_{T,\Lab} \cong  \RTbar_{T,\Lab^\sigma}$, the spaces $ \Tbar^\sigma_{e}(\Lab) $ and $ \Tbar_{e}(\Lab^\sigma) $ are spaces of perturbed gradient flow lines for Morse functions which are a priori different, over fibres which themselves may be different.
\end{rem}
For each flag $f$ containing $e$,  we have an evaluation map
\begin{equation}
\Tbar_{e}(\Lab) \to  X^f_{q_{\Lab_e}},
\end{equation}
while for each infinite end of $e$, we have an evaluation map
\begin{equation}
\Tbar_{e}(\Lab) \to  \Crit_{\Lab_e},
\end{equation}
so we obtain a decomposition of $\Tbar_{e}(\Lab)  $ into components $\Tbar_{e}(x; \Lab) $  labelled by such critical points.
\begin{lem}[c.f. Theorem 1.4 of \cite{FukayaOh1998}] 
For generic choices of admissible Morse data, the moduli spaces $\Tbar_{e}(\Lab) $,  $\Tbar^\sigma_{e}(\Lab) $, and $ \Tbar_{e}(\uLab)$ associated to an internal Morse edge $e$ are manifolds with boundary of dimension equal to 
\begin{equation} \label{eq:dimension_formula_moduli_space_interval}
n +|\Lab|-3 - \codim   
\end{equation}
where $\codim$ is as in Lemma \ref{lem:dimension_moduli_space_of_discs}. This manifold is  naturally oriented relative the tensor product of the orientation  line
\begin{equation}
   |X_{q_{\Lab_e}}| \coloneqq \Lambda^{n} T X_{q_{\Lab_e}}
\end{equation}
of $X_{q_{\Lab_e}}$ with the underlying abstract moduli spaces, and has boundary given by the inverse image of their boundaries under the evaluation map.

If $e$ is a external edge, then $\Tbar_{e}(x; \Lab)$,  $\Tbar^\sigma_{e}(x; \Lab)$, and $\Tbar_{e}(x; \uLab)$ are manifolds with boundary of dimension 
\begin{equation}
  \begin{cases} n-\deg(x)+|\Lab|-3 - \codim    & e \in E^{\inp}(T) \\
    \deg(x)+|\Lab|-3 - \codim & e \in E^{\out}(T)
  \end{cases}
\end{equation}
which are naturally oriented relative the tensor product of the orientation line of the underlying  moduli space with $|X_{q_{\Lab_e}}| \otimes \det^{\vee}_x$ in the first case, and $  \det_x$ in the second case (see Equation \eqref{eq:orientation_line_critical}).  The inverse image of each stratum of the underlying abstract moduli space  is again naturally a submanifold, of the same codimension. There is an additional boundary stratum of codimension $1$, given for $\Tbar(\Lab)$ by
\begin{equation} \label{eq:gradient_flow_line_breaking}
 \Tbar(\Lab_e)  \times_{\Crit(\Lab_e)} \Tbar_{e}(\Lab)
\end{equation}
and similarly for $\Tbar^\sigma_{e}(\uLab)$ and $\Tbar_{e}(\uLab)$. 
\end{lem}
\begin{proof}[Sketch of proof]
  Whenever $e$ is an internal edge, the interior of $\Tbar_{e}(\Lab) $,  $\Tbar^\sigma_{e}(\Lab) $, and $ \Tbar_{e}(\uLab)$  is by construction the product of $X_{q_{\Lab_e}}$ with the (interior of) moduli space $\RTbar_{T,\Lab}$, and hence has dimension given by Equation \eqref{eq:dimension_formula_moduli_space_interval}: the identification goes via picking a section of the universal interval over the interior of $\RTbar_{T,\Lab}$, which allows us to describe the moduli space as the space of solutions to a family of ODE's (prescribed by a point in $\RTbar_{T,\Lab}$) with initial conditions given by a point in $X_{q_{\Lab_e}}$.  Whenever the chosen Morse function is Morse-Smale, transversality of ascending and descending manifolds implies that the inverse images of each codimension $k$ boundary stratum of 
  $\Tbar_{e}(\Lab) $,  $\Tbar^\sigma_{e}(\Lab) $, and $ \Tbar_{e}(\uLab)$ is a manifold of dimension $n +|\Lab|-3- (\codim +k)  $. Gluing of (half)-gradient flow lines establishes that the union of the interior with these boundary strata is a topological manifold with boundary (in fact, it can be equipped with the structure of a smooth manifold with corners, but we shall not use this).

  The argument for external edges is entirely analogous, keeping in mind that $\det_x$ is the orientation line of the stable manifold of the critical point $x$, whose dimension is given by $\deg(x)$, and that the direct sums of the tangent spaces of the stable and unstable manifolds of $x$ is naturally isomorphic to $TX_{q_{\Lab_e}} $, so that the unstable manifold has orientation line isomorphic to $|X_{q_{\Lab_e}}| \otimes \det^{\vee}_x$. Equation \eqref{eq:gradient_flow_line_breaking} accounts for the boundary stratum corresponding to breaking of gradient trajectories along either end.
\end{proof}

Consistency of the choices of data implies that the two moduli spaces over each codimension $1$ boundary stratum of the boundary are naturally identified: in particular if $e$ is an internal edge, then the inverse image in $ \Tbar_{e}(\Lab)$ of the stratum $\ell(e) = \infty$, which is homeomorphic to the product $ \RTbar_{T^-_e, \Lab^-_e} \times  \RTbar_{T^+_e, \Lab^+_e} $, is naturally identified with the fibre product
\begin{equation}
  \Tbar_{e^-}(\Lab^-_e) \times_{\Crit(\Lab_e)} \Tbar_{e^+}(\Lab^+_e),
\end{equation}
 along the evaluation maps associated to the edges $e_\pm$ (see Equation \eqref{eq:definition_epm}). We have the same type of decomposition for boundary strata associated to edges of  $\Tbar^\sigma_{e}(\uLab)$ and $\Tbar_{e}(\uLab)$.

\subsection{Mixed moduli spaces}
\label{sec:mixed-moduli-spaces-1}

Let $T$ be a rooted ribbon tree labelled by $\Lab$. Given a choice of  Floer data as in Section \ref{sec:perturbed-equations} and Morse data as in Section \ref{sec:abstr-moduli-spac}, we obtain moduli spaces of holomorphic discs for all vertices and of gradient flow lines for all Morse edges. We shall define mixed moduli spaces as fibre products with respect to evaluation maps to the fibres. In order to specify these evaluation maps, we need to make some additional choices:

For any Morse flag $f$ of a tree $T$ labelling a boundary stratum of $\Rbar_\Lab$, $\Rbar^\sigma_\Lab$ or $\Rbar_{\uLab}$, we pick maps
\begin{align}
q^f_{\Lab} \co \RTbar_{T,\Lab} & \to \nu_Q \Lab_f  \\
q^{f,\sigma}_{\Lab} \co \RTbar^\sigma_{T,\Lab} & \to \nu_Q \Lab_f  \\
q^f_{\uLab} \co \RTbar_{T,\uLab} & \to \nu_Q \Lab_f
\end{align}
subject to the following constraints:
\begin{enumerate}
\item (Value at marked points) For each vertex $v$ such that $\Lab_v$ is degenerate, the values of $q^f_{\Lab}$, $q^{f,\sigma}_{\Lab}$, or $q^{f}_{\uLab} $ for flags containing $v$ agree. If $\Lab_v$ is non-degenerate,  $q^{f,\sigma}_{\Lab}$ is constant with value $q_\sigma$, while  $q^f_{\Lab}$ (or  $q^f_{\uLab}$) agrees with the value of the path $q_{\Lab_v}$ or $q^\phi_{\Lab_v}$ (resp.  $q_{\uLab_v}$ or $q^\phi_{\uLab_v}$) at the corresponding marked point of the surface in $\Sbar_{\Lab_v}$ (resp. $\Sbar_{\uLab_v}$).
\item (Compatibility with boundary) The two values of the maps $q^f_{\Lab}$ and $q^f_{\uLab}$ over each boundary stratum of $\RTbar_{T,\Lab}$, $\RTbar^\sigma_{T,\Lab}$, or  $\RTbar_{T,\uLab}$ agree.
\end{enumerate}
To clarify the last condition, recall that we have given a description of each boundary stratum as either a product of lower dimensional moduli spaces, corresponding to the breaking of a disc giving rise to a Floer edge or of a Morse edge having infinite length, or as a boundary stratum common to two different moduli spaces, corresponding to a Morse edge having length $0$, or the breaking of a disc giving rise to a Morse edge.

Let $X^f_\Lab$ denote the pullback of $X$ to a bundle over $\RTbar_\Lab$ under the map $q^f_\Lab$. For each Morse flag $f=(v,e)$, we have a natural evaluation map
\begin{equation}
 \Tbar_e(\Lab) \to X^f_\Lab
\end{equation}
obtained by applying $\psi_\Lab$ in order to identify the fibres of $X^f_\Lab$ with $X_{q_{\Lab_e}}$. We also have a map from $\Rbar(\Lab_v) \times_{\Rbar_{\Lab_v}} \RTbar_{T,\Lab}  $ to the same space, obtained by evaluation at the marked point associated to $e$. Let $X^T_\Lab$ define the fibre product over $\RTbar_\Lab  $ of the spaces $X^f_\Lab$ for all Morse flags over $\RTbar_\Lab $.   We then define the \emph{mixed moduli space} as the fibre product
\begin{equation} \label{eq:fibre_product_treed_discs}
  \begin{tikzcd}
     \RTbar_T(\Lab) \arrow{d} \arrow{r} &   \displaystyle{\prod_{e \in E^{\Mo}(T)} }  \Tbar_{e}(\Lab)  \times  \displaystyle{\prod_{e \in E^{\Fl}(T)} }  \Crit(\Lab_{e})   \arrow{d} \\
\displaystyle{\prod_{v \in V(T)}}  \Rbar(\Lab_v)  \times_{\Rbar_{\Lab_v}} \RTbar_{T,\Lab} \arrow{r} & X^T_\Lab \times  \displaystyle{\prod_{f \in F^{\Fl}(T)} }  \Crit(\Lab_{f}) .
  \end{tikzcd}
\end{equation}

Unwinding the definition of the fibre product,  we find that an element of $ \RTbar_T(\Lab)$ consists of (i) a point in $\RTbar_{T,\Lab} $, (ii) a gradient flow line in $\Tbar_e$ over this point in $\RTbar_{T,\Lab} $  for each Morse edge $e$ in $T$, (iii) an intersection point of the corresponding Lagrangians at each Floer edge $e$ of $T$,  and (iv) a holomorphic map in $  \Rbar(\Lab_v)   $   over the projection to $ \Rbar_{\Lab_v}$. At each Morse flag $v \in e$, we require that the evaluations of the gradient flow line in $\Tbar_e(\Lab)$ (at the end of $e$) and of the curve in $\Rbar(\Lab_v)$ (at the marked point corresponding to $e$) agree as points in $ X^f_{\Lab} $, while at each Floer flag $v \in e$ we require the asymptotic conditions for the elements of $\Rbar(\Lab_v) $ to be given by the  intersection point of Lagrangians chosen for $e$.

Taking the same fibre product construction yield moduli spaces $ \RTbar^\sigma_T(\Lab) \to  \RTbar^\sigma_{T,\Lab} $, as well as moduli spaces  with constraints
\begin{equation}
   \RTbar_T(\uLab)  \to  \RTbar_{T,\uLab}.
\end{equation}

Taking the evaluation map at all external edges yields the map
\begin{equation}
      \RTbar_T(\Lab)  \to \prod_{e \in E^{\ext}( T) } \Crit(\Lab_e),  
\end{equation}
whose fibre at a collection $x=\{x_e\}$ of critical and intersection points we denote $\RTbar_T(x)$, and similarly for the other two moduli spaces.  
\begin{lem}
For generic Floer and Morse data, the moduli spaces $\RTbar_T(x) $, $ \RTbar^\sigma_T(\Lab)$, and $ \RTbar_T(\uLab) $ are manifolds with boundary of dimension
  \begin{equation}
    |\Lab| - 3 - \codim + \deg( x_{e^\out}) - \sum_{e \in E^{\inp}(T)}    \deg( x_{e}),
  \end{equation}
where $\codim$ is as in Lemma \ref{lem:dimension_moduli_space_of_discs}. This space is naturally oriented relative the tensor product of the tangent space of the underlying moduli space with
\begin{equation} \label{eq:relative_orientation_mixed}
  \delta_{x_{e^\out}} \otimes \bigotimes_{e \in E^{\inp}(T_v)}  \delta^{\vee}_{x_e} .
\end{equation}

The codimension $1$ boundary strata of $  \RTbar_T(\Lab)$, $  \RTbar^\sigma_T(\Lab)$,  and $  \RTbar_T(\uLab)$ are given by the inverse images of the boundary strata of the corresponding abstract moduli spaces $\RTbar_{T,\Lab}$, $  \RTbar^\sigma_{T,\Lab}$,  and $  \RTbar_{T,\uLab}$, together with (i) boundary strata associated to each external edge $e \in E^{\ext}(T)  $\begin{equation} \label{eq:mixed_moduli_spaces_bubble}
  \begin{cases}
\RTbar_T(\Lab)   \times_{\Crit(\Lab_e) } \Tbar(\Lab_e) & e \in E^{\Mo}(T) \cap E^{\ext}(T)  \\
\RTbar_T(\Lab)   \times_{\Crit(\Lab_e) } \Rbar(\Lab_e) & e \in E^{\Fl}(T) \cap E^{\ext}(T)
  \end{cases}
\end{equation}
and (ii) boundary strata associated to each internal flag $(e,v)$ of Floer type
\begin{equation}
 \RTbar_{T^-_e}(\Lab^-_e)   \times_{\Crit(\Lab_e) } \Rbar(\Lab_e) \times_{\Crit(\Lab_e) } \RTbar_{T^+_e}(\Lab^+_e).
\end{equation}
We analogous descriptions for the boundary strata of $  \RTbar^\sigma_T(\Lab)$ and $\RTbar_T(\uLab)$. 
\end{lem}
\begin{proof}[Sketch of proof]
  We have already established that each of the moduli spaces in the bottom right and top left corner of Diagram \eqref{eq:fibre_product_treed_discs} are generically topological manifolds with boundary, equipped with smooth structure in the interior of each stratum; to prove the first statement, it suffices to show that one may in addition assume that their fibre product is transverse as stratified-smooth manifolds, i.e. that the restriction to the interior of each stratum is transverse.
  
The key point is that, in each fibre product over $X^f_{\Lab}$, one of the factors is a Morse gradient flow line, and that the inhomogeneous data for flow lines is allowed to vary with respect to the parameter in the abstract moduli space of stable discs and metric trees. The description of the boundary strata for generic choices is then a consequence of the description of the boundary strata of each constituent moduli space: the moduli space $ \Rbar(\Lab_v)  $ has boundary strata associated to breaking of pseudo-holomorphic strips at the ends, and the moduli space $  \Tbar_{e}(\Lab) $ has boundary strata associated to breaking of gradient trajectories. In the case of external edges, the corresponding stratum of the moduli space of maps incorporates a matching condition at one of the ends of the strip or the gradient flow line, corresponding to the fibre product over $\Crit_{\Lab_e}$. In the case of internal Morse edges, the breaking of gradient trajectories is incorporated into the length parameter of the corresponding flow line becoming infinite, and so does not give rise to a new boundary stratum. Finally, in the case of an internal Floer edge, we consider the decomposition of the tree $T$ into the trees $T^-_e$ and $T^+_e$  as in Equation \eqref{eq:decompose_tree_above_and_below_e}, and obtain a boundary stratum consisting of elements of the moduli spaces $\RTbar_{T^-_e}(\Lab^-_e) $ and $ \RTbar_{T^+_e}(\Lab^+_e)$ together with a pseudo-holomorphic strip which matches the asymptotic conditions of these moduli spaces; there are two such boundary strata, corresponding to whether we consider the pseudo-holomorphic strip as having bubbled from the moduli space of discs associated to either of the two vertices to which it is adjacent.
\end{proof}

Since the data are chosen consistently for different trees $T$, there are codimension $1$ strata which appear in pairs: given a Morse edge $e$, we have a map
\begin{align}\label{eq:map_collapse_edge_mixed_moduli}
 \RTbar_{T} (\Lab)   | \RTbar^{T/e}_{T,\Lab}  & \to  \RTbar_{T/e} (\Lab)   | \RTbar^{T/e}_{T,\Lab}   
\end{align}
corresponding on the left to locus where this Morse edge has length $0$, and on the right to the locus where it appears from a breaking of holomorphic discs.

Taking the union along the identification induced by Equation \eqref{eq:map_collapse_edge_mixed_moduli} over all $T$ with $d+1$ external edges and no interior edges of Floer type,  we obtain a moduli space $\RTbar(\Lab)$ which admits an evaluation map to $\Crit(\Lab_e) $ for each external edge as before. The same construction yields moduli spaces $\RTbar^\sigma(\Lab)$ and $\RTbar(\uLab)$.  We denote by $ \RTbar( x_{0},x_d, \ldots, x_{1}) $ the fibre over points in $\prod_{e \in E^{\ext}} \Crit(\Lab_e) $, with $x_{0}$ corresponding to the output. For consistency of notation, whenever $ \Lab$ is a pair of elements of our labelling set, and $(x,y)$ is a pair in $\Crit \Lab$, we write
\begin{equation}
   \RTbar(x, y) \coloneqq
   \begin{cases}
     \Tbar(x,y) & \textrm{ if $x$ and $y$ are critical points of a Morse function}\\
\Rbar(x,y)  & \textrm{ if $x$ and $y$ are intersection points of Lagrangians.}
   \end{cases}
\end{equation}

\begin{lem} \label{lem:boundary_moduli_space_{1}-dimensional}%
If the asymptotic conditions $ ( x_{0},x_d, \ldots, x_{1})$ satisfy
\begin{equation}
    d-2 - \codim +  \deg(x_0) - \sum_{i=1}^{d} \deg(x_i) = 1,
\end{equation}
then $\RTbar( x_{0},x_d, \ldots, x_{1})  $ is a $1$-dimensional manifold with boundary. The boundary decomposes into strata labelled by the codimension $1$ boundary strata of $\RTbar_\Lab$, $\RTbar^\sigma_\Lab$, or $\RTbar_{\uLab}$, together with the boundary strata
\begin{align}
 & \RTbar(x_0, y_0) \times  \RTbar( y_{0},x_d, \ldots, x_{1}) \\
 & \RTbar( x_{0},x_d, \ldots, y_i, \ldots,  x_{1}) \times \RTbar(y_i, x_i).
\end{align}  \qed
\end{lem}

\begin{rem} \label{rem:regular_moduli_spaces_despite-forgetful}
  As a special case of the above construction, we consider a subset $\Lab$ of $\Sigma$. The moduli spaces above will be used to define the $A_\infty$ structure on a Fukaya category with objects indexed by the elements of $\Sigma$. Note that all vertices of a tree labelling a stratum of this moduli space lie in $V^{\deg}(T)$ (see Equation \eqref{eq:degenerate_vertex}), which implies, according to the conditions imposed in the construction, that all pseudo-holomorphic discs with boundary on a fibre that appear in the construction of $\RTbar(\Lab)$ in this situation are ordinary holomorphic discs with trivial inhomogeneous term, and constant boundary conditions. Since we have assumed that all such discs are constant, we conclude that rigid elements of the moduli space $\RTbar(\Lab)$ can only contain holomorphic discs with three marked points, as any moduli space with more than three marked points passing through any point would have higher dimension. Since the moduli spaces of constant discs with three marked points form a copy of the fibre, they can be forgotten from the description of the rigid moduli spaces $\RTbar(\Lab) $, which thus agree with the corresponding moduli spaces of rigid gradient trees.
\end{rem}

\section{Chain level constructions}
\label{sec:homological-algebra}

The cover  $\Sigma$ of $Q$ chosen in Section \ref{sec:perturbed-equations} will be fixed for the remainder of this paper. Our first task is to refine the global structures from Section \ref{sec:cohom-constr} to the $A_\infty$ level. We then proceed to the local case, and establish the equivalence between local and global invariants. One minor difference with the cohomological construction is that the local category will have objects labelled by additional choices of basepoints. This leads to a quasi-equivalent category with many more objects, which is easier to make into the target of our functor. The classification of objects up to quasi-isomorphism will be the same as in Section \ref{sec:cohom-constr}.

\subsection{Statement of the cochain-level results}
\label{sec:statement-results}

We shall begin, in Section \ref{sec:directed-category}, by proving Proposition \ref{prop:cochain_diagram_commutes}: this means that we first construct the category $\Fuk$ with objects given by elements of $\Sigma$, and morphisms given  by Equation (\ref{eq:morphism_category_polytopes}), from which we see that cohomological category of $\Fuk$ is naturally identified with the category $H \Fuk$ introduced in Section \ref{sec:cohom-constr}. As discussed in Remark \ref{rem:regular_moduli_spaces_despite-forgetful}, the structure constants of this category can be expressed entirely in terms of Morse theory, because we have assumed that the fibres bound no non-constant holomorphic discs. Next, associated to an element $L$ of $\cA$, the left and right $A_\infty$ modules $\cL_L$ and $\scrR_L$ over $\Fuk$ have underlying cochain complexes that were introduced in Section \ref{sec:left-right-modules-1}, and we shall construct the map from Equation \eqref{eq:cochain_map_CF_to_left_module}, which we are trying to prove induces an isomorphism on cohomology. As described in the introduction, our strategy for proving that Equation \eqref{eq:cochain_map_CF_to_left_module} is an isomorphism proceeds by showing that its composition with Equation \eqref{eq:cochain_map_tensor_product_to_CF} factors as in Diagram \eqref{eq:cochain_diagram_modules}. The cochain complexes, morphisms, and the homotopy in this diagram, are also constructed in  Section \ref{sec:directed-category}, before the proof of Proposition \ref{prop:cochain_diagram_commutes} is given.

It is clear from the statement of Proposition \ref{prop:cochain_diagram_commutes} that Theorem \ref{thm:main_thm} would follow by showing that each of the three arrows which can be composed counterclockwise in Diagram \eqref{eq:cochain_diagram_modules} are isomorphisms. It is easiest to see this for the middle arrow:
\begin{lem} \label{lem:left_module_perfect}
For each $L \in \cA$, the left module $\cL_L$ admits a filtration whose associated graded modules are isomorphic to left Yoneda modules over $\Fuk$. 
\end{lem}
The above result is proved in  Section \ref{sec:perf-left-modul}. Using the tensor product of $A_\infty$ modules (see, e.g. \cite[Section 2]{Seidel2008} for a discussion that's compatible with our sign conventions), we conclude:
\begin{cor} \label{cor:reduction_to_isomorphism_Hom_tensor}
  The natural map
  \begin{equation}
         \Hom_{\Fuk}( \cL_{L'},  \Delbar) \otimes_{\Fuk}  \cL_L   \to  \Hom_{\Fuk}( \cL_{L'},  \Delbar\otimes_{\Fuk} \cL_L)  
       \end{equation}
       is a quasi-isomorphism.
\end{cor}
\begin{proof}
 This is a special case of the following general situation: we have a pair $\cL_0$ and $\cL_1$ of left modules over an $A_\infty$ category $\cC$ and a bimodule $\cP$, and consider the map
  \begin{equation}
         \Hom_{\cC}( \cL_{0},  \cP) \otimes_{\cC}  \cL_1   \to  \Hom_{\cC}( \cL_{L_0},  \cP \otimes_{\cC} \cL_1).
       \end{equation}
       The Yoneda Lemma implies that the above map is an isomorphism whenever $\cL_1$ is a Yoneda module. A filtration argument then implies that the above map is an isomorphism whenever $\cL_1$ admits a filteration with associated graded modules which are quasi-isomorphic to Yoneda modules, which proves the desired result.
\end{proof}

As in Section \ref{sec:stat-main-results}, the remaining two arrows are shown to be isomorphisms by reducing the global computation to a local one: given an element $\sigma \in \Sigma$, let $\Fuk_\sigma$ denote the full subcategory of $\Fuk$ with objects $\tau \in \Sigma_\sigma$.  We introduce as in Section \ref{sec:stat-main-results} the notion of a \emph{local} $A_\infty$ bimodule, which is an $A_\infty$ bimodule $\cP$ over $\Fuk$ with the property that
\begin{equation}
  \parbox{35em}{ $\cP(\sigma, \tau)$ is acyclic whenever $P_\tau \cap P_\sigma = \emptyset$.}
\end{equation}
By the computations of Sections \ref{sec:computing-bimodule-nearby} and \ref{sec:comp-bimod-dist}, the left and right modules $\Delbar(\sigma, \_)$ and  $\Delbar(\_, \sigma)$ satisfy this assumption. We have the following variant of Lemma \ref{lem:cohomological_compute_tensor_Hom_local}:
\begin{lem}
  \label{lem:cochain_compute_tensor_Hom_local}
 For each left module $\cL$ over $\Fuk$, and each local bimodule $\cP$, the natural maps
  \begin{align} 
  \Hom_{\Fuk}(\cL,\cP)(\sigma) &\to  \Hom_{\Fuk_\sigma}(\cL, \cP)(\sigma) \\ \label{eq:quasi-iso-tensor_product_local}
  \cP  \otimes_{\Fuk_\sigma}  \cL (\sigma)& \to \cP \otimes_{\Fuk}  \cL (\sigma)
 \end{align}
 are quasi-isomorphisms.
 \end{lem}
\begin{proof}
We consider the map in Equation \eqref{eq:quasi-iso-tensor_product_local}: the length filtration of the bar complex, with associated graded complex given by the direct sum of complexes
\begin{equation}
\cP(\sigma_d, \sigma)  \otimes_\Lambda  \Fuk(\sigma_{d-1}, \sigma_d) \otimes_\Lambda  \cdots  \otimes_\Lambda  \Fuk(\sigma_{0}, \sigma_{1}) \otimes_\Lambda    \cL(\sigma_{0}).
\end{equation}
By assumption, $\cP(\sigma_d, \sigma)$  is acyclic unless $\sigma_d \in \Sigma_\sigma$. On the other hand, if $\tau \in \Sigma_\sigma$, and there is a non-trivial morphism $\rho \to \tau$ in $\Fuk$, then $\rho \in \Sigma_\sigma$. Thus the only way for the above complex to not be acyclic is if all other elements of the sequence $(\sigma_d,\ldots, \sigma_0)$ to lie in $\Sigma_\sigma$, so that the inclusion of bar complexes induces an isomorphism on the cohomology of associated graded groups. The argument for $\Hom_{\Fuk}(\cL,\cP)$ is entirely analogous. 
\end{proof}

Having reduced the main result to a local computation, we now proceed, as suggested by the results of Section \ref{sec:cohom-constr}, by comparing the global constructions to local ones. To this end, we introduce a category $\Poly_{\basepoint}$  whose objects are pairs $(q,P)$ consisting of a polytope $P$ contained in the neighbourhood $\nu_Q \basepoint$ of $P_\sigma$ fixed in Section \ref{sec:famil-cont-equat}, and a point $q \in P$, and whose morphisms are constructed using (perturbed) gradient flow lines, of Morse functions that depend on the choice of pair $(q,P)$. We construct this category in such a way that we have a strict $A_\infty$ embedding
\begin{equation}
  j \co \Fuk_{\basepoint} \to \Poly_{\basepoint},
\end{equation}
which at the level of objects assigns to $P_\tau$ the pair $(q_\tau, P_\tau)$, and at the level of morphisms and compositions is achieved by copying the choice of Morse functions (and perturbations) used in the construction of the non-zero morphisms of $\Fuk_{\basepoint}$ when constructing the corresponding morphisms and compositions arising in the image of $j$.

We then construct modules $\cL_{L,\basepoint}$ and $\scrR_{L,\basepoint}$, as well as a map of right modules
\begin{equation}
  \label{eq:local_map_right_modules}
 \scrR_{L,\basepoint} \to \Hom_{\Fuk_{\basepoint}} \left( \cL_{L,\basepoint}, \Delta_{\Poly_{\basepoint}} \right). 
\end{equation}
The comparison between local and global constructions is summarised by the following result, from Section \ref{sec:from-local-global} below:
\begin{prop}  \label{lem:global_modules_iso_to_local}
There are quasi-isomorphisms of modules over $\Fuk_{\basepoint} $
 \begin{align}
j^* \Delta_{\Poly_{\basepoint}}  & \to \Delbar \\
   j^* \cL_{L,\basepoint} & \to \cL_{L} \\
   j^* \scrR_{L,\basepoint} & \to \scrR_{L} 
 \end{align}
which fit in homotopy commutative diagrams of right $\Fuk_{\basepoint} $ modules
\begin{equation}
  \begin{tikzcd} \label{eq:duality_map_compatible_local_global}
j^* \scrR_{L,\basepoint} \arrow{r}{} \arrow{d}{} &  \Hom_{\Fuk_{\basepoint}} \left( j^*  \cL_{L,\basepoint},   j^* \Delta_{\Poly_{\basepoint}} \right) \arrow{dr}{} &  \\
 \scrR_L  \arrow{r}{}   &  \Hom_{\Fuk_{\basepoint}} \left( \cL_{L} , \Delbar  \right) \arrow{r}{} & \Hom_{\Fuk_{\basepoint}} \left( j^* \cL_{L,\sigma} , \Delbar  \right)
  \end{tikzcd}
\end{equation}
and of left $\Fuk_{\basepoint}$ modules
\begin{equation} \label{eq:product_compatible_pertubed_diagonal_cochain}
  \begin{tikzcd}
j^*\Delta_{\Poly_{\basepoint}} \otimes_{\Fuk_{\basepoint}} j^* \cL_{L,\basepoint}   \arrow{r}{}  \arrow{d}{} & j^* \cL_{L,\basepoint} \arrow{d}{} \\
 \Delbar  \otimes_{\Fuk_{\basepoint}} \cL_L  \arrow{r}{} &  \cL_{L}. 
  \end{tikzcd}
\end{equation}
\end{prop}

This reduces the computations of the global modules over the local category, to computations of the local modules. In Section \ref{sec:local-computations}, we prove the $A_\infty$ analogues of Proposition
\ref{prop:tensor_Fuk_iso_projective} and Lemma \ref{lem:compute_right_module_section}:
\begin{lem} \label{lem:local_duality_isomorphism}
The natural maps
\begin{align} \label{eq:local_duality_isomorphism}
 j^* \scrR_{L,\basepoint}(\basepoint) & \to \Hom_{\Fuk_{\basepoint}} \left( j^*  \cL_{L,\basepoint},   j^* \Delta_{\Poly_{\basepoint}}(\basepoint, \_) \right) \\
j^*\Delta_{\Poly_{\basepoint}}(\_, \basepoint) \otimes_{\Fuk_{\basepoint}} j^* \cL_{L,\basepoint}  & \to j^* \cL_{L,\basepoint}(\basepoint)
\end{align}
are quasi-isomorphisms.
\end{lem}

This is the final ingredient which we need in order to prove the main result of this paper:
\begin{proof}[Proof of Theorem \ref{thm:main_thm}]
The results of \cite{Abouzaid2014a} imply that the functor is faithful, so that it suffices to show that Equation (\ref{eq:cochain_map_CF_to_left_module}) is surjective on cohomology. Proposition \ref{prop:cochain_diagram_commutes} and Corollary  \ref{cor:reduction_to_isomorphism_Hom_tensor} reduce this to showing that the bi-module $\Delbar$ behaves like the diagonal bimodule when tensored with $\cL_L$, and that $\scrR_L$ is the space of left module maps from $\cL_L$ to $\Delbar$. Both statements are reduced by Lemma \ref{lem:cochain_compute_tensor_Hom_local} to each local category $\Fuk_{\basepoint}$, and reduced further by Proposition \ref{lem:global_modules_iso_to_local} to the corresponding statement for the local modules $\cL_{L,\basepoint}$ and $\scrR_{L,\basepoint}$, and the pullback of the diagonal of $\Poly_{\basepoint}$.  Lemma \ref{lem:local_duality_isomorphism} thus completes the argument. \end{proof}

\subsection{Global constructions}
\label{sec:directed-category}

We start by constructing the category $\Fuk$: given a sequence of objects $\Lab = \{ \sigma_i \}_{i=0}^{d} $,  the $A_\infty$ operation
\begin{equation}
    \mu^d \co   \Fuk(\sigma_{d-1}, \sigma_d)  \otimes_\Lambda  \cdots \otimes_\Lambda  \Fuk(\sigma_{0}, \sigma_{1}) \to \Fuk(\sigma_{0}, \sigma_d)   
\end{equation}
is defined by counting rigid elements of $\RTbar(\Lab)$ as follows: given a collection $x= (x_{0}; x_{1},  \ldots, x_d)  \in \Crit(\sigma_{0}, \ldots, \sigma_d)$,  we note that every element of $\RTbar(x)$ induces a map
\begin{equation}
\Hom^c_\Lambda( U^{P_{\sigma_{d-1}}}_{\sigma_{d-1},x_d}, U^{P_{\sigma_d}}_{\sigma_d, x_d}) \otimes_\Lambda  \cdots \otimes_\Lambda  \Hom^c_\Lambda( U^{P_{\sigma_{0}}}_{\sigma_{0},x_{1}}, U^{P_{\sigma_{1}}}_{\sigma_{1},x_{1}}) \to     \Hom^c_\Lambda(U^{P_{\sigma_{0}}}_{\sigma_{0},x_{0}}, U^{P_{\sigma_d}}_{\sigma_d,x_{0}}) 
\end{equation}
obtained by parallel transport along all gradient flow line components (the boundaries of all discs are constant in this case). Tensoring with the map on orientation lines induced by Equation \eqref{eq:relative_orientation_mixed}, we obtain the map
\begin{equation}
\mu_u \co    \Fuk(\sigma_{d-1}, \sigma_d)  \otimes_\Lambda  \cdots \otimes_\Lambda  \Fuk(\sigma_{0}, \sigma_{1}) \to \Fuk(\sigma_{0}, \sigma_d).
\end{equation}
We define the higher products as the finite sum
\begin{align} 
\mu^k \coloneqq \sum_{ u \in \RTbar(x)} (-1)^{\maltese} \mu_u,
\end{align}
where $\maltese = 2- d  + \sum \deg x_i$, as in \cite{Seidel2008a}. The proof that these operations satisfy the $A_\infty$ relation goes back to Fukaya and Oh's work \cite{Fukaya1997b,FukayaOh1998} (for a discussion incorporating our choices of signs and perturbations, see \cite{Abouzaid2011a}).
We recall that the sign conventions in \cite{Seidel2008a} are associated to assigning  to each generator its \emph{reduced} grading.

Given a sequence $\Lab = (L, \sigma_{1}, \ldots, \sigma_\ell) $, with $L \in \cA$ and $\sigma_i \in \Sigma$, the module structure on $\cL_L$ is defined by counting elements of $\RTbar(\Lab)$: for each $x = (x_{0};  x_{1}, \ldots, x_\ell) \in \Crit (\Lab)$ and  $u \in \RTbar(x)$,  we obtain a map
\begin{equation}
\Hom^c_\Lambda( U^{P_{\sigma_{\ell-1}}}_{\sigma_{\ell-1},x_\ell}, U^{P_{\sigma_\ell}}_{\sigma_\ell, x_\ell}) \otimes_\Lambda  \cdots \otimes_\Lambda  \Hom^c_\Lambda( U^{P_{\sigma_{1}}}_{\sigma_{1},x_2}, U^{P_{\sigma_2}}_{\sigma_2,x_2})  \otimes_\Lambda  U^{P_{\sigma_{1}}}_{\sigma_{1},x_{1}}  \to   U^{P_{\sigma_\ell}}_{\sigma_\ell,x_{0}}
\end{equation}
 by parallel transport along the boundary of the disc components as well as along all constituent gradient flow lines of $u$. Together with the map on orientation lines given by Equation \eqref{eq:relative_orientation_mixed}, a rigid element thus induces a map
\begin{align} \label{eq:left_module_map_induced_by_polygon}
\mu_u \co \Fuk(\sigma_{\ell-1}, \sigma_\ell) \otimes_\Lambda  \cdots \otimes_\Lambda  \Fuk(\sigma_{1}, \sigma_2) \otimes_\Lambda  \cL_L(\sigma_{1}) & \to \cL_L(\sigma_\ell).
\end{align}
We define the higher left module structure maps as the sum
\begin{align} \label{eq:left_module_action_formula_sum}
\mu^{\ell-1|1}_{\cL_L} \coloneqq \sum_{\substack{y \in \RTbar(x) \\ x \in \Crit  (\Lab) }} (-1)^{\maltese +1} T^{E(u)} \mu_u 
\end{align}
where $E(u)$ is the sum of the (topological) energies of the underlying strip, and the sign $\maltese +1 $ accounts for the fact that $x_{0}$ and $x_{1}$ are assigned their usual gradings, and all other generators their reduced grading.

\begin{lem} \label{lem:convergence_structure_maps_modules}
The sum in Equation \eqref{eq:left_module_action_formula_sum} converges.
\end{lem}
\begin{proof}
Fix a path connecting a basepoint on $X_{q_{\sigma_{1}}}$ with the intersection point of this Lagrangian with each section $\triv_{\sigma_i}$. This induces a norm on each Floer complex appearing in Equation \eqref{eq:left_module_map_induced_by_polygon}, and we shall prove that there are only finitely many elements $u$ of the moduli space such that valuation of $\mu_u$ is uniformly bounded by any given constant. As $\mu_u$ is a composition of maps associated to gradient flow lines and to a holomorphic strip, and the valuation of the maps associated to gradient flow lines is  bounded, the valuation of $\mu_u$ is thus bounded, up to a constant independent of $u$, by the product of the diameter of elements of the cover $\{P_\sigma\}_{\sigma \in \Sigma}$ with the length of the boundary of the strip. The latter is bounded by the reverse isoperimetric inequality as in Lemma \ref{lem:universal_constant_isoperimetric}, and the result thus follows from Gromov compactness.
\end{proof}

For subsequent use, we note the following result, which relies essentially on the assumption that the Lagrangians that we consider are graded: Proposition \ref{prop:computation_inclusion} implies that the endomorphism $A_\infty$ algebra of each object of $\Fuk$ has cohomology supported in degree $0$, and moreover than we have a natural quasi-isomorphism
\begin{equation}
H\Fuk(\sigma,\sigma)  \to \Fuk(\sigma,\sigma).
\end{equation}
This implies that, for each $A_\infty$ module over $\Fuk$, the filtration by degree of the restriction to $\sigma$ is a filtration by submodules with respect to the action of $H\Fuk(\sigma,\sigma)$. 
\begin{lem} \label{lem:compute_associated_graded_module_Floer}
For each $\sigma \in \Sigma$, and each Lagrangian $L$, the associated graded $H\Fuk(\sigma,\sigma)$-module of $ \cL_L(\sigma)$ splits as a direct sum of free modules, indexed by the intersection points of $L$ with $X_{q_\sigma}$.
\end{lem}
\begin{proof}
Since we are considering a single element $\sigma \in \Sigma$, the only holomorphic discs that appear in the computation of the module structure of $  \cL_L(\sigma)$ over $\Fuk(\sigma,\sigma) $ are constant discs; this implies that the action of $H\Fuk(\sigma,\sigma)$ respects the direct sum decomposition, and the proof that the module is free then follows from Corollary \ref{cor:left_module_action_obvious}.
\end{proof}

We now briefly outline how the remaining operations are constructed: if we consider instead a sequence  $\Lab = (L,L', \sigma_{1}, \ldots, \sigma_\ell) $, with $L,L' \in \cA$ and $\sigma_i \in \Sigma$, then the count of elements of $\RTbar(\Lab)$ induces a map
\begin{equation}
CF^*(L,L') \to \Hom \left(  \Fuk(\sigma_{\ell-1}, \sigma_\ell) \otimes_\Lambda  \cdots \otimes_\Lambda  \Fuk(\sigma_{1}, \sigma_2) \otimes_\Lambda  \cL_{L'}(\sigma_{1}),  \cL_L(\sigma_\ell) \right) .
\end{equation}
Regarding signs, we use the convention that the generators of $CF^*(L,L')$ and $\cL_L(\sigma_\ell)$ are equipped with unreduced gradings, and all other are equipped with reduced grading.

Fixing $L$ and $L'$, and considering an arbitrary sequence of elements of $\Sigma$, we obtain the map
\begin{equation}
CF^*(L,L') \to \Hom_{\Fuk} \left(  \cL_{L'},  \cL_L \right).     
\end{equation}

Letting $\Lab = (\sigma_{-r}, \ldots, \sigma_{-2} ,  L )$, the count of rigid elements of $\RTbar(\Lab)$ yields the structure map
\begin{equation}
\mu^{1|r-1}_{\scrR_L}  \co \scrR_L(\sigma_{-2})  \otimes_\Lambda  \Fuk(\sigma_{-3}, \sigma_{-2}) \otimes_\Lambda  \cdots \otimes_\Lambda  \Fuk(\sigma_{-r}, \sigma_{-r+1}) \to \scrR_L(\sigma_{-r})   
\end{equation}
for the right module $\scrR_L$. Our sign conventions are again dictated by the requirement that generators of $\scrR_L(\sigma)$ are equipped with the unreduced grading; this corresponds to the fact that the moduli space $\RTbar( x_{-1};  x_{-r+1}, \ldots, x_{-2})  $ contributes with sign
\begin{equation}
  \maltese + 1 + \sum_{i=-r+1}^{-1} \deg(x_i)   
\end{equation}
in the definition of $\mu^{1|r}_{\scrR_L} $.  Considering the case $ \Lab = (L, \sigma_{1}, \ldots, \sigma_r, L') $ yields the map
\begin{equation}
\scrR_{L'}(\sigma_r) \otimes_\Lambda  \Fuk(\sigma_{r-1}, \sigma_r) \otimes_\Lambda  \cdots \otimes_\Lambda  \Fuk(\sigma_{1}, \sigma_2) \otimes_\Lambda  \cL_L(\sigma_{1})  \to   CF^*(L',L).  
\end{equation}

The structure map $\mu_{\Delbar}^{\ell|1|r}$ of the bimodule $\Delbar$
\begin{multline}
 \Fuk(\sigma_{\ell-1},\sigma_\ell) \otimes_\Lambda  \cdots \otimes_\Lambda   \Fuk(\sigma_{0},\sigma_{1}) \otimes_\Lambda   \Delbar(\sigma_{-1}, \sigma_{0}) \otimes_\Lambda   \Fuk(\sigma_{-2},\sigma_{-1}) \otimes_\Lambda  \cdots \otimes_\Lambda  \Fuk(\sigma_{-r},\sigma_{-r+1}) \\ \to  \Delbar(\sigma_{-r}, \sigma_\ell).
\end{multline}
is obtained by the count of rigid elements of $\RTbar(\Lab)$, with  $\Lab = (\sigma_{-r} , \ldots , \sigma_{-2},\sigma_{-1}, \sigma_{0}^\phi, \sigma_{1}^\phi , \ldots, \sigma_\ell^\phi)$. If we consider instead a sequence   $\Lab = (\sigma_{-r} , \ldots , \sigma_{-2},\sigma_{-1}, L, \sigma_{0}^\phi, \sigma_{1}^\phi , \ldots, \sigma_\ell^\phi)$, we obtain a map
\begin{multline}
 \Fuk(\sigma_{\ell-1},\sigma_\ell) \otimes_\Lambda  \cdots \otimes_\Lambda   \Fuk(\sigma_{0},\sigma_{1}) \otimes_\Lambda   \cL_L(\sigma_{0}) \otimes_\Lambda  \scrR_L(\sigma_{-1}) \otimes_\Lambda   \Fuk(\sigma_{-2},\sigma_{-1}) \otimes_\Lambda  \cdots \otimes_\Lambda  \Fuk(\sigma_{-r},\sigma_{-r+1}) \\ \to  \Delbar(\sigma_{-r}, \sigma_\ell). 
\end{multline}
The maps associated to all such sequences for fixed $L$ yield the map of bimodules:
\begin{equation}
\cL_L \otimes_\Lambda  \scrR_L  \to  \Delbar.
\end{equation}

Finally, the sequence $\Lab = (L, \sigma_{-s} , \ldots , \sigma_{-2},\sigma_{-1}, \sigma_{0}^\phi, \sigma_{1}^\phi , \ldots, \sigma_\ell^\phi) $ yields the map
\begin{equation}
 \Fuk(\sigma_{\ell-1},\sigma_\ell) \otimes_\Lambda  \cdots \otimes_\Lambda   \Fuk(\sigma_{0},\sigma_{1}) \otimes_\Lambda   \Delbar(\sigma_{-1}, \sigma_{0}) \otimes_\Lambda   \Fuk(\sigma_{-2},\sigma_{-1}) \otimes_\Lambda  \cdots \otimes_\Lambda  \Fuk(\sigma_{-r},\sigma_{-r+1}) \otimes_\Lambda  \cL_L(\sigma_{-r}) \to  \cL_L(\sigma_{\ell}). 
\end{equation}
Fixing $L$, we obtain the map of left modules:
\begin{equation}
  \cL_L \otimes_\Fuk \Delbar \to \cL_L.  
\end{equation}

We now prove the commutativity of the main diagram:
\begin{proof}[Proof of Proposition \ref{prop:cochain_diagram_commutes}]
Given a sequence $\Lab = (L, \sigma_{-r} , \ldots , \sigma_{-2},\sigma_{-1}, L', \sigma_{0}^\phi, \sigma_{1}^\phi , \ldots, \sigma_\ell^\phi) $, the count of rigid elements of $\RTbar(\Lab)$ defines a map
\begin{equation}
 \Fuk(\sigma_{\ell-1},\sigma_\ell) \otimes_\Lambda  \cdots \otimes_\Lambda   \Fuk(\sigma_{0},\sigma_{1}) \otimes_\Lambda   \cL_{L'}(\sigma_{0}) \otimes_\Lambda  \scrR_{L'}(\sigma_{-1}) \otimes_\Lambda   \Fuk(\sigma_{-2},\sigma_{-1}) \otimes_\Lambda  \cdots \otimes_\Lambda  \Fuk(\sigma_{-r},\sigma_{-r+1}) \otimes_\Lambda  \cL_L(\sigma_{-r}) \to  \cL_L(\sigma_{\ell}),
\end{equation}
which we rewrite by adjunction as a map
\begin{multline}
  \scrR_{L'}(\sigma_{-1}) \otimes_\Lambda   \Fuk(\sigma_{-2},\sigma_{-1}) \otimes_\Lambda  \cdots \otimes_\Lambda  \Fuk(\sigma_{-r},\sigma_{-r+1}) \otimes_\Lambda  \cL_L(\sigma_{-r})  \\ \to \Hom\left(  \Fuk(\sigma_{\ell-1},\sigma_\ell) \otimes_\Lambda  \cdots \otimes_\Lambda   \Fuk(\sigma_{0},\sigma_{1}) \otimes_\Lambda   \cL_{L'}(\sigma_{0}), \cL_L(\sigma_{\ell}) \right). 
\end{multline}
Fixing $L$ and $L'$, and letting the sequences of elements in $\Sigma$ vary, we obtain a map
\begin{equation}
\scrR_{L'}  \otimes_{\Fuk } \cL_L \to  \Hom_{\Fuk} \left(  \cL_{L'},  \cL_L \right)      
\end{equation}
which is a homotopy between the two compositions in Diagram \ref{prop:cochain_diagram_commutes}. 
\end{proof}

\subsection{Perfectness of left modules}
\label{sec:global-computations}
 \label{sec:perf-left-modul}

The left module $\scrY_\sigma$ associated to each element $\sigma$ of $\Sigma$ is given by
\begin{equation}
  \scrY_\sigma(\tau) \coloneqq \Fuk(\sigma,\tau).  
\end{equation}
Because $\Sigma$ is a partially ordered set, and the morphisms in $\Fuk$ from $\sigma$ to $\tau$ vanish unless $\sigma$ precedes $\tau$ with respect to the partial order, we can associate to each of its objects a different left module $\cM_\sigma$ given by
\begin{equation}
  \cM_\sigma(\tau) \coloneqq
  \begin{cases}
    \Fuk(\sigma,\sigma) & \tau = \sigma \\
0 & \textrm{otherwise,}
  \end{cases}
\end{equation}
where the module structure of $\Fuk(\sigma,\sigma)$ is the obvious one. The next result asserts that these modules can be built from Yoneda modules. For the proof, we introduce the notion of a module being \emph{supported at $\sigma \in \Sigma$} if its cohomology vanishes at every other object of $\Fuk$. By definition, the module $\cM_\sigma$ is supported at $\sigma$.
\begin{lem}
There is an iterated extension of Yoneda modules which is quasi-isomorphic to $\cM_\sigma$.  
\end{lem}
\begin{proof}
The proof is by induction on the number of elements larger than $\sigma$ in the partial ordering of $\Sigma$. By definition, the maximal elements correspond to objects $\sigma$ of $\Fuk$ such that $\Fuk(\sigma,\tau)$ vanishes whenever $\tau \neq \sigma$. This establishes the Lemma in the base case of maximal elements. By induction, we therefore assume that the Lemma holds for all $\tau$ larger than a fixed element $\sigma$. 

The fact that $\Sigma$ is partially ordered provides a filtration of every left module with subquotient direct sums of modules supported on elements of $\Sigma$. Applying this to the Yoneda module $\scrY_\sigma$, the quasi isomorphism between $\Fuk(\sigma, \tau)$ and $\Fuk(\tau,\tau)$ implies that $\scrY_\sigma$ is filtered by modules quasi-isomorphic to $\cM_\tau$ for $\sigma \leq \tau$. Applying the induction hypothesis, we conclude that the Yoneda module $\scrY_\sigma$ admits a filtration so that one subquotient is quasi-isomorphic to $\cM_\sigma$, and all others are iterated extensions of Yoneda modules. This implies that $\cM_\sigma$ is itself obtained as an iterated extension of Yoneda modules, proving the result.
\end{proof}

From the above, we conclude:
\begin{proof}[Proof of Lemma \ref{lem:left_module_perfect}]
  Consider the filtration of $\cL_L$ associated to the partial order on $\Sigma$. The subquotient associated to each object $\sigma \in \Sigma$ is the complex $\cL_L(\sigma)$ as a module over $\Fuk(\sigma,\sigma)$. By Lemma \ref{lem:compute_associated_graded_module_Floer},  $\cL_L(\sigma)$ admits a filtration as a module over $\Gamma^{P_\sigma}$, with associated graded quotient that is free. Since the inclusion map $\Gamma^{P_\sigma} \to \Fuk(\sigma,\sigma)$, is a quasi-isomorphism, the conclusion is immediate.
\end{proof}

\subsection{Local constructions}
\label{sec:from-local-global}

In this section, we revisit the local constructions from Section \ref{sec:cohom-constr}, and lift them to the $A_\infty$ level. We start in Section \ref{sec:local-category-based} by constructing the category $\Poly_\sigma$ associated to an element of $\Sigma$, which admits $\Fuk_\sigma$ as an embedded subcategory (recall that this is the full subcategory of $\Fuk$ with objects $\tau \in \Sigma_\sigma$). We then briefly indicate how the constructions of Section \ref{sec:directed-category} can be adapted to produce left and right modules over $\Poly_\sigma$ associated to each Lagrangian $L \in \cA$, and the local duality map from Equation (\ref{eq:local_map_right_modules}). We then use the moduli spaces constructed in Section \ref{sec:mixed-moduli-spaces-1} to prove Proposition \ref{prop:cochain_diagram_commutes}.

\subsubsection{The local category of (based) polytopes}
\label{sec:local-category-based}

We shall give a relatively convoluted definition of the category $\Poly_\sigma$, which ensures that there is a strict $A_\infty$ embedding $\Fuk_\sigma \subset \Poly_\sigma$. The basic idea is that we can think of the structure maps of $\Poly_\sigma$ as counting either (perturbed) gradient flow segments on $X_{q_\sigma}$, glued together along choices of diffeomorphisms with specified isotopies to the identity, or as collections of (perturbed) gradient flow segments on fibres $X_q$ over different points in $\cN_\sigma$, glued together along identifications of these fibres.  The diffeomorphisms $\psi_\sigma^{q,q'}$ associated to the chosen section of $X$ over $\cN_\sigma$ will be essential in comparing these two points of view. In order to achieve an embedding of $\Fuk_\sigma$, we also choose a homotopy between the Lagrangian sections associated to each $\tau \in \Sigma_\sigma$, and that associated to $\sigma$ itself.

The objects of the category $\Poly_\sigma$ are pairs $(q,P)$, where $q \in P \subseteq \cN_\sigma$. The choice of basepoint $q$ is of no real consequence; the objects corresponding to all such choices will be quasi-isomorphic, and the choice is only included to give us enough flexibility to produce the desired strict $A_\infty$ embedding.

For each pair $\Lab = ((q_{0},P_{0}), (q_{1},P_{1}))$, we choose a fibre $q_\Lab$ in $\cN_\sigma$, a metric $g_\Lab$ on $X_{q_\Lab}$, and a Morse--Smale function $f_\Lab$ on $X_{q_\Lab}$. We assume that these choices are subject to the following constraint:
\begin{equation} \label{eq:condition_extend_Morse_data_Fuk_to_Poly}
\parbox{35em}{If $(q_{0},q_{1})  = (q_{\sigma_{0}}, q_{\sigma_{1}})$, the data $(q_\Lab, g_\Lab, f_\Lab)$ agree with the data used to define the Floer complex $ CF^*((\sigma_{0},P_{0}), (\sigma_{1},P_{1}))$ in Equation (\ref{eq:definition_Floer_complex_polytopes}).}
\end{equation}
To be more precise, since the cover $\Sigma$ that we are considering arises from the construction of Section \ref{sec:famil-cont-equat}, the data that we consider in this situation agree with the choice of Morse data specified in Section \ref{sec:morse-data}.
\begin{rem}
 We implicitly assume that, whenever $\tau \neq \rho$, the basepoints $q_\rho$ and $q_\tau$ are distinct. Of course, this condition can be easily achieved by generic choices, but more importantly, we can always enlarge the category $\Poly_\sigma$ so that we can associate distinct objects of $\Poly_\sigma$ to distinct elements $\Sigma_\sigma$, as is required in order for Condition \eqref{eq:condition_extend_Morse_data_Fuk_to_Poly} to hold.
\end{rem}
The morphisms are then given by the Morse complexes
\begin{equation}
  \Poly_{\sigma}( \Lab) \coloneqq CM^{*}\left(X_{q_{\Lab}}, \Hom^c_\Lambda(U^{P_{0}}_\sigma, U^{P_{1}}_\sigma ) \otimes \lambda \right). 
\end{equation}

To define the compositions, we pick, for each sequence $\Lab$ of objects of $\Poly_\sigma$, metric tree $T$ labelled by $\Lab$, and discs with marked points $S_v \in \Rbar_{\Lab_v}$ for each vertex, the following data: (i) perturbed gradient equations on $X_{q_e} \coloneqq X_{q_{\Lab_e}}$ for each edge of $T$, (ii) a point $q_v \in \cN_\sigma$ for each vertex of $T$, and (iii) a diffeomorphism
\begin{equation}
 \psi^{f}  \co X_{q_e} \to X_{q_v}
\end{equation}
for each flag $f = (v \in e)$, with a specified homotopy to $\psi_\sigma^{q_e,q_v}$. We again require the perturbations for sequences of objects of $\Poly_\sigma$ which are the images of totally ordered subset of $\Sigma_\sigma$ to agree with the choices made in Section \ref{sec:abstr-moduli-spac}.
\begin{rem}
The choice of discs associated to vertices is completely unnecessary, and only appears here to make the construction exactly the same as that for the construction of the $A_\infty$ structure on $\Fuk$.
\end{rem}

Because the path space is contractible, there is no obstruction to making such choices in families, compatibly with degenerations of discs and breaking of edges. In addition, to ensure compatibility with the construction of the category $\Fuk_\sigma$, we require that:
\begin{equation}
\parbox{35em}{If $\Lab = ((q_{\sigma_{0}}, P_{0}), \ldots, (q_{\sigma_n}, P_n))$, with $\sigma_i \in \Sigma_\sigma$ such that $\sigma_{0} \leq \cdots \leq \sigma_n$, the choice for any tree labelled by $\Lab$ agrees with that in Sections \ref{sec:morse-moduli-spaces-1} and \ref{sec:mixed-moduli-spaces-1}.}
\end{equation}
The specified homotopy for sequences coming from objects of $\Sigma$ arises from the isotopy, fixed above, between the Lagrangian sections associated to elements of $\Sigma_\sigma$, and the section for $\sigma$ itself.

It is now entirely straightforward to see that the counts of such rigid configurations equip $\Poly_\sigma$ with the structure of an $A_\infty$ category. We obtain an embedding of $\Fuk_\sigma$ as a subcategory of $\Poly_\sigma$, corresponding to the objects $(q_\sigma, P_\sigma)$, by using for each pair $\Lab = (\tau_{0}, \tau_{1})$ of objects of $\Fuk_\sigma$, the isomorphism
\begin{equation}
    CM^{*}\left(X_{q_\Lab}, \Hom^c_\Lambda(U^{P_{\tau_{0}}}_{\tau_{0}}, U^{P_{\tau_{1}}}_{\tau_{1}} ) \otimes \lambda \right) \cong CM^{*}\left(X_{q_{\Lab}}, \Hom^c_\Lambda(U^{P_{\tau_{0}}}_\sigma, U^{P_{\tau_{1}}}_\sigma ) \otimes \lambda \right)
\end{equation}
induced by the choice of isotopy of sections associated to $\sigma$ and $\tau_i$.
\subsubsection{Modules and structure maps}
\label{sec:modul-struct-maps}

Given $L \in \cA$, we have already fixed a Hamiltonian isotopic Lagrangian $L_{\sigma}$ which is transverse to $X_{q_\sigma}$. Given an element $(q,P)$ of $\Poly_\sigma$, we define
\begin{align}
  \cL_{L}(q, P) & \coloneqq CF^*(L,(\sigma,P)) \\
 \scrR_{L}(q, P) & \coloneqq CF^*((\sigma,P),L)
\end{align}
as in Section \ref{sec:left-right-modules-1}. In other words, all the groups are defined using holomorphic strips with Lagrangian boundary conditions $(L, X_{q_\sigma})$.

Let $\Lab$ be a sequence of the form
\begin{align} \label{eq:left_module_local_sequence}
& (L,  (q_{0},P_{0}), \ldots , (q_\ell, P_\ell) ) \\ \label{eq:right_module_local_sequence}
& ( (q_{-r}, P_{-r}),  \ldots,  (q_{-1}, P_{-1})  , L) \\  \label{eq:bimodule_map_local_sequence}
&  (q_{-r}, P_{-r}),  \ldots,  (q_{-1}, P_{-1}) , L, (q_{0},P_{0}), \ldots , (q_\ell, P_\ell)  ).
\end{align}
As in Section \ref{sec:perturbed-equations}, let $\Lab^\sigma$ denote the sequence obtained from $\Lab$ by replacing all objects of $\Poly_\sigma$ by the element $\sigma$ of $\Sigma$. We have a corresponding moduli space $\Rbar(\Lab^\sigma)$ of holomorphic discs boundary conditions $L$ and $X_{q_\sigma}$,  with two ends (adjacent to the segment labelled $L$), equipped with a collection of boundary marked points.

As in Section \ref{sec:morse-floer-edges}, we consider a moduli space $\RTbar^\sigma_\Lab$ of discs and metric trees, which has a top-dimensional stratum  of the form $\Rbar^T_{\Lab^\sigma} \times \prod_{e \in E^{\Mo}(T)} \Tbar_e$ for each tree $T$ with label $\Lab$. We then proceed to inductively choose Morse data on all edges of the tree $T$, compatibly with the choice of Morse data on the trees which define the $A_\infty$ structure. As in Equation \eqref{eq:fibre_product_treed_discs}, we obtain a mixed moduli space $\RTbar^\sigma_T(\Lab)$ as a fibre product of moduli spaces $\Rbar(\Lab^\sigma)$ and the moduli spaces of (perturbed) gradient trajectories over the common evaluation maps to Lagrangian torus fibres. For sequences $\Lab$ of the form
\begin{align} 
& (L,  (q_{\sigma_{0}},P_{\sigma_{0}}), \ldots , (q_{\sigma_\ell}, P_{\sigma_\ell}) ) \\
& ( (q_{\sigma_{-r}}, P_{\sigma_{-r}}),  \ldots,  (q_{\sigma_{-1}}, P_{\sigma_{-1}})  , L) \\ 
&  (q_{\sigma_{-r}}, P_{\sigma_{-r}}),  \ldots,  (q_{\sigma_{-1}}, P_{\sigma_{-1}}) , L, (q_{\sigma_{0}},P_{\sigma_{0}}), \ldots , (q_{\sigma_\ell}, P_{\sigma_\ell})  ),
\end{align}
with $\sigma_{0} \leq \ldots \leq  \sigma_\ell$, and $\sigma_{-r} \leq \ldots \leq  \sigma_{-1}$ ordered subsets of $\Sigma$, we assume that the choices agree with those made in Section \ref{sec:famil-cont-equat}.

The count of rigid elements of the moduli space associated to Equation (\ref{eq:left_module_local_sequence}) defines a map
\begin{equation}
  \Poly_\sigma((q_{\ell-1},P_{\ell-1}), (q_\ell,P_\ell)) \otimes_\Lambda  \cdots \otimes_\Lambda  \Poly_\sigma((q_{0},P_{0}), (q_{1},P_{1})) \otimes_\Lambda  \cL_{L,\sigma}(q_{0} ,P_{0}) \to \cL_{L,\sigma}(q_\ell, P_\ell),
\end{equation}
which makes $\cL_{L,\sigma}$ into a left module over $\Poly_\sigma$. We similarly obtain the structure of a right module on $\scrR_{L,\sigma}$ by considering sequences as in Equation (\ref{eq:right_module_local_sequence}).

We note that these modules allow us to formulate a generalisation of Lemma \ref{lem:compute_associated_graded_module_Floer}: namely, consider the cohomoloogical category $H\Fuk_{\sigma}$, with objects given by elements $\tau \in \Sigma_\sigma$. The morphism spaces in this category were computed in Equation \eqref{eq:morphisms_cohomological_directed} to be given by affinoid rings, and in particular to be supported in degree $0$. Since the morphism spaces in $\Fuk$ are supported in non-negative degree, we have a natural quasi-isomorphism $H\Fuk_\sigma \to \Fuk_\sigma$, so that we can pull back $\cL_{L,\sigma}$ to a module over $ H\Fuk_\sigma$. With this in mind, the next result follows from Corollary \ref{cor:left_module_action_obvious}: 
\begin{lem} \label{lem:compute_associated_module_structure_local}
The associated graded $H\Fuk_\sigma$-module of the filtration of $\cL_{L,\sigma}$ by degree is quasi-isomorphic to the direct sum, over all intersection points of $L$ with $X_{q_\sigma}$, of shifts of the module $\tau \mapsto \Gamma^{P_\tau}$. \qed
\end{lem}

Finally, the sequence in Equation (\ref{eq:bimodule_map_local_sequence}) gives rise to a map
\begin{multline}
\Poly_\sigma((q_{\ell-1},P_{\ell-1}), (q_\ell,P_\ell)) \otimes_\Lambda  \cdots \otimes_\Lambda  \Poly_\sigma((q_{0},P_{0}), (q_{1},P_{1}))   \otimes_\Lambda  \cL_{L,\sigma}(q_{0},P_{0}) \otimes_\Lambda   \scrR_{L,\sigma}(q_{-1}, P_{-1}) \\ \otimes_\Lambda  \Poly_\sigma((q_{-2},P_{-2}), (q_{-1},P_{-1})) \otimes_\Lambda  \cdots \otimes_\Lambda  \Poly_\sigma((q_{-r},P_{-r}),(q_{-r+1}, P_{-r+1}))  \to \Poly_\sigma((q_{-r},P_{-r}), (q_\ell,P_\ell)).
\end{multline}
Fixing $L$, and letting $r$ and $\ell$ vary, we obtain the map of bimodules
\begin{equation}
    \cL_{L,\sigma} \otimes_\Lambda   \scrR_{L,\sigma} \to \Delta_{\Poly_\sigma} .
\end{equation}

\subsubsection{The global-to-local comparison}
\label{sec:glob-local-comp}

The purpose of this section is to complete the reduction of the global computations to local ones. The argument is an entirely straightforward use of the moduli spaces $\RTbar(\uLab)$ from Section \ref{sec:mixed-moduli-spaces-1}.
\begin{proof}[Proof of  Proposition \ref{lem:global_modules_iso_to_local}]
Fix an element $\sigma \in \Sigma$. We first begin by considering sequences $\Lab$ of the form $(L, \sigma_{0}, \ldots, \sigma_\ell)$ or $(\sigma_{-r}, \ldots, \sigma_{-1}, L)$, where $\sigma_i \in \Sigma_\sigma$. The counts of rigid elements of the corresponding moduli spaces define maps
\begin{align}
\Fuk(\sigma_{\ell-1},\sigma_\ell) \otimes_\Lambda  \cdots \otimes_\Lambda   \Fuk(\sigma_{0},\sigma_{1}) \otimes_\Lambda   \cL_{L}(q_{\sigma_{0}}, P_{\sigma_{0}}) & \to \cL_{L}(\sigma_\ell) \\
 \scrR_{L}(q_{\sigma_{-1}}, P_{\sigma_{-1}}) \otimes_\Lambda   \Fuk(\sigma_{-2},\sigma_{-1}) \otimes_\Lambda  \cdots \otimes_\Lambda  \Fuk(\sigma_{-r},\sigma_{-r+1}) & \to  \scrR_{L}(\sigma_{-r})
\end{align}
which are the structures maps of left and right module homomorphisms $j^* \cL_{L} \to \cL_L$ and $j^* \scrR_{L} \to \scrR_L$. Here, we use the identification 
\begin{equation}
  \Fuk(\sigma_i, \sigma_j) \cong \Poly_\sigma((q_{\sigma_i}, P_{\sigma_i}), (q_{\sigma_j}, P_{\sigma_j}))
\end{equation}
giving rise to the $A_\infty$ embedding noted at the end of Section \ref{sec:local-category-based}.

Considering instead a sequence $\Lab = (\sigma_{-r} , \ldots , \sigma_{-2},\sigma_{-1}, \sigma_{0}^\phi, \sigma_{1}^\phi , \ldots, \sigma_\ell^\phi)$, we obtain a map
\begin{multline}
 \Fuk(\sigma_{\ell-1},\sigma_\ell) \otimes_\Lambda  \cdots \otimes_\Lambda   \Fuk(\sigma_{0},\sigma_{1})  \otimes_\Lambda   \Poly_\sigma((q_{\sigma_{-1}}, P_{\sigma_{-1}}), (q_{\sigma_{0}}, P_{\sigma_{0}})) \\
\otimes_\Lambda  \Fuk(\sigma_{-2},\sigma_{-1}) \otimes_\Lambda  \cdots \otimes_\Lambda  \Fuk(\sigma_{-r},\sigma_{-r+1}) \to \Delbar(\sigma_{-r}, \sigma_\ell),
\end{multline}
which defines the map of bimodules $j^* \Delta_{\Poly_\sigma} \to \Delbar$.

The homotopies in Diagrams (\ref{eq:duality_map_compatible_local_global}) and (\ref{eq:product_compatible_pertubed_diagonal_cochain}) follow in exactly the same way by counting rigid elements of $\RTbar(\uLab)$ for sequences:
\begin{align}
 &  (L,  \sigma_{-r}, \ldots,  \sigma_{-1}, \sigma_{0}^\phi,\sigma_{1}^\phi, \ldots, \sigma_\ell^\phi)   \\  \label{eq:sequence_Hom_comparison}
&  (\sigma_{-r}, \ldots,  \sigma_{-1}, L, \sigma_{0}^\phi, \sigma_{1}^\phi, \ldots, \sigma_\ell^\phi).
\end{align}
\end{proof}

\subsection{Local computations}
\label{sec:local-computations}
We conclude the main part of the paper by performing to the necessary local computations:

\begin{proof}[Proof of Lemma \ref{lem:local_duality_isomorphism}]
  The two arguments are entirely analogous; we explain the proof of the duality isomorphism in Equation \eqref{eq:local_duality_isomorphism}.  First, using the quasi-isomorphism $ H \Fuk_\sigma \to \Fuk_{\basepoint}$, we reduce the problem so showing that the map
  \begin{equation}
         j^* \scrR_{L,\basepoint}(\basepoint) \to \Hom_{H \Fuk_{\basepoint}} \left( j^*  \cL_{L,\basepoint},   j^* \Delta_{\Poly_{\basepoint}}(\basepoint, \_) \right)
  \end{equation}
is a quasi-isomorphism.  By Lemma \ref{lem:compute_associated_module_structure_local}, the decomposition of the intersection points of $L$ with $X_{q_\sigma}$ by degree provides a filtration of $ j^* \cL_{L,\sigma} $ as a left module over $H \Fuk_\sigma $, hence of $\Hom_{H \Fuk_\sigma}\left( j^* \cL_{L,\sigma},   j^* \Delta_{\Poly_\sigma}( \sigma, \_)  \right) $ as a cochain complex.  It also provides a filtration of $ j^* \scrR_{L,\sigma}(\sigma)   $ as a cochain complex. It suffices to prove that the map at the level of each associated graded group is a quasi-isomorphism. Since each associated graded summand of $ j^* \cL_{L,\sigma}  $ is the module which assigns to $\tau$ the affinoid ring $\Gamma^{P_\tau}$,  %
  the complex $\Hom_{H \Fuk_\sigma}\left( j^* \cL_{L,\sigma},   j^* \Delta_{\Poly_\sigma}( \sigma, \_)  \right) $ is thus quasi-isomorphic to the \v{C}ech complex
\begin{equation}
 \left( \bigoplus_{\tau_{0} < \cdots < \tau_k} \Hom^c_\Lambda(\Gamma^{P_\sigma}, \Gamma^{\tau_k})[-k] , \check{\delta} \right)  .
\end{equation}
As in Section \ref{sec:comp-right-module}, we may replace $P_{\tau_k}$ in the above expression by $P_{\tau_k} \cap P$ for a polytope $P$ containing $P_\sigma$ in its interior, and covered by the elements of $\Sigma_\sigma$. By Tate acyclicity, we conclude that each associated graded group of $\Hom_{H \Fuk_\sigma}\left( j^* \cL_{L,\sigma},   j^* \Delta_{\Poly_\sigma}( \sigma, \_)  \right) $  is quasi-isomorphic to the complex $ CF^*(P, P^{\sigma})$ whose cohomology is quasi-isomorphic to the linear dual of $\Gamma^{P_\sigma}$ by Proposition \ref{prop:computation_reverse_inclusion}.
\end{proof}

\appendix

\section{Reverse isoperimetric inequalities}
\label{sec:geometric-setup}

The purpose of this appendix is to extend the reverse isoperimetic inequalities established by Groman-Solomon \cite{GromanSolomon2014} and Duval \cite{Duval2016} for holomorphic curves with Lagrangian boundary conditions, to the case of moving Lagrangian boundary conditions. Since our main result requires a rather long list of technical hypotheses, we postpone its statement in the form of Proposition \ref{prop:main_result_family_reverse_isoperimetric} to the end, and develop the basic ideas one step at a time.

\subsection{The basic inequality}
\label{sec:basic-inequality}
Let $X$ be a closed symplectic manifold, $S$ a Riemann surface obtained from a closed Riemann surface with boundary by removing finitely many boundary punctures. Let $K \subset X$ be a closed codimension $0$ submanifold with boundary, and assume that we have a labelling 
\begin{equation}
  \Lab_S \co  \pi_{0}(\partial S) \to \{K \} \cup \cL
\end{equation}
of  $\partial S$ by $\{K\} \cup  \cL$, where $\cL$ is a collection of Lagrangians in $X$. Given $i \in \pi_{0}(\partial S)$, we write $\partial_i S$ for the corresponding component of the boundary. We write $\partial_K S$ for the union of components labelled by $K$, and $L_i$ for the Lagrangian labelled by a component $i$. We assume that all intersections among Lagrangians appearing in $\cL$ are contained in the interior of $K$, and choose another closed subset $K'$ containing this intersection, and contained in the interior of $K$:
 \begin{equation}
 K' \Subset K \subset X.
\end{equation}
In addition, we fix a Riemannian metric $g_X$ with respect to which we shall compute all norms.

Given an almost complex structure $J$, consider a family
\begin{equation}
  J_S \co S \to \scrJ  
\end{equation}
of almost complex structures which agree with $J$ away from $K'$. There is an associated moduli space 
\begin{equation}
\scrM(\Lab_S ,J_S)
\end{equation}
of $J_S$-holomorphic curves with boundary conditions given by $\Lab_S$ in the sense that a point $z \in \partial_K S$ maps to $K'$, while a point $z \in \partial_i S$ on a component labelled by a Lagrangian maps to $L_i$.

Recall that the geometric energy $E^{geo}(u)$ of any element $u$ of the moduli space is the area
\begin{equation}
  E^{geo}(u) = \int \| du\|^2.
\end{equation}
Moreover, given a component $i$ of $\partial S$, we define
\begin{equation}
   \ell_{K}  \left( \partial _i u \right)
\end{equation}
to be the length, with respect to $g_X$, of the collection of curves obtained removing $u^{-1} K$ from $\partial_i S$.

\begin{lem} \label{lem:isoperimetric_constant_lagrangian}
For each choice of labels $\Lab_S$, there exists a constant $C$, depending on $J$ but not on the family $J_S$, such that for each element $u \in \scrM(\Lab_S,J_S)$ and each component $i \in \pi_{0}(\partial S)$ which is labelled by a Lagrangian, we have
\begin{equation}
   \ell_{K}   \left(  \partial_i u \right) \leq C  E^{geo}(u).
\end{equation}
This constant is independent of the restriction of $L_i$ to $K$.
\end{lem}
\begin{proof}
  This is a minor modification of the results of Groman-Solomon \cite[Theorem 1.1]{GromanSolomon2014} and Duval \cite{Duval2016}. We briefly indicate how to adapt Duval's proof:

  We choose a tubular neighbourhood  $\nu_X L_i$ of $L_i$, which is equipped with a non-negative weakly plurisubharmonic function $\rho_i$ (with respect to $J$) that vanishes on the intersection with $K' \cup L_i$, and does not vanish away from the intersection with $K \cup L_i$. Moreover, away from $K$,   we require that $\rho_i$  be (strictly) plurisubharmonic with weakly plurisubharmonic square root.  Near $\partial K' \cap L_i$, such a function  can be locally modelled after the function
\begin{equation}
  \chi(x_{1}) \left(|y_{1}|^2 + \cdots + |y_n|^2\right) 
\end{equation}
for the Lagrangian $[0,\infty) \times \bR^{n-1} \subset \bC^n$, where $\chi$ is a smooth function vanishing for $x_{1} \leq 0$. For simplicity, we assume that $g_X$ everywhere dominates the (semi)-metric induced by $\rho$ and $J$.

Let $\ell_{\rho_i}(\partial u)$ denote the length of $\partial u$ with respect to the (semi)-metric induced by $d d^c{\rho_i}$, and $E_{{\rho_i}}$ the integral of $d d^c {\rho_i}$ over the part of $u$ with image in $\nu_X L_i$. There is a constant $C_{0}$ such that
\begin{equation}
 \ell_{{\rho_i}}(\partial u) \leq C_{0} E_{\rho_i}(u)  . 
\end{equation}
The basic idea is that the function $\frac{E_{{\rho_i}}(r,u)}{r^2}$ which measures the area of $u$ in the domain ${\rho_i} \leq r$ is monotonic, and by Fubini's theorem, has limit bounded above by a constant multiple $\ell_{{\rho_i}}(\partial u)$ (see \cite{Duval2016} for  details); the fact that ${\rho_i}$ vanishes to order greater than $1$ along $K'$ implies that this part of boundary does not contribute to $\ell_{{\rho_i}}(\partial u)$. 

Away from $K$, the metric induced by $d d^c{\rho_i}$ is uniformly comparable to $g_X$, so we may assume the existence of a constant $C_{1}$ such that
\begin{equation}
    \ell_{K} \partial_i u \leq C_{1} \ell_{{\rho_i}}(\partial u)
\end{equation}
On the other hand, the assumption that $g_X$ dominates $d d^c{\rho_i}$ implies that
\begin{equation}
   E_{{\rho_i}}(u) \leq E^{geo}(u). 
\end{equation}
The result follows from combining these inequalities.
\end{proof}

For applications, it will be convenient to state the estimate in a different way: define $L_i /{\sim}$ to be the quotient of $L_i$ by the relation which identifies points in the same component of $K \cap L_i$.  Given a length metric on $L_i /{\sim}$ (see the discussion preceding Lemma \ref{lem:good-metric-quotient}), we have:
\begin{cor} \label{cor:isoperimetric_constant_lagrangian_quotient}
There exists a constant $C$ such that, for each element $u \in \scrM(\Lab_S,J_S)$, we have
\begin{equation}
   \ell \left(  \partial_i u /{\sim} \right) \leq C  E^{geo}(u).
\end{equation} \qed
\end{cor}

\subsection{Moving Lagrangian boundary conditions}
\label{sec:moving-lagr-bound}

It will be necessary to have a generalisation for moving boundary conditions, and families of almost complex structures. Let $S$ be a Riemann surface equipped with a collection of subsets $\Theta \subset S$, which we call the components of the thin part,  that include neighbourhoods of all punctures. By abuse of notation, we will refer to this as a thick-thin decomposition, even though the thin part that we consider is much larger than the intrinsic one coming from hyperbolic geometry.

Let $\Lab_S$ be moving boundary conditions on $S$, i.e. a smooth assignment of a subset of $X$ to each point on $\partial S$, which we assume is locally constant on the thin-part. We shall only consider the situation in which the restriction to each component is either a moving family of Lagrangians, or is a constant family given by a compact subset $K \subset X$ as in the previous section. We assume that any component labelled by $K$ is contained in a component of the thin part.

For each component $i \in \pi_{0}(\partial S)$ labelled by a moving family of Lagrangians, we denote the graph by  $\tilde{L}_i \subset \partial_i S \times X$. By construction, this map is independent of the first factor in the thin part, so that we obtain a Lagrangian $L_{i,\Theta}$ in $X$ for each component $\Theta$ of the thin part.

Assume that we are given, for each such component, an almost complex structure $J_\Theta$  on $X$ as well as a pair of codimension $0$ manifolds with boundary 
\begin{equation}
 \nu'_X \Theta \subseteq \nu_X \Theta \subseteq X
\end{equation}
which contain the intersection of any pair of Lagrangians labelled by boundary components of $S$ which meet $\Theta$. In addition, we assume that $K \subset \nu_X \Theta$  if $\Theta$ intersects a component labelled by the subset $K$. We pick a family $J_S$ of almost complex structures on $X$ parametrised by $S$, such that
\begin{equation}
  \label{eq:condition_J_S_fixed_ends}
\parbox{35em}{the restriction of $J_S$ to the product of a component $\Theta$ of the thin part with $  X \setminus  \nu'_\Theta X$ agrees with $J_\Theta$.}
\end{equation}

Given this data, we shall consider the moduli space 
\begin{equation} \label{eq:moduli_space_moving_boundary_conditions}
\scrM(\Lab_S,J_S)
\end{equation}
of $J_S$ holomorphic curves with boundary conditions $\Lab_S$: it is convenient to introduce the almost complex structure $\tilde{J}_S$ on $S \times X$ induced by $J_S$, and describe this moduli space as the space of $\tilde{J}_S$ holomorphic sections with boundary conditions given by $\tilde{L}_i $ over $\partial_i S  $.

Our goal is to prove a reverse isoperimetric inequality for these moduli spaces. The basic idea is that Lemma \ref{lem:isoperimetric_constant_lagrangian} provides an estimate for the moduli spaces holomorphic strips corresponding to each end; we shall extend this estimate to $S$, at the cost of a possibly weaker proportionality constant, as well as the addition of a constant term. The main point is to provide a proof that extends to families of (broken) holomorphic curves.

We shall compare the geometric energy of each curve $u$
\begin{equation} \label{eq:geometric_energy_moving}
  E^{geo}(u) = \int |du|^2 = \int u^* \omega
\end{equation}
to the length of the boundary, which we formulate as follows:
we denote by $\tilde{L}_i/{\sim}$ the quotient of $\tilde{L}_i$ by the equivalence relation which
\begin{equation}
  \label{eq:collapsing_boundary_moving_Lagrangians}
  \parbox{33em}{
(i) collapses each component of the inverse image of $\Theta$ in $\tilde{L}_i$ to a single fibre,  and (ii) identifies points in the same component of the intersection of $\tilde{L}_i $ with $\nu_X \Theta$. }
\end{equation}
 In particular, $ \tilde{L}_i/{\sim}$ maps to the quotient of $\partial_i S$ by the components of the intersection with the thin part, agrees with $\tilde{L}_i$ away from the inverse image of these components, and agrees with the quotient considered in the previous section over each component of the thin part.

\begin{lem} \label{lem:basic_isoperimetric_moving_L}
There exists a constant $C$ 
such that, for each $u \in \scrM(\Lab_S, J_S) $ and for each component $i$ of $\partial S$, we have
\begin{equation}
     \ell( \partial_i u / {\sim}) \leq C   E^{geo}(u)   + \textrm{ a constant independent of }u.
\end{equation}
The constant depends only on the restriction of the Lagrangian boundary conditions and the almost complex structures to a neighbourhood of $\partial_i S$, but is independent of their restriction to $\Theta \times \nu'_X \Theta $ for each component $\Theta$ of the thin part.
\end{lem}
\begin{proof}
The graph $\tilde{L}_i$ is a totally real submanifold with respect to the almost complex structure $\tilde{J}_S$ which at every point in $z \in S$ is the product of the complex structure on $S$ with the value of the family $J_S$ at $x$. This almost complex structure is compatible with the symplectic form
  \begin{equation}
    \omega_X  + \omega_S,    
  \end{equation}
where  $\omega_S$ is a symplectic form on $S$ of total area $1$.  If we define $E(\tilde u)$ to be the area with respect to such a symplectic form,  we have
\begin{equation} \label{eq:bound_energy_by_section}
  E(\tilde u) = E^{geo}(u) + 1,
\end{equation}
which will be one origin for the constant term in the statement of the Lemma. The remainder of the proof proceeds in essentially the same way as that of Lemma \ref{lem:isoperimetric_constant_lagrangian}:

For each component $\Theta$ of the thin part which is adjacent to $i$, equipped with the family of almost complex structures $J_\Theta$, fix the neighbourhood $\nu_X L_{i,\Theta}$ and the function  $\rho_{i,\Theta} \co \nu_X L_{i,\Theta} \to \bR$ considered in the proof of Lemma \ref{lem:isoperimetric_constant_lagrangian}.

By projection to the second factor, we obtain a function on $\Theta \times X$  which we denote $\rho_{i,\Theta} $. We consider a neighbourhood $\nu_{S \times X} \tilde{L}_i$ of $\tilde{L}_i$ which contains the product $\Theta \times L_{i,\Theta}$ for all components $\Theta$ of the thin part which meet $\tilde{L}_i$, and is equipped with a weakly plurisubharmonic function
\begin{equation}
\rho_i \co  \nu_{S \times X} \tilde{L}_i \to [0,\infty)
\end{equation}
such that the following properties hold
\begin{enumerate}
\item Over $\Theta \subset S$, $\rho_i$ agrees with  $\rho_{i,\Theta}$.
\item Away from $ \Theta \times \nu_X \Theta $, the function $\rho_i$ is strictly plurisubharmonic, and only vanishes on $\tilde{L}_i$.
\item The square root of $\rho_i$ is everywhere weakly plurisubharmonic.
\end{enumerate}
The existence of such a function follows from the usual patching argument for plurisubharmonic functions  as in \cite{Duval2016}. For simplicity, we also assume that there is a finite diameter metric on $S$ whose product with $g_X$ everywhere dominates the metric induced by $\rho_i$.

Given a section $\tilde{u}$ corresponding to an element of $\scrM(\Lab_S, J_S) $,  let $\ell_{\rho_i}(\partial_i \tilde{u})$ denote the length of $\partial_i \tilde{u}$ with respect to the (semi)-metric induced by $d d^c\rho_i$, and $E_{\rho_i}$ the integral of $d d^c \rho_i$ over the part of $\tilde{u}$ with image in  $\nu_{S \times X} \tilde{L}_i$. As before, there is a constant $C_{0}$ such that
\begin{equation} \label{eq:duval_isoperimetric_rho}
 \ell_{\rho_i}(\partial_i \tilde{u}) \leq C_{0}  E_{\rho_i}(u).
\end{equation}
Away from a small neighbourhood of the inverse image of $ \Theta$, the metric induced by $d d^c\rho_i$ is uniformly comparable to the product of $g_X$ with a metric on $S$. Because $\rho_i$ is non-degenerate in the fibre direction over $\Theta$,  we therefore obtain a constant $C_{1}$ such that
\begin{equation} \label{eq:iso_perimetric_rho}
    \ell( \partial_i u/ {\sim}) \leq C_{1} \ell_{\rho_i}(\partial_i  \tilde{u})  + \textrm{ a constant independent of }u,
\end{equation}
where the constant  term accounts for the fact that the derivatives of $\rho_i$ in the base direction vanish identically on $\Theta_i$.

On the other hand, the assumption that $g_X$ and the metric on $S$ dominate $d d^c\rho_i$ implies that
\begin{equation}
   E_{\rho_i}(u) \leq E(\tilde{u}). 
\end{equation}
The result follows from combining this inequality with Equations \eqref{eq:bound_energy_by_section}, \eqref{eq:duval_isoperimetric_rho}, and \eqref{eq:iso_perimetric_rho}.
\end{proof}

\subsection{Compatibility with gluing}
\label{sec:comp-with-gluing}

We shall require a uniform estimate for families of Riemann surfaces. Let us therefore consider a rooted ribbon tree $T$, and a collection of Riemann surfaces with punctures $S_v$ for each vertex of $T$, with an identification of the ends of $S_v$ with the edges adjacent to $v$. We write $E^{\fin}(T)$ for the edges of $T$ which are adjacent to two vertices. Given gluing parameters
\begin{equation} \label{eq:gluing_map_interior_edges}
R \co  E^{\fin}(T) \to [0,\infty]  
\end{equation}
and a choice of disjoint strip-like ends for each end of $S_v$, we obtain a Riemann surface $S_{T,R}$ by gluing, which is the image of a surjective map with domain the disjoint union of the surfaces $S_v$. Assuming that each $S_v$ is equipped with a thick-thin decomposition, we obtain a thick-thin decomposition of $S_{T,R}$. We allow the possibility that the decomposition of a component $S_v$ be degenerate, in the sense that the thin part consists of the entire surface.

The ribbon structure determines an embedding $T \subset \bR^2$ up to isotopy;  assume that we have moving boundary conditions $\Lab_{S_v}$, which are locally constant on all components of the thin part, and are consistent for adjacent vertices, in the sense that the subsets of $X$ assigned to the two ends corresponding to each interior edge of $T$ agree. We write $\Lab_{S_{T,R}}$ for the moving boundary conditions on $S_{T,R}$ obtained by gluing. As in the previous section, we pick, for each component $\Theta$ of the thin part of a curve $S_v$, closed subsets $\nu'_X \Theta \subseteq \nu_X \Theta$ of $X$, which contains the intersections of all labels, and an almost complex structure $J_\Theta$ on $X$; we assume that these choices are the same for the thin parts meeting the two  punctures corresponding to a given edge in $T$.

Finally, we pick families of almost complex structures $J_{S_v}$ which agree with $J_\Theta$ on $\Theta \times X \setminus \nu'_X \Theta $, and are also compatible across the edges of $T$.  Let  $J_{S_{T,R}}$ be a family of almost complex structure on $S_{T,R}$ which
\begin{equation}
  \label{eq:almost_complex_structure_obtained_by_gluing}
  \parbox{33em}{agrees with the family obtained by gluing (i) in the thick part near each boundary stratum, and (ii) in the thin part away from  $\Theta \times \nu'_X \Theta$ for each component $\Theta$ of the thin part.}
\end{equation}

Given this data, we shall consider the moduli space 
\begin{equation}
\scrM(\Lab_{S_{T,R}} ,J_{S_{T,R}})
\end{equation}
which is given as in the previous section if all gluing parameters are finite, and is given for infinite gluing parameters by the fibre products, along the evaluation maps for the edges, of the moduli spaces $ \scrM(\Lab_{S_v}, H_{S_v}, J_{S_v})  $ corresponding to the vertices.

For each $i \in  \pi_{0}(\bR^2 \setminus T)$ with Lagrangian label, and finite gluing parameter $R$, consider the quotient $\tilde{L}_{i,R}/ {\sim}$ of the moving Lagrangian boundary condition over the component of the boundary corresponding to $i$, by the relation considered in the previous section: collapse each component of the inverse image of the thin part, then identify points in the same component of $\nu'_X \Theta$ in the fibres over a component $\Theta$ of the thin part.  By construction, the spaces we obtain for different choices of gluing parameters are naturally homeomorphic; we write $\tilde{L}_{i}/ {\sim}$ for any of these spaces, and note that we have an evaluation map
\begin{equation}
  \partial_i u/{\sim}   \co \partial_i S_{T,R} \to \tilde{L}_{i}/ {\sim} 
\end{equation}
for any choice of gluing parameter. 
\begin{lem} \label{lem:isoperimetric_constant_indep_gluing}
There exists a constant $C$ such that, for each $R \in [0,\infty]^{E^{\fin}(T)}$ and  $u \in \scrM(\Lab_{ S_{T,R}},  J_{S_{T,R}}) $, we have
\begin{equation}
  \ell \left( \partial_i u/{\sim} \right) \leq C   E^{geo}(u)   + \textrm{ a constant independent of } u \textrm{ and } R.
\end{equation}
This constant depends on the restriction of $J_{S_{T,R}}$ to a neighbourhood of the corresponding boundary component, and is independent of the restriction of $(\tilde{L}_{i,R}, J_{S_{T,R}}) $ to the product of each component $\Theta$ of the thin part with $\nu'_X \Theta$. 
\end{lem}
\begin{proof}
It suffices to check that the constants in the proof of Lemma \ref{lem:basic_isoperimetric_moving_L} can be chosen in terms of the corresponding constants for the moduli spaces $ \scrM(\Lab_{S_v}, J_{S_v})  $ and independently of $R$.  For Equation \eqref{eq:bound_energy_by_section}, this follows from the fact that the choice of symplectic form on $X \times S_V$ induces a symplectic form on $X \times S_{T,R}$.  For Equation \eqref{eq:duval_isoperimetric_rho}, we note that a choice of function $\rho_{i,S_v}$ for all $v$ induces, by gluing, a function $\rho_{S_{T,R}}$, for which we have the same estimate. The same argument applies to Equation \eqref{eq:iso_perimetric_rho}, which completes the argument.
\end{proof}

\subsection{Reverse isoperimetric inequality for families}
\label{sec:reverse-isop-ineq}

We now assemble the previous discussion into a statement for families of Riemann surfaces: consider a toplogical space $P$ parametrising a family of pre-stable Riemann surfaces with boundary marked points. This means that we are given a space $\bar{S}_P \to P$, whose fibre $\bar{S}_p$ for $p \in P$ is a compact pre-stable Riemann surface with boundary $\bar{S}_p$, equipped with finitely many distinct marked points on $\partial \bar{S}_p$. We assume that the conformal structure on $\bar{S}_p$ as a Riemann surface with marked points varies continuously with respect to $p$. Note that $P$ inherits a stratification by the number of components of $\bar{S}_p$, and we shall impose the technical condition that each such stratum $\partial^n P$ admits a neighbourhood deformation retract
\begin{equation}
  \pi_n \co  \nu \partial^n P \to \partial^n P
\end{equation}
 on which the family of Riemann surfaces is obtained by gluing, in the sense that we are given a map
\begin{equation} \label{eq:gluing_parameters_near_n-stratum}
    \nu \partial^n P \to [0,\infty]^n
\end{equation}
and we can choose strip-like ends on either side of each node of $\bar{S}_p$ for $p \in \partial^n P$, so that $\bar{S}_{q}$ for $q \in \nu \partial^n P$ agrees with the result of gluing the components of $ \bar{S}_{\pi_n(q)}$ with gluing parameters given by Equation \eqref{eq:gluing_parameters_near_n-stratum}.  

We denote by $S_P \to P$ the space obtained by removing all nodes and all marked points, and fix a collection of subsets $\Theta \subset S_P$ such that each puncture of a fibre $S_p$ lies in one of these subsets, and we require in addition that, whenever $p$ lies in $ \nu \partial^n P  $, each component of the gluing region is contained in some subset $\Theta$.

Given the above setup, we consider a moving family $ \Lab_{S_P} $ of boundary conditions on the family $S_P$, i.e. a map from $\partial S_P$ to the subsets of a closed symplectic manifold $X$. We assume that we are given a fixed a closed codimension $0$ subset $K \subset X$, such that the restriction of $  \Lab_{S_P}$ to each component of $\partial S_P$ is either constant with value $K$, or a moving family of Lagrangians. In the second case, we assume that the family of Lagrangians is obtained by gluing near $ \partial^n P $.

For each component $\Theta$ of the thin part, we consider a continuously varying family of almost complex structures $J_{\Theta_P}$ on $X$, parametrised by $P$, as well as open subsets
\begin{equation}
   \nu'_X \Theta \subseteq \nu_X \Theta \subseteq X \times P,
\end{equation}
which, over a point $p \in P$, contain the intersection of the Lagrangian labels of the components of $ \partial S_p$ meeting $\Theta$. We finally pick a family $J_{S_P}$ of almost complex structures on $X$, parametrised by $S_P$, such that Condition \eqref{eq:condition_J_S_fixed_ends} holds for the restriction to $S_p$ for each $p \in P$. 

We now define the moduli space
\begin{equation}
    \scrM(\Lab_{S_{P}} ,J_{S_P}) \equiv \coprod_{p \in P} \scrM(\Lab_{S_p} ,J_{S_p}),
\end{equation}
to be the disjoint union over all elements of $p$ of the moduli spaces of $J_{S_p}$ holomorphic maps with moving Lagrangian boundary conditions considered in Equation \eqref{eq:moduli_space_moving_boundary_conditions}. This space can be equipped with a natural topology as a parametrised moduli space, but we shall not require it in our discussion.

We have an assignment of a geometric energy $E^{geo}(u)$ to each element of this moduli space (see Equation \eqref{eq:geometric_energy_moving}), as well as a reduced length $ \ell( \partial_i u / {\sim}) $ of each boundary component, obtained by collapsing the part of the boundary lying in the thin part as in Equation \eqref{eq:collapsing_boundary_moving_Lagrangians}. The compatibility of our proof of the reverse isoperimetric inequality in Lemma \ref{lem:basic_isoperimetric_moving_L} with gluing (as dicussed in Lemma \ref{lem:isoperimetric_constant_indep_gluing}) implies:
\begin{prop} \label{prop:main_result_family_reverse_isoperimetric}
  If $P$ is compact, there exists a constant $C$ such that, for each $u \in \scrM(\Lab_{S_P}, J_{S_P}) $ and for each component $i$ of $\partial S_P$, we have
\begin{equation}
     \ell( \partial_i u / {\sim}) \leq C   E^{geo}(u)   + \textrm{ a constant independent of }u.
\end{equation}
The constant depends only on the restriction of the Lagrangian boundary conditions and the almost complex structures to a neighbourhood of $\partial_i S_P$, but is independent of their restriction to $\Theta \times \nu'_X \Theta $ for each component $\Theta$ of the thin part. \qed
\end{prop}

\section{Tate's acyclicity theorem}
\label{sec:null-homotopy-tates}

Given an affinoid covering of an affinoid domain, Tate showed in \cite{Tate1971} that the augmented \v{C}ech complex of the rings of functions of the domains of the cover is acyclic, and more generally for the \v{C}ech complex with coefficients in a complex of coherent sheaves. Tate's argument starts by constructing a null-homotopy for Laurent coverings (see Section \ref{sec:tates-null-homotopy} below), then proceeds to use standard tools of homological algebra to conclude the general case.

\subsection{Statement of the main result}

In this section, we imitate the strategy of Tate's proof in order to be able to use \v{C}ech  methods to compute morphism spaces in the analytic Fukaya category. Since the construction is completely local, we consider a torus $\bT^n$ equipped with a basepoint and Morse function. Let $\Poly$ denote the $A_\infty$-category with objects integral affine polytopes $P \subset H_{1}(\bT^n, \bR)$, morphisms for a pair $(P_{0},P_{1})$ given by the complex
\begin{equation}
 CF^*(P_{0}, P_{1}) \coloneqq  CM^{*}(X_{q}, \Hom^c_\Lambda(U^{P_{0}}, U^{P_{1}}) \otimes \lambda)
\end{equation}
 and $A_\infty$ operations as in Section \ref{sec:local-category-based}.   We do not require that the objects of $P$ have non-empty interior, which implies that the object of $\Poly$ are closed under intersections.

Let $A$ be a finite partially ordered set indexing a cover $\{ P_{a} \}_{a \in A}$ of a polytope $P$, with the property that the polytopes $P_a$ and $P_b$ are disjoint whenever $a$ and $b$ are not comparable (note that the construction outlined after Equation \eqref{eq:nerve_cover_locally_contractible} satisfies this property). Given a totally ordered subset $\tau \subseteq A$, we denote by $P_{\tau}$ the intersection of the polytopes $P_{a}$ for $a \in \tau$. If  $\tau$ is a subset of $\sigma$, the inclusion $P_\sigma \subseteq P_\tau$ gives rise to a canonical element
\begin{equation} \label{eq:differential_cech}
  \delta^{\sigma}_{\tau} \in CF^0(P_\tau, P_\sigma)
\end{equation}
whose construction is recalled in Section \ref{sec:differentials} below.

We obtain a twisted complex
\begin{equation} \label{eq:cech-twisted-complex}
  \check{T}(\{ P_{\alpha} \}) \coloneqq \left( \bigoplus_{\emptyset \neq \sigma \subset A } P_{\sigma}[-|\sigma|], \delta \right),
\end{equation}
where $\delta$ is the \v{C}ech  differential. Explicitly, if $\sigma$ equals $ (a_{1}, \ldots, a_m)$ as an ordered set, then the restriction of the differential to $P_{\sigma}$ is given by
\begin{equation}
  \sum_{a \in  A \setminus \sigma} (-1)^{i} \delta_{\sigma}^{\sigma \cup \{ a \}}.
\end{equation}
\begin{prop}
  \label{prop:cech-complex-acyclic}
The natural map
\begin{equation}
  P \to     \check{T}(\{ P_{\alpha} \} )
\end{equation}
induced by $ \bigoplus_{a \in A} \delta^{a}_{\varnothing}$ is a quasi-isomorphism in the category of twisted complexes over $\Poly$.
\end{prop}

The proof of this proposition is given in Section \ref{sec:twist-compl-laur} below.

\subsection{Restriction maps on Floer cochains}
\label{sec:differentials}

We begin by defining the map in Equation \eqref{eq:differential_cech} more precisely: if $P_{1} \subseteq P_{0}$, the restriction map $\Gamma^{P_{0}} \to \Gamma^{P_{1}}$ induces a map of local systems
\begin{equation}
  U^{P_{0}} \to U^{P_{1}},
\end{equation}
which is a continuous inclusion (we point out again that, unlike in Section \ref{sec:loops-paths-local}, we do not decorate local systems by elements of an underlying cover of $Q$). These maps are natural in sense that given a triple $P_2 \subseteq P_{1} \subseteq P_{0} $ the map $ U^{P_{0}} \to U^{P_2} $ is given by composition.

Taking the sum of these elements over all maxima of the Morse function on $\bT^n$, we obtain
\begin{equation}
  \delta_{P_{0}}^{P_{1}} \in CF^0(P_{0}, P_{1}). 
\end{equation}
To tie back to Equation \eqref{eq:differential_cech}, we define
\begin{equation}
   \delta^{\sigma}_{\tau}   \coloneqq  \delta_{P_\tau}^{P_\sigma}.
\end{equation}

\subsection{Tate's null-homotopy}
\label{sec:tates-null-homotopy}

Following Tate, we begin by considering an integral affine function $u$ on $H_{1}(\bT^n, \bR)$, which we assume is primitive in the sense that the class of $du \in H^{1}(\bT^n, \bZ)$ is not divisible.

If $P$ is an integral affine polytope, let $P_+$ and $P_-$ denote the subsets of $P$ where $u$ is non-negative, respectively non-positive, and let $P_{0}$ denote their intersection. In the language of rigid geometry, this corresponds to a \emph{Laurent cover} with two terms (Tate calls these \emph{two term special affine coverings}, see \cite[Lemma 8.3]{Tate1971}).  

Consider the two-term \v{C}ech complex $ \check{C}^*(\Gamma^P; \{ +,- \}) $
\begin{equation}
  \begin{tikzcd}
     \Gamma^{P_+} \oplus  \Gamma^{P_-}  \arrow{r}{\hat{d}} & \Gamma^{P_0}. 
  \end{tikzcd}
\end{equation}
This complex is natural in the sense that every inclusion $P \subseteq P'$ induces a natural map of complexes
\begin{equation}
 \check{C}^*(\Gamma^{P'}; \{ +,- \}) \to  \check{C}^*(\Gamma^P; \{ +,- \}).   
\end{equation}
\begin{lem} \label{lem:elementary_laurent_cover_acyclic}
A choice of complementary subspace to the line spanned by $du$  in $ H^{1}(\bT^n, \bZ)$  determines a continuous null-homotopy for the augmented \v{C}ech complex
 \begin{equation}
\Gamma_P \to  \check{C}^*(\Gamma^P; \{ +,- \}) 
 \end{equation}
which is natural in $P$.
\end{lem}
\begin{proof}
  Applying a change of coordinates, we may assume that $u$ is a coordinate function on $H_{1}(\bT^n, \bR)$, corresponding to a monomial $z$ in the group ring $\Gamma$. We formally write every element of the ring of functions on $P_0$, $P_+$, $P_-$, and $P$ as
  \begin{equation}
    \label{eq:express_F_series}
    F(z,w) = \sum_{i=\infty}^{+\infty} z^i 
f_i(w)
  \end{equation}
 where $w=(w_2, \ldots, w_n)$ are the other coordinates, and define
\begin{align}
  F_+(z,w) & \coloneqq   \sum_{i=1}^{+\infty} z^i f_i(w) \\
F_-(z,w) & \coloneqq   \sum_{i=-\infty}^{0} z^i f_i(w).
\end{align}
Evidently, $F_- + F_+ = F$. On $ \Gamma^{P_0} $, the null homotopy is provided by 
\begin{align}
     \Gamma^{P_+} \oplus  \Gamma^{P_-} &  \leftarrow  \Gamma^{P_0} \co \hat{h}  \\
(F_+, - F_-) &\mapsfrom F
\end{align}
The valuation of this map is non-negative because the minimal valuation of $F_\pm$ on $P_\pm$ is achieved on $P_\pm$, and $\hat{d} \circ \hat{h} |\Gamma^{P_0} $ is the identity, where $\hat{d}$ is the differential on the augmented  \v{C}ech complex. On $  \Gamma^{P_+} \oplus  \Gamma^{P_-} $, the null homotopy is
\begin{align}
  \Gamma^{P}   &  \leftarrow  \Gamma^{P_+} \oplus  \Gamma^{P_-}  \co \hat{h} \\
F_- +  G_+  &\mapsfrom  (F,G),
\end{align}
which again has non-negative valuation. The reader may easily compute that  $\hat{h} \circ \hat{d} |\Gamma^{P} $ is the identity. On $\Gamma^{P_+} \oplus  \Gamma^{P_-}$, we have
\begin{equation}
  \begin{tikzcd}[column sep=-30]
    & (F,G) \ar[rd] \ar[ld]  & \\
F_- + G_+ \ar[dr] & & F - G \ar[dl] \\
& \left(F_- + G_+ + (F - G)_+,   F_- + G_+ - (F - G)_-  \right),
  \end{tikzcd}
\end{equation}
from which the equation for a null-homotopy follows.

Naturality with respect to restriction maps is automatic from the fact that we did not appeal to any property of $P$ in constructing $\hat{h}$, and the only choice we made is the choice of complementary subspace that determines the decomposition in Equation \eqref{eq:express_F_series} (this decomposition clearly does not depend on the actual coordinates $(w_2, \ldots, w_n)$, but only on subspace of $ H^{1}(\bT^n, \bZ)$ spanned by their derivatives).
\end{proof}

Given a topological vector space $V$, we consider the complex
\begin{equation} \label{eq:null_homotopy_map_to_V}
  \Hom^c_\Lambda(\check{C}^*(\Gamma^P; \{+,-\}), V) \coloneqq \Hom^c_\Lambda(\Gamma^{P_0},V) \to \Hom^c_\Lambda(\Gamma^{P_+}, V)  \oplus \Hom^c_\Lambda(\Gamma^{P_-}, V),
\end{equation}
where $\Hom^c_\Lambda$ is the space of continuous maps. Composing with $\tilde{h}$, we obtain a null-homotopy for  the augmented complex
\begin{equation}
  \Hom^c_\Lambda(\check{C}^*(\Gamma^P; \{+,-\}), V) \to \Hom^c_\Lambda(\Gamma^P, V) 
\end{equation}
which is natural in $P$ and $V$.

\subsection{Acyclicity of the augmented complex}
\label{sec:twist-compl-laur}

We next consider a general Laurent cover: let $\{ u_m \}_{m \in M}$ be a collection of integral affine functions on $H_{1}(\bT^n, \bR)$ indexed by a finite set $M$. Given a polytope $P$, we associate to each element of $M \times \{ +,-\} $ the polytope $P_{m,\pm}$ given by the subset where $u_m$ is non-negative (or non-positive). We say that the elements $ P_{m,\pm}$ are the \emph{Laurent cover} associated to $M \times \{+,-\} $,  so that we obtain a twisted complex $\check{T}(P, M \times \{ +,-\} )$ on $\Poly$ given by Equation \eqref{eq:cech-twisted-complex}. 

Let $(P,P')$ be a pair of polytopes. By definition, the space of morphisms in the category of twisted complexes on $\Poly$ from  $\check{T}(P, M \times \{ +,-\})$ to $P'$ is given by
\begin{align}
  \left( \bigoplus_{\sigma \subset M \times \{ +,-\}} CF^*( P_\sigma, P')[-n] , \delta \right)
\end{align}
with the differential induced by restriction. 
\begin{lem}
For each pair of polytopes $(P,P')$, the natural map
\begin{equation}
    \left( \bigoplus_{\sigma \subset M \times \{ +,-\}} CF^*( P_\sigma, P')[-n] , \delta \right) \to   CF^*( P, P')
\end{equation}
is a quasi-isomorphism.
\end{lem}
\begin{proof}
Filtering by the degree of critical points of the Morse function, it suffices to prove that the augmented complex
\begin{equation}
 \bigoplus_{\sigma \subset M \times \{ +,-\}}    \Hom^c_\Lambda(\Gamma^{P_\sigma}, \Gamma^{P'}) \to \Hom^c_\Lambda(\Gamma^{P}, \Gamma^{P'})
\end{equation}
is a quasi-isomorphism. As in \cite[Lemma 8.4]{Tate1971}, induction on the number of elements of $M$ reduces this to the case of a singleton, which follows immediately from Lemma \ref{lem:elementary_laurent_cover_acyclic}.
\end{proof}
\begin{cor}
  \label{cor:cech-complex-Laurent-acyclic}
  If $M \times \{ +,-\}$ indexes a Laurent cover of a polytope $P$, there is a natural quasi-isomorphism
  \begin{equation}
P \to   \check{T}(P, M \times \{ +,-\}). 
  \end{equation} \qed
\end{cor}

We now prove the main result of this section:
\begin{proof}[Proof of  Proposition \ref{prop:cech-complex-acyclic}]
We essentially follow the method introduced by Tate in \cite[Section 8]{Tate1971}: every cover $\Sigma$ admits a refinement by a Laurent cover $ M \times \{ +,-\} $ obtained by considering the functions defining all boundary facets of polytopes appearing in the cover. The naturality of the construction of the complexes $\check{T}$ implies that we can write the map from $P$ to the \v{C}ech twisted complex associated to $ M \times \{ +,-\} $ as a composition:
\begin{equation}
P \to   \check{T}(P,\Sigma)  \to \check{T}(P, M \times \{ +,-\}),
\end{equation}
where the first arrow is the map which we would like to show is a quasi-isomorphism. By Corollary \ref{cor:cech-complex-Laurent-acyclic}, it suffices to show that the second map is a quasi-isomorphism. Filtering by the number of elements of a subset $\sigma \in \Sigma$, this follows by applying Corollary \ref{cor:cech-complex-Laurent-acyclic} to
\begin{equation}
 P_\sigma \to   \check{T}(P_\sigma, M \times \{ +,-\}).  
\end{equation}
\end{proof}

\section{Computations of cohomology groups with twisted coefficients}
\label{sec:based-loops-laurent}

The main goal of this section is to prove Propositions \ref{prop:computation_inclusion} and \ref{prop:computation_reverse_inclusion}. As in Appendix \ref{sec:null-homotopy-tates}, we consider the local situation: given a Morse function on the torus, we define a Floer group $ CF^*(P_{0}, P_{1}) $ for each pair of polytopes in $H_{1}(\bT^n; \bR) $.
\begin{prop} \label{prop:computation_inclusion-restated}
  If $P_{1} \subseteq  P_{0}$, the Floer group $HF^*(P_{0}, P_{1}) $ is supported in degree $0$, and we have a natural quasi-isomorphism
  \begin{equation}
    \Gamma^{P_{1}}  \to CF^*(P_{0}, P_{1})
  \end{equation}
given by the inclusion of $  \Gamma^{P_{1}}$ in $ CF^0(P_{0}, P_{1})$.  If $P_{0} $ is contained in the interior of $P_{1}$, the Floer group $HF^*(P_{0}, P_{1}) $ is supported in degree $n$, and a choice of orientation of $P_0$ determines a quasi-isomorphism
\begin{equation}
 CF^*(P_{0}, P_{1}) \cong   \Hom_{\Lambda}^{c}(\Gamma^{P_{0}}, \Lambda)
\end{equation}
given by the trace map on $ CF^n(P_{0}, P_{1}) $.
\end{prop}
Along the way, we shall prove that morphisms between disjoint polytopes vanish.

\begin{rem}
A version of the results of this section hold for the completions of the homology of the based loops space of any topological space having the homotopy type of a finite $CW$ complex, where the polytopes are integral affine subsets of first cohomology. The proof would take us too far afield, so we give a computational and explicit proof in the case of tori.
\end{rem}

Throughout this section, we shall write $\Hom(A,B)$ and $A \otimes B$ for the group of homomorphisms and for the tensor product of two abelian groups $A$ and $B$. If $A$ and $B$ are equipped with (semi)-norms, we write $\Hom^c(A,B)$ for the group of bounded homomorphisms, and $A \hat{\otimes} B$ for the completed tensor product (i.e. for the completion of $A \otimes B$ with respect to the induced semi-norm).  When $R$ is a (commutative) ring, and $A$ and $B$ are modules over $R$, we write  $\Hom_R(A,B)$ and $A \otimes_R B$ for the $R$-homomorphisms, and the tensor product of modules. Finally, assuming that $R$ is a normed ring, and $A$ and $B$ are $R$-modules equipped with (semi)-norms, we write $\Hom^c_{R}(A,B)$ for the group of bounded $R$-module homomorphisms, and $A \hat{\otimes}_{R} B$ for the completion of $A \otimes_R B$.

\subsection{The $1$-dimensional case}
\label{sec:1-dimensional-case}

Consider the circle equipped with the standard Morse function $f$ with a unique minimum and maximum, and let $U$ denote the local system corresponding to the regular representation of the fundamental group. Identifying the space of paths from the minimum to the maximum with the space of based loops at the maximum via the choice of a path connecting these two points, and the homology of the latter with the Laurent polynomial ring $\Gamma = \bZ[z,z^{-1}]$, the differential in the Morse complex
\begin{equation}
  CM^*(f;   \Hom(U,U))
\end{equation}
with coefficients in $\Hom(U,U)$ can be expressed as the map
\begin{align}
\Hom(\Gamma, \Gamma) & \to \Hom(\Gamma, \Gamma) \\
  \label{eq:differential_hom_loops_S^1}
 \phi & \mapsto \phi - z \cdot \phi \cdot z^{-1},
\end{align}
and the kernel of this differential is naturally isomorphic to
\begin{equation} \label{eq:inclusion_Gamma}
\Gamma \cong   \Hom_{\Gamma}( \Gamma,    \Gamma) \subset \Hom( \Gamma , \Gamma).
\end{equation}

Consider the map
\begin{equation}
  \Hom( \Gamma,    \Gamma) \leftarrow \Hom( \Gamma,  \Gamma) \thinspace \colon \! h  ,
\end{equation}
which is recursively defined by the formula
\begin{equation} \label{eq:null_htpy_circle}
\psi(z^i) + z (h \psi)(z^{i-1}) = (h\psi)(z^i)
\end{equation}
which we normalise by setting $(h \psi)(1) = 0$. Note in particular that the above formula implies that $(h \psi)(z) = \psi(z)$, and $(h \psi)(z^{-1}) = - z^{-1} \psi(1) $.
\begin{lem}
  The map $h$ defines a homotopy from the identity of $ CM^*(f;   \Hom(U,U))$ to the composition
  \begin{equation}
     CM^*(f;   \Hom(U,U)) \to \Gamma \to  CM^*(f;   \Hom(U,U)) ,
   \end{equation}
   where the first map is the projection
\begin{align}
 \Hom( \Gamma , \Gamma) & \to \Hom_{\Gamma}( \Gamma,    \Gamma) \coloneqq \Gamma \\
\phi & \mapsto \phi(1)
\end{align}
from $CM^0(f;   \Hom(U,U)) $  to $\Gamma$, and the second map is the inclusion in Equation \eqref{eq:inclusion_Gamma}. \qed
\end{lem}

For later purposes, it is convenient to derive this complex, and the corresponding null-homotopy, from a version of the Koszul complex: namely, consider
\begin{equation} \label{eq:Koszul_circle}
  \Gamma \otimes \Gamma \to \Gamma \otimes \Gamma  
\end{equation}
with differential
\begin{equation}
d ( f \otimes g) = f \otimes g -  z^{-1} \cdot f \otimes g \cdot z.  
\end{equation}
The complex $  CM^*(f;   \Hom(U,U))$ is naturally isomorphic to the complex of $\Gamma$-module maps from Equation \eqref{eq:Koszul_circle} to $\Gamma$, and the homotopy $h$ arises from a null homotopy of the extended complex
\begin{equation}
    \Gamma \otimes \Gamma \to \Gamma \otimes \Gamma  \to \Gamma,
  \end{equation}
  where the last map sends $f \otimes g$ to the product $f \times g$.
\begin{rem}
Note that the choice of Koszul complex depends on a choice of decomposition of Laurent polyonomials into positive and negative powers. In particular, the differential $d_-$ associated to swapping the r\^oles of $z$ and $z^{-1}$ is related to the above differential by the equation
\begin{equation}
  d_- = - z \cdot  d \cdot z^{-1}.  
\end{equation}
The minus sign above accounts for the need to choose an orientation when we define the trace map on the dual Floer cohomology group.
\end{rem}
\subsection{The standard Morse complex on the torus}
\label{sec:stand-morse-compl}

We consider the torus $\bT^n \coloneqq S^1 \times \cdots \times S^1$, and denote by $\Gamma$ the group ring of $H_1(\bT^n; \bZ)$. The product decomposition, and an orientation of each factor, induce an isomorphism
\begin{equation} 
  \Gamma \cong \bZ[z_{1}^\pm, \cdots, z_n^\pm].  
\end{equation}
It is convenient to switch back and forth between this notation, and the notation wherein we write elements of $\Gamma$ as $z^{\alpha}$ for $\alpha \in \bZ^n$.

Consider the standard Morse function, i.e. the sum of the standard Morse functions on each factor, having  a unique minimum and maximum, and pick on each factor a path from the minimum to the maximum. Consider the universal local system whose fibre at a point is the space of paths to a basepoint, which we choose to be the maximum. Our choice of paths identifies the Morse complex with coefficients in the endomorphisms of this local system with
\begin{equation} \label{eq:hom_local_system_torus}
  \Hom(\Gamma,\Gamma) \otimes H^*(\bT^n; \bZ) ,
\end{equation}
with a differential given by
\begin{equation} \label{eq:differential_endomorphism_loop_space}
\partial ( \phi \otimes \alpha ) = \sum_i \partial_i ( \phi \otimes \alpha )  \coloneqq \sum_{i} \left( \phi - z_j \cdot \phi \cdot z_j^{-1} \right) \otimes b_j \wedge \alpha,
\end{equation}
where $\{ b_j \}_{j=1}^{n}$ is the standard basis of $H^1(\bT^n; \bZ)$.

There is a natural subcomplex of Equation \eqref{eq:hom_local_system_torus} given by the inclusion
\begin{equation}
\Gamma \to   \Hom(\Gamma,\Gamma) \otimes H^*(\bT^n; \bZ) 
\end{equation}
whose image lies is $ \Hom_{\Gamma}(\Gamma,\Gamma) \otimes H^0(\bT^n; \bZ)  $.  We shall construct an explicit retraction from the right to the left hand side.

To this end, we define a map
\begin{align}
  h_j \co \Hom(\Gamma,\Gamma) & \to \Hom(\Gamma,\Gamma) \\
h_j \psi (z^{\alpha}) & = \begin{cases} 0 & \textrm{ if } \alpha_j = 0 \\
\psi(z^\alpha) + z_j  \cdot \psi(z^{\alpha - e_j}) & \textrm{ otherwise}
\end{cases}
\end{align}
where $\alpha \in \bZ^n$, and $e_j$ is the $j$\th basis element. Note that this is a recursive definition of $h_j \psi$, and that the explicit formula is 
\begin{equation}
  h_j \psi (z^{\alpha})  =
  \begin{cases}
   \sum_{i=0}^{\alpha_j-1} z^{ie_j} \psi(z^{\alpha - i e_j}) & 0 \leq \alpha_j \\
  - \sum_{i= \alpha_j}^{-1} z^{ie_j} \psi(z^{\alpha - i e_j}) &  \alpha_j < 0.
  \end{cases}
\end{equation}

We then define
\begin{align}
  h \co    \Hom(\Gamma,\Gamma) \otimes H^*(\bT^n; \bZ) & \to  \Hom(\Gamma,\Gamma) \otimes H^*(\bT^n; \bZ)  \\
h & = \sum_{j=1}^{n} h_j \otimes \iota_j,
\end{align}
where $\iota_j$ is the slant product 
\begin{equation}
 H^*(\bT^n; \bZ) \to H^{*-1}(\bT^n; \bZ)
\end{equation}
with the basis element of $e_j \in H_{1}(\bT^n; \bZ)$.

\begin{lem} \label{lem:standard_map_is_null_htpy}
The map $h$ is a homotopy between the identity on $ \Hom(\Gamma,\Gamma) \otimes H^*(\bT^n; \bZ) $  and the composition of the inclusion $\Gamma \subset \Hom(\Gamma,\Gamma) \otimes H^0(\bT^n; \bZ) $ with the projection
\begin{align}
   \Hom(\Gamma,\Gamma) \otimes H^*(\bT^n; \bZ) & \mapsto \Gamma, 
 \end{align}
 which vanishes on the graded components of strictly positive degree, and is given on the degree $0$ component by the map
 \begin{equation}
    \phi \otimes 1  \mapsto  \phi(1).
 \end{equation}
\end{lem}
\begin{proof}
The complex $\Hom(\Gamma,\Gamma) \otimes H^*(\bT^n; \bZ)$ is naturally isomorphic to the complex of $\Gamma$-homomorphisms from the $n$-fold tensor product of Equation \eqref{eq:Koszul_circle} to $\Gamma$, and the homotopy to the projection is induced from the corresponding homotopy in Equation \eqref{eq:null_htpy_circle}.
\end{proof}

\subsection{Construction of the homotopy for inclusions: I}
\label{sec:constr-bound-null}

Our goal is to extract from Lemma \ref{lem:standard_map_is_null_htpy} a bounded homotopy for the completions. To this end, we now use $\Gamma$ to denote the ring of Laurent polynomials over the Novikov field $\Lambda$.

Let $P_{0}$ and $P_{1}$ be integral affine polytopes in $H^1(\bT^n; \bR)$. The standard Morse complex computing morphisms between the corresponding local systems is given by
\begin{equation} %
CF^{*}(P_{0}, P_{1}) \cong  \Hom^c_\Lambda(\Gamma^{P_{0}},\Gamma^{P_{1}}) \otimes H^*(\bT^n; \bZ) ,
\end{equation}
with differential given by Equation \eqref{eq:differential_endomorphism_loop_space}. Moreover, we have the inclusion of $\Hom^c_\Gamma(\Gamma^{P_{0}},\Gamma^{P_{1}})$ in degree $0$.

\begin{lem} \label{lem:map_big_to_small}
If $P_{1} \subseteq P_{0}$, the homotopy $h$ is continuous, hence induces a retraction
\begin{equation}
    \Hom^c_\Lambda(\Gamma^{P_{0}},\Gamma^{P_{1}}) \otimes H^*(\bT^n; \bZ)  \to \Hom^c_\Gamma(\Gamma^{P_{0}},\Gamma^{P_{1}}) = \Gamma^{P_{1}}
\end{equation}
\end{lem}
\begin{proof}
It suffices to prove that each map $h_j$ is continuous, in which case is suffices to bound
\begin{equation}
  \val(h_j \psi) - \val \psi
\end{equation}
for any (continuous) map $\psi$ from $\Gamma^{P_{0}}$ to $\Gamma^{P_{1}} $. A straightforward computation reduces this to proving that 
\begin{equation}
  \val_{P_{0}} z_j \leq  \val_{P_{1}} z_j   
\end{equation}
which follows immediately from the inclusion $P_{1} \subseteq P_{0}$.
\end{proof}

\subsection{Construction of the homotopy for disjoint sets}
Let us now consider the case where $P_{0}$ and $P_{1}$ are disjoint integral affine polytopes. By a change of coordinates, we may assume that the first lies in the region where the first coordinate is strictly positive, and the second in the region where it is strictly negative. We then define
\begin{align}
h \co \Hom^c_\Lambda(\Gamma^{P_{0}},\Gamma^{P_{1}}) \otimes H^*(\bT^n)  & \to  \Hom^c_\Lambda(\Gamma^{P_{0}},\Gamma^{P_{1}}) \otimes H^*(\bT^n) \\
\psi \otimes v  & \mapsto \sum_{i=1}^{\infty} z^i_{1} \cdot \psi \cdot z^{-i}_{1} \otimes \iota_{1} v.  
\end{align}
We note that the infinite series on the right hand side is convergent because
\begin{equation}
  \val(z^i_{1} \cdot \psi \cdot z^{i-1}_{1}  )  = \val(\psi) + i \left( \val_{P_{0}}(z) + \val_{P_{1}}(z^{-1}) \right)
\end{equation}
and the assumptions on $P_0$ and $P_{1}$ respectively imply that
\begin{equation}
 0 <  \val_{P_{0}} z  \textrm{ and } 0 < \val_{P_{1}} z^{-1}.
\end{equation}

\begin{lem}
The map $h$ defines a null-homotopy of $\Hom^c_\Lambda(\Gamma^{P_{0}},\Gamma^{P_{1}}) \otimes H^*(\bT^n)$. 
\end{lem}
\begin{proof}
Let us write $H^*(\bT^n)$ as the direct sum $H^0(S^1) \otimes H^*(\bT^{n-1}) \oplus H^1(S^1) \otimes H^*(\bT^{n-1})$. The compositions
\begin{align*}
h \circ \partial_{1} \co   \Hom^c_\Lambda(\Gamma^{P_{0}},\Gamma^{P_{1}}) \otimes H^0(S^1) \otimes H^*(\bT^{n-1})  & \to    \Hom^c_\Lambda(\Gamma^{P_{0}},\Gamma^{P_{1}}) \otimes H^0(S^1) \otimes H^*(\bT^{n-1}) \\
 \partial_{1} \circ h \co   \Hom^c_\Lambda(\Gamma^{P_{0}},\Gamma^{P_{1}}) \otimes H^1(S^1) \otimes H^*(\bT^{n-1})  & \to    \Hom^c_\Lambda(\Gamma^{P_{0}},\Gamma^{P_{1}}) \otimes H^1(S^1) \otimes H^*(\bT^{n-1}),
\end{align*}
both agree with the identity by an explicit computation. On the other hand, $h$ commutes with each differential $\partial_i$ for $i \neq 1$. The result follows.
\end{proof}
\begin{cor} \label{cor:disjoint_morphisms_vanish}
If $P_{0} \cap P_{1} = \emptyset$, the cohomology group $HF^*(P_{0}, P_{1})$ vanishes. \qed
\end{cor}

\subsection{Computation of morphisms for inclusions: II}
\label{sec:comp-morph-incl}
Consider the $1$-dimensional case, with bounded closed intervals $P_{1} \Subset P_{0}$ (i.e. so that $P_1$ is included in the interior of $P_0$) . Recall that 
\begin{equation}
  \val_{P_{0}}z < \val_{P_{1}} z \textrm{ and }   \val_{P_{0}}z^{-1} < \val_{P_{1}} z^{-1},
\end{equation}
thus, there is a constant $c$ such that
\begin{equation} \label{eq:linear_bound_difference_norms}
  \val_{P_{1}} z^{i} - \val_{P_{0}} z^i > c |i|.  
\end{equation}
For the next statement, recall that $\hat{\otimes}$ denote the completed tensor product.
\begin{lem}
The natural inclusions
\begin{equation}
\Hom^c_\Lambda(\Gamma^{P_{0}}, \Gamma^{P_{0}}) \leftarrow   \Hom^c_\Lambda(\Gamma^{P_{1}}, \Gamma^{P_{0}}) \to  \Hom^c_\Lambda(\Gamma^{P_{1}}, \Gamma^{P_{1}})
\end{equation}
factor through the inclusions
\begin{align}
\Gamma^{P_i} \hat{\otimes} \Hom^c_\Lambda(\Gamma^{P_i}, \Lambda)  & \subset  \Hom^c_\Lambda(\Gamma^{P_i}, \Gamma^{P_i}).
\end{align}
\end{lem}
\begin{proof}
We can formally write any element $ \phi \in  \Hom^c_\Lambda(\Gamma^{P_{1}}, \Gamma^{P_{0}}) $ as
\begin{equation}
  \sum_{i,j} \phi_{i}(z^j) z^i \otimes \rho_j,  
\end{equation}
where $\phi_{i}(z^j)$ is the coefficient of $z^i$ in $\phi(z^j)$, and $\rho_j$ is the homomorphism which assigns $1$ to $z^j$ and $0$ to every other monomial basis element of $\Gamma$. The assumption that $\phi$ is bounded implies that there exists a constant $K$ such that
\begin{equation}
  \min_{i,j} \left( \val  \phi_{i}(z^j) + \val_{P_{0}}  z^i - \val_{P_{1}} z^j \right) \geq K.
\end{equation}
The result now follows from Equation \eqref{eq:linear_bound_difference_norms}, since replacing $ \val_{P_{0}}  z^i  $ by $ \val_{P_{1}}  z^i $ allows us to add $c |i|$ to the right hand side, so that, when considered as an element of $\Hom^c_\Lambda(\Gamma^{P_{1}}, \Gamma^{P_{1}})$, there are only finitely many terms with valuation bounded above by any given number. Replacing  $\val_{P_{1}} z^j$ by $\val_{P_{0}} z^j $ implies the same for $\Hom^c_\Lambda(\Gamma^{P_{0}}, \Gamma^{P_{0}}) $.
\end{proof}
We conclude that there is a natural trace
\begin{equation}
  \tr \co   \Hom^c_\Lambda(\Gamma^{P_{1}}, \Gamma^{P_{0}}) \to \Lambda,
\end{equation}
given by the composition of the inclusions into $ \Gamma^{P_i} \hat{\otimes} \Hom^c_\Lambda(\Gamma^{P_i}, \Lambda)$  with the evaluation map
\begin{align}
   \Gamma^{P_i} \hat{\otimes} \Hom^c_\Lambda(\Gamma^{P_i}, \Lambda) & \to \Lambda\\
  f \otimes \rho & \mapsto \rho(f).
\end{align}
Explicitly, the trace can be written as
\begin{equation}
  \psi \mapsto    \sum_{i=-\infty}^{+\infty} \psi( z^i)_{i}, 
\end{equation}
where the subscript records the coefficient of $z^{i}$. This in particular shows that the traces defined via endomorphisms of $\Gamma^{P_{0}}$ and $ \Gamma^{P_{1}} $ agree. 

We now consider the map
\begin{align}
\epsilon \co  \Hom^c_\Lambda(\Gamma^{P_{1}}, \Gamma^{P_{0}}) & \to \Hom^c_\Lambda(\Gamma^{P_{1}}, \Lambda) \\ \label{eq:definition_map_to_dual_space}
\epsilon \psi (f)  & \coloneqq  \tr \left( \psi \circ f \right).
\end{align}

It is easy to see that $\epsilon$ composes trivially with the differential on $ CF^*(P_{1}, P_{0})$ hence defines a cochain map to $\Hom^c_\Lambda(\Gamma^{P_{1}}, \Lambda) $. Consider the map
\begin{align}
 \Hom^c_\Lambda(\Gamma^{P_{1}}, \Gamma^{P_{0}})  &  \leftarrow  \Hom^c_\Lambda(\Gamma^{P_{1}}, \Lambda)  \co  \delta \\
\rho(f)   & =  \delta (\rho) (f) .
\end{align}
\begin{lem}
The composition  $\epsilon \circ \delta$ is the identity on $\Hom^c_\Lambda(\Gamma^{P_{1}}, \Gamma)$.  
\end{lem}
\begin{proof}
Consider the map $ \delta (\rho)  \cdot z^i$.  We compute that
\begin{equation}
   \delta (\rho) (z^i \cdot z^j) = \rho(z^{i+j}) .
\end{equation}
Thus the trace of this map is given by setting $j=0$, and is equal to $\rho(z^i)$.  
\end{proof}

We now define a map $\hbar$ which will serve as a homotopy between the identity and the composition $\delta \circ \epsilon $, and which is determined by the expression
\begin{align}
\Hom^c_\Lambda( \Gamma^{P_{1}},    \Gamma^{P_{0}}) & \leftarrow \Hom^c_\Lambda( \Gamma^{P_{1}},  \Gamma^{P_{0}})  \\ \label{eq:recursive_definition_hbar}
\psi +  \hbar(z^{-1} \circ \psi \circ z) = \hbar(\psi) &  \mapsfrom  \psi,    
\end{align}
which we normalise by requiring that $\hbar(\psi)$ vanish if the image of $\psi$ is contained in the image of $\delta$.  We obtain an explicit expression for this map as follows: formally write 
\begin{equation}
  \psi = \sum_{i \in \bZ} \psi_i  
\end{equation}
with $\psi_i$ having image in the line spanned by $z^i$. We have
\begin{equation}
\hbar \psi_i =   \begin{cases} \sum_{j=0}^{i-1} z^{-j} \circ \psi_i \circ z^j & 1 \leq i\\
    0 & i=0 \\
 \sum_{j=i}^{-1} z^{-j} \circ \psi_i \circ z^j & i \leq -1,
  \end{cases}
\end{equation}
and we (formally) define
\begin{align}
  \hbar \psi &  = \sum_{i \in \bZ} \hbar \psi_i \\  \label{eq:infinite_sum_expression_hbar}
& = \sum_{j \leq -1} z^{-j} \circ \psi_{\leq j} \circ z^j +  \sum_{0 \leq j} z^{-j} \circ \psi_{j+1 \leq} \circ z^j,
\end{align}
where $\psi_{\leq j}$ is the sum of all components $\psi_i$ with $i \leq j$, and $\psi_{j+1 \leq}$ is the sum of all components with $j+1 \leq i$.  
\begin{lem}
The expression in Equation \eqref{eq:infinite_sum_expression_hbar} is convergent, and the map $\hbar$ is continuous.
\end{lem}
\begin{proof}
 We compute that, for $j$ strictly negative, the valuation of $z^{-j} \cdot \psi_{\leq j} \cdot z^{j} $ is given by the infimum over all Laurent polynomials $f$ of 
\begin{align}
 \val_{P_{0}} z^{-j} \cdot \psi_{\leq j} \cdot z^{j} f - \val_{P_{1}} f & =  j  \val_{P_{0}} z^{-1}  + \val_{P_{0}} \psi_{\leq j} \cdot z^{j} f - \val_{P_{1}} z^{j} f \\
& \qquad + \val_{P_{1}} z^{j} f - \val_{P_{1}}  f \\
& \geq  j  \val_{P_{0}} z^{-1}  +  \val \psi_{\leq j}  + \val_{P_{1}} z^{j}  \\
& \geq -j  \left(  \val_{P_{1}} z^{-1} -  \val_{P_{0}} z^{-1} \right) + \val \psi.
\end{align}
The desired bound in this case thus follows from Equation \eqref{eq:linear_bound_difference_norms}. The case of positive monomials is similar, and the result immediately follows.
\end{proof}

\begin{lem}
The map $\hbar$ defines a homotopy between the identity on $CF^*(P_{1}, P_{0})$, and the projection $\delta \circ \epsilon $. 
\end{lem}
\begin{proof}
The fact that $\hbar \circ d$ is the identity follows trivially from Equation \eqref{eq:recursive_definition_hbar}. The equality $d \circ \hbar = id - \delta\circ \epsilon$ follows from a computation using the formal decomposition used above. In particular, for $\phi$ with image in the line spanned by $z^i$ with $i$ positive, we have
\begin{align}
  d \hbar \phi & = d\left( \phi + z^{-1} \phi t z + \cdots + z^{-i+1} \phi z^{i-1} \right) \\
&= \phi -  z^{-1} \phi z +  z^{-1} \phi z - z^{-2} \phi z^2 + \cdots - z^{-i} \phi z^{i}  \\
& = \phi - z^{-i} \phi z^{i}. 
\end{align}
The result thus follows from the equality $\delta \circ \epsilon(\phi) = z^{-i} \phi z^{i}  $.
\end{proof}

We now consider the higher dimensional situation: let $P_{1} \Subset P_{0}$ be nested integral affine polytopes in $\bR^n$, with the property that the interior of $P_0$ contains $P_1$. As before, we have a map
\begin{align}
\epsilon \co  \Hom^c_\Lambda(\Gamma^{P_{1}}, \Gamma^{P_{0}}) & \to \Hom^c_\Lambda(\Gamma^{P_{1}}, \Lambda) \\ 
\epsilon \psi (f)  & =  \tr \left( \psi \circ f \right).
\end{align}
with one-sided inverse given by the inclusion
\begin{equation}
  \delta \co  \Hom^c_\Lambda(\Gamma^{P_{1}}, \Lambda) \to \Hom^c_\Lambda(\Gamma^{P_{1}}, \Gamma^{P_{0}})
\end{equation}
whose image consists of maps factoring through $\Lambda \cdot 1  \subseteq \Gamma^{P_{0}}$.

Consider the maps $\hbar_j$ given by
\begin{align}
\Hom( \Gamma^{P_{1}},    \Gamma^{P_{0}}) & \leftarrow \Hom( \Gamma^{P_{1}},  \Gamma^{P_{0}})  \\ 
\psi +  \hbar_j(z^{-1}_j \circ \psi \circ z_j) = \hbar_j(\psi) &  \mapsfrom  \psi,    
\end{align}
which we normalise by requiring that $\hbar_j(\psi)$ vanish if the image of $\psi$ is contained in the space spanned by monomials with trivial power of $z_j$ (i.e. convergent Laurent series in the variables $z_i$ for $i \neq j$). As in Section \ref{sec:constr-bound-null}, we set
\begin{equation}
  \hbar = \sum_{j=1}^{n} \hbar_j \otimes \iota_j.
\end{equation}
This map provides a homotopy between the identity on $CF^*(P_{1}, P_{0})$ and the projection to $ \Hom^c_\Lambda(\Gamma^{P_{1}}, \Lambda)$ given by the composition
\begin{equation}
    \begin{tikzcd}
    CF^n(P_{1}, P_{0}) \ar[r,"\epsilon"] &    \Hom^c_\Lambda(\Gamma^{P_{1}}, \Lambda)  \ar[r,"\delta"] &  CF^n(P_{1}, P_{0}),
    \end{tikzcd}
\end{equation}
on the degree $n$ graded component of the Floer complex, which we extend by $0$ to the other graded components.

\begin{proof}[Proof of Proposition \ref{prop:computation_inclusion-restated}]
   For Morse functions on $\bT^n$ which are products of the standard Morse function on $S^1$, the first part of Proposition \ref{prop:computation_inclusion-restated} follows from Lemma \ref{lem:map_big_to_small}, and the second from the homotopy $\hbar$ constructed above. To conclude the result for general Morse functions, we use the standard argument that continuation maps in Morse theory induce homotopy equivalences of Morse complexes; these are automatically continuous, because each continuation map and each homotopy of continuation maps is a sum of finitely many terms.
\end{proof}

\begin{rem}
The constructions of this section are formally dual to those of Section \ref{sec:constr-bound-null}, in the sense that the formulae we use can be derived from those of that section by dualising with respect to the pairing
\begin{align}
\Gamma \otimes_\Lambda  \Gamma & \to \Lambda \\ 
f \otimes g & \mapsto \mathrm{Res}\left( fg \frac{dz}{z} \right) 
\end{align}
where the symbol $\mathrm{Res}( h dz)$ assigns $1$ to the monomial $h=z^{-1}$,  and $0$ to every non-trivial monomial. One can thus link the discussion of this section with the theory of residues via Tate's approach \cite{Tate1968}.
\end{rem}

\section{Invariance of the Family Floer functor}
\label{sec:invar-family-floer}

In this Appendix, we briefly outline an argument for the independence of the family Floer functor on the choices we have made (after placing further constraints on them). For simplicity, and as in \cite[Appendix A]{Abouzaid2014a} where the $A_\infty$ functor is constructed, we consider the functor on a single tautologically unobstructed immersed Lagrangian $L$ equipped with a compatible almost complex structure $J_L$ so that $L$ bounds no holomorphic discs and no holomorphic $1$-gons; the case of a finite collection of mutually transverse tautologically unobstructed Lagrangians is only more complicated for notational reasons.

Using the methods of this paper, the essential choice in the construction of the functor is that of a \emph{system of paths of fibres} as in Equation \eqref{eq:paths_of_fibres_functor}: for each holomorphic disc $\Sigma$ with two punctures and a choice of $\ell\leq \dim Q +1$ boundary marked points on one of the two components of the boundary, together with a labelling of the segments between the marked points by a collection $ ( q_1, \ldots, q_{\ell})$  of points in $Q$ which lie in a ball of radius $1$, we choose a smooth path $q_{\Lab}$ in $Q$, parametrised by the part of $\partial \Sigma$ carrying the marked points, and subject to the conditions listed in Section \ref{sec:isop-const}. We write $\scrS$ for such a system of paths, to which we can assign according to Lemma \ref{lem:universal_constant_isoperimetric}, a reverse isoperimetric constant $C_{\scrS}$ for the families of moduli spaces of strips with one boundary condition along $L$, and the other along the paths of Lagrangian fibres prescribed by $\scrS$, parametrised by the choices of points $ ( q_1, \ldots, q_{\ell})$.

We obtain from these moduli spaces a module $\scrL_{L}$ over the category $\Fuk_{A}$ from Section \ref{sec:directed-category}, whenever $A$ is a cover of $Q$ whose associated cover $\Sigma_{A}$ by the partially ordered set of non-empty intersections satisfies Condition \eqref{eq:nerve_cover_locally_contractible}, and which is sufficiently fine with respect to the constant $C_{\cS}$ so that Condition \eqref{eq:diameter_open_star_cover} holds. Moreover, we have a map
\begin{equation}
  CF^*(L,L) \to \Hom_{\Fuk_{A}}(  \scrL_{L},  \scrL_{L}).
\end{equation}

Following the methods of \cite[Appendix A]{Abouzaid2014a}, we can extend this map to an $A_\infty$-homomorphism from the category $\Fuk_L$ with one object (with endomorphism $ CF^*(L,L)$) to the category of modules on $\Fuk_{A} $.  The key point is that this does not require any further constraint on the cover because the higher terms of the functor are defined using moduli spaces of strips with marked points along both boundary segments; we use the system of paths $\scrS$ along one, and boundary condition $L$ along the other. The reverse isoperimetric constant which we used for the construction of the linear part of the functor thus applies to the higher parts as well.

\begin{prop}
  If  $\scrS_1$ and $\scrS_2$ are systems of paths,  there is a commutative diagram
  \begin{equation} \label{eq:commutative_diagram_comparison}
    \begin{tikzcd}
      & \Fuk_L \ar[d, hookrightarrow] \ar[dr,hookrightarrow] \ar[dl,hookrightarrow] & \\
      \mod_{ \Fuk_{A_1}} & \mod_{ \Fuk_{A_{1} \amalg A_{2}}} \ar[l] \ar[r] &  \mod_{ \Fuk_{A_2}},
    \end{tikzcd}
  \end{equation}
whenever $A_1$ and $A_2$ are sufficiently fine covers, in which the arrows from $\Fuk_L$ to the categories of modules over $\Fuk_{A_1}$ and $\Fuk_{A_2} $ are respectively defined using $\scrS_1$ and $\scrS_2$.
\end{prop}
\begin{proof}[Sketch of proof:]
  Consider a system of paths $\scrS_{12}$, defined for sequences
\begin{align}
  & (L,q^1_0, \ldots, q^1_{\ell}, q^{12}_0, \ldots, q^{12}_k)    \\
  & (L,q^2_0, \ldots, q^2_{\ell}, q^{12}_0, \ldots, q^{12}_k)
\end{align}
with all points $q_i^1$ and $q_j^{12}$ (or $q_i^2$ and $q_j^{12}$) lying within a ball of radius $1$, and where $\ell$  is bounded by $\dim Q $ and $k$ by $2\dim Q$. We require that these satisfy again the conditions listed in Section \ref{sec:isop-const}, and that they extend the choices of paths $\scrS_{1}$ and  $\scrS_2$ in the sense that they agree with them whenever there are no points labelled $q_j^{12}$. According to Lemma \ref{lem:universal_constant_isoperimetric}, there is a reverse isoperimetric constant $C_{\scrS_{12}}$ for the moduli spaces of strips with one boundary condition along $L$, and the other along the paths prescribed by $\scrS_{12}$.

We now consider covers $A_1$ and $A_2$ with the property that any sequence of nested non-empty intersections of elements of these sets has length smaller than or equal to $\dim Q + 1$. This implies that any sequence of nested non-empty intersections of elements of $A_1 \cup A_2$ has length smaller than or equal to $2\dim Q + 2$. These non-empty intersections are the objects of the category $ \Fuk_{A_{1} \amalg A_{2}}$, which is equipped with natural embeddings of $ \Fuk_{A_{1}} $ and $ \Fuk_{A_{2}}$.

Using the system of paths $\scrS_{12}$, we may construct a functor
\begin{equation}
  \Fuk_L \to  \Fuk_{A_{1} \amalg A_{2}}.
\end{equation}
The key point is to use the points labelled $q_i^1$ and $q_i^2$ as basepoints of the objects of $\Fuk_{A_{1} \amalg A_{2}}$ which respectively come from  $ \Fuk_{A_{1}} $ and $ \Fuk_{A_{2}}$, and $q^{12}_j$  for the remaining objects. This choice ensures that Diagram \eqref{eq:commutative_diagram_comparison} commutes. The main results of this paper and of \cite{Abouzaid2014a} imply that the functors we have constructed from $\Fuk_{L}$ to these categories of modules are all fully faithful embeddings (upon imposing a further constraint on the size of the covers).
\end{proof}

In order to conclude from the above result that the family Floer functor is independent of choices, we need as well a comparison between the categories of modules over $\Fuk_{A_1 \amalg A_2}$  and 
$\Fuk_{A_i}$. While these categories are not equivalent, Lemma \ref{lem:left_module_perfect} and the fact that continuation maps induce equivalences of family Floer homology groups (c.f. \cite[Corollary 3.13]{Abouzaid2014}) together imply that the image of this functor lies in our model for the derived category of (twisted) coherent sheaves (c.f. \cite[Definition 2.7]{Abouzaid2014a}): this subcategory consists of those modules $\scrM$ for which (i) the value on any object $\sigma$ is a perfect module over $\Fuk(\sigma, \sigma)$, and (ii) the induced map
  \begin{equation} \label{eq:O-module}
        \scrM(\sigma) \otimes_{\Fuk(\sigma, \sigma) }\Fuk(\sigma, \tau) \to  \scrM(\tau)
      \end{equation}
is an equivalence.

\begin{lem}
  The functor $\mod_{\Fuk_{A_1 \amalg A_2}} \to  \mod_{\Fuk_{A_i}}$ restricts to a fully faithful embedding on the subcategory corresponding to complexes of (twisted) coherent sheaves. \end{lem}
\begin{proof}[Sketch of proof:]
 Let $\Fuk^{\perp}_{A_1}$ denote the subcategory of $ \Fuk_{A_1 \amalg A_2}$ consisting of objects which do not lie in $\Fuk_{A_1}$. By construction, there are no morphisms from objects of $\Fuk^{\perp}_{A_1}$ to those of $\Fuk_{A_1}$ (i.e. this is a semi-orthogonal decomposition).  The morphism space between modules $\scrM$ and $\scrN$ over $  \Fuk_{A_1 \amalg A_2}$ is thus naturally isomorphic to the cone of the map
  \begin{equation}
 \Hom_{\Fuk_{A_1}}(\scrM, \scrN) \oplus  \Hom_{\Fuk^{\perp}_{A_1}}(\scrM, \scrN) \to        \Hom_{\Fuk^{\perp}_{A_1}}(_{\Fuk^{\perp}_{A_1} }\Delta_{\Fuk_{A_1}} \otimes_{ \Fuk_{A_1}} \scrM, \scrN),
\end{equation}
where  $_{\Fuk^{\perp}_{A_1} }\Delta_{\Fuk_{A_1}} $ is the bimodule associated to the semi-orthogonal decomposition of $\Fuk_{A_1 \amalg A_2} $, and we omit any notation for the restriction of the modules to $\Fuk_{A_1}$ and $\Fuk^{\perp}_{A_1}$. This isomorphism identifies the kernel of the restriction map (on cohomology) with the cohomology of the cone of
\begin{equation}
 \Hom_{\Fuk^{\perp}_{A_1}}(\scrM, \scrN) \to        \Hom_{\Fuk^{\perp}_{A_1}}(_{\Fuk^{\perp}_{A_1} }\Delta_{\Fuk_{A_1}} \otimes_{ \Fuk_{A_1}} \scrM, \scrN).
\end{equation}
The finiteness assumptions on $\scrM$ and $\scrN$, imply that this complex is quasi-isomorphic to
\begin{equation}  \label{eq:morphisms_as_hom_from_cone}
 \Hom_{\Fuk^{\perp}_{A_1}}(\Cone(_{\Fuk^{\perp}_{A_1} }\Delta_{\Fuk_{A_1}} \otimes_{ \Fuk_{A_1}} \scrM \to \scrM), \scrN).
\end{equation}
Equation \eqref{eq:O-module} implies that the cone of the map $ _{\Fuk^{\perp}_{A_1} }\Delta_{\Fuk_{A_1}} \otimes_{ \Fuk_{A_1}} \scrM \to \scrM$ vanishes on every object of $\Fuk^{\perp}_{A_1}$ except those which are objects of $ \Fuk_{A_1}$. We shall use this to show that the complex in Equation \eqref{eq:morphisms_as_hom_from_cone} is acyclic.

The filtration of the category $\Fuk^{\perp}_{A_1}$ by the number of elements $A_2$ appearing in the nested intersection associated to each object reduces the result to the following situation: $P$ is an integral affine polytope (i.e. a polytope associated to an object of $\Fuk_{A_2} $), $A_P$ is a cover of $P$ by integral affine polytopes (i.e. the cover obtained from intersecting the elements of $A_1$ with $P$), and $\scrC$ is a module over  the category obtained from $\Fuk_{A_P}$ by adding one object $\sigma_P$ corresponding to $P$ (with endomorphism algebra $\Gamma^{P}$), so that $\scrC$  vanishes on every object other than $\sigma_P$, and whose value on $\sigma_P$ is a perfect complex. Given another module $\scrN$ whose value on $\sigma_P$ is perfect, and such that Equation \eqref{eq:O-module} holds, our goal is then to show that the space of morphisms from $\scrC$ to $\scrN$ is acyclic. Given the vanishing assumption on $\scrC$ and Equation \eqref{eq:O-module}, this complex is quasi-isomorphic to the cone of the map
\begin{multline}
    \Hom_{\Gamma^{P}}(\scrC(\sigma_P), \scrN(\sigma_P)) \to \\  \Hom_{\Fuk_{A_P}}( _{\Fuk_{A_P} }\Delta_{\Fuk(\sigma_P,\sigma_P)}\otimes_{\Gamma^{P} } \scrC(\sigma_P), _{\Fuk_{A_P} }\Delta_{\Gamma^{P}} \otimes_{\Gamma^{P} }\scrN(\sigma_P)).
\end{multline}
The result now follows from Tate acyclicity using the perfectness of $\scrC(\sigma_P)$ and $\scrN(\sigma_P)$.
\end{proof}

\def\cprime{$'$}
\providecommand{\bysame}{\leavevmode\hbox to3em{\hrulefill}\thinspace}
\providecommand{\href}[2]{#2}

\end{document}